\documentclass[a4paper,10pt,reqno]{amsart}
\usepackage{amsmath}%
\usepackage{amsfonts}%
\usepackage{amscd}
\usepackage{amssymb}%
\usepackage{graphicx}
\usepackage{enumerate}
\usepackage[all]{xy}
\usepackage{hyperref}
\usepackage{mathtools}
\usepackage{ stmaryrd }
\usepackage{xcolor}
\usepackage[pageref]{backref}
\usepackage{dsfont}

\usepackage{xfrac}          
\usepackage{nicefrac}       
\usepackage[normalem]{ulem}       
\usepackage{stackrel}

\newtheorem{theorem}{Theorem}[section]
\theoremstyle{plain}
\newtheorem{proposition}[theorem]{Proposition}
\newtheorem{lemma}[theorem]{Lemma}
\newtheorem{corollary}[theorem]{Corollary}

\theoremstyle{definition}

\newtheorem{example}[theorem]{Example}

\theoremstyle{remark}

\newtheorem{remark}[theorem]{Remark}

\numberwithin{equation}{section}

\DeclareMathOperator{\im}{im}
\DeclareMathOperator{\h}{h}

\DeclareMathOperator{\supp}{supp}
\DeclareMathOperator{\ann}{ann}

\DeclareMathOperator{\Hom}{Hom}

\DeclareMathOperator{\modR}{mod\textrm{-}\!}
\DeclareMathOperator{\projR}{proj\textrm{-}\!}
\DeclareMathOperator{\fgmodR}{f.g.\!-mod\textrm{-}\!}

\DeclareMathOperator{\Rmod}{\!\textrm{-}mod}

\DeclareMathOperator{\Fun}{Fun}
\DeclareMathOperator{\Fungr}{Fun_{gr}}
\DeclareMathOperator{\Bifun}{Bifun}

\DeclareMathOperator{\Homgr}{Hom_{gr-\mathit{R}}}
\DeclareMathOperator{\Endgr}{End_{gr-\mathit{R}}}
\DeclareMathOperator{\End}{End}
\DeclareMathOperator{\HOM}{HOM}
\DeclareMathOperator{\END}{END}
\DeclareMathOperator{\Hgr}{\mathbb{H}}
\DeclareMathOperator{\cfm}{CFM}
\DeclareMathOperator{\M}{M}
\DeclareMathOperator{\UT}{UT}
\DeclareMathOperator{\pdim}{pdim}
\DeclareMathOperator{\sdim}{sdim}
\DeclareMathOperator{\spdim}{spdim}

\DeclareMathOperator{\grR}{gr\textrm{-}\!}
\DeclareMathOperator{\Sgr}{\!\textrm{-}gr}
\DeclareMathOperator{\RgrR}{\!\textrm{-}gr\textrm{-}\!}

\DeclareMathOperator{\nGamma}{\varepsilon\Gamma}
\DeclareMathOperator{\Gamman}{\Gamma\varepsilon}

\begin{document}
	
	\title[Groupoid Graded Semisimple Rings]{Groupoid Graded Semisimple Rings}
	\author[Z. Cristiano]{Zaqueu Cristiano}
	\address[Zaqueu Cristiano]{Department of Mathematics - IME, University of S\~ao Paulo,
		Rua do Mat\~ao 1010, S\~ao Paulo, SP, 05508-090, Brazil}
	\email[Zaqueu Cristiano]{zaqueucristiano@usp.br}%
	
	\author[W. Marques de Souza]{Wellington Marques de Souza}
	\address[Wellington Marques de Souza]{Department of Mathematics - IME, University of S\~ao Paulo,
		Rua do Mat\~ao 1010, S\~ao Paulo, SP, 05508-090, Brazil}
	\email[Wellington Marques de Souza]{tomarquesouza@usp.br}%
	
	\author[J. S\'anchez]{Javier S\'anchez}
	\address[Javier S\'anchez]{Department of Mathematics - IME, University of S\~ao Paulo,
		Rua do Mat\~ao 1010, S\~ao Paulo, SP, 05508-090, Brazil}
	\email[Javier S\'anchez]{jsanchez@ime.usp.br}%
	\thanks{The first author was supported by grant \texttt{\#}2021/14132-2, S\~ao Paulo Research Foundation (FAPESP), Brazil}
	\thanks{The second author was supported by the Coordenação de Aperfeiçoamento de Pessoal de Nível Superior – Brasil (CAPES) – Finance Code 001}
	\thanks{The third author was supported by grant \texttt{\#}2020/16594-0, S\~ao Paulo Research Foundation (FAPESP), Brazil}
	
	\date{July 11, 2025}

	\begin{abstract}
    We develop the theory of groupoid graded semisimple rings. Our rings are neither unital nor one-sided artinian. Instead, they exhibit a strong version of having local units and being locally artinian, and we call them $\Gamma_0$-artinian. 
    One of our main results  is a groupoid graded version of the Wedderburn-Artin Theorem, where we characterize groupoid graded semisimple rings as direct sums of graded simple $\Gamma_0$-artinian rings and we exhibit the structure of this latter class of rings. 
    In this direction, we also prove a groupoid graded version of Jacobson-Chevalley density theorem. 
    We need to define and study properties of groupoid gradings on matrix rings (possibly of infinite size) over groupoid graded rings,  and specially over groupoid graded division rings. Because of that, we study groupoid graded division rings and their graded modules.  We consider a natural notion of freeness for groupoid graded modules that, when specialized to group graded rings, gives the usual one, and show that for a groupoid graded division ring all graded modules are free (in this sense). Contrary to the group graded case, there are groupoid graded rings for which all graded modules are free according to our definition, but they are not  graded division rings. We exhibit an easy example of this kind of rings and characterize such class among groupoid graded semisimple rings. 
    We also relate groupoid graded semisimple rings with the notion of semisimple category defined by B. Mitchell. For that, we show the link between functors from a preadditive category to abelian groups and graded modules over the groupoid graded ring associated to this category, generalizing a result of P. Gabriel. We characterize simple artinian categories and categories for which every functor from them to abelian groups is free in the sense of B. Mitchell.
	\end{abstract}

	\maketitle

 \tableofcontents

\section{Introduction}
Semisimple rings and the Wedderburn-Artin Theorem lie at the basis of classical Ring Theory. 
Later, the study of group graded structures became important for the development of different branches of Mathematics \cite{livroNastasescu}, \cite{Hazrat}, \cite{Elduque}. Many classical concepts and results in Ring and Module Theory have a group graded version, for example the Wedderburn-Artin Theorem \cite{Liu_Beattie_Fang}, \cite[Theorem 2.10.10]{livroNastasescu} or \cite[Theorem~2.12]{artigoNastasescu}.
Nowadays, groupoids \cite{Higgins} play an important role in Mathematics \cite{Renault}, \cite{Mackenzie}, \cite{Exel}, and groupoid graded structures have already appeared in different contexts such as partial actions \cite{BFP}, \cite{BP}, \cite{MoreiraOinert} and Leavitt path algebras \cite{Goncalves_Yoneda}.

A groupoid, together with a zero, can be regarded as a semigroup. Thus, some results about group graded rings might be shrouded in the theory of semigroup graded rings. To knowledge of the authors, the first work that began the study of groupoid graded rings and modules as a separate class was \cite{Lund}. Usually, first results in the theory 
were obtained supposing that either the graded ring has an identity element and/or that the groupoid has a finite number of objects \cite{Liu_Li}, \cite{Lund2}, \cite{Verhulst}. A systematic study of groupoid graded rings and modules that do not satisfy these conditions began in \cite{CLP} where the first results about groupoid graded semisimple modules were also given. In that work,
graded rings are supposed to have a good set of local units. They refer to such rings  as object unital rings, a term that we adopt in this work (see Subsection~\ref{sec:grupoid graded rings}). This concept was already used in \cite{Oinert_Lundstrom} under the name of locally unital ring. 

The general purpose of this work is  twofold: to continue the systematic study of object unital groupoid graded rings and to offer an application of such objects to the study of (small) preadditive categories. A more precise description of our aim is given in the following paragraphs.

Naturally, in the groupoid graded context, rings do not necessarily have a unity.  In fact, imposing the existence of an identity element implies that there is only a finite number of idempotents of the groupoid inside the support of the graded ring \cite[Proposition~2.1.1]{Lund}. The possibility of an infinite number of idempotents in a groupoid does not allow some classical definitions to be generalized as directly  as they do in the group graded case.  For instance, many interesting examples of groupoid graded rings are decomposed as an infinite direct sum of nonzero graded right modules. The usual definition of artinianity does not apply in this case, but another descending chain condition is useful. It is the natural generalization of the notion of a categorically artinian ring, introduced in  \cite[Definition~1.1]{Abrams_ArandaPino_Perera_SilesMolina} and \cite[Section~4.2]{Abrams_Ara_SilesMolina}, to the setting of groupoid graded rings. We will talk about $\Gamma_0$-artinian rings, see Subsection~\ref{sec:general results about gr-semisimple}. In other cases, the natural generalization of certain concepts lead to unexpected properties. For example, groupoid graded division rings  are not necessarily graded prime rings. Other definitions  are not useful in the groupoid graded context. Also, if a module is graded by a groupoid but not by a group, then no element  is linearly independent.  For this reason, we define pseudo-free modules instead of free modules, see Subsection~\ref{sec:pseudo-free modules} for more details.

Another difficulty that appears when dealing with  rings $R$ graded by a groupoid $\Gamma$ is how to 
induce $\Gamma$-gradings on matrix rings with entries in $R$. In the group graded context, a finite sequence of $n$ elements of the group is enough to produce a group grading on the matrix ring (of finite size $n\times n$). In our case, we will require  sequences of suitable subsets of the groupoid in order to obtain groupoid gradings on matrix rings  with entries in a groupoid graded ring. On the other hand, these graded matrix rings have similar properties to those that appear in the group graded context.

One of our main results is a groupoid graded version of the Wedderburn-Artin theorem. We prove that groupoid graded semisimple rings are precisely certain (may be infinite) direct products  of gr-simple $\Gamma_0$-artinian rings, see Theorem~\ref{teo: wa para ss -resumo}. These rings are the groupoid graded generalization of simple artinian rings and we characterize them as graded matrix rings (perhaps of infinite size) over graded prime division rings, see Theorem~\ref{teo: WA para simp -resumo}. 

It is a well-known result in Ring Theory that every unital right (left) module over a division ring has a basis (see, for example, \cite[Theorem 2.4]{Hunger}). One can use the Wedderburn-Artin theorem to prove the converse. In fact, suppose that $R$ is a ring with unity such that every unital right (left) $R$-module is free. This implies that every unital right (left) $R$-module is projective. By \cite[Theorem 2.8]{Lam1}, $R$ is a semisimple ring and it follows that $R$ must be a finite product of matrix rings over division rings. But all modules over such ring are free if and only if the ring is a division ring.  Analogous arguments show the correspondent characterization of group graded division rings (cf. \cite[Theorem 3.3]{BalabaMik} for one implication). We were surprised to build a non-sophisticated example of a groupoid graded ring (a $3\times 3$ matrix ring with a suitable grading) over which all graded right modules are pseudo-free, but that is not a graded division ring. With this example in mind and using our graded Wedderburn-Artin Theorem,  we are able to characterize graded rings  whose modules are all pseudo-free.

One of our main examples of groupoid graded ring is the ring obtained from a small preadditive category. In small preadditive categories, one can define ideals, simplicity \cite{Facchini}, semisimplicity and artinianity \cite{Mit}. This suggests that concepts of ring theory can be defined in category theory and that graded results about the ring of the category  can be used to get results about the category. We show that semisimplicity, artinianity and simplicity of the category are equivalent to the graded semisimplicity, graded artinianity and graded simplicity, respectively, of the ring of a small preadditive category.

\medskip

In Section~\ref{sec:Preliminaries}, we present the basic definitions, examples and results about groupoids and groupoid graded rings and  modules that will be needed in the paper. In addition, we fix the conventions that will be adopted throughout the text. 

Section~\ref{sec:rings_of_matrices} is dedicated to study groupoid gradings on matrix rings, homomorphisms between groupoid graded modules and the identification of graded matrix rings with graded endomorphism rings of certain graded modules. We begin this section by presenting pseudo-free modules. Next, we define and prove some results about gr-homomorphisms and homomorphisms with degree. Then, we   describe a key method to obtain groupoid gradings in matrix rings over graded rings. One of the main results of this section is the expression of homomorphisms with degree as matrices. Later, in order to obtain results about left modules from those about right modules, we consider the opposite ring of a graded ring as a graded ring. These techniques are useful to get isomorphisms between categories of graded left modules over a graded ring and graded right modules over the opposite ring. We also show other isomorphisms between categories of graded modules that are obtained when the groupoid is connected. We finish this third section with groupoid gradings in additive groups of rectangular matrices. 

In Section~\ref{sec: gr-div rings}, we deal with properties of groupoid graded division rings. The first main result is the decomposition of a graded division ring as a sum of graded prime (simple) division rings. Another important result is the fact that all graded modules over a graded division ring are pseudo-free. We finish proving some results about graded division rings that are similar to the ones in basic linear algebra. 

We study groupoid graded semisimple rings in Section~\ref{sec: art simp}. We begin with basic results about graded simple modules, such as graded Schur's Lemma, and about graded semisimple modules and rings. Then we find the structure of graded prime semisimple rings as certain graded matrix rings over graded division rings and prove that graded semisimple rings are direct sums of these. We end the section studying the uniqueness of this matrix representation.

In Section~\ref{sec: teo da den}, we get another proof of the structure of gr-simple $\Gamma_0$-artinian rings, proving a groupoid graded version of the Jacobson-Chevalley density theorem.

Section~\ref{sec:Pseudo_divsision_rings} is devoted to study  pseudo-free module (pfm) rings, that is, those graded rings over which all graded modules are pseudo free. The main result of this section is the characterization of such rings. This theorem is obtained from the one that characterizes graded prime (simple) pfm rings. We also find an invariant for pfm rings: the gr-simple dimension. In the final part of the section, we study some relations between pfm rings and graded division rings. 

Example~\ref{ex: aneis graduados2} tells us how to obtain a groupoid graded ring from a small preadditive category. This example and \cite{Mit, Gab} motivate us to apply the concepts discussed in previous sections to the categorical context. This is the main objective of Section~\ref{sec: semisimple categories}.  We begin by showing that additive contravariant functors from a small preaddtive category to abelian groups can be regarded as graded modules over the ring of the category. After that, we obtain new characterizations of semisimple categories. We also define and characterize simple artinian categories and division categories. We show that the concept of free functors given in \cite{Mit} is linked with our concept of pseudo-free modules and we use this to get a characterization of categories for which all functors to abelian groups are free.

\section{Preliminaries}\label{sec:Preliminaries}

In this section we collect some known definitions and results that will be used throughout the paper, sometimes without reference.

\subsection{Groupoids}
A \emph{groupoid} $\Gamma$ is a small category in which every morphism is invertible. In other words, every morphism of $\Gamma$ is an isomorphism. We will denote the set of \emph{objects} of $\Gamma$ by $\Gamma_0$. The identity morphism of $e\in\Gamma_0$ will be denoted again by $e$. In this way, we identify the groupoid with its set of morphisms and sometimes we will refer to $\Gamma_0$ as the set of idempotents of $\Gamma$. If $e,f\in\Gamma_0$ and $\gamma$ is a morphism from $e$ to $f$, we will write $d(\gamma)=e$ and $r(\gamma)=f$. Notice that $d(\gamma^{-1})=r(\gamma)=f$ and $r(\gamma^{-1})=d(\gamma)=e$.
If, moreover, $d(\delta)=f$, then $\delta\gamma\in\Gamma$ with $d(\delta\gamma)=d(\gamma)=e$ and $r(\delta\gamma)=r(\delta)=f$. If $d(\delta)\neq f$, we will say that $\delta\gamma$ \emph{is not defined} in $\Gamma$.  For a detailed study of groupoids the interested reader is referred to \cite{Higgins} or \cite{IR}. 

Groupoids with a unique object
can be identified with groups. Some other important examples of groupoids that we will deal with are the following.

\begin{example}\label{ex:groupoids}
    \begin{enumerate}[(1)]
      \item Let $\{G_i:i\in I\}$ be a family of groups where $e_{G_i}$ is the identity element of $G_i$ for each $i\in I$. Let $\Gamma$ be the disjoint union of $\{G_i:i\in I\}$. Then $\Gamma$ is a groupoid with $\Gamma_0=\{e_{G_i}:i\in I\}$  if we define the inverse of each element in the natural way and the composition only for elements in the same group. Thus $r(g_i)=d(g_i)=e_{G_i}$ for all $i\in I$ and $g_i\in G_i$. 
    
        \item Let $X$ be a non-empty set. We endow $\Gamma:=X\times X$ with a structure of groupoid in the following way. For 
        $(y,x)\in \Gamma$, define $d(y,x)=(x,x)$, $r(y,x)=(y,y)$ and $(y,x)^{-1}=(x,y)$. Thus $\Gamma_0=\{(x,x)\colon x\in X\}$ and it can be identified with $X$ in the natural way. 
        Note that there exists a unique morphism between two objects in $\Gamma_0$. For each
        $(y,x),(z,w)\in \Gamma$, the composition  $(z,w)(y,x)$ is defined if and only if $w=y$. In that case, 
        $(z,y)(y,x)=(z,x)$.

        \item Generalizing the previous example, let $X$ be a non-empty set and $G$ be a group with identity element $e$. We endow 
        $\Gamma:=X\times G\times X$ with a structure of groupoid in the following way. For each
        $(y,g,x)\in \Gamma$, 
        define $d(y,g,x)=(x,e,x)$, $r(y,g,x)=(y,e,y)$ and
        $(y,g,x)^{-1}=(x,g^{-1},y)$. Thus $\Gamma_0=\{(x,e,x)\colon x\in X\}$ and it can be identified with $X$. Note that the morphisms between two objects in $\Gamma_0$ can be identified with the
        elements of $G$. For each $(z,h,w),(y,g,x)\in \Gamma$, the composition  $(z,h,w)(y,g,x)$ is defined if and only if $w=y$. In that case, 
        $(z,h,y)(y,g,x)=(z,hg,x)$. Notice  that we obtain the foregoing example when  $G$ is the trivial group.

        This example is important because  any groupoid is the disjoint union of subgroupoids isomorphic (in a non-canonical way) to the groupoid presented in this example, see for example \cite[p.~125]{Brown}.\qed  
    \end{enumerate}   
\end{example}


\medskip

Now we introduce some more notation on groupoids.
Let $\Gamma$ be a groupoid.
For $X,Y\subseteq \Gamma$ and $\gamma,\delta\in\Gamma$, we define the following subsets of $\Gamma$:
$$X^{-1}:=\{\alpha^{-1}\colon \alpha\in X\},$$
$$\gamma X:=\{\gamma\alpha\colon \alpha\in X \textrm{ with } d(\gamma)=r(\alpha)\},$$ 
$$X\delta:=\{\alpha\delta\colon \alpha\in X \textrm{ with } d(\alpha)=r(\delta)\},$$
$$XY:=\{\alpha\beta\mid \alpha\in X, \beta\in Y \textrm{ and } d(\alpha)=r(\beta)\},$$
$$\gamma X \delta:=\{ \gamma \alpha \delta\colon  \alpha \in X \textrm{ with }
d(\gamma)=r(\alpha) \textrm{ and } d(\alpha)=r(\delta)\}.$$

We say that the groupoid $\Gamma$ is \emph{connected} if $e\Gamma f\neq\emptyset$ for all $e,f\in\Gamma_0$. In other words, given $e,f\in\Gamma_0$, there exists $\sigma\in\Gamma$ such that $r(\sigma)=e$ and $d(\sigma)=f$.

 Note that if $e\in\Gamma_0$, then  $e\Gamma e=\{\gamma\in\Gamma\colon d(\gamma)=r(\gamma)=e\}$ is a group with identity element $e$. While $\Gamma(e)$ is the standard notation for this group in the literature, we choose the more suggestive notation $e\Gamma e$ to better serve our purposes and to avoid potential confusion with other concepts.

\subsection{Groupoid graded additive groups}\label{subsec:Groupoid graded addtivie groups}
Let $\Gamma$ be a groupoid. An additive group $G$ is \emph{$\Gamma$-graded} if there exists a family $\{G_\gamma:\gamma\in\Gamma\}$ of subgroups of $G$ such that $G=\bigoplus_{\gamma\in\Gamma}G_\gamma$. 
Graded rings and graded modules (which we define below) are examples of graded additive groups. 

We continue fixing some general notation. Let $X=\bigoplus_{\gamma\in\Gamma}X_\gamma$ be a $\Gamma$-graded additive group. For each $\gamma\in\Gamma$, $X_\gamma$ is called the \emph{homogeneous component of degree $\gamma$ of $X$}, its nonzero elements are said to be \emph{homogeneous of degree $\gamma$} and we write $\deg(x)=\gamma\in\Gamma$ when $0\neq x\in X_\gamma$. For each $\gamma\in\Gamma$ and $x\in X$, we write $x=\sum_{\gamma\in\Gamma}x_\gamma$ with $x_\gamma\in X_\gamma$ (with all but finitely many $x_\gamma$ nonzero) and we call $x_\gamma$ the \emph{homogeneous component of degree $\gamma$ of $x$}. The set of the \emph{homogeneous elements} of $X$ is $\h(X):=\bigcup_{\gamma\in\Gamma}X_\gamma$. We define the \emph{support} of $X$ as $\supp(X):=\{\gamma\in\Gamma:X_\gamma\neq0\}$ and the \emph{support} of $x\in X$ as $\supp(x):=\{\gamma\in\Gamma:x_\gamma\neq0\}$. When $\sigma,\tau\in\Gamma$ are such that $d(\sigma)\neq r(\tau)$, that is, $\sigma\tau$ is not defined, we adopt the convention $X_{\sigma\tau}:=\{0\}$. If $Y$ is a $\Gamma$-graded additive group, then a homomorphism of groups $g:X\to Y$ it is said to be a \emph{gr-homomorphism of groups} if $g(X_\sigma)\subseteq Y_\sigma$ for all $\sigma\in\Gamma$. If, moreover, $g$ is bijective, then we say that $g$ is a \emph{gr-isomorphism of groups}, that $X$ is \emph{gr-isomorphic} to $Y$ as groups and we will denote it by $X\cong_{gr}Y$.

\medskip

\subsection{Groupoid graded rings}\label{sec:grupoid graded rings}
\emph{Throughout this work, rings are assumed to be associative but not necessarily unital.}

\medskip 

Let $R$ be a ring. We say that $R$ is a \emph{$\Gamma$-graded ring} if there is a family $\{R_\gamma\}_{\gamma\in\Gamma}$ of additive subgroups of $R$ such that 
$R=\bigoplus\limits_{\gamma\in\Gamma}R_\gamma$ and
$R_\gamma R_\delta\subseteq  R_{\gamma\delta}$, for each $\gamma,\delta\in\Gamma$. Following \cite{CLP2}, we say that $R$ is \emph{object unital}
if, for all $e\in\Gamma_0$,  the ring $R_e$ is unital  with identity element $1_e$, and for all $\gamma\in \Gamma$ and $r\in R_{\gamma}$, we have
$1_{r(\gamma)}r=r1_{d(\gamma)}=r$. This concept has received other names in the semigroup graded context, see \cite{AMdR}. 
It follows from \cite[Proposition 2.1.1]{Lund} that the object unital $\Gamma$-graded ring $R=\bigoplus_{\gamma\in\Gamma}R_\gamma$ is unital if and only if $\Gamma_0'=\{e\in\Gamma_0 \colon 1_e\neq 0\}$ is finite. In this event, $1_R=\sum_{e\in\Gamma_0'}1_e$.

Now we present some examples of (object unital) groupoid graded rings.

\begin{example}
\label{ex: aneis graduados}
    \begin{enumerate}
        \item If $R$ is a group graded (unital) ring, then $R$ is a (object unital) groupoid graded ring. More generally, if $I$ is a non-empty set and, for each $i\in I$, $R_i$ is a (unital) ring graded by the group $G_i$, then $R:=\bigoplus_{i\in I}R_i$ is (object unital) graded by the groupoid obtained as the disjoint union of the groups $G_i$, $i\in I$, via $R_{g_i}:=(R_i)_{g_i}$ if $g_i\in G_i$. 

        \item Suppose that a ring $R$ has \emph{enough idempotents}, i.e., there exists a set $\{e_i:i\in I\}$ of pairwise orthogonal idempotents of $R$ such that $R=\bigoplus_{i\in I}Re_i=\bigoplus_{i\in I}e_iR$ \cite{Fuller}. Then $R$ is an object unital $I\times I$-graded ring via $R_{(i,j)}:=e_iRe_j$ for all $i,j\in I$. Note that $e_i$ is the unity of the ring $R_{(i,i)}$ for each $i\in I$.

        \item An \emph{object crossed system} $(A,\Gamma,\alpha,\beta)$ consists of a family $A=\{A_e:e\in\Gamma_0\}$ of nonzero unital rings, a family $\alpha=\{\alpha_\sigma:A_{d(\sigma)}\to A_{r(\sigma)}\}$ of isomorphisms of rings (respecting identity elements) and a family $\beta=\{\beta_{\sigma,\tau}\in U(A_{r(\sigma)}):\sigma,\tau\in\Gamma, d(\sigma)=r(\tau)\}$ of invertible elements, satisfying
        \begin{enumerate}[(i)]
            \item $\alpha_e=id_{A_e}$ for all $e\in\Gamma_0$.
            \item $\beta_{\sigma,d(\sigma)}=\beta_{r(\sigma),\sigma}= 1_{A_{r(\sigma)}}$ for all $\sigma\in\Gamma$.
            \item $\alpha_\sigma(\alpha_\tau(a))=\beta_{\sigma,\tau}\alpha_{\sigma\tau}(a)\beta_{\sigma,\tau}^{-1}$ for all $\sigma,\tau\in\Gamma$ with $d(\sigma)=r(\tau)$ and $a\in A_{d(\tau)}$.
            \item $\beta_{\sigma,\tau}\beta_{\sigma\tau,\rho}=\alpha_\sigma(\beta_{\tau,\rho})\beta_{\sigma,\tau\rho}$ for all $\sigma,\tau,\rho\in\Gamma$ with $d(\sigma)=r(\tau)$ and $d(\tau)=r(\rho)$.
        \end{enumerate}
        Given an object crossed system $(A,\Gamma,\alpha,\beta)$ and a copy $\{u_\sigma:\sigma\in\Gamma\}$ of $\Gamma$, consider the set of formal sums
        \[A\rtimes^\alpha_\beta\Gamma:=\left\{\sum_{\sigma\in\Gamma}a_\sigma u_\sigma\colon(a_\sigma)_{\sigma\in\Gamma}\in\bigoplus_{\sigma\in\Gamma}A_{r(\sigma)}\right\}.\]
        The sum in $A\rtimes^\alpha_\beta\Gamma$ is defined by
        \[\sum_{\sigma\in\Gamma}a_\sigma u_\sigma+\sum_{\sigma\in\Gamma}a'_\sigma u_\sigma:=\sum_{\sigma\in\Gamma}(a_\sigma+a'_\sigma) u_\sigma\]
        and the product is defined to be the natural extension of
        \[(a_\sigma u_\sigma)\cdot(b_\tau u_\tau):=\begin{cases}
            a_\sigma\alpha_\sigma(b_\tau)\beta_{\sigma,\tau}u_{\sigma\tau},& \textrm{~if $d(\sigma)=r(\tau)$}\\
            0,& \textrm{~if $d(\sigma)\neq r(\tau)$}.
        \end{cases}\]
        By \cite[Proposition 16]{CLP2}, $A\rtimes^\alpha_\beta\Gamma$ is an (object unital) groupoid graded ring via $(A\rtimes^\alpha_\beta\Gamma)_\sigma:=A_{r(\sigma)}u_\sigma$. 
        
        \item  Given a partial action $\alpha=(\{D_\gamma\}_\gamma\in\Gamma,\{\alpha_\gamma\}_{\gamma\in\Gamma})$ of a groupoid $\Gamma$ on a ring $R$ as in \cite[Section~3]{BP}, the skew groupoid ring $R\ast_\alpha \Gamma$ is defined. It is $\Gamma$-graded, associative and object unital.  See also \cite[Section~3]{BFP} for mild conditions under which $R\ast_\alpha \Gamma$ is associative and object unital.
        
        \item In \cite{Goncalves_Yoneda},  a free groupoid grading on  Leavitt path algebras is presented.  \qed
        \end{enumerate}
\end{example}
        
Now we provide what we believe is the most important class of examples of groupoid graded rings. As far as we know, this family of examples has not been studied as graded rings until now. The first example will be very important in Section~\ref{sec: semisimple categories}.        
        
        \begin{example} \label{ex: aneis graduados2}
        \begin{enumerate}
        \item Let $\mathcal{C}$ be a small preadditive category, i.e., $Obj(\mathcal{C})$ is a set and every morphism set has an additive group structure such that the composition of morphisms is $\mathbb{Z}$-bilinear.  
        We define the \emph{ring of the category $\mathcal{C}$} as \footnote{Our definition of preadditive category is called additive category in \cite[p. 9]{Mit}. Our definition of ring of a category is the \emph{Gabriel functor ring} of \cite[p. 123]{Facchini} and is denoted by $\mathbf{Z}[\mathcal{C}]$ in \cite[p. 346]{Gab}. We point out that our ring of a category is not the same as the one in \cite[Section 7]{Mit}.}        
        \[R[\mathcal{C}]:=\bigoplus_{A,B\in Obj(\mathcal{C})}\Hom_{\mathcal{C}}(A,B)\] 
        where, given morphisms $f$ and $g$, the product $fg$ is defined by $f\circ g$ if the composition is possible and $0$ otherwise. The ring (possibly without unity) $R[\mathcal{C}]$ has a natural structure of ring graded by the groupoid $Obj(\mathcal{C})\times Obj(\mathcal{C})$ via 
        \[R[\mathcal{C}]_{(A,B)}:=\Hom_{\mathcal{C}}(B,A),\]
        for each $A,B\in Obj(\mathcal{C})$. Considering the identity morphisms $I_A\in \Hom_\mathcal{C}(A,A)$, $A\in Obj(\mathcal{C})$, it is easy to show that $R[\mathcal{C}]$  is object unital.

        \item Let $\mathcal{C}$ be a small preadditive category and $G$ be a group. In the literature,  the category $\mathcal{C}$ is \emph{$G$-graded} if every morphism set is given the structure of a $G$-graded additive group and  for $A,B,C\in Obj(\mathcal{C})$ the composition is $\mathbb{Z}$-bilinear and induces a homomorphism 
        \[\Hom_\mathcal{C}(B,C)\otimes\Hom_\mathcal{C}(A,B)\rightarrow \Hom_\mathcal{C}(A,C)\]
        of $G$-graded abelian groups (i.e., preserving degrees). Then $R[\mathcal{C}]$ is also graded by the groupoid $Obj(\mathcal{C})\times G\times Obj(\mathcal{C})$ via
        \[R[\mathcal{C}]_{(A,g,B)}:=\Hom_{\mathcal{C}}(B,A)_g,\]
        for each $A,B\in Obj(\mathcal{C})$ and $g\in G$. Again, considering the identity morphisms $I_A\in \Hom_\mathcal{C}(A,A)$, $A\in Obj(\mathcal{C})$, it is easy to show that $R[\mathcal{C}]$  is object unital.

        Conversely, if $X$ is a non-empty set and $G$ is a group, then every $X\times G\times X$-graded ring $R$ is the ring of some $G$-graded small preadditive category. In fact, it suffices to consider the category $\mathcal{C}$ whose set of objects is $X$ and, given $x,y\in X$, we set $\Hom_\mathcal{C}(x,y):=\bigoplus_{g\in G}R_{(y,g,x)}$. In particular, the $X\times X$-graded rings are exactly the rings of small preadditive categories with set of objects $X$. Thus, to study the rings graded by groupoids of the form $X\times X$ and $X\times G \times X$ ($X$ a set, $G$ a group) means to study rings of categories.
         \qed
    \end{enumerate}        
\end{example}


\emph{In this work, unless otherwise stated (Section~\ref{sec: teo da den}), all groupoid graded rings are supposed to be object unital.}

 \bigskip
Given $\Gamma$-graded rings $R$ and $S$, a \emph{gr-homomorphism of rings} is a homomorphism of rings $\varphi:R\to S$ which is also a gr-homomorphism of $\Gamma$-graded abelian groups and satisfies $\varphi(1^R_e)=1^S_e$ for all $e\in\Gamma_0$, where $1^R_e$ (resp. $1^S_e$) denotes the unity of the ring $R_e$ (resp. $S_e$). A \emph{gr-isomorphism} of $\Gamma$-graded rings is a bijective gr-homomorphism of $\Gamma$-graded rings. When there exists a gr-isomorphism of $\Gamma$-graded rings $\varphi: R\to S$, we say that $R$ is \emph{gr-isomorphic} to $S$ as rings and we write $R\cong_{gr}S$.

Let $R$ be a $\Gamma$-graded ring and $S\subseteq R$. We will say that $S$ is a \emph{graded subring} of $R$ if $S$ is closed under sums and products, $S=\bigoplus_{\gamma\in\Gamma}(S\cap R_\gamma)$ and, for all $e\in\Gamma_0$, if $S_e\neq0$, then $1_e\in S_e$.

A way to produce a new $\Gamma$-graded ring from other $\Gamma$-graded rings is  the \emph{graded direct product} that we proceed to define. Let $\{R_j:j\in J\}$ be a family of $\Gamma$-graded rings. We denote by 
\[\sideset{}{^{gr}}\prod_{j\in J}R_j\] 
the $\Gamma$-graded ring whose homogeneous component of degree $\gamma\in\Gamma$ is the additive group
\[\prod_{j\in J}(R_j)_\gamma.\]
It is easy to see that this defines a direct product in the category of  $\Gamma$-graded rings and it coincides with the usual direct product of rings if $J$ is finite.

\subsection{Groupoid graded modules}
In this section, most of our definitions and results are formulated for modules on the right. For modules on the left, they can be stated in the natural way. 

Throughout this section, let $\Gamma$ be a groupoid and $R=\bigoplus_{\gamma\in \Gamma}R_\gamma$ be a $\Gamma$-graded ring. 

If $M$ is a right $R$-module, we say that $M$ is \emph{$\Gamma$-graded} if there exists a family $\{M_\gamma:\gamma\in\Gamma\}$ of additive subgroups of $M$ such that $M=\bigoplus_{\gamma\in\Gamma}M_\gamma$ and, for each $\sigma,\tau\in\Gamma$, 
\[M_\sigma R_\tau \subseteq 
\begin{cases}
    M_{\sigma\tau},& \textrm{if $\sigma\tau$ is defined}\\
    \{0\},& \textrm{if $\sigma\tau$ is not defined}
\end{cases}.\]
The \emph{support of $M$} is the subset of $\Gamma$ defined by $\supp(M):=\{\gamma\in \Gamma\colon M_\gamma\neq 0\}$.

A submodule $N$ of $M$ is a \emph{graded submodule} if $N=\bigoplus_{\gamma\in\Gamma}N_\gamma$ where $N_\gamma:=N\cap M_\gamma$ for each $\gamma\in\Gamma$. Equivalently, if $n_{\gamma_1}+\cdots+n_{\gamma_t}\in N$ with $n_{\gamma_i}\in M_{\gamma_i}$ and $\gamma_i\neq\gamma_j$ for different $i,j=1,...,t$, then $n_{\gamma_i}\in N$ for all $i=1,...,t$. 


Clearly $R_R$ is a $\Gamma$-graded right $R$-module. 
A right ideal $U$ of $R$ is a \emph{graded right ideal} if $U_R$ is a graded submodule of $R_R$. 
A \emph{graded ideal} of $R$ is a bilateral ideal of $R$ that is, simultaneously, a graded right ideal and a graded left ideal of $R$.

Let $R$ and $S$ be $\Gamma$-graded rings and $M$ be an $(S,R)$-bimodule. We say that $M$ is a $\Gamma$-\emph{graded} $(S,R)$-\emph{bimodule} if $M$ is a $\Gamma$-graded right $R$-module and a $\Gamma$-graded left $S$-module such that $(sx)r=s(xr)$ for all $s\in S, r\in R$ and $x\in M$.

Let $M$, $N$ be $\Gamma$-graded right  $R$-modules. A homomorphism of modules $g:M\to N$ it is said to be a \emph{gr-homomorphism of modules} if $g(M_\sigma)\subseteq N_\sigma$ for all $\sigma\in\Gamma$. We denote by $\Homgr(M,N)$ the abelian group consisting of the gr-homomorphisms of modules $g:M\to N$.  We also define the ring $\Endgr(M):=\Homgr(M,M)$. 
It is easy to show that if $g\in\Homgr(M,N)$, then $\ker g$ is a graded submodule of $M$ and $\im g$ is a graded submodule of $N$.
A bijective gr-homomorphism of modules $g:M\to N$ is called a \emph{gr-isomorphism of modules}. We say that $M$ is \emph{gr-isomorphic} to $N$ as modules (denoted by $M\cong_{gr}N$) if there exists a gr-isomorphism of modules $g:M\to N$.

If $\{M_j:j\in J\}$ is a family of $\Gamma$-graded right $R$-modules, then the direct sum $\bigoplus_{j\in J}M_j$ is also a $\Gamma$-graded right $R$-module via $(\bigoplus_{j\in J}M_j)_\gamma:=\bigoplus_{j\in J}(M_j)_\gamma$ for each $\gamma\in\Gamma$. 

Let $M$ be a $\Gamma$-graded right $R$-module and $N$ be a graded submodule of $M$. 
The quotient module $M/N$ is $\Gamma$-graded via $(M/N)_\gamma:=\frac{M_\gamma+N}{N}$ for each $\gamma\in\Gamma$. We say that $N$ is a \emph{graded direct summand} of $M$ if there exists a graded submodule $X$ of $M$ such that $M=N\oplus X$. In this event,  $M/N\cong_{gr}X$.

A right $R$-module $M$ is \emph{unital} if $MR=M$. In other words, for each $m\in M$, there exist $m_1,\dotsc,m_n\in M$ and $a_1,\dotsc,a_n\in R$ such that $m_1a_1+\dotsb+m_na_n=m$. 
\begin{proposition}{\cite[Proposition 5]{CLP}}
     Let $R$ be a $\Gamma$-graded ring and let $M$ be a $\Gamma$-graded right 
     $R$-module. Then $M$ is unital if and only if the equality $m_\sigma1_{d(\sigma)}=m_\sigma$ 
     holds for all $\sigma\in\Gamma$ and  $m_\sigma\in M_\sigma$. \qed
\end{proposition}

In this work, \emph{unless otherwise stated (Section~\ref{sec: teo da den}), all groupoid graded modules are supposed to be unital.
Moreover, in expressions such as ``let $M$ be a $\Gamma$-graded $R$-module'' we assume that $M$ is a $\Gamma$-graded right $R$-module.}

\begin{remark}\label{rem:examples of gradation}
\begin{enumerate}[(1)]
    \item   Let $I$ be a set. Endow $I\times I$ with the groupoid structure of Example~\ref{ex:groupoids}(2).  Let $R$ be an $I\times I$-graded $R$-module. For each $i\in I$, recall that $1_{(i,i)}$ is the identity element of $R_{(i,i)}$.
    Fix $i_0\in I$. Then every unital right $R$-module $M$ can be regarded as an $I\times I$-graded $R$-module via $M_{(i_0,i)}:=M1_{(i,i)}$ for each $i\in I$. Notice that $\supp(M)\subseteq \{i_0\}\times I$.  Furthermore, any unital $(R,R)$-bimodule (i.e. $R\cdot M=M=M\cdot R$) can be regarded as an $I\times I$-graded $(R,R)$-bimodule via $M_{(i,j)}:=1_{(i,i)}M1_{(j,j)}$, for each $i,j\in I$. In particular this applies to the ring $R[\mathcal{C}]$ of a small preadditive category $\mathcal{C}$. In this case, every ideal of $R[\mathcal{C}]$ is $\mathcal{C}_0\times \mathcal{C}_0$-graded.
    \item  Consider the commutative  ring $\mathbb{Z}^{(\Gamma_0)}:=\bigoplus_{e\in\Gamma_0}{\mathbb{Z}}$
    with the obvious $\Gamma$-grading with support $\Gamma_0$. 
    Every $\Gamma$-graded abelian group is a $\Gamma$-graded right and left $\mathbb{Z}^{(\Gamma_0)}$-module and vice-versa. In fact, let $X$ be a $\Gamma$-graded abelian group. 
    Then, considering $X$ as a right (resp. left) $\mathbb{Z}$-module, we make $X$  a $\Gamma$-graded right (resp. left) $\mathbb{Z}^{(\Gamma_0)}$-module via 
    \[x_\gamma\cdot (n_e)_{e\in\Gamma_0}:=x_\gamma n_{d(\gamma)} \quad (\mathrm{resp.~} (n_e)_{e\in\Gamma_0}\cdot x_\gamma:=n_{r(\gamma)}x_\gamma),\]
    for each $\gamma\in\Gamma$, $x_\gamma\in X_\gamma$ and $(n_e)_{e\in\Gamma_0}\in\mathbb{Z}^{(\Gamma_0)}$.\qed
\end{enumerate}
   \end{remark}

\medskip

Let  $M$ be a $\Gamma$-graded right  $R$-module and $\sigma\in\Gamma$.  For each $\gamma\in\Gamma$, define $M(\sigma)_\gamma:=M_{\sigma\gamma}$, where we follow the convention $M_{\sigma\gamma}=\{0\}$ if $d(\sigma)\neq r(\gamma)$. The \emph{shift  of $M$ by $\sigma$} is the $\Gamma$-graded right $R$-module $M(\sigma):=\bigoplus_{\gamma\in\Gamma}M(\sigma)_\gamma$. Note that $\supp (M(\sigma))\subseteq d(\sigma)\Gamma$. Analogously, if $M$ is a $\Gamma$-graded left $R$-module and  $\gamma\in\Gamma$, set $((\sigma)M)_\gamma:=M_{\gamma\sigma}$. The shift  of $M$ by $\sigma$ is the $\Gamma$-graded left $R$-module $(\sigma)M:=\bigoplus_{\gamma\in\Gamma}((\sigma)M)_\gamma$. Notice that $\supp ((\sigma)M)\subseteq \Gamma r(\sigma)$.
The following property of shifts of modules will be very useful. For a left version, see \cite[Proposition 10(a)]{CLP}.


\begin{lemma}
\label{lem: M(e)=M(gamma) como conj}
    Let $M$ be a $\Gamma$-graded $R$-module and $\sigma\in\Gamma$. Then $M(\sigma)$ equals $M(r(\sigma))$ as $R$-modules. More precisely, $M(\sigma)$  and $M(r(\sigma))$ have the same homogeneous components (albeit labeled with a different degree). 
\end{lemma}

\begin{proof}
    If $\gamma\in\supp(M(\sigma))$, then $r(\gamma)=d(\sigma)$ and  $M(\sigma)_\gamma=M_{\sigma\gamma}=M(r(\sigma))_{\sigma\gamma}$ is a homogeneous component of $M(r(\sigma))$.
    On the other hand, if $\gamma\in\supp(M(r(\sigma)))$ then $r(\gamma)=r(\sigma)$ and it follows that $M(r(\sigma))_\gamma=M_{\gamma}=M_{\sigma\sigma^{-1}\gamma}=M(\sigma)_{\sigma^{-1}\gamma}$ is a homogeneous component of $M(\sigma)$.
\end{proof}

Note that if $R$ is a $\Gamma$-graded ring and $M$ is a $\Gamma$-graded right 
$R$-module, then $M=\bigoplus_{e\in\Gamma_0}M(e)$ 
. This decomposition will be important, so we define $$\Gamma'_0(M):=\{e\in\Gamma_0:M(e)\neq0\}.$$
 The next result gives two more ways of expressing $\Gamma'_0(R_R)$ and shows that $\Gamma'_0(R_R)=\Gamma'_0(_RR)$. Thus we can write $\Gamma'_0(R)$  or simply $\Gamma'_0$, when the $\Gamma$-graded ring $R$ is clear, instead of $\Gamma'_0(R_R)$ or $\Gamma'_0(_RR)$.

\begin{lemma}
\label{lem: obj unit ---> unit}
    Let $R$ be a $\Gamma$-graded ring. Then $$1_e\neq0\iff R_e\neq\{0\}\iff R(e)\neq\{0\}\iff (e)R\neq\{0\}$$ for each $e\in\Gamma_0$.
\end{lemma}

\begin{proof}
 It suffices to note that, for each $e\in\Gamma_0$, $R_e$ is a ring with unity $1_e$ and $R(e)$ (resp. $(e)R$) is a right (resp. left) $R$-module generated by $1_e$.
%
\end{proof}


\medskip

Given $\gamma\in\Gamma$, we say that a homomorphism of right $R$-modules $g:M\to N$ is a \emph{homomorphism of degree $\gamma$} if $g(M_\sigma)\subseteq N_{\gamma\sigma}$  for each $\sigma\in\Gamma$, where we understand that $N_{\gamma\sigma}=\{0\}$ whenever $\gamma\sigma$ is not defined in $\Gamma$. For all $\gamma\in\Gamma$, $\HOM_R(M,N)_\gamma$ will denote the additive group of the homomorphisms $g:M\to N$ of degree $\gamma$. Hence, we can define the $\Gamma$-graded additive group $\HOM_R(M,N):=\bigoplus_{\gamma\in\Gamma}\HOM_R(M,N)_\gamma$. 

A word of caution is needed for graded left modules.
If $M,N,P$ are $\Gamma$-graded left $R$-modules and $g\colon M\rightarrow N$, $h\colon N\rightarrow P$ are  homomorphisms of $\Gamma$-graded modules, $g$ will act on the right. Thus, the image of $x\in M$ will be denoted by $(x)g$, and the composition $g\circ h$ means that $g$ acts first. Now, given $\gamma\in\Gamma$, we say that the homomorphism of left $R$-modules $g:M\to N$ is a \emph{homomorphism of degree $\gamma$} if $(M_\sigma)g\subseteq N_{\sigma\gamma}$  for each $\sigma\in\Gamma$. For all $\gamma\in\Gamma$, $\HOM_R(M,N)_\gamma$ will denote the additive group of the homomorphisms $g:M\to N$ of degree $\gamma$. Hence, we can define the $\Gamma$-graded additive group $\HOM_R(M,N):=\bigoplus_{\gamma\in\Gamma}\HOM_R(M,N)_\gamma$.

The following results  follow immediately from the definition (see \cite[Proposition 13(c)]{CLP}).

\begin{lemma}
\label{lem: g_ah_b tem grau ab}
    Let $R$ be a $\Gamma$-graded ring and $\gamma,\sigma\in\Gamma.$
    \begin{enumerate}[\rm(1)]
        \item Let $M$, $N$ and $P$ be $\Gamma$-graded right $R$-modules and consider homomorphisms $g\in\HOM_R(M,N)_\gamma$, $h\in\HOM_R(N,P)_\sigma$.   If $\sigma\gamma$ is defined, then $h\circ g:M\to P$ is a homomorphism of degree $\sigma\gamma$, and $h\circ g=0$ otherwise.
        \item Let $M$, $N$ and $P$ be $\Gamma$-graded left $R$-modules and consider homomorphisms $g\in\HOM_R(M,N)_\gamma$, $h\in\HOM_R(N,P)_\sigma$. If $\gamma\sigma$ is defined, then $h\circ g:M\to P$ is a homomorphism of degree $\gamma\sigma$, and $ h\circ g=0$ otherwise. \qed
    \end{enumerate} 
\end{lemma}

\begin{lemma}
\label{lem: ker e im de homo de grau gamma}
    Let $R$ be a $\Gamma$-graded ring, $M$ and $N$ be $\Gamma$-graded $R$-modules and $\sigma\in\Gamma$. If $g\in\HOM_R(M,N)_\sigma$, then $\bigoplus\limits_{\substack{e\in\Gamma_0\\e\neq d(\sigma)}}M(e)\subseteq \ker g$ and $\im g\subseteq N(r(\sigma))$. In particular, $\HOM_R(M,N)_\sigma = \HOM_R(M,N(r(\sigma)))_\sigma$.\qed
\end{lemma}


Let  $M$ be a $\Gamma$-graded right (resp. left) $R$-module and $e\in\Gamma_0$. Considering the decomposition $M=\bigoplus_{e\in\Gamma_0}M(e)$ (resp. $M=\bigoplus_{e\in\Gamma_0}(e)M$), we denote by $\mathds{1}_e$ the canonical projection $M\to M(e)$ (resp. $M\to (e)M$). It is easy to see that $\mathds{1}_e\in\END_R(M)_e$.

\begin{proposition}\label{prop: 1 of END}
    Let $R$ be a $\Gamma$-graded ring and $M$ be a $\Gamma$-graded $R$-module. Then $\END_R(M)$ is an object unital $\Gamma$-graded ring with $g\cdot h:=g\circ h$ for all $g,h\in\END_R(M)$. Moreover,  the ring $\END_R(M)_e$ has unity $\mathds{1}_e$ for each $e\in\Gamma_0$.
\end{proposition}

\begin{proof}
    We know that $\END_R(M)=\bigoplus_{\gamma\in\Gamma}\END_R(M)_\gamma$ and, by Lemma \ref{lem: g_ah_b tem grau ab}, we have $\END_R(M)_\gamma\END_R(M)_\delta\subseteq \END_R(M)_{\gamma\delta}$ for each $\gamma,\delta\in\Gamma$. In particular, $\END_R(M)$ is a ring. It follows easily from Lemma~\ref{lem: ker e im de homo de grau gamma} that  $\mathds{1}_{r(\gamma)}\circ g=g=g\circ\mathds{1}_{d(\gamma)}$ for each $\gamma\in\Gamma$ and $g\in\END_R(M)_\gamma$. Thus, $\END_R(M)$ is an object unital $\Gamma$-graded ring in which $\mathds{1}_e$ is the unity of the ring $\END_R(M)_e$ for each $e\in\Gamma_0$.    
\end{proof}


A graded module and its endomorphism ring $\END$ are linked by the following important relation.

\begin{lemma}
    Let $R$ be a $\Gamma$-graded ring and $M$ be a $\Gamma$-graded $R$-module. Then $\Gamma'_0(M)=\Gamma'_0(\END_R(M))$. 
\end{lemma}

\begin{proof}
   By definition of $\Gamma'_0(M)$, $e\in \Gamma'_0(M)$ if and only if $M(e)\neq0$. By definition of $\mathds{1}_e$,  $M(e)\neq0$ if and only if $\mathds{1}_e\neq0$. By Proposition~\ref{prop: 1 of END} and Lemma~\ref{lem: obj unit ---> unit}, $\mathds{1}_e\neq0$ if and only if $\END_R(M)_e\neq0$. Again by Lemma~\ref{lem: obj unit ---> unit}, $\END_R(M)_e\neq0$ if and only if $e\in \Gamma'_0(\END_R(M))$, as desired.    
\end{proof}

Note that if  $M$, $N$ are $\Gamma$-graded $R$-modules such that $\varphi:M\to N$ is a gr-isomorphism of modules, then $\END_R(M)\cong_{gr}\END_R(N)$ as graded rings via $g\mapsto\varphi\circ g\circ \varphi^{-1}$ for all $g\in\END_R(M)$.

\begin{remark}\label{rem:graded_homomorphisms}
    It is easy to show that if $R$ is a $\Gamma$-graded ring and $M,N$ are $\Gamma$-graded $R$-modules, then
    \begin{enumerate}
        \item 
        \[\Homgr(M,N)\hookrightarrow\prod_{e\in\Gamma_0}\HOM_R(M,N)_e\]
        via $h\mapsto (h_e)_{e\in\Gamma_0}$ where $h_e=h$ on $M(e)$ and $h_e=0$ on $M(f)$ for all $e\neq f\in\Gamma_0$.
        \item     
        \[\Homgr(M,N)\cong\prod_{e\in\Gamma_0}\Homgr(M(e),N(e))\]
        via the homomorphism of additive groups 
         $h\mapsto (h|_{M(e)})_{e\in\Gamma_0},$
        where $h|_{M(e)}$ denotes the restriction of the gr-homomorphism $h\colon M\rightarrow N$ to the graded submodule $M(e)$ for each $e\in\Gamma_0$. Notice that $h(M(e))\subseteq N(e)$ because $h$ is a gr-homomorphism. Thus, $h|_{M(e)}\in \Homgr(M(e),N(e))$ for each $e\in\Gamma_0$. \qed 
       %
    \end{enumerate}
\end{remark}

\medskip

The following result is a graded version of the \emph{generation Lemma} \cite[(1.1)]{Lam2}. If $R$ is a $\Gamma$-graded ring, $M$ is a $\Gamma$-graded $R$-module and $\{m_i:i\in I\}\subseteq \h(M)$ is a set of homogeneous generators of $M$, we will say that $\{m_i:i\in I\}$ is a \emph{minimal set of generators} if, for each $i_0\in I$, $m_{i_0}$ does not belong to the graded submodule of $M$ generated by $\{m_i:i\in I\setminus\{i_0\}\}$.
\begin{lemma}
\label{lem: generation}
    Let $R$ be a $\Gamma$-graded ring, $M$ be a $\Gamma$-graded $R$-module and $\{m_i:i\in I\}\subseteq \h(M)$ a minimal set of homogeneous generators of $M$ where $I$ is infinite. Then $M$ cannot be generated by a set of homogeneous elements with cardinality less than $|I|$. In particular, every minimal set of homogeneous generators of $M$ has cardinality $|I|$.
\end{lemma}

\begin{proof}
    Suppose that $\{x_j:j\in J\}$ is another set of homogeneous generators of $M$. Each $x_j$ is a linear combination of $m_i$'s. Then, let $I_0$ be a minimal subset of $I$ such that $\{x_j:j\in J\}\subseteq\{m_i:i\in I_0\}R$. Thus, $M=\{x_j:j\in J\}R\subseteq\{m_i:i\in I_0\}R$ and it follows that $I_0=I$. If $J$ is a finite set, then $I_0$  is also a finite set, contradicting that $|I|$ is infinite. Therefore, $|J|$ is infinite and it follows that $|I|=|I_0|\leq |J|\cdot\aleph_0=|J|$. 
\end{proof}


\begin{corollary}
\label{coro: generation lemma para gr-ciclicos}
    Let $R$ be a $\Gamma$-graded ring, $M$ be a $\Gamma$-graded $R$-module and $\{m_i:i\in I\},\{m'_j:j\in J\}$ be subsets of $\h(M)\setminus\{0\}$ such that $M=\bigoplus\limits_{i\in I}m_iR=\bigoplus\limits_{j\in J}m'_jR$. If $I$ or $J$ is infinite, then $|I|=|J|$.
\end{corollary}

\begin{proof}
    It is suffices to note that $\{m_i:i\in I\}$ and $\{m'_j:j\in J\}$ are minimal set of homogeneous generators of $M$. Now the result  follows from Lemma \ref{lem: generation}.
\end{proof}

\subsection{Conventions}
Throughout this work, rings are assumed to be associative but not necessarily unital.

Unless otherwise stated (Section~\ref{sec: teo da den}), all groupoid graded rings are supposed to be object unital.

Unless otherwise stated (Section~\ref{sec: teo da den}), all groupoid graded modules are supposed to be unital.

In expressions such as ``let $M$ be a $\Gamma$-graded $R$-module,'' we assume that $M$ is a $\Gamma$-graded right $R$-module.

We will consider left modules. In this event, we will make explicit the word left.
If $M,N,P$ are  left modules and $g\colon M\rightarrow N$, $h\colon N\rightarrow P$ are  homomorphisms of  modules, $g$ will act on the right. Thus, the image of $x\in M$ will be denoted by $(x)g$, and the composition $g\circ h$ means that $g$ acts first.


Let $X = \bigoplus_{\gamma \in \Gamma} X_\gamma$ be a $\Gamma$-graded additive group; for instance, $X$ could be a $\Gamma$-graded ring or a $\Gamma$-graded (right or left) module. When $\sigma, \tau \in \Gamma$ are such that $d(\sigma) \neq r(\tau)$ (that is, $\sigma\tau$ is not defined) we adopt the convention $X_{\sigma\tau} := \{0\}$.

\section{Gr-homomorphisms and rings of matrices}\label{sec:rings_of_matrices}
	
	\emph{Throughout this section, let $\Gamma$ be a groupoid and $R=\bigoplus\limits_{\gamma\in\Gamma}R_\gamma$ be  
	a $\Gamma$-graded ring.}

\subsection{Pseudo-free modules}\label{sec:pseudo-free modules}

Let $M$ be a $\Gamma$-graded $R$-module.
Consider a sequence of homogeneous elements $(x_i)_{i\in I}\in \prod\limits_{i\in I} M_{\gamma_i}$ where $(\gamma_i)_{i\in I}$ is  a sequence of elements of $\Gamma$.
A homogeneous element $x\in M_\sigma$ is said to be a \emph{genuine linear combination} of $(x_i)_{i\in I}$ if there exists $(a_i)\in \bigoplus\limits_{i\in I}R_{\gamma_i^{-1}\sigma}$ such that $x=\sum\limits_{i\in I} x_ia_i$. 
We say that $(x_i)_{i\in I}$ is \emph{pseudo-linearly independent} if the only sequence  $(a_i)_{i\in I}\in \bigoplus\limits_{i\in I} 1_{d(\gamma_i)}R$ such that
$\sum\limits_{i\in I}x_ia_i=0$ is $(a_i)_{i\in I}=0$. Equivalently,   the only sequence $(a_i)_{i\in I}\in \bigoplus\limits_{i\in I} 1_{d(\gamma_i)}R$, with $a_i\in\h(1_{d(\gamma_i)}R)$ for all $i\in I$, for which 
$\sum\limits_{i\in I}x_ia_i=0$  is $(a_i)_{i\in I}=0$. 
Analogously, the only genuine linear combination of $(x_i)_{i\in I}$ that equals $0\in M$ is with $a_i=0$ for all $i\in I$.
The sequence $(x_i)_{i\in I}$ is a \emph{pseudo-basis} of $M$ if $(x_i)_{i\in I}$ generates $M$ as an $R$-module and it is pseudo-linearly independent.   If $M$ has a pseudo-basis,  we say that $M$ is a \emph{pseudo-free} module.

 It is worth noting that the $R$-module $R_R=\bigoplus\limits_{e\in\Gamma_0} R(e)$ is pseudo-free, with pseudo-basis $\{1_e\}_{e\in\Gamma_0'(R)}$. Furthermore, for each $\sigma\in\Gamma$ such that $r(\sigma)\in\Gamma'_0(R)$, $R(\sigma)$ is pseudo-free with the pseudo-basis formed by $1_{r(\sigma)}\in (R(\sigma))_{\sigma^{-1}}$.

Modules of the form $\bigoplus\limits_{i\in I} R(\sigma_i)$ were called free  in \cite{Lund} and free by suspension in \cite{CLP}. In the next result, we show the universal property of these modules and that they are (gr-isomorphic to) what we have just called pseudo-free modules.

\begin{proposition}\label{prop:pseudo_free}
	Let $M$ be a $\Gamma$-graded $R$-module.
	Consider a sequence of homogeneous elements $(x_i)_{i\in I}\in \prod\limits_{i\in I} M_{\gamma_i}$ where $(\gamma_i)_{i\in I}$ is  a sequence of elements of $\Gamma$.	
	The following statements are equivalent.
	
	\begin{enumerate}[\rm (1)\phantom{'}]
		\item $(x_i)_{i\in I}$ is a pseudo-basis of $M$
		\item For each $x\in M$, there exists a unique sequence $(a_i)_{i\in I}\in \bigoplus\limits_{i\in I} 1_{d(\gamma_i)}R$ such that $$x=\sum_{i\in I} x_ia_i.$$
		\item Every homogenous element $x\in M$ can be uniquely expressed as a genuine linear combination of $(x_i)_{i\in I}$. That is, for each $\sigma \in \Gamma$ and $x\in M_\sigma$, there exists a unique sequence $(a_i)_{i\in I}\in \bigoplus\limits_{i\in I} R_{{\gamma_i}^{-1}\sigma}$ such that 
		$$x=\sum_{i\in I} x_ia_i.$$
		\item For any $R$-module $N$ and sequence $(y_i)_{i\in I}\in \prod\limits_{i\in I} N1_{d(\gamma_i)}$, there exists a unique homomorphism of $R$-modules  $f\colon M\rightarrow N$ such that $f(x_i)=y_i$ for all $i\in I$.
		
		\item For any $\Gamma$-graded $R$-module $N$ and sequence $(y_i)_{i\in I}\in \prod\limits_{i\in I} N_{\gamma_i}$, there exists a unique $f\in \Homgr(M,N)$ such that $f(x_i)=y_i$ for all $i\in I$.
		
		\item For any $\sigma\in \Gamma$,  $\Gamma$-graded $R$-module $N$ and sequence $(y_i)_{i\in I}\in \prod\limits_{i\in I} N_{\sigma\gamma_i}$, there exists a unique $f\in \HOM(M,N)_\sigma$ such that $f(x_i)=y_i$ for all $i\in I$.
		
		\item There exists a unique gr-isomorphism $\varphi:M\longrightarrow\bigoplus\limits_{i\in I}R({\gamma_i}^{-1})$ such that $\varphi(x_i)=1_{d(\gamma_i)}$ for all $i\in I$.
	
		\end{enumerate}

\end{proposition}	

\begin{proof}
	$(1)\implies (2)$ Since $(x_i)_{i\in I}\in \prod\limits_{i\in I} M_{\gamma_i}$ generates $M$,
	$x=\sum\limits_{i\in I}x_id_i$
	 for some $(d_i)_{i\in I}\in \bigoplus\limits_{i\in I}R$. If we (uniquely) express
	 each $d_i$ as $d_i=a_i+a'_i$ where $a_i\in 1_{d(\gamma_i)}R$ and
	 $a_i'\in \bigoplus\limits_{e\in\Gamma_0\setminus{d(\gamma_i)}}1_e R$, then
	 $$x=\sum_{i\in I}x_id_i=\sum_{i\in I}x_i(a_i+a'_i)=\sum_{i\in I}x_ia_i$$
	 because, for all $i\in I$, $x_ia_i'=0$   due to the fact that $M_{\gamma_i} \cdot\left(\bigoplus\limits_{e\in\Gamma_0\setminus{d(\gamma_i)}}1_e R\right)=\{0\}$. 
	 
	 Now suppose that there exists $(b_i)_{i\in I}\in \bigoplus\limits_{i\in I}1_{d(\gamma_i)}R$ such that
	 $x=\sum_{i\in I}x_ib_i.$ Then 
	 $$0=\sum_{i\in I}x_i(a_i-b_i).$$
	 Since $a_i-b_i\in  1_{d(\gamma_i)}R$ and $(x_i)_{i\in I}$ is pseudo-linearly independent, we obtain that $a_i=b_i$ for all $i\in I$.


	 $(2)\implies (4)$ Let $x\in M$. By (2), there exists a unique sequence $(a_i)_{i\in I}\in \bigoplus\limits_{i\in I} 1_{d(\gamma_i)}R$ such that $$x=\sum_{i\in I} x_ia_i.$$  Define $f(x)=f(\sum_{i\in I} x_ia_i):=\sum_{i\in I}y_ia_i$. Now it is routine to see that $f$ is a homomorphism of $R$-modules.

  $(4)\implies(1)$: First, we show that $(x_i)_{i\in I}$ generates $M$. Let $M'$ be the graded submodule of $M$ generated by $(x_i)_{i\in I}$ and consider $N:=M/M'$. By (4) there exists a unique homomorphism of $R$-modules  $f\colon M\rightarrow N$ such that $f(x_i)=0$ for all $i\in I$. Thus, $f$ is the canonical projection and it follows that $N=0$, i.e., $M'=M$. Now we prove that $(x_i)_{i\in I}$ is pseudo-linearly independent. Suppose that a sequence $(a_i)_{i\in I}\in \bigoplus\limits_{i\in I} 1_{d(\gamma_i)}R$ is such that $\sum\limits_{i\in I}x_ia_i=0$. For each $i\in I$, let $f_i\colon M\rightarrow 1_{d(\gamma_i)}R$ be the unique homomorphism of $R$-modules such that $f_i(x_i)=1_{d(\gamma_i)}$ and $f_i(x_j)=0$ for all $j\in I\setminus\{i\}$. So, for all $i\in I$, we have
  \[0=f_i\left(\sum_{j\in I}x_ja_j\right)=\sum_{j\in I}f_i(x_j)a_j=a_i.\]

  $(2)\implies(3)$ is clear.

	 $(3)\implies (5)$ is shown in the same way that $(2)$ implies $(4)$ and observing that the homomorphism $f$ so defined is a gr-homomorphism.
  
	 
	$(5) \implies (7)$ Set $N=\bigoplus\limits_{i\in I} R(\gamma_i^{-1})$. First note that  $1_{d(\gamma_i)}\in N_{\gamma_i}$	for each $i\in I$. 
	 By (5), there exists a unique gr-homomorphism of  $R$-modules
	 $\varphi\colon M\rightarrow N$ such that $\varphi(x_i)=1_{d(\gamma_i)}$ for each $i\in I$.
	 Now observe that the
	 sequence $(1_{d(\gamma_i)})_{i\in I}\in \prod\limits_{i\in I}N_{\gamma_i}$ is a pseudo-basis of $N$. Hence $\varphi$ is a gr-isomorphism.
	 
	 $(7)\implies (1)$ follows because $\bigoplus_{i\in I}R(\gamma_i^{-1})$ is pseudo-free.
	 
	 $(5)\implies (6)$ For each $i\in I$, $y_i\in N_{\sigma\gamma_i}=N(\sigma)_{\gamma_i}$. By (5), there exists a $f\in\Homgr(M,N(\sigma))$ such that $f(x_i)=y_i$ for all $i\in I$. Now $f$ can be regarded as a homomorphism of $R$-modules from $M$ to $N$ such that $f(M_\gamma)\subseteq N_{\sigma\gamma}$ for all $\gamma\in \Gamma$.
  
	 
	 $(6)\implies (5)$ Let $\Delta_0=\{r(\gamma_i)\mid i\in I\}$. Let $e\in \Delta_0$. By (6), there exists a unique $f_e\in \HOM(M,N)_e$ such that 
	 $$f_e(x_i)=\left\{ \begin{array}{ll}
	 y_i	& \textrm{if } r(\gamma_i)=e\\ 
	 0	& \textrm{if } r(\gamma_i)\neq e
	 \end{array}\right.$$
	 Notice that $f(M(e))\subseteq N(e)$ and  $f_e(M(e'))=\{0\}$ for all $e'\in\Delta_0$, $e'\neq e$. Define also $f_e\colon M\rightarrow N$ as zero for all $e\in\Gamma_0\setminus \Delta_0$. 
	 	 Now set $f\in \Homgr(M,N)$ as the unique $R$-module homomorphism such that
	 $f=f_e$ when restricted to $M(e)$ for all $e\in\Gamma_0$. By Remark~\ref{rem:graded_homomorphisms}, (5) follows.
	\end{proof}

\begin{corollary}
\label{coro: M livre sse todo M(e) livre}
    Let $M$ be a $\Gamma$-graded $R$-module. Then $M$ is a pseudo-free module if and only if $M(e)$ is a pseudo-free module  for every $e\in\Gamma_0$.
\end{corollary}

\begin{proof}
    It follows from condition (7) in Proposition \ref{prop:pseudo_free}.
\end{proof}

In order to state the next result, we will say that a $\Gamma$-graded $R$-module $P$ is \emph{gr-projective} if, for each $\Gamma$-graded $R$-modules $M,N$, a surjective $g\in\Homgr(M,N)$ and $h\in\Homgr(P,N)$, there exists $f\in\Homgr(P,M)$ such that $h=gf$ \cite[Proposition 35]{CLP}.
Now we give a direct proof of \cite[Lemma 37]{CLP} for object unital graded rings.

\begin{corollary}\label{coro:pseudo-free is gr-projective}
Let $P$ be a $\Gamma$-graded $R$-module. If $P$ is pseudo-free module, then $P$ is gr-projective.   
\end{corollary}
\begin{proof}
Let $M,N$ be $\Gamma$-graded $R$-modules. Suppose we are given a surjective $g\in\Homgr(M,N)$ and $h\in\Homgr(P,N)$. Let $(x_i)_{i\in I}\in\prod_{i\in I}P_{\gamma_i}$, where $(\gamma_i)_{i\in I}\in \Gamma^I$, be a pseudo-basis of $P$. Let $(m_i)_{i\in I}\in\prod_{i\in I}M_{\gamma_i}$ be such that $g(m_i)=h(x_i)$ for all $i\in I$. By Proposition~\ref{prop:pseudo_free}(5), there exists $f\in\Homgr(P, M)$ such that $f(x_i)=m_i$ for all $i\in I$. Then $h=gf$, as desired.
\end{proof}

\subsection{Technical results on  homomorphisms of modules with degree}
In this section, we study in more detail some aspects of the graded objects  $\END(M)$ and $\HOM(M,N)$ for $\Gamma$-graded modules $M,N$.

We begin with the regular $R$-module $R_R$. 
In future sections, it will be useful to regard the graded ring $R$ as a ring of  endomorphisms with degree.

\begin{lemma}
\label{lem: R=END(R)}
The map  
\[R\longrightarrow\END(R_R), \  a\mapsto m_a,\]
where $m_a$ is the homomorphism given by multiplication on the left by $a$, is a gr-isomorphism of $\Gamma$-graded rings.
Moreover,  gr-isomorphisms of $\Gamma$-graded rings  \[1_eR1_e\cong_{gr}\END_R(R(e))\]
are induced by restriction for each $e\in\Gamma_0$.

\end{lemma}

\begin{proof}
Consider
$ \phi:R\rightarrow\END(R_R),\ 
    a\mapsto m_a,$ with $m_a$ as defined in the statement.
Note that if
$a\in R_\gamma$ for some $\gamma\in\Gamma$, then $m_a$ is a homomorphism of degree $\gamma$. Clearly, $m_{ab}=m_a\circ m_b$ for $a,b\in R$. Now we prove that $\phi$ is an isomorphism.
If $m_a=0$, then $a1_e=m_a(1_e)=0$ for all $e\in\Gamma_0$ and it follows that $a=0$.
Let $\gamma\in\Gamma$ and $g\in \END(R_R)_\gamma$. Then $g(R_\alpha)\subseteq R_{\gamma\alpha}$, for all $\alpha\in\Gamma$. Hence $g(x)\neq0$ implies $x\in R(d(\gamma))$.  In this event, $g(x)=g(1_{d(\gamma)}x)=g(1_{d(\gamma)})x$. 
Therefore, $g=\phi(g(1_{d(\gamma)}))$.

Let $e\in\Gamma_0$. If $a\in 1_eR1_e$, then $m_a(R(e))\subseteq R(e)$. Thus,
$$\phi'\colon 1_eR1_e \rightarrow\END_R(R(e)), \
    a\mapsto m_a|_{R(e)},$$ is well-defined.
Let now be $0\neq g\in \END_R(R(e))_\gamma$. 
As before, 
$g=m_{g(1_e)}$, where $g(1_e)=g(1_e)1_e\in R(e)1_e= 1_eR1_e$.
\end{proof}

Let $M$ be a $\Gamma$-graded $R$-module. Denote by $\mathcal{P}(\Gamma)$ the power set of $\Gamma$ and consider a subset $\Sigma\in\mathcal{P}(\Gamma)$. 

Following \cite[Section~2.2]{Lund}, we define $M(\Sigma):=\bigoplus_{\sigma\in\Sigma}M(\sigma)$. When $M=R$ we write $R(\Sigma):=R_R(\Sigma)$.
Notice that if $r(\sigma)=f$ for some $\sigma\in \Sigma$ and $f\in\Gamma_0\setminus\Gamma_0'(M)$, then $M(\sigma)=\{0\}$. Thus, if $\Sigma'=\{\sigma\in\Sigma\colon r(\sigma)\in\Gamma_0'(M)\}$, then $M(\Sigma)=M(\Sigma')$.

We say that $\Sigma$ is \emph{$r$-unique for $M$} if $\{r(\sigma):\sigma\in\Sigma\}\subseteq\Gamma'_0(M)$ and, for each $e\in\Gamma_0'(M)$, there exists at most one $\sigma\in \Sigma$  with $r(\sigma)=e$. This condition implies that, for $\sigma,\tau,\rho,\lambda\in\Sigma$, the equality $\sigma\gamma\tau^{-1}=\rho\delta\lambda^{-1}$  holds in $\Gamma$ if and only if
 $$\sigma=\rho,\ \tau=\lambda \textrm{ and } \gamma=\delta.$$ 
Indeed, if $\sigma\gamma\tau^{-1}=\rho\delta\lambda^{-1}$, then $r(\sigma)=r(\sigma\gamma\tau^{-1})=r(\rho\delta\lambda^{-1})=r(\rho)$. Since $\Sigma$ is $r$-unique, then $\sigma=\rho$. This implies that $\gamma\tau^{-1}=\delta\lambda^{-1}$. Thus,
$r(\tau)=d(\tau^{-1})=d(\gamma\tau^{-1})=d(\delta\lambda^{-1})=d(\lambda^{-1})=r(\lambda)$. Since $\Sigma$ is $r$-unique, then $\tau=\lambda$. As a consequence, we also have $\gamma=\delta$.
 
We say that $\Sigma$ is \emph{fully $r$-unique for $M$} if the correspondence $\sigma\mapsto r(\sigma)$ defines a bijective function $\Sigma\rightarrow\Gamma_0'(M)$. In other words,  
$\Sigma$ is an $r$-unique subset for $M$ that satisfies the following  condition: for each $e\in \Gamma_0'(M)$ there exists $\sigma\in\Sigma$ such that $r(\sigma)=e$. 
 
We say that $\Sigma\in\mathcal{P}(\Gamma)$ is \emph{$d$-unique} if, for each $e\in\Gamma_0(M)$ (different from $e\in\Gamma_0'(M)$!), there exists at most one $\sigma\in \Sigma$  with $d(\sigma)=e$.

Let $H$ be a $\Gamma$-graded additive group and $\Sigma,\Delta\in\mathcal{P}(\Gamma)$ $d$-unique subsets. Given $\gamma\in\Gamma$, if there are (unique) $\sigma\in\Sigma$ and $\delta\in\Delta$ such that $r(\gamma)=d(\delta)$ and $d (\gamma)=d(\sigma)$, we denote $H_{\Delta\gamma\Sigma^{-1}}:=H_{\delta\gamma\sigma^{-1}}$. Otherwise, $H_{\Delta\gamma\Sigma^{-1}}:=\{0\}$.

 \medskip

The following technical result is concerned with shifts of modules and their gr-homomorphisms.  It will be very useful in the sequel and some aspects of it can be regarded as an improvement of \cite[Proposition~13]{CLP}. 
The equality of Proposition~\ref{prop: HOM(,) e Hom_gr}(1) was given in \cite[Proposition~13(a)]{CLP}. 
We include a proof  for completion.

\begin{proposition}
\label{prop: HOM(,) e Hom_gr}

Let  $M,N$ be $\Gamma$-graded  $R$-modules. The following assertions hold:

\begin{enumerate}[\rm(1)]

    \item  For each $\gamma \in \Gamma$, 
\[
\Homgr(M,N(\gamma))=   \HOM_R(M,N)_\gamma \cong  \Homgr(M(\gamma^{-1}),N)
\]
    and the isomorphism above is the identity if  $M=M(\gamma^{-1})$ as sets.
   
    \item Let $\Sigma,\Delta\in\mathcal{P}(\Gamma)$ be such that $\Sigma$ is $r$-unique for $M$, $\Delta$ is $r$-unique for $N$ and both $\Sigma,\Delta$ are $d$-unique. Then, for each $\gamma\in\Gamma$,
    \[\HOM_R(M(\Sigma),N(\Delta))_\gamma\cong \HOM_R(M,N)_{\Delta\gamma\Sigma^{-1}},\]
    and such isomorphism  is the identity if there exists $\sigma\in\Sigma$ such that $M=M(\sigma)$ as sets.
\end{enumerate}
\end{proposition}

\begin{proof}
(1) If $g:M\longrightarrow N$ is a homomomorphism of $R$-modules then, for each $\alpha\in\Gamma$, we have
$g(M_\alpha)\subseteq N(\gamma)_\alpha$ if and only if $g(M_\alpha)\subseteq N_{\gamma\alpha}$. Therefore,
 the equality of the statement follows.

   If $g\in\HOM_R(M,N)_\gamma$ then, for all $\alpha\in\Gamma$,
    \[g\left(M(\gamma^{-1})_\alpha\right)=g(M_{\gamma^{-1}\alpha})\subseteq N_{\gamma\gamma^{-1}\alpha}\subseteq N_\alpha.\] 
    Thus, we can consider the function
    \begin{align*}
        \varphi:\HOM_R(M,N)_\gamma&\longrightarrow \Homgr(M(\gamma^{-1}),N)\\
        g&\longmapsto g|_{M(\gamma^{-1})}
    \end{align*}
    By Lemma \ref{lem: M(e)=M(gamma) como conj}, $M(\gamma^{-1})=M(d(\gamma))$ as $R$-modules. Thus
    \[M=M(\gamma^{-1})\oplus \bigoplus_{\substack{e\in\Gamma_0\\e\neq d(\gamma)}}M(e)\]
    as $R$-modules.
    Hence, if $g,g'\in\HOM_R(M,N)_\gamma$ coincide in $M(\gamma^{-1})$, then $g=g'$ by Lemma~\ref{lem: ker e im de homo de grau gamma}. That is, $\varphi$ is injective.
    Furthermore, every $g\in \Homgr(M(\gamma^{-1}),N)$ can be extended to a gr-homomorphism of $R$-modules $g:M\to N$ by defining $g(M(e))=0$ for all $e\in\Gamma_0\setminus\{d(\gamma)\}$. Such $g$ has degree $\gamma$ because, for all $\alpha\in\Gamma$,
     if $r(\alpha)\neq d(\gamma)$, then $g(M_\alpha)\subseteq g(M(r(\alpha)))=0$, and if $r(\alpha)= d(\gamma)$, then $g(M_\alpha)=g(M_{\gamma^{-1}\gamma\alpha})=g(M(\gamma^{-1})_{\gamma\alpha})\subseteq N_{\gamma\alpha}$.
    Thus, $\varphi$ is surjective, and therefore an isomorphism. Clearly $\varphi$ is the identity function if $M=M(\gamma^{-1})$ as sets.

(2) Suppose there exist (unique) $\sigma\in\Sigma$ and $\delta\in\Delta$ such that $d(\sigma)=d(\gamma)$ and $d(\delta)=r(\gamma)$.
Using (1), we obtain \begin{eqnarray*}    \HOM\left(M(\Sigma),N(\Delta)\right)_\gamma & = &
    \Homgr\left(M(\Sigma),N(\Delta)(\gamma)\right) \\
    & \stackrel{(*)}{=} & \Homgr\left(M(\Sigma),N(\delta\gamma)\right)  \\
    & = & \HOM\left(M(\Sigma),N\right)_{\delta\gamma} \\ 
    & \cong &  \Homgr\left(M(\Sigma)(\gamma^{-1}\delta^{-1}), N\right) \\
    & 	\stackrel{(*)}{=} & \Homgr\left(M(\sigma\gamma^{-1}\delta^{-1}), N\right) \\
    & \cong & \HOM\left(M,N\right)_{\delta\gamma\sigma^{-1}} \\
    & = & \HOM(M,N)_{\Delta\gamma\Sigma^{-1}}.
\end{eqnarray*}
Notice that, if $\sigma$ or $\delta$ as before do not exist, then $(*)$ equals zero.
Hence, $$\HOM\left(M(\Sigma),N(\Delta)\right)_\gamma\cong \HOM(M,N)_{\Delta\gamma\Sigma^{-1}}$$ and the first part of (2) follows.
 And, again by (1), if  
  there exists $\sigma\in\Gamma$ such that $M=M(\sigma)$ as sets (in this case, $\Gamma'_0(M)=\{r(\sigma)\}$ and $\Sigma=\{\sigma\}$), then the isomorphisms above are the identity. 
\end{proof}

\subsection{Graded matrix rings}

	Now we turn our attention to the representation of  homomorphisms with degree of pseudo-free modules  as matrices. For this purpose, given a ring $R$ and a set $I$, we consider two rings of $I\times I$ matrices with entries in  $R$. The first one is 
    $\cfm_I(R)$, consisting of the column finite $I\times I$ matrices with  entries in $R$. The second one is  $\M_I(R)$, consisting of the $I\times I$ matrices with only a finite number of nonzero entries in $R$. In the ungraded context, if $R$ is a unital ring these matrix rings can be identified with certain rings of endomorphisms of free $R$-modules, as the next example shows.

\begin{example}\label{ex: cfm_I and M_I}
    Let $R$ be a unital ring and let $M$ be a free right $R$-module with basis $\mathcal{B}=\{e_i\colon i\in I\}.$ Let $f\in \End_R(M)$. We can associate to $f$ the matrix $\left[f\right]_\mathcal{B}\in \cfm_I(R)$, whose $i$-th column consists of the coordenates of $f(e_i)$ with respect to the basis $\mathcal{B}$ for each $i\in I$. Then the map $\End_R(M)\rightarrow \cfm_I(R)$, given by $f\mapsto \left[f\right]_\mathcal{B}$, defines a ring isomorphism. Under this isomorphism, the subring $\M_I(R)$ of $\cfm_I(R)$ is then identified with the subring of $\End_R(M)$ consisting of the endomorphisms $f$ such that the set $I_f=\{i\in I\colon f(e_i)\neq 0\}$ is finite. Notice that if $I$ is finite, then $\cfm_I(R)=\M_I(R)$. \qed
\end{example}


 In this subsection, we aim to generalize Example~\ref{ex: cfm_I and M_I} to the groupoid graded context. To this end, we have to define first  the graded objects that will play the role of $\cfm_I(R)$ and $\M_I(R)$.

 Let $I$ be a non-empty set and $\overline{\Sigma}=(\Sigma_i)_{i\in I}\in\mathcal{P}(\Gamma)^I$ be a sequence of non-empty subsets.
 We say that $\overline{\Sigma}$ is \emph{d-finite} if, for each $e\in\Gamma_0$, the set $$\{i\in I\colon  d(\sigma)=e \textrm{ for some } \sigma\in\Sigma_i\}$$
 is finite.

 We say that 
 the sequence $\overline{\Sigma}=(\Sigma_i)_{i\in I}\in\mathcal{P}(\Gamma)^I$ is \emph{cf-matricial for $R$} if   $\Sigma_i$ is $d$-unique and $r$-unique for $R$  for each $i\in I$. If, moreover, $\overline{\Sigma}$ is  $d$-finite, we say that $\overline{\Sigma}$ is \emph{matricial for $R$}. Notice that if $I$ is finite, then 	$\overline{\Sigma}$ is cf-matricial for $R$ if and only if it is matricial for $R$.
The sequence $\overline{\Sigma}$ is called \emph{fully matricial for $R$} if it is matricial for $R$ and $\Sigma_i$ is fully $r$-unique for $R$ for all $i\in I$.

	Suppose that $\overline{\Sigma}=(\Sigma_i)_{i\in I}\in\mathcal{P}(\Gamma)^I$ is cf-matricial for $R$. 	
	Let $\cfm_I(R)$ be the set of column finite $I\times I$ matrices with  entries in $R$. For each $\gamma\in \Gamma$, consider the subset 
	$\cfm_I(R)(\overline{\Sigma})_\gamma$ of $\cfm_I(R)$ where
	$$\cfm_I(R)(\overline{\Sigma})_\gamma=\left\{(a_{ij})\in \cfm_I(R)\mid a_{ij}\in R_{\Sigma_i \gamma \Sigma_j^{-1}} \right\}.$$
	Note that each  $\cfm_I(R)(\overline{\Sigma})_\gamma$ is an additive subgroup of $\cfm_I(R)$. Thus
	$$\cfm_I(R)(\overline\Sigma):=\sum_{\gamma\in \Gamma} \cfm_I(R)(\overline{\Sigma})_\gamma$$
	is an additive subgroup of $\cfm_I(R)$. 

Let $i,j\in I$. If $a\in R_{\sigma_i\gamma\tau_j^{-1}}$ for some $\gamma\in \Gamma$, $\sigma_i\in\Sigma_i$ and $\tau_j\in\Sigma_j$ with $d(\sigma_i)=r(\gamma)$ and $d(\gamma)=d(\tau_j)$, then the matrix whose $(i,j)$-entry is $a$ and all its other entries are zero will be denoted by $aE_{ij}$. Note that $aE_{ij}\in\cfm_I(R)(\overline{\Sigma})_\gamma$. As an important special case, we define the \emph{matrix units} $E_{ij}^{e}$  as follows. Suppose that $e\in\Gamma'_0$ and there exist $\sigma_i\in\Sigma_i,\tau_j\in\Sigma_j$ such that $r(\sigma_i)=r(\tau_j)=e$. Then $E_{ij}^{e}:= 1_eE_{ij}\in \cfm_I(R)(\overline{\Sigma})_{\sigma_i^{-1}\tau_j}$.

\begin{lemma}\label{lem:matrixrings}
Let $\overline{\Sigma}=(\Sigma_i)_{i\in I}\in\mathcal{P}(\Gamma)^I$ be a cf-matricial sequence for $R$.
    \begin{enumerate}[\rm(1)]
        \item $\cfm_I(R)(\overline\Sigma) =\bigoplus_{\gamma\in \Gamma} \cfm_I(R)(\overline{\Sigma})_\gamma.$ 
        
        \item  The product in $\cfm_I(R)$ induces a product in $\cfm_I(R)(\overline\Sigma)$ that endows \linebreak $\cfm_I(R)(\overline\Sigma)$ with a natural structure of $\Gamma$-graded ring with identity elements 
        $\mathbb{I}_e:=\sum\limits_{\substack{i\in I_e \\ \sigma_i\in\Sigma_ie}}E_{ii}^{r(\sigma_i)}\in \cfm_I(R)(\overline\Sigma)_e$, where 
        $I_e:=\{i\in I\colon \Sigma_ie \neq\emptyset \}$, for each $e\in\Gamma_0$.
    \end{enumerate}
\end{lemma}
	
\begin{proof}
    (1) It is enough to prove that $\sum\limits_{\gamma\in\Gamma} R_{\Sigma_i\gamma\Sigma_j^{-1}}$ is direct for each $i,j\in I$. Let $i,j\in I$ and $\gamma,\delta\in \Gamma$ such that $R_{\Sigma_i\gamma\Sigma_j^{-1}}$ and $R_{\Sigma_i\delta\Sigma_j^{-1}}$ are nonzero. Let $\sigma_i,\sigma'_i\in \Sigma_i$ be the unique elements in $\Sigma_i$ such that $d(\sigma_i)=r(\gamma)$ and $d(\sigma'_i)=r(\delta)$, and let $\tau_j,\tau'_j\in \Sigma_j$ be the unique elements in $\Sigma_j$ such that $d(\tau_j)=d(\gamma)$ and  $d(\tau'_j)=d(\delta)$. 
    Then, since $\Sigma_i$ and $\Sigma_j$ are $r$-unique for $R$ and $d$-unique, it follows that
    \begin{align*}
        R_{\Sigma_i\gamma\Sigma_j^{-1}}=R_{\Sigma_i\delta\Sigma_j^{-1}}&\iff R_{\sigma_i\gamma\tau_j^{-1}}=R_{\sigma'_i\delta{\tau'_j}^{-1}}\\
        &\iff \sigma_i\gamma\tau_j^{-1}=\sigma'_i\delta{\tau'_j}^{-1}\\
        &\iff \sigma_i=\sigma'_i\textrm{, $\tau_j=\tau'_j$ and $\gamma=\delta$}\\
        &\iff \gamma=\delta\,.
    \end{align*}


		(2) Fix $\gamma,\delta\in \Gamma.$ Let $A\in \cfm_I(R)(\overline{\Sigma})_\gamma$ and
		$B\in \cfm_I(R)(\overline{\Sigma})_\delta$. Then $AB=(c_{ij})\in \cfm_I(R)$, where 
		$$c_{ij}\in \sum_{k\in I}R_{\Sigma_i\gamma\Sigma_k^{-1}}\cdot R_{\Sigma_k\delta\Sigma_j^{-1}}.$$
		
		Let $k\in I$ be such that
		$R_{\Sigma_i\gamma\Sigma_k^{-1}}$ and $R_{\Sigma_k\delta\Sigma_j^{-1}}$ are nonzero. Then
		$R_{\Sigma_i\gamma\Sigma_k^{-1}}=R_{\sigma_i\gamma\rho_k^{-1}}$ and $R_{\Sigma_k\delta\Sigma_j^{-1}}=R_{\rho'_k\delta\tau_j^{-1}}$ for certain
		$\sigma_i\in \Sigma_i$, $\rho_k,\rho'_k\in \Sigma_k$ and $\tau_j\in\Sigma_j$, where $\rho_k$ is the unique element in $\Sigma_k$ with $d(\rho_k)=d(\gamma)$ and
		$\rho'_k$ is the unique element in $\Sigma_k$ with $d(\rho'_k)=r(\delta)$. 
		
		Suppose first that $\gamma\delta$ is defined in $\Gamma$. 
		Now, since
		$d(\gamma)=r(\delta)$, we get that $d(\rho_k)=d(\rho'_k)$ which implies $\rho_k=\rho'_k$ because $\Sigma_k$ is $d$-unique. Hence
		$$R_{\Sigma_i\gamma\Sigma_k^{-1}}R_{\Sigma_k\delta\Sigma_j^{-1}}=R_{\sigma_i\gamma\rho_k^{-1}}R_{\rho_k\delta\tau_j^{-1}} \subseteq R_{\sigma_i\gamma\delta\tau_j^{-1}}=R_{\Sigma_i\gamma\delta\Sigma_j^{-1}}.$$ This implies 
		$AB\in \cfm_I(R)(\overline{\Sigma})_{\gamma\delta}$.
		
		Suppose now that $\gamma\delta$ is not defined in $\Gamma$.  Thus, $d(\gamma)\neq r(\delta)$.
		This implies $d(\rho_k)\neq d(\rho'_k)$.  Since $\Sigma_k$ is $r$-unique for $R$, we obtain that $r(\rho_k)\neq r(\rho'_k)$. Thus		$$R_{\Sigma_i\gamma\Sigma_k^{-1}}R_{\Sigma_k\delta\Sigma_j^{-1}}=R_{\sigma_i\gamma\rho_k^{-1}}R_{\rho'_k\delta\sigma_j^{-1}}=0.$$ This implies $AB=0\in \cfm_I(R)(\overline{\Sigma})_{\gamma\delta}.$

		Now we prove that the  ring is object unital.  Let $e\in\Gamma_0$. 
		Consider the matrix $$\mathbb{I}_e= \sum\limits_{\substack{i\in I_e \\ \sigma_i\in\Sigma_ie}}E_{ii}^{r(\sigma_i)}\in \cfm_I(R)(\overline{\Sigma})_e,$$ consisting of 
		$1_{r(\sigma_{i})}\in R_{r(\sigma_{i})}=R_{\sigma_{i}e\sigma_{i}^{-1}}$ in the $(i,i)$-entry, $i\in I_e$, and zero everywhere else. Let $A=(a_{kl})\in \cfm_I(R)(\overline{\Sigma})_\gamma$ where $r(\gamma)=e$. 
		If $k\notin I_e$, then there is no element $\sigma\in\Sigma_k$ with $d(\sigma)=e$. Thus, $a_{kl}\in R_{\Sigma_k \gamma \Sigma_l^{-1}}=\{0\}$. 
   If $k\in I_e$, then, taking $\sigma_k\in\Sigma_k$ such that $d(\sigma_k)=e$, we obtain $a_{kl}\in R_{\Sigma_k \gamma \Sigma_l^{-1}}=R_{\sigma_k \gamma \Sigma_l^{-1}}$. Thus, $1_{r(\sigma_k)}a_{kl}=a_{kl}$.
  Then, it is routine to show that 
		$\mathbb{I}_e A=A$. In the same way, one can show that $B\mathbb{I}_e=B$ for any 
		$B\in \cfm_I(R)(\overline{\Sigma})_\delta$ with $d(\delta)=e$.
	\end{proof}

	Let $\M_I(R)$ be the set of $I\times I$ matrices with only a finite number of nonzero entries in $R$. Let $\overline{\Sigma}=(\Sigma_i)_{i\in I}\in\mathcal{P}(\Gamma)^I$ be a cf-matricial sequence for $R$. 
 For each $\gamma\in \Gamma$, consider the subset 
	$\M_I(R)(\overline{\Sigma})_\gamma$ of $\M_I(R)\subseteq \cfm_I(R)$ where
	$$\M_I(R)(\overline{\Sigma})_\gamma=\left\{(a_{ij})\in \M_I(R)\mid a_{ij}\in R_{\Sigma_i \gamma \Sigma_j^{-1}} \right\}\subseteq \cfm_I(R)_\gamma.$$
	Note that each  $\M_I(R)(\overline{\Sigma})_\gamma$ is an additive subgroup of $\M_I(R)$  and $\cfm_I(R)_\gamma$. Thus
	$$\M_I(R)(\overline\Sigma):=\bigoplus_{\gamma\in \Gamma} \M_I(R)(\overline{\Sigma})_\gamma$$
	is  an additive subgroup of $\M_I(R)$ and a $\Gamma$-graded additive group, by Lemma \ref{lem:matrixrings}(1).

	\begin{proposition}\label{prop:matrixrings}
	Let $\overline{\Sigma}=(\Sigma_i)_{i\in I}\in\mathcal{P}(\Gamma)^I$ be a matricial sequence for $R$. Then $\cfm(R)(\overline{\Sigma})=\M_I(R)(\overline{\Sigma})$, and therefore the following statements hold true.

	\begin{enumerate}[\rm(1)]
		\item  The product in $\M_I(R)$ induces a product in $\M_I(R)(\overline\Sigma)$ that endows $\M_I(R)(\overline\Sigma)$ with a natural structure of $\Gamma$-graded ring
with identity elements \linebreak $\mathbb{I}_e:=\sum\limits_{\substack{i\in I_e \\ \sigma_i\in\Sigma_ie}}E_{ii}^{r(\sigma_i)}\in \M_I(R)(\overline\Sigma)_e$, where $I_e:=\{i\in I\colon  d(\sigma)=e \textrm{ for some } \sigma\in~\Sigma_i\}$, for each $e\in\Gamma_0$.

		\item If $\overline{\Sigma}$ is fully matricial for $R$, then $\M_I(R)= \M_I(R)(\overline\Sigma)$ as rings. Thus, $\M_I(R)$ can be endowed with a $\Gamma$-graded ring structure.
	\end{enumerate}
		
	\end{proposition}
	
	\begin{proof}
	In order to prove $\cfm_I(R)(\overline{\Sigma})=\M_I(R)(\overline{\Sigma})$, it is enough to show that $\cfm(R)(\overline{\Sigma})_\gamma=\M_I(R)(\overline{\Sigma})_\gamma$ for each $\gamma\in\Gamma$.

For each $\gamma\in\Gamma$, since $\overline{\Sigma}$ is $d$-finite, $I_{r(\gamma)}$ and $I_{d(\gamma)}$ are finite and therefore $R_{\Sigma_i\gamma\Sigma_j^{-1}}\neq 0$ for only a finite number of $i,j\in I$. Thus, $\cfm(R)(\overline{\Sigma})_\gamma\subseteq\M_I(R)(\overline{\Sigma})_\gamma$.


	(1) follows from Lemma~\ref{lem:matrixrings}(2). Note that, since the sequence $\overline{\Sigma}$ is  $d$-finite, it follows that the sets $I_e$ are finite and, therefore,  the matrices $\mathbb{I}_e=\sum\limits_{\substack{i\in I_e\\ \sigma\in\Sigma_ie}}E_{ii}^{r(\sigma_i)}\in\M_I(R)(\overline{\Sigma})_e$ for all $e\in \Gamma_0$.
		
	(2) Let $0\neq a\in R_\gamma$ for some $\gamma\in\Gamma$. 
Since $\overline{\Sigma}$ is fully matricial for $R$, there exist
	$\sigma_i\in \Sigma_i$ and $\tau_j\in \Sigma_j$ such that $r(\sigma_i)=r(\gamma)$ and
	$r(\tau_j)=d(\gamma)$. Then $a\in R_\gamma=R_{\sigma_i (\sigma_i^{-1}\gamma \tau_j) \tau_j^{-1}}$. This shows that 
 $aE_{ij}\in \M_I(R)(\overline{\Sigma})_{\sigma_i^{-1}\gamma \tau_j}$. Therefore any matrix in $\M_I(R)$ can be expressed as sum of a finite number of homogeneous elements in $\M_I(R)(\overline\Sigma)$.
	\end{proof}
	
\begin{remark}
\begin{enumerate}[(1)]
	\item In Proposition~\ref{prop:matrixrings}(1), if the sequence
	$\overline{\Sigma}$ is cf-matricial for $R$, but not matricial for $R$, then $\M_I(R)$ is a $\Gamma$-graded ring which is not object unital.
	\item In Proposition~\ref{prop:matrixrings}(2), if the sequence $\overline{\Sigma}$ is matricial for $R$, but not fully matricial for $R$, then $\M_I(R)\neq \M_I(R)(\overline{\Sigma})$.
	
	\item    
    Assume that $\overline{\Sigma}$ is cf-matricial for $R$ and all $\Sigma_i$ are fully $r$-unique for $R$. Then
    \begin{enumerate}[(a)]
        \item If $I$ is finite, then $\cfm_I(R)(\overline{\Sigma})=\M_I(R)(\overline{\Sigma})=\M_I(R)=\cfm_I(R)$.
        \item If $\supp R$ is finite, then $\M_I(R)(\overline{\Sigma})=\M_I(R)$. \qed
    \end{enumerate}  
\end{enumerate}
\end{remark}

\subsection{Graded endomorphism rings as graded matrix rings}

Now we want to show that the matrix rings just defined correspond to rings of  endomorphisms with degree. For that, we need Proposition~\ref{prop:endomorphism_rings} that will be useful in Section~\ref{sec: art simp} too.

\medskip
 
	Let $I$, $J$ be non-empty sets and $\{M_j\colon j\in J\}$, $\{N_i\colon i\in I\}$ be families of $\Gamma$-graded $R$-modules. Set $M=\bigoplus\limits_{j\in J}M_j$ and $N=\bigoplus\limits_{i\in I}N_i$.
	
	For each $\gamma\in \Gamma$, we will denote by $\Hgr_{I\times J}(M,N)_\gamma$ the set of $I\times J$ column finite matrices $(f_{ij})$ where $f_{ij}\in \HOM(M_j,N_i)_\gamma$ for each $(i,j)\in I\times J$.
	Clearly each   $\Hgr_{I\times J}(M,N)_\gamma$ is an additive group. Thus 
	$$\Hgr_{I\times J}(M,N)=\bigoplus\limits_{\gamma\in\Gamma} \Hgr_{I\times J}(M,N)_\gamma$$
	is a $\Gamma$-graded additive group. Moreover, $\Hgr_{J\times J}(M,M)$ is a $\Gamma$-graded ring with the usual product of matrices. 
 
 That is, given $F=(f_{ij})\in \Hgr_{J\times J}(M,M)_\gamma$ and 
	$G=(g_{ij})\in \Hgr_{J\times J}(M,M)_\delta$, then
	\[FG=\left( h_{ij}  \right)\in \Hgr_{J\times J}(M,M)_{\gamma\delta}\]
	where $h_{ij}=\sum\limits_{k\in J}f_{ik}g_{kj}$. Note that, if $e\in\Gamma_0$, then the matrix whose $(j,j)$-entry is $\mathds{1}_{je}$, the identity element 
 of $\END(M_j)_e$, for each $j$ and zero any other entry is the identity element of $\Hgr_{J\times J}(M,M)_e$.

	\begin{proposition}\label{prop:endomorphism_rings}
Let $I$, $J$ be non-empty sets, $\{M_j\colon j\in J\}$ and $\{N_i\colon i\in I\}$ be families of $\Gamma$-graded $R$-modules and $M=\bigoplus\limits_{j\in J}M_j$, $N=\bigoplus\limits_{i\in I}N_i$. Suppose that, for all $j\in J$ and $g \in \h(\HOM(M_j,N))$, there exists a finite $I_g\subseteq I$ such that $\im g\subseteq\bigoplus\limits_{i\in I_g}N_i$. Consider  the natural inclusions $\rho_j\colon M_j\rightarrow M$, $\rho_i'\colon N_i\rightarrow N$ and the natural projections $\pi_j\colon M\rightarrow M_j$, $\pi'_i\colon N\rightarrow N_i$ for all $i\in I$, $j\in J$.  The following statements hold true.
		\begin{enumerate}[\rm (1)]
			\item The natural map 
			$$\HOM(M,N)\rightarrow \Hgr_{I\times J}(M,N), \quad f\mapsto (\pi'_if\rho_j)_{ij},$$
			is a gr-isomorphism of $\Gamma$-graded additive groups.
			
			\item The natural map 
			$$\END(M)\rightarrow \Hgr_{J\times J}(M,M), \quad f\mapsto (\pi_if\rho_j)_{ij},$$
			is a gr-isomorphism of $\Gamma$-graded rings.
			
			\item If, moreover, for all $e\in \Gamma_0$ there exists a finite $J_e\subseteq J$ such that $M(e)=\bigoplus\limits_{j\in J_e}M_j(e)$, then $\Hgr_{I\times J}(M,N)$ consists of matrices with only a finite number of nonzero entries. 
			
		\end{enumerate}
	\end{proposition}
	
	\begin{proof}
Let
$$\Phi: \HOM_R(M,N)\rightarrow \mathbb{H}_{I\times J}(M,N),\quad
    f\mapsto (\pi'_if\rho_j)_{ij}.$$
First we prove that $\Phi$ is well-defined. For that it is enough to show that $\Phi$ is well defined for homogeneous elements.  Let $\gamma\in\Gamma$ and $f\in \HOM_R(M,N)_\gamma$. For each $i\in I, j\in J$ we have that $\pi'_if\rho_j\in \HOM_R(M_j,N_i)_\gamma$. 
Fix now $j\in J$. 
By hypothesis, there exists a finite $I_{f\rho_j}\subseteq I$ such that $\im (f\rho_j)\subseteq\bigoplus\limits_{i\in I_{f\rho_j}}N_i$.
Hence, $\pi'_if\rho_j\neq0$ implies $i\in I_{f\rho_j}$ and it follows that the sequence $(\pi'_if\rho_j)_{i\in I}$ is almost zero.

Clearly, $\Phi$ is a gr-homomorphism of $\Gamma$-graded abelian groups. Thus, in order to prove (1), it is enough to show that  $\Phi$ is bijective. First we show the injectivity of $\Phi$. Suppose that $f\in\ker\Phi$. Then, $\pi'_if\rho_j=0$, for all $i\in I,j\in J$. It implies
\[f(m)=\sum_{i\in I}\rho'_i\pi'_if(m)=\sum_{i\in I}\rho'_i\pi'_if\left(\sum_{j\in J}\rho_j\pi_j(m)\right)=\sum_{i\in I}\sum_{j\in J}\rho'_i\pi'_if\rho_j\pi_j(m)=0,\]
for all $m\in M$, that is, $f=0$. Now we prove that $\Phi$ is onto. Let $\gamma\in\Gamma$ and $(g_{ij})\in \Hgr_{I\times J}(M,N)_\gamma$.  Since $(g_{ij})$ is a column finite matrix, it follows that, for all $l\in J$, the sum $\sum\limits_{k\in I}\rho'_k g_{kl}\pi_l$ is finite and, therefore, it is an element of $\HOM_R(M,N)_\gamma$. Since, for each $m\in M=\bigoplus\limits_{j\in J}M_j$, we have $\pi_l(m)\neq0$ only for a finite number of $l\in J$, we can define 
\[f=\sum_{l\in J}\left(\sum_{k\in I}\rho'_k g_{kl}\pi_l\right)\in \HOM_R(M,N)_\gamma.\]
Then
\[\Phi(f)=(\pi'_if\rho_j)_{ij}=\left(\sum_{l\in J}\sum_{k\in I}\pi'_i\rho'_k g_{kl}\pi_l\rho_j\right)_{ij}=\left(\pi'_i\rho'_i g_{ij}\pi_j\rho_j\right)_{ij}=(g_{ij})_{ij},\]
which shows (1). 

For the proof of (2), changing $N$ for $M$, it is enough to show that $\Phi$ respects products. 
Indeed, if $f,g\in \END_R(M)$ then
\begin{align*}
\Phi(fg)=(\pi_ifg\rho_j)_{ij}&=\left(\pi_if\left(\sum_{l\in J}\rho_l\pi_l\right)g\rho_j\right)_{ij} \\ &=\left(\sum_{l\in J}\pi_if\rho_l\pi_lg\rho_j\right)_{ij}=\left(\pi_if\rho_j\right)_{ij}\cdot\left(\pi_ig\rho_j\right)_{ij}=\Phi(f)\Phi(g).
\end{align*}

    (3) Since $\Hgr_{I\times J}(M,N)_\gamma$ consists of column finite matrices, to prove this statement it is enough to show that, for each $\gamma\in\Gamma$, if $F=(f_{ij})\in \Hgr_{I\times J}(M,N)_\gamma$, then only a finite number of columns of $F$ is nonzero. There exists only a finite number of $j$'s such that $M_j(d(\gamma))\neq 0$, say $j_1,j_2\dotsc,j_r$. Fix $j\in J$ different from $j_1,\dotsc,j_r$. 
   Let $\alpha \in \Gamma$. If $r(\alpha)\neq d(\gamma)$, then $f_{ij}((M_j)_\alpha)\subseteq (N_i)_{\gamma\alpha}=\{0\}.$ 
    And if $r(\alpha)=d(\gamma)$, then $M_j(r(\alpha))=0$ and $f_{ij}((M_j)_{\alpha})\subseteq f_{ij}(M_j(r(\alpha))=\{0\}$. This shows that $f_{ij}=0$ if $j$ is not one of $j_1,\dotsc,j_r$.
\end{proof}

We say that a $\Gamma$-graded $R$-module $M$ is \emph{$\Gamma_0$-finitely generated} if the $R$-module $M(e)$ is finitely generated for all $e\in \Gamma_0$. This concept will play an important role in the sequel.

 \begin{remark}
 \label{rem: end rings for Gamma0 f.g.}
 \begin{enumerate}[(1)]
     \item In the statement of Proposition~\ref{prop:endomorphism_rings}, the following hypotheses are equivalent:
     \begin{enumerate}[(a)]
     \item For all $j\in J$ and $g \in \h(\HOM(M_j,N))$, there exists a finite $I_g\subseteq I$ such that $\im g\subseteq\bigoplus\limits_{i\in I_g}N_i$.
     \item For all $j\in J$ and $g \in \HOM(M_j,N)$, there exists a finite $I_g\subseteq I$ such that $\im g\subseteq\bigoplus\limits_{i\in I_g}N_i$. 
     \end{enumerate}
     \item The foregoing hypotheses are satisfied when $M_j$ is $\Gamma_0$-finitely generated. This holds because if $g\in \HOM(M_j,N)_\tau$ for some $\tau\in\Gamma$, then $g(M_j(e))=0$ for all  $e\in \Gamma_0\setminus \{d(\tau)\}$, and $g(M_j(d(\tau)))$ is fully determined by the values of $g$ in the generators of the finitely genarated $R$-module $M_j(d(\tau))$. \qed
 \end{enumerate}

 \end{remark}

Below we present a series of results that are consequence of Proposition \ref{prop:endomorphism_rings}.

\begin{corollary}
\label{coro: END(+,+) infinito ortogonal}
Let $I$ be a non-empty set and $\{M_i\colon i\in I\}$ be a family of $\Gamma$-graded $R$-modules such that $\HOM(M_j,M_i)=0$ for  different $i,j\in I$. Consider $M=\bigoplus\limits_{i\in I}M_i$. Then
\[\END(M)\cong_{gr}\sideset{}{^{gr}}\prod_{i\in I}\END(M_i).\]
\end{corollary}

\begin{proof}
    For each $j\in I$ and $g\in \HOM(M_j,M)$, we have $\im g\subseteq M_j$ because $\pi_i\circ g\in\HOM(M_j,M_i)=0$ for all $i\in I\setminus\{j\}$ (where $\pi_i:M\to M_i$ denotes the canonical projection). It follows from Proposition \ref{prop:endomorphism_rings} that $\END(M)\cong_{gr} \Hgr_{I\times I}(M,M)$. But $\Hgr_{I\times I}(M,M)$ consists of diagonal matrices and therefore $$\Hgr_{I\times I}(M,M)\cong_{gr}\bigoplus\limits_{\gamma\in\Gamma}\prod\limits_{i\in I}\END_R(M_i)_\gamma=\sideset{}{^{gr}}\prod\limits_{i\in I}\END_R(M_i),$$
    as desired.
\end{proof}

Let $M$ be a $\Gamma$-graded $R$-module and $\overline{\Sigma}=(\Sigma_i)_{i\in I}\in\mathcal{P}(\Gamma)^I$ be a sequence of non-empty subsets of $\Gamma$. We will denote
\[M(\overline{\Sigma}):=\bigoplus_{i\in I}M(\Sigma_i).\]

\begin{corollary}
\label{coro:graded_endomorphism_ring of modules}
Let $M$ be a $\Gamma$-graded $R$-module and $\overline{\Sigma}=(\Sigma_i)_{i\in I}\in\mathcal{P}(\Gamma)^I$ be a sequence of subsets of $\Gamma$. 
    The following statements hold true.
    \begin{enumerate}[\rm(1)]
        \item If $\overline{\Sigma}$ is cf-matricial for $\END(M)$ and each $M(\Sigma_i)$ is $\Gamma_0$-finitely generated, then 
        $\END (M(\overline{\Sigma})) \cong_{gr} \cfm_I(\END(M))(\overline{\Sigma})$. 
        \item If $\overline{\Sigma}$ is matricial for $\END(M)$, then
        $\END (M(\overline{\Sigma}))\cong_{gr} \M_I(\END(M))(\overline{\Sigma})$.
        If, moreover, $\overline{\Sigma}$ is fully matricial for $\END(M)$, then $\END(M(\overline{\Sigma}))$ is isomorphic to $\M_I(\END(M))$ as graded rings (with the induced grading from $\M_I(\END(M))(\overline{\Sigma})$).  
       \item  If $\overline{\Sigma}$ is cf-matricial for $R$, then $\END(R(\overline{\Sigma}))\cong_{gr}\cfm_I(R)(\overline{\Sigma})$. 
        \item If $\overline{\Sigma}$ is matricial for $R$, then
			$\END (R(\overline{\Sigma}))\cong_{gr} \M_I(R)(\overline{\Sigma})$.
			If, moreover, $\overline{\Sigma}$ is fully matricial for $R$, then $\END(R(\overline{\Sigma}))$ is isomorphic to $\M_I(R)$ as graded rings (with the induced grading from $\M_I(R)(\overline{\Sigma})$). 
    \end{enumerate}
\end{corollary}

\begin{proof}
(1) By Proposition \ref{prop:endomorphism_rings}(2) and Remark \ref{rem: end rings for Gamma0 f.g.}, $\END(M(\overline{\Sigma}))\cong\Hgr_{I\times I}(M(\overline{\Sigma}), M(\overline{\Sigma}))$. For each $(i,j)\in I\times I$ and $\gamma\in\Gamma$, the $(i,j)$-entry of  a matrix in $\Hgr_{I\times I}(M(\overline{\Sigma}), M(\overline{\Sigma}))_\gamma$ is an  element of $\HOM(M(\Sigma_j),M(\Sigma_i))_\gamma$. But by Proposition \ref{prop: HOM(,) e Hom_gr}(2), 
\[\HOM(M(\Sigma_j),M(\Sigma_i))_\gamma\cong\END(M)_{\Sigma_i\gamma\Sigma_j^{-1}}.\]

(2) For each $e\in\Gamma_0$, as $\overline{\Sigma}$ is $d$-finite, we have $M(\Sigma_i)(e)\neq0$ only for the elements of the finite set  
\[I_e:=\{i\in I:d(\sigma)=e \textrm{ for some } \sigma\in\Sigma_i\}.\]
So, for each $e\in\Gamma_0$ we have 
\[M(\overline{\Sigma})(e)=\bigoplus\limits_{i\in I}M(\Sigma_i)(e)=\bigoplus\limits_ {i\in I_{e}}M(\Sigma_i)(e).\] 
In particular, if $\gamma\in\Gamma$ and $g\in \HOM(M(\Sigma_j),M(\overline{\Sigma}))_\gamma$, then $$\im g\subseteq M(\overline{\Sigma})(r(\gamma))=\bigoplus\limits_ {i\in I_{r(\gamma)}}M(\Sigma_i)(r(\gamma))\subseteq \bigoplus\limits_ {i\in I_{r(\gamma)}}M(\Sigma_i).$$
By Proposition \ref{prop:endomorphism_rings}, $\END(M(\overline{\Sigma}))\cong\Hgr_{I\times I}(M(\overline{\Sigma}), M(\overline{\Sigma}))$ and this ring consists of matrices with only a finite number of nonzero entries. Finally, for each $(i,j)\in I\times I$ and $\gamma\in\Gamma$, the $(i,j)$-entry of a matrix in $\Hgr_{I\times I}(M(\overline{\Sigma}), M(\overline{\Sigma}))_\gamma$ is an element of $\HOM(M(\Sigma_j),M(\Sigma_i))_\gamma\cong\END(M)_{\Sigma_i\gamma\Sigma_j^{-1}}$.
The second part of the statement is a consequence of  Proposition~\ref{prop:matrixrings}(2).

(3)
Let  $e\in\Gamma_0$. 	Then, for each $i\in I$,  $R(\Sigma_i)(e)=
\bigoplus\limits_{\sigma\in \Sigma_i}R(\sigma e)$
is a cyclic $R$-module because  $\Sigma_i$ is $d$-unique. Thus, $R(\Sigma_i)$ is $\Gamma_0$-finitely generated. Now apply (1) to the $R$-module $M=R$ and recall that
$\END(R)\cong_{gr} R$ by Lemma~\ref{lem: R=END(R)}.

(4)  By Lemma~\ref{lem: R=END(R)}, $R\cong_{gr}\END(R_R)$. Thus, if we make $M=R_R$ in (2), we obtain $\END (R(\overline{\Sigma}))\cong_{gr}\M_I(\END(R_R))(\overline{\Sigma}) \cong_{gr}\M_I(R)(\overline{\Sigma})$ and the first part of the statement is proved. 

Suppose now that $\overline{\Sigma}$ is fully matricial for $R$. Again, if we make $M=R_R$ in (2), we get $\END(R(\overline{\Sigma}))$ is isomorphic to $\M_I(\END(R_R))$ as graded rings (with the induced grading from $\M_I(\END(R_R))(\overline{\Sigma})$).  Therefore, $\END(R(\overline{\Sigma}))$ is isomorphic to $\M_I(R)$ as graded rings (with the induced grading from $\M_I(R)(\overline{\Sigma})$).
\end{proof}

Before stating the next results we will need some notation.

	Let $\overline{\Sigma}=(\Sigma_i)_{i\in I}\in\mathcal{P}(\Gamma)^I$ be a sequence of subsets.
	When $\Sigma_i=\{\sigma_i\}$ for all $i\in I$, that is, when each $\Sigma_i$, $i\in I$, consists of a unique element, we will denote the sequence by $\overline{\sigma}=(\sigma_i)_{i\in I}$. 
	If $\overline{\Sigma}$ is finite, suppose
	$\overline{\Sigma}=(\Sigma_1,\Sigma_2,\dotsc,\Sigma_n)$, then we will usually write 
	$\M_n(R)(\Sigma_1,\dotsc,\Sigma_n)$ instead of 
	$\M_I(R)(\overline{\Sigma})$. Likewise, when $\overline{\sigma}=(\sigma_1,\dotsc,\sigma_n)$, we will write $\M_n(R)(\sigma_1,\dotsc,\sigma_n)$.

\begin{corollary}
\label{coro: END M(sigma)}
Let $I$ be a non-empty set, $e\in \Gamma_0$  and $\overline{\sigma}=(\sigma_i)_{i\in I}\in (e\Gamma)^I$. The following statements hold true.
\begin{enumerate}[\rm (1)]
    \item Let $M$ be a $\Gamma$-graded $R$-module with $M=M(e)$.   If $\overline{\sigma}$ is $d$-finite, then $\overline{\sigma}$ is fully matricial for $\END(M)$ and therefore
\[\END(M(\overline{\sigma}))\cong_{gr}
\cfm_I(\END(M))(\overline{\sigma})= \M_I(\END(M))(\overline{\sigma})=\M_I(\END(M)),\]
where the last equality is as rings.
\item Suppose  $\supp (R)\subseteq e\Gamma e$.  If $\overline{\sigma}$ is $d$-finite, then $\overline{\sigma}$ is fully matricial  for $R$ and therefore
$$\END(R(\overline{\sigma}))\cong_{gr}
\cfm_I(R)(\overline{\sigma})= \M_I(R)(\overline{\sigma})=\M_I(R),$$
where the last equality is as rings.
\end{enumerate}
\end{corollary}

\begin{proof}
(1) As $\Gamma'_0(\END(M))=\Gamma'_0(M)=\{e\}$, since $\overline{\sigma}$ is $d$-finite, it is fully matricial  for $\END(M)$.   Then the result follows from Proposition~\ref{prop:matrixrings} and Corollary~\ref{coro:graded_endomorphism_ring of modules}.

(2) Follows from (1) and Lemma~\ref{lem: R=END(R)}.
\end{proof}

\begin{corollary}
\label{coro: END M(Sigma)}
Let $\overline{\Sigma}=(\Sigma_1,\dotsc,\Sigma_n)\in\mathcal{P}(\Gamma)^n$ be a sequence of subsets. 
 \begin{enumerate}[\rm (1)]
     \item If $\Sigma_i$ is $d$-unique and $r$-unique for $M$ for all $i=1,\dotsc,n$, then $\overline{\Sigma}$ is matricial for $\END(M)$ and 
     \[\END(M(\overline{\Sigma}))\cong_{gr} \cfm_n(\END(M))(\Sigma_1,\dotsc,\Sigma_n)=\M_n(\END(M))(\Sigma_1,\dotsc,\Sigma_n).\]
 If, moreover, $\Sigma_i$ is fully $r$-unique for $M$ for all $i=1,\dotsc,n$, then we also have the equality as rings $\M_n(\END(M))=\M_n(\END(M))(\Sigma_1,\dotsc,\Sigma_n)$.

     \item If $\Sigma_i$ is $d$-unique and $r$-unique for $R$ for all $i=1,\dotsc,n$, then $\overline{\Sigma}$ is matricial  for $R$ and $$\END(R(\overline{\Sigma}))\cong_{gr} \cfm_n(R)(\Sigma_1,\dotsc,\Sigma_n)=\M_n(R)(\Sigma_1,\dotsc,\Sigma_n).$$
If, moreover, $\Sigma_i$ is fully $r$-unique  for $R$ for all $i=1,\dotsc,n$, then we also have the equality as rings $\M_n(R)=\M_n(R)(\Sigma_1,\dotsc,\Sigma_n)$.
 \end{enumerate}
\end{corollary}

\begin{proof}
Clearly, $\overline{\Sigma}$ is $d$-finite. Now (1)  follows from Proposition~\ref{prop:matrixrings} and Corollary~\ref{coro:graded_endomorphism_ring of modules}.  

(2) follows from (1) together with Lemma~\ref{lem: R=END(R)}.
\end{proof}

 We end this subsection with the following remark.
 
 \begin{remark} 
 \label{rem: R=M1(R)}
 If $\Gamma'_0:=\Gamma'_0(R)$, then
\begin{enumerate}[(1)]
    \item $R\cong_{gr}\M_1(R)(\Gamma'_0)$ via $a\mapsto(a)$;
\item $R\cong_{gr}\M_{\Gamma'_0}(R)(\overline{e})$, where $\overline{e}:=(e)_{e\in\Gamma'_0}$, via $\sum\limits_{e,f\in\Gamma'_0}a_{ef}\mapsfrom(a_{ef})_{e,f\in\Gamma'_0}$.
\end{enumerate}
\end{remark}

\subsection{Some results concerning the opposite ring}\label{subsec:anel_oposto}

 The aim of this section is to present techniques for deriving results about gr-semisimple left 
$R$-modules from those obtained for gr-semisimple right $R$-modules.

The \emph{opposite ring} $R^{op}$ equals the ring $R$ as an additive group, but it is endowed with a new operation given by $$a\cdot^{op}b=ba$$
for all $a,b\in R$. It is not difficult to show that if we define $(R^{op})_\gamma:=R_{\gamma^{-1}}$ for all $\gamma\in\Gamma$, then $R^{op}$ becomes a $\Gamma$-graded ring.

If $_RL$ is a $\Gamma$-graded left $R$-module, we denote by $L^{op}$ the right $R^{op}$-module whose underlying additive group equals $L$, but with multiplication defined by
$$x\cdot^{op}a=ax$$
for all $a\in R$, $x\in L$. Moreover, we  endow $L^{op}$ with a structure of $\Gamma$-graded right $R^{op}$-module defining $(L^{op})_\gamma=L_{\gamma^{-1}}$ for all $\gamma\in\Gamma$.

Let $g\colon L\rightarrow L'$ be a homomorphism of left $R$-modules. Recall that, since 
$L$ is a left $R$-module, the homomorphism $g$ acts on elements of $L$ from the right by convention. We define 
$\hat{g}\colon L^{op}\rightarrow (L')^{op}$ by $\hat{g}(x)=(x)g$ for all $x\in L$. Since 
\begin{equation*}\label{eq:oposto_homomorfismo}
    \hat{g}(x\cdot^{op}a)=(ax)g=a(x)g=\hat{g}(x)\cdot^{op} a
\end{equation*}
for all $x\in L$ and $a\in R$, we get that $\hat{g}$ is a homomorphism of right $R^{op}$-modules. 

\begin{proposition}
\label{prop: oposto}
Let $L$ be a $\Gamma$-graded left $R$-module and $\overline{\Sigma}=(\Sigma_i)_{i\in I}\in\mathcal{P}(\Gamma)^I$ be a sequence of $d$-unique sets. The following assertions hold:
    \begin{enumerate}[\rm (1)]
        \item If $\overline{\Sigma}$ is matricial for $R$, then  $\overline{\Sigma}$ is matricial for $R^{op}$ and $$\M_I(R)(\overline{\Sigma})^{op}\cong_{gr}\M_I(R^{op})(\overline{\Sigma})$$ via the homomorphism defined by  transposition of matrices.
        \item  $(\END_R L)^{op}\cong_{gr}\END_{R^{op}}(L^{op})$ via the map $g\mapsto\hat{g}$.
        \item $L^{op}(\overline{\Sigma})=\left(\left(\overline{\Sigma}^{-1}\right)L\right)^{op}$, where $\overline{\Sigma}^{-1}:=\left(\Sigma_i^{-1}\right)_{i\in I}$.
        \item If $\overline{\Sigma}$ is matricial for $\END_R(L)$, then $\END_R\left(\left(\overline{\Sigma}^{-1}\right)L\right)\cong_{gr}\M_I(\END_R(L))(\overline{\Sigma})$.
        
        \item Suppose that there exists a family $\{R_j:j\in J\}$ of $\Gamma$-graded rings such that $R=\sideset{}{^{gr}}\prod\limits_{j\in J}R_j$. Then $R^{op}=\sideset{}{^{gr}}\prod\limits_{j\in J}(R_j)^{op}$.
            \end{enumerate}
\end{proposition}

\begin{proof}
    (1) The first part follows from $\Gamma'_0(R)=\Gamma'_0(R^{op})$. Let
    \begin{align*}
        \Phi: \M_I(R)(\overline{\Sigma})^{op}&\longrightarrow\M_I(R^{op})(\overline{\Sigma})\\
        A&\longmapsto A^t,
    \end{align*}
    where $A^t$ denotes the transpose matrix of $A$. Let $\gamma\in\Gamma$ and $A=(a_{ij})\in (\M_I(R)(\overline{\Sigma})^{op})_\gamma$. Then $A\in\M_I(R)(\overline{\Sigma})_{\gamma^{-1}}$, that is, $a_{ij}\in R_{\Sigma_i\gamma^{-1}\Sigma_j^{-1}}$ for all $i,j\in I$. Thus, the $(i,j)$-entry of $A^t$ is
    \[a_{ji}\in R_{\Sigma_j\gamma^{-1}\Sigma_i^{-1}}=(R^{op})_{\Sigma_i\gamma\Sigma_j^{-1}}.\]
    Hence, $A^t\in \M_I(R^{op})(\overline{\Sigma})_\gamma$. Therefore, $\Phi$ is well-defined and it is also a gr-homomorphism of additive groups. 
    Given $A,B\in\M_I(R)(\overline{\Sigma})^{op}$, we have
    \[\Phi(A\cdot^{op}B)=\Phi(BA)=(BA)^t\stackrel{(*)}{=}A^tB^t=\Phi(A)\Phi(B),\]
where, in $(*)$, we have used that the $(i,j)$-entry of $(BA)^t$ is
$$\sum_{k\in I}b_{jk}a_{ki}=\sum_{k\in I}a_{ki}\cdot^{op}b_{jk}.$$
 Furthermore, since $\Phi$ fixes diagonal matrices, it follows that $\Phi(\mathbb{I}_e)=\mathbb{I}_e$ for all $e\in\Gamma_0$. Therefore, $\Phi$ is a gr-homomorphism of rings. In order to prove that $\Phi$ is bijective, it is enough to show that matrix transposition also defines a map 
  $\M_I(R^{op})(\overline{\Sigma})\longrightarrow\M_I(R)(\overline{\Sigma})^{op}$.
    Let $\gamma\in\Gamma$ e $B=(b_{ij})\in \M_I(R^{op})(\overline{\Sigma})_\gamma$. Then, for all $i,j\in I$ we have 
    \[b_{ij}\in (R^{op})_{\Sigma_i\gamma\Sigma_j^{-1}}=R_{\Sigma_j\gamma^{-1}\Sigma_i^{-1}}.\]
    Thus, the $(i,j)$-entry of $B^t$ is 
    $b_{ji}\in R_{\Sigma_i\gamma^{-1}\Sigma_j^{-1}}$ and $B^t\in \M_I(R)(\overline{\Sigma})_{\gamma^{-1}}=(\M_I(R)(\overline{\Sigma})^{op})_\gamma$ as desired. 

    (2) First note that if
    $h\colon L^{op}\to L^{op}$ is a homomorphism of right $R^{op}$-modules, the map $\tilde{h}\colon L\rightarrow L$ defined by
    $(x)\tilde{h}=h(x)$ is a homomorphism of left $R$-modules. Indeed,    $$(ax)\tilde{h}=h(x\cdot^{op}a)=h(x)\cdot^{op}a=ah(x)=a(x)\tilde{h}$$
    for all $x\in L$ and $a\in R$. Moreover, $\tilde{\hat{g}}=g$ for each endomorphism of left $R$-modules $g\colon L\rightarrow L$ and
    $\hat{\tilde{h}}=h$ for each endomorphism of right $R^{op}$-modules $h\colon L^{op}\rightarrow L^{op}$.    
 Furthermore
    \begin{align*}
       g\in ((\END_R L)^{op})_\gamma &\Longleftrightarrow g\in (\END_R L)_{\gamma^{-1}}\\ 
        &\Longleftrightarrow (L_{\alpha^{-1}})g\subseteq L_{\alpha^{-1}\gamma^{-1}},\forall \alpha\in\Gamma\\ 
        &\Longleftrightarrow \hat{g}((L^{op})_\alpha)\subseteq (L^{op})_{\gamma\alpha},\forall \alpha\in\Gamma\\
        &\Longleftrightarrow \hat{g}\in \END_{R^{op}}(L^{op})_\gamma. 
    \end{align*}
This implies that the map $(\END_R L)^{op}\to\END_{R^{op}}(L^{op})$,  $g\mapsto\hat{g}$, is well-defined and it is bijective. Finally, observe that if $g_1,g_2\in \END_R (L)$, then
$\widehat{g_2\circ^{op} g_1}=\widehat{g_1g_2}=\hat{g}_2\hat{g}_1$, as desired.

    (3) For each $\gamma\in\Gamma$, we have
    \begin{align*}
        L^{op}(\overline{\Sigma})_\gamma&=\bigoplus_{i\in I}(L^{op})_{\Sigma_i\gamma}\\
        &=\bigoplus_{i\in I}L_{\gamma^{-1}\Sigma_i^{-1}}\\
        &=\bigoplus_{i\in I}((\Sigma_i^{-1})L)_{\gamma^{-1}}\\
        &=((\overline{\Sigma}^{-1})L)_{\gamma^{-1}}\\
        &=(((\overline{\Sigma}^{-1})L)^{op})_\gamma.
    \end{align*}

    (4) If $\overline{\Sigma}$ is matricial for $\END_R(L)$, then, from the previous items and Corollary~\ref{coro:graded_endomorphism_ring of modules}(2), we get
    \begin{align*}
        \left(\END_R\left(\overline{\Sigma}^{-1}\right)L\right)^{op}&\cong_{gr}\END_{R^{op}}(((\overline{\Sigma}^{-1})L)^{op})\\
        &=\phantom{_{gi}}\END_{R^{op}}(L^{op}(\overline{\Sigma}))\\
        &\cong_{gr}\M_I(\END_{R^{op}}(L^{op}))(\overline{\Sigma})\\
        &\cong_{gr}\M_I((\END_R(L))^{op})(\overline{\Sigma})\\
        &\cong_{gr}\M_I(\END_R(L))(\overline{\Sigma})^{op}.
    \end{align*}
    Therefore, $\END_R((\overline{\Sigma}^{-1})L)\cong_{gr}\M_I(\END_R(L))(\overline{\Sigma})$.

    (5) Just notice that for each $\gamma\in\Gamma$ we have 
    \[(R^{op})_{\gamma}=R_{\gamma^{-1}}=\prod_{j\in J}(R_j)_{\gamma^{-1}}=\prod_{j\in J}((R_j)^{op})_\gamma=\left(\sideset{}{^{gr}}\prod\limits_{j\in J}(R_j)^{op}\right)_\gamma.\qedhere\]
    
\end{proof}

\subsection{Some isomorphisms between categories of graded modules}\label{subsec:modules_graded_groupoid}

\emph{Throughout this subsection, let $\Gamma$ be a groupoid}.

Let $R$ be a $\Gamma$-graded ring. The category whose objects are the $\Gamma$-graded right (resp. left) $R$-modules and morphisms are gr-homomorphisms will be denoted by $\Gamma-\grR R$ (resp. $\Gamma-R\Sgr$). If $e\in\Gamma_0$ is fixed, then the full subcategory of $\Gamma-\grR R$ (resp. $\Gamma-R\Sgr$) whose objects are the $\Gamma$-graded right (resp. left) $R$-modules $M$ such that $\supp M\subseteq e\Gamma$ (resp. $\supp M\subseteq\Gamma e$) will be denoted by $e\Gamma-\grR R$ (resp. $\Gamma e-R\Sgr$). We denote by $\nGamma-\modR$-$R$ (resp. $\Gamman-R\Rmod$) the category whose objects are the $\Gamma$-graded right (resp. left) $R$-modules $M$ for which there exists $\varepsilon_M\in\Gamma_0$ satisfying $\supp M\subseteq \varepsilon_M\Gamma$ (resp. $\supp M\subseteq\Gamma \varepsilon_M$) and, for objects $M$ and $N$, the set of morphisms from $M$ to $N$ is $\HOM_R(M,N)$. Note that if $M=M(e)$ for some $e\in\Gamma_0$, then $\mathds{1}_e$ is the unity of the ring $\HOM_R(M,M)$. In some cases, we will refer to objects of $e\Gamma-\grR R$ as $e\Gamma$-graded $R$-modules.


 We aim to describe categories of left modules using categories of right modules.
 One way to proceed is to induce gradings on opposite rings and modules using $\Gamma$  as in the previous subsection. An alternative approach, which we adopt here, is to work with $\Gamma^{op}$, the opposite category of $\Gamma$,  which  is again a groupoid and is called the \emph{opposite groupoid} of $\Gamma$.
It comes equipped with a natural bijection $\Gamma \rightarrow \Gamma^{op}$, given by $\gamma \mapsto \gamma^o$, where $d(\gamma^o)=r(\gamma)^o$ and $r(\gamma^o)=d(\gamma)^o$. Furthermore, if $\gamma,\delta\in\Gamma$ are such that $\gamma\delta$ is defined in $\Gamma$, then $(\gamma\delta)^o=\delta^o\gamma^o$ and $(\gamma^{-1})^o=(\gamma^o)^{-1}$.

 Let $R$ be a $\Gamma$-graded ring. The next result considers the ring $R^{op}$ as a $\Gamma^{op}$-graded ring via the grading $(R^{op})_{\gamma^{o}}:=R_\gamma$. In this context, if $L$ is a $\Gamma$-graded left $R$-module, then $L^{op}$ will be regarded as a $\Gamma^{op}$-graded right $R^{op}$-module via $(L^{op})_{\gamma^{o}}:=L_\gamma$. Note that
\[(L^{op})_{\gamma^{o}}\cdot^{op}(R^{op})_{\delta^{o}}=R_\delta L_\gamma\subseteq L_{\delta\gamma}=(L^{op})_{(\delta\gamma)^{o}}=(L^{op})_{\gamma^{o}\delta^{o}}.\]

\begin{proposition}
\label{prop: categorias de mod a esq}
    Let $R$ be a $\Gamma$-graded ring. We have the following category isomorphisms:
    \begin{enumerate}[\rm (1)]
        \item $\Gamma-R\Sgr\cong \Gamma-\grR R^{op}\cong \Gamma^{op}-\grR R^{op}$.
        \item $\Gamma e-R\Sgr\cong e\Gamma-\grR R^{op}\cong e^{o}\Gamma^{op}-\grR R^{op}$.
        \item $\Gamman-R\Rmod\cong \nGamma-\modR$-$R^{op}\cong \nGamma^{op}-\modR$-$R^{op}$.
    \end{enumerate}
\end{proposition}

\begin{proof}
    In all three cases, we define a functor from the corresponding category of $\Gamma$-graded left $R$-modules to the category of $\Gamma$-graded (or $\Gamma^{op}$-graded) right $R$-modules. This functor is defined on  objects by $L\mapsto L^{op}$, and sends any morphism 
    $g:L\to L'$ to the morphism $\hat{g}: L^{op}\rightarrow (L')^{op}$ defined on Section~\ref{subsec:anel_oposto}. In the third case, observe that if $g$ is of degree $\gamma$  in $\Gamman-R\Rmod$, then $\hat{g}$ is of degree $\gamma^{o}$  in $\nGamma^{op}-\modR$-$R^{op}$ and of degree $\gamma^{-1}$ in $\nGamma-\modR$-$R^{op}$.
    \end{proof}

 Now, we turn our attention to categories of modules graded by connected groupoids. In view of Proposition \ref{prop: categorias de mod a esq}, we will focus on right modules.

\begin{lemma}
\label{lem:eG_fG_modules}
Let $R$ be a $\Gamma$-graded ring. Let $e,f\in \Gamma_0$ such that there exists $\sigma\in \Gamma$ with 
$d(\sigma)=f$, $r(\sigma)=e$. Then the categories $e\Gamma-\grR R$ and $f\Gamma-\grR R$ are isomorphic. 
Hence, if $\Gamma$ is connected,  $e\Gamma-\grR R$ and $f\Gamma-\grR R$ are isomorphic for all $e,f\in\Gamma_0$. 
\end{lemma}

\begin{proof}
Let $M,N\in e\Gamma-\grR R$. Then $M(\sigma),N(\sigma)\in f\Gamma-\grR R$. If $h\in\Hom_{gr}(M,N)$, then $h_\sigma\colon M(\sigma)\rightarrow N(\sigma)$ defined by 
$h_\sigma(x)=h(x)$ for all $x\in M(\sigma)_\delta=M_{\sigma\delta}$, $\delta\in f\Gamma$, is such that $h_\sigma \in \Hom_{gr}(M(\sigma),N(\sigma))$. Hence 
\[
\begin{array}{rcl}
T_\sigma\colon e\Gamma-\grR R &\rightarrow & f\Gamma-\grR R \\
M & \mapsto & M(\sigma) \\
h\in \Hom_{gr}(M,N) &\mapsto & h_\sigma\in \Hom_{gr}(M(\sigma),N(\sigma))
\end{array}
\]
is a functor with inverse $T_{\sigma^{-1}}$. 
\end{proof}

Let $\Gamma$ be a connected groupoid and fix an idempotent $e_0\in\Gamma_0$. Consider the group $G:=e_0\Gamma e_0$. Set $\sigma_{e_0}=e_0$ and,  for each $e\in\Gamma_0\setminus\{e_0\}$, pick $\sigma_e\in\Gamma$ with $d(\sigma_e)=e$ and $r(\sigma_e)=e_0$. Thus, for each $\gamma\in\Gamma$, there exists a unique $g\in G$ such that $\gamma=\sigma_{r(\gamma)}^{-1}g\sigma_{d(\gamma)}$.  For each $\gamma\in\Gamma$, let $g_\gamma:=\sigma_{r(\gamma)}\gamma \sigma^{-1}_{d(\gamma)}\in G$. Then
\begin{equation}
\label{eq:iso_connected_groupoid}
\begin{array}{ccc}
	\Gamma & \longrightarrow  & \Gamma_0\times G\times\Gamma_0 \\
	\gamma & \longmapsto & (r(\gamma),g_\gamma,d(\gamma)) \\
    \sigma_e^{-1} g\sigma_f & \longmapsfrom &  (e,g,f)
\end{array}
\end{equation}
is an isomorphism of groupoids, see for example \cite[p.125]{Brown}. Of course, if $\Gamma$ is a groupoid of the form $I\times G\times I$ for some group $G$ and set $I$, there is a natural choice of  isomorphism in \eqref{eq:iso_connected_groupoid}
making $\sigma_e=(e_0,1_G,e)$ for all $e\in I\equiv\Gamma_0$. 
We remark on passing that  we obtain a relation between rings graded by a connected groupoid and categories graded by a group from Example~\ref{ex: aneis graduados2}(2) and \eqref{eq:iso_connected_groupoid}.



\medskip

Keep in mind the context of the previous paragraph and let $R=\bigoplus_{\gamma\in\Gamma}R_\gamma$ be a $\Gamma$-graded ring.  We proceed to define the category of right $R$-modules
graded by the group $G$, as well as the category of right $R$-modules graded by 
$G\times \Gamma_0$ (both notions will be explicitly defined below). We then relate these categories with $e\Gamma-\grR R$ in Proposition~\ref{prop:G_graded_eGamma_graded}. Later, in Corollary~\ref{coro: semisimplicty of R, eGamma, G and GxGamma)}, we will use this Proposition to relate the gr-semisimplicity of the $\Gamma$-graded ring $R$ with that of the objects in those categories.

The objects of the category $G-\grR R$ are the $G$-graded $R$-modules. That is, right $R$-modules $M$ such that $MR=M$ and for which there exists a family $\{M_g\colon g\in G\}$ of additive subgroups of $M$ such that $M=\bigoplus_{g\in G}M_g$ as additive groups and 
$M_gR_\gamma \subseteq M_{gg_\gamma}$ for each $\gamma\in\Gamma$.
If $N$ is another $G$-graded right $R$-module, a homomorphism of $G$-graded modules is a homomorphism of modules $p\colon M\rightarrow N$ such that $p(M_g)\subseteq N_g$ for all $g\in G$. An example of a $G$-graded $R$-module is $R(e)$, where $e\in\Gamma_0$, via $R(e)_g:=\bigoplus_{f\in\Gamma_0}R_{\sigma_e^{-1}g\sigma_f}$ for each $g\in G$.

The objects of the category $(G\times\Gamma_0)-\grR R$ are the  $(G\times \Gamma_0)$-graded $R$-modules $M$.  That is,  right $R$-modules $M$ such that $MR=M$ and for which there exists a family of additive subgrups $\{M_{(g,e)}\colon (g,e)\in G\times\Gamma_0\}$ such that $M=\bigoplus\limits_{(g,e)\in G\times\Gamma_0}M_{(g,e)}$ as additive groups and,  for each $\gamma\in\Gamma$,
$M_{(g,e)}R_\gamma\subseteq M_{(gg_\gamma,d(\gamma))}$ if $r(\gamma)=e$ and $M_{(g,e)}R_\gamma=0$ otherwise. 
Given  $(G\times \Gamma_0)$-graded  $R$-modules $M, N$, a homomorphism  of $(G\times\Gamma_0)$-graded modules is a homomorphism of $R$-modules $p\colon M\rightarrow N$ such that $p(M_{(g,e)})\subseteq N_{(g,e)}$ for all $(g,e)\in G\times\Gamma_0$. An example of a $(G\times \Gamma_0)$-graded $R$-module is $R(e)$, where $e\in\Gamma_0$, via $R(e)_{(g,f)}:=R_{\sigma_e^{-1}g\sigma_f}$ for each $(g,f)\in G\times\Gamma_0$.

\begin{proposition}
\label{prop:G_graded_eGamma_graded}
    Let $\Gamma$ be a connected groupoid and
    $R=\bigoplus_{\gamma\in\Gamma}R_\gamma$  be a $\Gamma$-graded ring. Then the categories $G-\grR R$, $(G\times \Gamma_0)-\grR R$ and $e\Gamma-\grR R$ are isomorphic for any $e\in\Gamma_0$ and any  $e_0\in \Gamma_0$, $G:=e_0\Gamma e_0$ and $\{\sigma_e\}_{e\in\Gamma_0}$ as in \eqref{eq:iso_connected_groupoid}. 
\end{proposition}

\begin{proof}
Fix $e_0\in \Gamma_0$, $G:=e_0\Gamma e_0$ and $\{\sigma_e\}_{e\in\Gamma_0}$ as in \eqref{eq:iso_connected_groupoid}.

First we show that $G-\grR R$ and $(G\times \Gamma_0)-\grR R$ are isomorphic.

A $G$-graded right $R$-module $M=\bigoplus_{g\in G}M_g$ has a natural structure of $(G\times\Gamma_0)$-graded module. Indeed, for each $g\in G$ and $f\in\Gamma_0$,  consider the identity element $1_f\in R_f$ and define 
\[M_{(g,f)}=M_g1_f.\]
Note that $M_{(g,f)}=M_g1_f\subseteq M_{gg_f}=M_{ge_0}=M_g$, for all $g\in G$, $f\in\Gamma_0$. 
Let $x\in M_g$. Since $MR=M$, then $x=\sum_{s=1}^tx_{s}a_{s}$ for some $x_s\in M_{g_s}$, $a_s\in R_{\gamma_s}$ such that
$g_sg_{\gamma_s}=g$ for each $s=1,\dotsc,t$. Note that $x_sa_s\in M_{g_s}R_{\gamma_s}1_{d(\gamma_s)}\subseteq M_g1_{d(\gamma_s)}=M_{(g,d(\gamma_s))}$ for each $s$.  Hence $M_g=\sum_{f\in\Gamma_0}M_{(g,f)}$. Moreover, since $\{1_f\}_{f\in\Gamma_0}$ is an orthogonal set of idempotents, we get $M_g=\bigoplus_{f\in\Gamma_0}M_{(g,f)}$. Therefore $M=\bigoplus\limits_{(g,f)\in G\times\Gamma_0}M_{(g,f)}$. Moreover 
$$M_{(g,f)}R_\gamma=M_g1_fR_\gamma\subseteq M_{gg_\gamma}1_{d(\gamma)}= M_{(gg_\gamma,d(\gamma))}$$ 
if $r(\gamma)=f$ and zero otherwise. Also, it is not difficult to realize that any homomorphism of $G$-graded $R$-modules is in fact a homomorphism of $(G\times \Gamma_0)$-graded $R$-modules.

Conversely, any $(G\times\Gamma_0)$-graded right $R$-module $M=\bigoplus\limits_{(g,f)\in G\times\Gamma_0}M_{(g,f)}$ can be regarded as a $G$-graded module defining $M_g=\bigoplus_{f\in \Gamma_0}M_{(g,f)}$. And any homomorphism of $(G\times\Gamma_0)$-graded $R$-modules is a homomorphism of $G$-graded modules.

Now we fix an $e\in\Gamma_0$ and show that $(G\times \Gamma_0)-\grR R$ and $e\Gamma-\grR R$ are isomorphic.

Let  $M=\bigoplus_{(g,f)\in G\times\Gamma_0}M_{(g,f)}$ be a $(G\times\Gamma_0)$-graded $R$-module. Set $\widetilde{M}:=\bigoplus_{\gamma\in e\Gamma}\widetilde{M}_\gamma$ where $\widetilde{M}_\gamma=M_{(g_\gamma,d(\gamma))}$ for each $\gamma\in e\Gamma$.  Since for each $(g,f)\in G\times \Gamma_0$ there exists a unique $\gamma\in e\Gamma f$ such that $(g,f)=(g_\gamma,f)$, then $\widetilde{M}=M$ as additive groups. Define the product of elements of $\widetilde{M}$ by elements of $R$ in the natural way, that is, if $x\in \widetilde{M}_\gamma=M_{(g_\gamma,d(\gamma))}$ and $a\in R_{\delta}$ for some $\gamma\in e\Gamma$ and $\delta\in \Gamma$, then $xa$ is the product given by action of $R$ in $M$. Thus, $xa\in M_{(g_\gamma g_\delta,d(\delta))}=\widetilde{M}_{\gamma\delta}$ if $d(\gamma)=r(\delta)$ and $xa=0$ otherwise. Thus $\widetilde{M}$ is an $e\Gamma$-graded right $R$-module. Moreover, if $h\colon M\rightarrow N$ is a homomorphism of  $(G\times\Gamma_0)$-graded right $R$-modules, then $\tilde{h}\colon \widetilde{M}\rightarrow \widetilde{N}$, where $\tilde{h}=h$ as maps, is a homomorphism of $e\Gamma$-graded right $R$-modules. 

Conversely, given an $e\Gamma$-graded right $R$-module $\widetilde{X}=\bigoplus_{\gamma\in e\Gamma} \widetilde{X}_\gamma$, then 
$X=\bigoplus_{(g,f)\in G\times\Gamma_0}X_{(g,f)}$ where $X_{(g,f)}=\widetilde{X}_\gamma$, where $\gamma$ is the unique element in $e\Gamma f$ such that $g=g_\gamma$, is a $(G\times\Gamma_0)$-graded right $R$-module. Also, any homomorphism of $e\Gamma$-graded right $R$-modules $\widetilde{X}\rightarrow \widetilde{Y}$ can be regarded as a homomorphism of $(G\times \Gamma_0)$-graded $R$-modules $X\to Y$. 	
\end{proof}


\subsection{Gradings on retangular matrix groups}
\label{subsec: matrizes retangulares}

\medskip

Let $I,J$ be non-empty sets. 
We denote by $\M_{I\times J}(R)$ the additive group consisting of the $I\times J$ matrices with entries in $R$ and with at most a finite number of nonzero entries. Suppose, moreover, that $\supp(R)\subseteq e\Gamma e$ for some $e\in\Gamma_0$. If $\overline{\sigma}=(\sigma_{i})_{i\in I}\in (e\Gamma)^{I},$ $\overline{\tau}=(\tau_{j})_{j\in J}\in (e\Gamma)^{J}$. We will denote by $\M_{I\times J}(R)(\overline{\sigma})(\overline{\tau})$ the $\Gamma$-graded additive group whose homogeneous component of degree $\gamma\in\Gamma$, $\M_{I\times J}(R)(\overline{\sigma})(\overline{\tau})_\gamma$, is the subset of $M_{I\times J}(R)$ consisting of the matrices whose $(i,j)$-entry belongs to $R_{\sigma_i\gamma\tau_j^{-1}}$ for all $i\in I, j\in J$. 
Note that $\sum_{\gamma\in\Gamma}\M_{I\times J}(R)(\overline{\sigma})(\overline{\tau})_\gamma$ is, in fact, a direct sum because, for each $i\in I$ and $j\in J$, the sum $\sum_{\gamma\in\Gamma}R_{\sigma_i\gamma\tau_j^{-1}}$ is direct. Furthermore, if $\gamma\in\Gamma$ and $0\neq a\in R_\gamma$, then, for each $i\in I$ and $j\in J$, $aE_{ij}\in \M_{I\times J}(R)(\overline{\sigma})(\overline{\tau})_{\sigma_i^{-1}\gamma\tau_j}$. Therefore, this indeed defines a grading of $\M_{I\times J}(R)$ by $\Gamma$. 
Note that 
\[\M_{I\times I}(R)(\overline{\sigma})(\overline{\sigma})=\M_I(R)(\overline{\sigma}).\]
If $I$ and $J$ are finite with $|I|=m,|J|=n\in\mathbb{Z}_{>0}$, we will also write $\M_{m\times n}(R)$, and $\M_{m\times n}(R)(\overline{\sigma})(\overline{\tau})$. 
If $K$ is a non-empty subset of $I$, we will write $\overline{\sigma}_K$ to denote the sequence $(\sigma_i)_{i\in K}\in (e\Gamma)^K$.

\begin{lemma}
\label{lem: M_n(D)_gamma}
    Let $I$ be a non-empty set, $e_0\in\Gamma_0$ such that $\supp(R)\subseteq e_0\Gamma e_0$ and $\overline{\sigma}=(\sigma_i)_{i\in I}\in (e_0\Gamma)^I$ be a $d$-finite sequence. For each $e\in\Gamma_0$, consider the finite set $I_e:=\{i\in I: d(\sigma_i)=e\}$. Then, for all $e,f,g\in\Gamma_0$, $\gamma\in e\Gamma f$ and $\delta\in f\Gamma g$, there exist isomorphisms of additive groups
    \[\phi_\gamma:\M_I(R)(\overline{\sigma})_\gamma\longrightarrow \M_{|I_e|\times|I_f|}(R)(\overline{\sigma}_{I_e})(\overline{\sigma}_{I_f})_\gamma\]
    such that the following diagram is commutative
    {\footnotesize
    \[\xymatrix{ 
    \M_I(R)(\overline{\sigma})_\gamma \times \M_I(R)(\overline{\sigma})_\delta\ar[r]^(0.35){(\phi_\gamma,\phi_\delta)}\ar[d]^{}  &    \M_{|I_e|\times|I_f|}(R)(\overline{\sigma}_{I_e})(\overline{\sigma}_{I_f})_\gamma \times \M_{|I_f|\times|I_g|}(R)(\overline{\sigma}_{I_f})(\overline{\sigma}_{I_g})_\delta  \ar[d]^{}\\   
    \M_I(R)(\overline{\sigma})_{\gamma\delta} \ar[r]_{\phi_{\gamma\delta}}   &  \M_{|I_e|\times|I_g|}(R)(\overline{\sigma}_{I_e})(\overline{\sigma}_{I_g})_{\gamma\delta} } \] 
    }where vertical arrows indicate matrix multiplication.
    In particular, there exist gr-isomorphisms of $\Gamma$-graded additive groups 
    \[\phi_{e,f}:\bigoplus_{\gamma\in e\Gamma f}\M_I(R)(\overline{\sigma})_\gamma\longrightarrow \M_{|I_e|\times|I_f|}(R)(\overline{\sigma}_{I_e})(\overline{\sigma}_{I_f}).\]
\end{lemma}

\begin{proof}
    For each $e,f\in\Gamma_0$ and $\gamma\in e\Gamma f$,  $\M_I(R)(\overline{\sigma})_\gamma =(R_{\sigma_i\gamma\sigma_j^{-1}})_{ij}$. Moreover, $R_{\sigma_i\gamma\sigma_j^{-1}}\neq 0$ implies that ${\sigma_i\gamma\sigma_j^{-1}}$ is defined,  and thus, $i\in I_e$ and $j\in I_f$. Therefore, we obtain a natural isomorphism of additive groups
    \[\phi_\gamma:\M_I(R)(\overline{\sigma})_\gamma\longrightarrow \M_{|I_e|\times|I_f|}(R)(\overline{\sigma}_{I_e})(\overline{\sigma}_{I_f})_\gamma\] 
    that associates to a matrix of degree $\gamma$, the matrix obtained considering just the entries in $I_e\times I_f$ because all other entries are zero. Then $\phi_{e,f}:=\bigoplus_{\gamma\in e\Gamma f}\phi_\gamma$ is well-defined and a gr-isomorphism of $\Gamma$-graded additive groups. The commutativity of the diagram follows by the way matrices are multiplied.  
\end{proof}


\section{Graded division rings}
\label{sec: gr-div rings}

Throughout this section, let $\Gamma$ be a groupoid.

\subsection{General facts on graded division rings}

 Let $D=\bigoplus\limits_{\gamma\in\Gamma} D_\gamma$ be a $\Gamma$-graded ring. 

 We say that $D$ is a \emph{gr-domain} or a \emph{graded domain}  if  $D\neq \{0\}$ and for all $\gamma,\delta\in\Gamma$ with
$d(\gamma)=r(\delta)$ and nonzero elements $a\in D_\gamma$, $b\in D_\delta$,  we have
$ab\neq 0$.
 We say that $D$ is a \emph{gr-division ring} or a \emph{graded division ring} if $D\neq \{0\}$ and for all $\gamma\in\supp D$ and nonzero $a\in D_\gamma$, there exists an element $a^{-1}\in D_{\gamma^{-1}}$ such that $aa^{-1}=1_{r(\gamma)}$ and
 $a^{-1}a=1_{d(\gamma)}$. In \cite[Section~2]{Verhulst}, the element $a^{-1}$ is called the $\Gamma$-inverse of $a$, but we prefer to say that $a^{-1}$ is the \emph{inverse} of $a$.  We will also write that $a$ is invertible (with inverse $a$). Note that no confusion will arise with the usual concept of invertibility in ring theory because we will not deal with it. 
A gr-division ring is a gr-domain. Indeed, given $\gamma,\delta\in\Gamma$ with
$d(\gamma)=r(\delta)$ and elements $a\in D_\gamma$, $b\in D_\delta$, suppose that 
$ab=0$. If $b\neq 0$, then 
$0=abb^{-1}=a1_{d(\gamma)}=a$. 

It is important to note that if $D$ is a gr-division ring and $a,b\in\h(D)$ are such that $ab\neq 0$, then $(ab)^{-1}=b^{-1}a^{-1}$.

We say that $D$ is a \emph{gr-prime ring} if for all nonzero graded ideals $I$ and $J$ of $D$, we have $IJ\neq0$. Equivalently, $D$ is a gr-prime ring if and only if, for all $a,b\in \h(D)\setminus\{0\}$, we have $aDb\neq0$. $D$ will be called a \emph{gr-simple ring} if $D\neq\{0\}$ and its only graded ideals are $\{0\}$ and $D$. Clearly, every gr-simple ring is a gr-prime ring.

 In general, a gr-division ring can have nonzero graded  ideals different from $D$, even if $\Gamma$ is connected. For example, let $\Gamma=\{1,2,3,4\}\times \{1,2,3,4\}$ and
 $F$ be a field. Then 
 $$D=\begin{pmatrix} F & F & 0 & 0 \\
 F & F & 0 & 0 \\
 0 & 0 & F & F \\
 0& 0 & F & F\end{pmatrix}$$
is a $\Gamma$-graded division ring, via $D_{(i,j)}:=E_{ii}DE_{jj}$ for all $1\leq i,j\leq 4$, but 
$$I= \begin{pmatrix} F & F & 0 & 0 \\
 F & F & 0 & 0 \\
 0 & 0 & 0 & 0 \\
 0& 0 & 0 & 0\end{pmatrix},\quad J=\begin{pmatrix} 0 & 0 & 0 & 0 \\
 0 & 0 & 0 & 0 \\
 0 & 0 & F & F \\
 0& 0 & F & F\end{pmatrix}$$
are graded ideals of $D$ such that $IJ=0$. Note that
$D$ is not strongly graded because $D_{(1,3)}D_{(3,1)}=\{0\}\neq D_{(1,1)}$. Also notice that $I$ and $J$ are $\Gamma$-graded division rings.

\medskip

Let $D=\bigoplus\limits_{\gamma\in\Gamma}D_\gamma$ be a gr-division ring. We define the \emph{gr-primality relation} on $\Gamma_0':=\Gamma'_0(D)$ in the following way.
For $e,f\in\Gamma_0'$
$$e\sim f\ \ \textrm{ if and only if }\ \ 1_eD1_f\neq 0.$$
Equivalently, $e\sim f$ if and only if $e\Gamma f\cap \supp(D)\neq \emptyset$  if and only if $\supp(1_eD1_f)\neq\emptyset$.

\begin{proposition}
\label{prop: quando D e'simples}
Let $D$ be a $\Gamma$-graded division ring. The following assertions hold:
\begin{enumerate}[\rm (1)]
    \item $\sim$ is an equivalence relation. 

    \item If we define 
$$\Gamma_{[e]}=\left\{\gamma\in\Gamma \mid d(\gamma),r(\gamma)\in [e]\right\},\quad D_{[e]}=\bigoplus\limits_{\gamma\in \Gamma_{[e]}} D_\gamma$$
for each equivalence class $[e]\in \Gamma_0'/\sim$, then 
$D_{[e]}$ is a nonzero graded ideal of $D$ and $D_{[e]}D_{[f]}=\{0\}$ for $[e]\neq [f]\in \Gamma_0'/\sim$. 

\item $D=\bigoplus\limits_{[e]\in \Gamma_0/\sim}D_{[e]}$ and  $D_{[e]}$ is a gr-simple gr-division ring for each equivalence class $[e]\in \Gamma_0'/\sim$

\item The following assertions are equivalent about $D$.
\begin{enumerate}[\rm (a)]
    \item $D$ is gr-simple
    \item $D$ is gr-prime
    \item $\Gamma_0'/\sim$ possesses  only one equivalence class.
\end{enumerate}
\end{enumerate}

\end{proposition}

\begin{proof}
(1) The relation $\sim$ is reflexive because $1_e\in 1_eD1_e$ for all $e\in\Gamma_0'$. It is also symmetric. Indeed,  suppose that $1_eD1_f\neq 0$ for $e,f\in\Gamma_0'$, then there exists a nonzero element $a\in D_\gamma$ for some $\gamma\in e\Gamma f$. Now $a^{-1}\in D_{\gamma^{-1}}$ belongs to $1_fD1_e$.  Suppose that $e,f,g\in\Gamma_0'$ are such that $e\sim f$ and $f\sim g$. Then there exist nonzero homogeneous elements 
$a\in 1_eD1_f$, $b\in 1_fD1_g$. Since $D$ is a gr-domain, we obtain that the relation $\sim$ is transitive.

(2) It is enough to show that $D_{[e]}D_{[e]}\subseteq D_{[e]}$ and $D_{[e]}D_{[f]}=0$ for all $e,f\in\Gamma'_0$ with $[e]\neq[f]$. Thus, 
fix different $[e],[f]\in {\Gamma'_0}/{\sim}$. Given $\gamma,\gamma'\in\Gamma$ such that $r(\gamma),d(\gamma),r(\gamma'),d(\gamma')\in[e]$, we have  $D_\gamma D_{\gamma'}\subseteq D_{\gamma\gamma'}\subseteq D_{[e]}$. This shows $D_{[e]}D_{[e]}\subseteq D_{[e]}$.  Now let $\delta\in\Gamma$ be such that $r(\delta),d(\delta)\in[f]$. If we had $D_\gamma D_\delta\neq0$, we would have $D_{\gamma\delta}\neq0$ and therefore,  $1_{r(\gamma)}D1_{d(\delta)}\neq0$ and consequently $e\sim r(\gamma)\sim d(\delta)\sim f$, a contradiction. Hence, $D_\gamma D_\delta=0$ for all $\gamma,\delta\in\Gamma$ such that $r(\gamma),d(\gamma)\in[e]$ and $r(\delta),d(\delta)\in[f]$. That is, $D_{[e]}D_{[f]}=0$. 

(3) Since $\Gamma_{[e]}\cap \Gamma_{[f]}=\emptyset$ if $e\nsim f$, and $\supp D\subseteq\operatornamewithlimits{\bigcup}\limits_{[e]\in \Gamma_0'/\sim} \Gamma_{[e]}$, we obtain that $D=\bigoplus\limits_{[e]\in \Gamma_0'/\sim}D_{[e]}$. 
Fix $e\in\Gamma'_0$. Since $\sim$ is an equivalence relation, if $\gamma\in \Gamma_{[e]}$, then 
$\gamma^{-1}\in\Gamma_{[e]}$. Thus, $D_{[e]}$ is a $\Gamma$-graded (or even a $\Gamma_{[e]}$-graded) division ring.
We now show that $D_{[e]}$ is gr-simple. Let $\gamma, \delta \in \Gamma_{[e]}$ and
$a\in D_\gamma$, $b\in D_\delta$ with $a\neq 0$. We will prove that any graded ideal containing $a$ must contain $b$ too.  First note that $d(\gamma)\sim r(\delta)$. Hence there exists a nonzero homogeneous element  $u\in 1_{d(\gamma)}D1_{r(\delta)}$. Then
$b= (au)^{-1}a(ub)$, as desired.

(4) (a)$\implies$(b): It was observed above right after the definition of gr-simple ring.

(b)$\implies$(c): This implication follows from (2). 

(c)$\implies$(a): This implication follows from (3).
\end{proof}

There exists a way of describing gr-prime components of a gr-division ring as crossed products. It can be done as in \cite[Proposition~2.7]{Verhulst} with maps

$$\alpha\colon \supp D_{[e]}\times \supp D_{[e]}\rightarrow D_e\setminus\{0\} \textrm{ and  } \sigma\colon \supp  D_{[e]}\rightarrow Aut(D_e)$$
satisfying the properties in that paper. In the next result, we describe the  gr-prime components of a gr-division ring using a group graded division ring and the rings of matrices introduced in Section~\ref{sec:rings_of_matrices}. 

\begin{theorem}
\label{theo: anel com div primo = anel de matr}
Let $D$ be a $\Gamma$-graded division ring which is gr-prime. Fix $e\in\Gamma_0'(D)$ and set
$H=1_eD1_e$. Then $H$ is a $e\Gamma e$-graded division ring and 
$D\cong_{gr}\M_{\Gamma_0'(D)}(H)(\overline{\sigma})$ where $\overline{\sigma}=(\sigma_f)_{f\in\Gamma'_0(D)}\in\prod\limits_{f\in\Gamma'_0(D)}\supp (1_eD1_f)$.

Conversely, let $e\in\Gamma_0$, $H$ be an $e\Gamma e$-graded division ring and $\overline{\sigma}\in\prod\limits_{f\in\Delta_0}e\Gamma f$ for some
$\Delta_0\subseteq \Gamma_0$. Then
$D:=\M_{\Delta_0}(H)(\overline{\sigma})$ is a gr-prime $\Gamma$-graded division ring with $\Gamma_0'(D)=\Delta_0$.
\end{theorem}

\begin{proof}
    $H$ is clearly an $e\Gamma e$-graded division ring. 
    For each $f\in\Gamma_0'(D)$, fix $\sigma_f\in e\Gamma f$ such that $D_{\sigma_f}\neq \{0\}$ and
    $0\neq u_{f}\in D_{\sigma_f}$. Notice that $\overline{\sigma}=(\sigma_f)_{f\in \Gamma'_0(D)}$ is $d$-finite because for each $e_0\in\Gamma_0$ there exists at most one $f\in \Gamma'_0(D)$ such that $d(\sigma_f)=e_0$. Thus $\overline{\sigma}$ is matricial for $H$ and $\M_{\Gamma'_0(D)}(H)(\overline{\sigma})$ is a $\Gamma$-graded ring by Proposition~\ref{prop:matrixrings}.

    Let $\gamma\in\supp D$. Then $d(\gamma),r(\gamma)\in \Gamma_0'(D)$. If $a\in D_\gamma$, set
    $$h_a=u_{r(\gamma)}au_{d(\gamma)}^{-1}\in D_{\sigma_{r(\gamma)}\gamma \sigma_{d(\gamma)}^{-1}}\subseteq H.$$
Thus, $h_aE_{r(\gamma)d(\gamma)}\in \M_{\Gamma'_0(D)}(H)(\overline{\sigma})_\gamma.$ We define $\Phi(a)=h_aE_{r(\gamma)d(\gamma)}$ for each $a\in D_\gamma$, $\gamma\in\supp D$. Extending $\Phi$ by additivity, we obtain a gr-homomorphism of $\Gamma$-graded additive groups $\Phi\colon D\rightarrow \M_{\Gamma'_0(D)}(H)(\overline{\sigma})$. If $a\in D_\gamma$, $b\in D_\delta$ with $\gamma,\delta\in \supp D$, then
\begin{align*}
    \Phi(a)\Phi(b) &=  h_a E_{r(\gamma)d(\gamma)}h_b E_{r(\delta)d(\delta)} \\
    &=  \left\{\begin{array}{ll}
0     & \textrm{if } d(\gamma)\neq r(\delta)  \\
h_ah_b E_{r(\gamma)d(\delta)}     & \textrm{if } d(\gamma)=r(\delta) 
\end{array} \right.  \\
    & \stackrel{(*)}{=}  \left\{\begin{array}{ll}
0     & \textrm{if } d(\gamma)\neq r(\delta)  \\
h_{ab} E_{r(\gamma\delta)d(\gamma\delta)}     & \textrm{if } d(\gamma)=r(\delta) 
\end{array} \right. \\
    &=  \Phi(ab),
\end{align*}
where we have used that, when $d(\gamma)=r(\delta)$,
\begin{align*}
    h_ah_b & = u_{r(\gamma)}au_{d(\gamma)}^{-1}u_{r(\delta)}bu_{d(\delta)}^{-1} \\
    & = u_{r(\gamma)}abu_{d(\delta)}^{-1} \\
    & = u_{r(\gamma\delta)}abu_{d(\gamma\delta)}^{-1}\\
    & = h_{ab}
\end{align*}
Moreover, for each $f\in\Gamma_0'(D)$, $\Phi(1_f)=u_{f}1_fu_{f}^{-1}E_{ff}=1_eE_{ff}=\mathbb{I}_f$, the identity element of $\M_{\Gamma_0'(D)}(H)(\overline{\sigma})_f$.
Therefore $\Phi$ is a gr-homomorphism of rings. It remains to show that $\Phi$ is bijective. For that, we construct an inverse $\Psi\colon \M_{\Gamma'_0(D)}(H)(\overline{\sigma})\rightarrow D$  as follows. 
Note that if $\gamma\in\supp \M_{\Gamma'_0(D)}(H)(\overline{\sigma})$, then $\sigma_f\gamma\sigma_{f'}^{-1}$ is defined in $\Gamma$ for some $f,f'\in\Gamma'_0(D)$. But this happens if and only if $f=r(\gamma)$ and $f'=d(\gamma)$, i.e., $r(\gamma),d(\gamma)\in\Gamma'_0(D)$. If $h\in D_{\sigma_{r(\gamma)}\gamma\sigma_{d(\gamma)}^{-1}}$, then $hE_{r(\gamma)d(\gamma)}\in \M_{\Gamma'_0(D)}(H)(\overline{\sigma})_\gamma$ and we define $\Psi(hE_{r(\gamma)d(\gamma)})=u_{r(\gamma)}^{-1} h u_{d(\gamma)}\in D_\gamma$.

Conversely, let $e\in\Gamma_0$,  $\Delta_0\subseteq\Gamma_0$, $H$ be an $e\Gamma e$-graded division ring and 
\[\overline{\sigma}=(\sigma_f)_{f\in \Delta_0}\in\prod\limits_{f\in\Delta_0}e\Gamma f.\] 
By construction, $\overline{\sigma}$ is (fully) matricial for $H$ and, by Proposition~\ref{prop:matrixrings}, $D:=\M_{\Delta_0}(H)(\overline{\sigma})$ is a $\Gamma$-graded ring. Notice that if $\gamma\in\Gamma$ with either $d(\gamma)\notin \Delta_0$ or $r(\gamma)\notin\Delta_0$, then $H_{\sigma_{f_1}\gamma \sigma_{f_2}^{-1}}=0$ for all $f_1,f_2\in \Delta_0$. This implies $D_\gamma=0$. Now, if $\gamma\in\supp D$, there exist unique $f_1=r(\gamma), f_2=d(\gamma)\in\Delta_0$ such that $\sigma_{f_1}\gamma \sigma_{f_2}^{-1}$ is defined. Thus the homogeneous matrices of $D$ have at most one nonzero entry. In particular, if $f\in\Gamma_0$, we have $\mathbb{I}_f=0$ if $f\notin\Delta_0$ and, if $f\in\Delta_0,$ $\mathbb{I}_f=1_eE_{ff}$, that is the matrix with $1_e$ in the $(f,f)$-entry and zero everywhere else.  From this we obtain $\Gamma'_0(D)=\Delta_0$. Now if $A\in D_\gamma\setminus\{0\}$, then $A=aE_{f_1f_2}$ where $f_1=r(\gamma)$, $f_2=d(\gamma)$ and $0\neq a\in H_{\sigma_{f_1}\gamma\sigma_{f_2}^{-1}}$. This homogeneous matrix is invertible with $A^{-1}=a^{-1}E_{f_2f_1}$. Indeed, $A^{-1}A=1_eE_{f_2f_2}=\mathbb{I}_{d(\gamma)}$ and $AA^{-1}=1_eE_{f_1f_1}=\mathbb{I}_{r(\gamma)}$.
Now $D$ is gr-prime because, for all $a,b\in\h(H)\setminus\{0\}$ and $f_1,f_2,f_3,f_4\in\Delta_0$, we have $0\neq E^e_{f_1f_4}=(aE_{f_1f_2})(a^{-1}b^{-1}E_{f_2f_3})(bE_{f_3f_4})\in(aE_{f_1f_2})D(bE_{f_3f_4})$.
\end{proof}

More generally than the converse in Theorem~\ref{theo: anel com div primo = anel de matr} one can show, with a similar proof, the following result.

\begin{remark}
     If $H$ is a (gr-prime) gr-division ring and $\overline{\Sigma}:=(\Sigma_i)_{i\in I}\in\mathcal{P}(\Gamma)^I$ is a matricial sequence for $H$ such that, for all $f\in\Gamma_0$ there exists at most one $\sigma\in\bigcup_{i\in I}\Sigma_i$ satisfying $d(\sigma)=f$, then $D:=\M_I(H)(\overline{\Sigma})$ is a (gr-prime) $\Gamma$-graded division ring with $\Gamma_0'(D)=\bigcup_{i\in I}\{d(\sigma_i):\sigma_i\in \Sigma_i\}$.\qed
\end{remark}

\subsection{Graded modules over gr-division rings}
In this subsection, our aim is to show that the behaviour of graded modules over groupoid graded division rings is similar to the one  graded modules over group graded division rings.

The proofs of  Theorem \ref{theo:modules_over_division_rings} and Corollary \ref{coro:pseudodimension_submodules} follow very much the pattern for group graded division  rings \cite[Section~1.4]{Hazrat}.

\begin{theorem}\label{theo:modules_over_division_rings}
Let $D$ be a $\Gamma$-graded division ring and $M$ be a $\Gamma$-graded $D$-module. The following assertions hold:
\begin{enumerate}[\rm (1)]
    \item $M$ is pseudo-free.
    \item Any pseudo-linearly independent sequence of $M$ can be extended to a pseudo-basis of $M$.
    \item Any two pseudo-basis of $M$ have the same cardinality.
    \end{enumerate}
\end{theorem}

\begin{proof}
First we show (2). Let $(x_i)_{i\in I}\in \prod\limits_{i\in I} M_{\gamma_i}$ be a pseudo-linearly independent sequence of homogeneous elements in $M$. Consider the set 

$$\mathcal{F}=\left\{(x_i)_{i\in Q}\in \prod\limits_{i\in Q}M_{\gamma_i}\mid I\subseteq Q,\, (x_i)_{i\in Q} \textrm{ is pseudo-linearly independent} \right\}.$$
Clearly $\mathcal{F}$ is not empty and it is a partially ordered set with inclusion and every chain has an upper bound. By Zorn's lemma, $\mathcal{F}$ has a maximal element. Let $(x_i)_{i\in K}\in\prod\limits_{i\in K}M_{\gamma_i}$ be one such maximal element. Let $N$ be the graded $D$-submodule of $M$ generated by $\{x_i:i\in K\}$. Suppose that $N\neq M$. Thus there exists $x_0\in M_{\gamma_0}\setminus N$ for some $\gamma_0 \in \Gamma$. We claim that the sequence $(x_i)_{i\in K\cup\{0\}}$
is pseudo-linearly independent. Indeed, suppose that there exists a sequence of homogeneous elements
$(a_i)_{i\in K\cup\{0\}}\in \bigoplus\limits_{i\in K\cup\{0\}} 1_{d(\gamma_i)}D$ such that
$x_0a_0+\sum\limits_{i\in K}x_ia_i=0$. If $a_0\neq 0,$ then
$x_0=-\sum\limits_{i\in K}x_ia_ia_0^{-1}\in N$, a contradiction. Thus, $a_0=0$. Hence  $a_i=0$ for all $i\in K$, and the  claim is proved. But the claim contradicts the maximality of $(x_i)_{i\in K}$ in $\mathcal{F}$. Therefore $M=N$ and $(x_i)_{i\in K}$ is a pseudo-basis of $M$, as desired.

(1) Suppose that $M\neq \{0\}$. There exist $\gamma \in \Gamma$ and $ x\in M_\gamma$ such that $x\neq 0$. If there exists a homogeneous element $a\in 1_{d(\gamma)}D$ such that
$xa=0$, then $a\neq 0$ would imply that $x=x1_{d(\gamma)}=x aa^{-1}=0$. Thus $\{x\}$ is a pseudo-linearly independent sequence and we can extend it to obtain a pseudo-basis of $M$ by (2). Hence, $M$ is a pseudo-free $D$-module.

 (3) If $M$ has a pseudo-basis consisting on an infinite number of elements, then  Lemma \ref{lem: generation} imply the result. Thus, we can suppose that the pseudo-bases of $M$ have a finite number of elements. 
Let $(x_i)_{i=1}^n\in \prod\limits_{i=1}^n M_{\gamma_i}$ and $(y_j)_{j=1}^p\in \prod\limits_{j=1}^p M_{\delta_j}$ be two pseudo-bases of $M$. We will show that $n=p$. Let $(a_i)_{i=1}^n$ be the unique sequence in $\bigoplus\limits_{i=1}^n D_{\gamma_i^{-1}\delta_1}$  such that
$$y_1=x_1a_1+\dotsb+x_na_n.$$
Since $y_1\neq 0,$ there exists $1\leq i_1\leq n$ such that $a_{i_1}\neq 0$. Then
$$x_{i_1}=y_1a_{i_1}^{-1}-(x_1a_1+\cdots+x_{i_1-1}a_{i_1-1}+x_{i_1+1}a_{i_1+1}+\cdots+x_na_n)a_{i_1}^{-1},$$
and therefore $(y_1)\cup ((x_i)_{i=1}^n\setminus \{x_{i_1}\})$ generates $M$. If $p=1$, then we have $p\leq n$. If $p>1$, there exist
$b_1\in D_{\gamma_1^{-1}\delta_{2}},\dotsc, b_{i_1}\in D_{\delta_1^{-1}\delta_{2}}, \dotsc, b_n\in D_{\gamma_n^{-1}\delta_{2}}$ such that 
\[y_{2}=y_1b_{i_1}+x_1b_1+\cdots+x_{i_1-1}b_{i_1-1}+x_{i_1+1}b_{i_1+1}+\cdots+x_nb_n.\]
Since $(y_1,y_2)$  is pseudo-linearly independent, there exists
$i_2\in \{1,\dotsc,n\}\setminus \{i_1\}$ such that $b_{i_2}\neq 0$. As before,
$x_{i_2}$ is a $D$-linear combination
of $\{y_{1},y_2,x_1,\dotsc,x_n\}\setminus \{x_{i_1},x_{i_2}\}$. After the $k$-th step of this process of changing one $x_{i_j}$ by one $y_j$, we obtain that
$\{y_1,\dotsc,y_k,x_1\dotsc,x_n\}\setminus \{x_{i_1},\dotsc,x_{i_k}\}$ generates $M$. Thus, $n<p$ implies that $\{y_1,\dotsc,y_n\}$ generates $M$. But this contradicts the fact that $\{y_1,\dotsc,y_p\}$ is a pseudo-basis of $M$. Therefore $p\leq n$. 
An analogous argument shows that $n\leq p$.
\end{proof}

If $D$ is a $\Gamma$-graded division ring, then every $\Gamma$-graded $D$-module has a pseudo-basis by  Theorem~\ref{theo:modules_over_division_rings}(1). Unlike in the group graded context \cite[Theorem~3.3]{BalabaMik}, the converse is not true. We will deal with this situation in Section~\ref{sec:Pseudo_divsision_rings}.

\begin{corollary}
\label{coro: ext de homo de D mod}
    Let $D$ be a $\Gamma$-graded division ring and $M$, $N$ be $\Gamma$-graded $D$-modules. For all 
     $(\gamma_i)_{i\in I}\in\Gamma^I$, pseudo-linear independent sequence $(x_i)_{i\in I}\in\prod\limits_{i\in I} M_{\gamma_i}$, $\sigma\in\Gamma$ and $(y_i)_{i\in I}\in \prod\limits_{i\in I}N_{\sigma\gamma_i}$, there exists $g\in\HOM_D(M,N)_\sigma$ such that $g(x_i)=y_i$ for each $i\in I$.
\end{corollary}

\begin{proof}
   Follows from Theorem~\ref{theo:modules_over_division_rings}(2) and Proposition~\ref{prop:pseudo_free}(6).
\end{proof}

By Theorem~\ref{theo:modules_over_division_rings}, every graded module $M$ over a $\Gamma$-graded division ring $D$ has a pseudo-basis and any two pseudo-basis have the same number of elements. Such cardinality will be called the \emph{pseudo-dimension} of $M$ and it will be denoted by $\pdim_D(M)$.

\begin{corollary}\label{coro:pseudodimension_submodules}
Let $D$ be a $\Gamma$-graded division ring and $M$ be a $\Gamma$-graded $D$-module. If $N$ is a graded submodule of $M$, then
 $$\pdim_D(N)+\pdim_D(M/N)=\pdim_D(M).$$
\end{corollary}

\begin{proof}
Let $(x_i)_{i\in I}\in \prod\limits_{i\in I}N_{\gamma_i}$ be a pseudo-basis of $N$, which exists by Theorem~\ref{theo:modules_over_division_rings}(1). By  Theorem~\ref{theo:modules_over_division_rings}(2), it can be extended to a pseudo-basis $(x_i)_{i\in I\cup J}\in \prod\limits_{i\in I\cup J}M_{\gamma_i}$ of $M$ with $I$ and $J$ disjoint sets. It is enough to show that
$(x_i+N)_{i\in J}$ is a pseudo-basis of $M/N$. Since $(x_i)_{i\in I\cup J}$ generates $M$, $(x_i+N)_{i\in J}$ generates $M/N$. If a sequence of elements $(a_i)_{i\in J}\in \bigoplus 1_{d(\gamma_i)}D$ is such that 
$\sum\limits_{i\in J}(x_i+N)a_i=0$, then $\sum\limits_{i\in J} x_i a_i\in N$. But then, there exists a sequence of elements $(b_i)_{i\in I}\in \bigoplus 1_{d(\gamma_i)}D$ such that
$\sum\limits_{i\in J} x_i a_i=\sum\limits_{i\in I} x_i b_i\in N$. If $(a_i)_{i\in J}\neq 0$, this is a contradiction with the fact that $(x_i)_{i\in I\cup J}$ is a pseudo-basis of $M$. 

Now clearly, $\pdim_D(M)=|I\cup J|=|I|+|J|=\pdim_D(N)+\pdim_D(M/N)$.
\end{proof}

\begin{remark}\label{rem:basis_from_generators} Let $D$ be a $\Gamma$-graded division ring and $M$ be a $\Gamma$-graded $D$-module.
    It can be shown that if a sequence o homogeneous elements $(x_i)_{i\in I}\in \prod\limits_{i\in I} M_{\gamma_i}$ generates $M$, there exists a subset $J\subseteq I$ such that $(x_i)_{i\in J}$ is a pseudo-basis of $M$. Indeed, any subsequence $(x_i)_{i\in J}$ maximal among the subsequences of $(x_i)_{i\in I}$ that are pseudo-linearly independent works. \qed
\end{remark}

\subsection{Graded linear algebra over gr-division rings}

 Now we turn our attention to gr-homomorphisms of finitely generated pseudo-free modules.
We showed in Proposition~\ref{prop:pseudo_free} that each finitely generated pseudo-free $R$-module is gr-isomorphic to a graded module of the form $\bigoplus\limits_{i=1}^mR(\alpha_i)$.

Let $R$ be a $\Gamma$-graded ring.
Set $\overline{\alpha}=(\alpha_1,\dotsc,\alpha_m)\in\Gamma^m$, 
$\overline{\beta}=(\beta_1,\dotsc,\beta_n)\in\Gamma^n$, $M=\bigoplus\limits_{j=1}^nR(\beta_j)$ and $N=\bigoplus\limits_{i=1}^mR(\alpha_i)$. Set also $$\M_{m\times n}(R)[\overline{\alpha}][\overline{\beta}]=\left\{\begin{pmatrix}
    a_{11} & \cdots & a_{1n} \\
    \vdots & \ddots & \vdots \\
    a_{m1} & \cdots & a_{mn}
\end{pmatrix} \in  \M_{m\times n}(R)\ \bigg|\ a_{ij}\in R_{\alpha_i\beta_j^{-1}}\right\},$$ where we follow the convention $R_{\alpha_i\beta_j^{-1}}=\{0\}$ if $d(\alpha_i)\neq r(\beta_j^{-1})=d(\beta_j)$.

It is important to note that if $A\in \M_{m\times n}(R)[\overline{\alpha}][\overline{\beta}]$ and $B\in \M_{n\times p}(R)[\overline{\beta}][\overline{\tau}]$ for some 
$\overline{\tau}\in\Gamma^p$, then $AB\in \M_{m\times p}(R)[\overline{\alpha}][\overline{\tau}]$.

One can show as in the proof of 
Proposition~\ref{prop:endomorphism_rings} and Corollary~\ref{coro: END M(Sigma)}(2) that \linebreak
$\Homgr(M,N)$ is isomorphic to the additive group of matrices $\M_{m\times n}(R)[\overline{\alpha}][\overline{\beta}].$
Moreover, if  $\overline{\alpha}=\overline{\beta}$, one obtains the isomorphism of rings
$\Endgr(M)\cong \M_{n\times n}(R)[\overline{\beta}][\overline{\beta}]$. But we will prove this in a more traditional way.

    Let $M$ be a $\Gamma$-graded $R$-module. Suppose that $M$ is a finitely generated pseudo-free module with pseudo-basis $\mathcal{B}=(u_j)_{j=1}^n\in \prod\limits_{j=1}^nM_{\beta_j^{-1}}$ for some $\beta_1,\beta_2,\dotsc,\beta_n\in\Gamma$. If $x\in M_\delta$, there exists a unique $(a_j)_{j=1}^n\in\bigoplus\limits_{j=1}^nR_{\beta_j\delta}$ such that 
$$x=u_1a_1+u_2a_2+\dotsb+u_na_n.$$
We will write $$(x)_\mathcal{B}=\begin{pmatrix}
    a_1 \\ a_2\\ \vdots\\ a_n
\end{pmatrix}.$$
Note that $(a_1,\dotsc,a_n)\in (R(\beta_1)\oplus\dots\oplus R(\beta_n))_\delta$. 
Furthermore, the map 
$$M\longrightarrow R(\beta_1)\oplus\dots\oplus R(\beta_n), \  x\mapsto (x)_{\mathcal{B}},$$  defines a gr-isomorphism of modules.

Let  $N$ be a $\Gamma$-graded pseudo-free $R$-module
with pseudo-basis $\mathcal{C}=(v_i)_{i=1}^m\in\prod\limits_{i=1}^m N_{\alpha_i^{-1}}$ for some $\alpha_1,\alpha_2,\dotsc,\alpha_m\in\Gamma$. Suppose that $f\in\Homgr(M,N)$. Then
$$(f(u_j))_{\mathcal{C}}=\begin{pmatrix}
    a_{1j} \\ a_{2j} \\ \vdots \\ a_{mj}
\end{pmatrix}\in (R(\alpha_1)\oplus \dotsb \oplus R(\alpha_m))_{\beta_j^{-1}}=R_{\alpha_1\beta_j^{-1}}\oplus \dotsb\oplus R_{\alpha_m\beta_j^{-1}}.$$
Define $[f]_{\mathcal{B}\mathcal{C}}\in \M_{m\times n}(R)[\overline{\alpha}][\overline{\beta}]$  as the matrix
$$[f]_{\mathcal{BC}}=\Big[(f(u_1))_\mathcal{C}\ (f(u_2))_{\mathcal{C}}\ \dotsb (f(u_n))_{\mathcal{C}} \Big]=\begin{pmatrix}
    a_{11} & a_{12} & \hdots & a_{1n} \\
    a_{21} & a_{22} & \hdots & a_{2n} \\
    \vdots & \ddots & \ddots & \vdots \\
    a_{m1} & a_{m2} & \hdots & a_{mn}
\end{pmatrix}.$$
Then it is routine to show that
\begin{equation}\label{eq:matrix_of_homomorphism}
    [f]_{\mathcal{BC}}\cdot (x)_\mathcal{B} = (f(x))_{\mathcal{C}}\ \textrm{ for all } x\in \h(M)
\end{equation}
and that $[f]_{\mathcal{BC}}$ is the unique matrix in $\M_{m\times n}(R)[\overline{\alpha}][\overline{\beta}]$ that satisfies \eqref{eq:matrix_of_homomorphism}. 

Conversely, if $A\in \M_{m\times n}(R)[\overline{\alpha}][\overline{\beta}]$, then there exists a unique $f\in\Homgr(M,N)$ such that $A=[f]_{\mathcal{BC}}$.

Moreover, if $L$ is a pseudo-free module with pseudo-basis $\mathcal{A}=(t_k)_{k=1}^p\in\prod\limits_{k=1}^p L_{\gamma_k^{-1}}$ and $g\in \Homgr(L,M)$, then
$$[f]_{\mathcal{BC}}\cdot [g]_{\mathcal{AB}}=[f\circ g]_{\mathcal{AC}}.$$
All in all, we have proved
\begin{proposition}\label{prop:matrices_grhomomorphisms}
Let $R$ be a $\Gamma$-graded ring.
Let $M$ and $N$ be pseudo-free $R$-modules with corresponding pseudo-bases $$\mathcal{B}=(u_j)_{j=1}^n\in \prod\limits_{j=1}^nM_{\beta_i^{-1}} \ \textrm{ and } \ \mathcal{C}=(v_i)_{i=1}^m\in\prod\limits_{i=1}^m N_{\alpha_i^{-1}}$$ for some $\beta_1,\beta_2,\dotsc,\beta_n,\alpha_1,\alpha_2,\dotsc,\alpha_m\in\Gamma$, respectively.
 Then the map $f\mapsto [f]_{\mathcal{BC}}$ defines a gr-isomorphism 
$\Homgr(M,N)\rightarrow \M_{m\times n}(R)[\overline{\alpha}][\overline{\beta}].$ 
Moreover, if  $M=N$ and $\mathcal{B}=\mathcal{C}$, we obtain the gr-isomorphism of rings 
$\Endgr(M)\rightarrow \M_{n\times n}(R)[\overline{\beta}][\overline{\beta}]$, $f\mapsto [f]_{\mathcal{BB}}$. \qed
\end{proposition}

\bigskip




Let now $I_{r(\overline{\alpha})}\in \M_{m\times m}(R)[\overline{\alpha}][\overline{\alpha}]$ be the matrix whose $(i,i)$-entry is $1_{r(\alpha_i)}$ and whose $(i,j)$-entry, with $i\neq j$, is zero for all
$i,j\in \{1,\dotsc,m\}$. Note that $I_{r(\overline{\alpha})}A=A$ for all $A\in \M_{m\times n}(R)[\overline{\alpha}][\overline{\beta}]$ and that the matrix  $I_{r(\overline{\alpha})}$ corresponds to the identity of $N:=\bigoplus_{i=1}^mR(\alpha_i)$ in the gr-isomorphism $\Endgr(N)\rightarrow \M_{m\times m}(R)[\overline{\alpha}][\overline{\alpha}]$ of Proposition~\ref{prop:matrices_grhomomorphisms}.
We say that $A\in \M_{n\times n}(R)[\overline{\alpha}][\overline{\beta}]$ is \emph{invertible} if there exists $B\in \M_{n\times n}(R)[\overline{\beta}][\overline{\alpha}]$ such that $AB=I_{r(\overline{\alpha})}$, $BA=I_{r(\overline{\beta})}$ and $1_{r(\alpha_i)},1_{r(\beta_i)}\neq 0$ for all $i=1,\dotsc,m$. Notice that such $B$ corresponds to the inverse gr-homomorphism of the one represented by $A$ in the gr-isomorphism $\Homgr(M,N)\rightarrow \M_{m\times n}(R)[\overline{\alpha}][\overline{\beta}]$ of Proposition~\ref{prop:matrices_grhomomorphisms}. Because of this,  such matrix $B$ is unique and it will be called the \emph{inverse} of $A$.

\begin{corollary}\label{coro:theo_of_grhomomorphism}
  Let $D$  be a $\Gamma$-graded division ring,   $M$ and $N$ be $\Gamma$-graded $D$-modules and $f\colon M\rightarrow N$ a gr-homomorphism. Then
  $$\pdim(M)=\pdim(\ker f )+\pdim(\im f).$$
If, moreover, $\pdim(M)=\pdim(N)=n<\infty$, then $f$ is a gr-isomorphism if and only if either $f$ is surjective or $f$ is injective. 
  
  As a consequence, if $A\in \M_{n\times n}(D)[\overline{\alpha}][\overline{\beta}]$ and $B\in \M_{n\times n}(D)[\overline{\beta}][\overline{\alpha}]$ for some $\overline{\alpha},\overline{\beta}\in\Gamma^n$, then $$AB=I_{r(\overline{\alpha})} \iff
  BA=I_{r(\overline{\beta})}.$$
\end{corollary}

\begin{proof}
 Since $M/\ker(f)\cong_{gr} \im f$, Corollary~\ref{coro:pseudodimension_submodules} implies the first part. 

For the second part, suppose that $f$ is surjective. Then $\im f=N$. Hence
$$n=\pdim(M)=\pdim(\ker f)+\pdim(N)=\pdim(\ker f)+n.$$
Thus, $\pdim(\ker f)=0$ and $\ker f=0$. If $f$ is  injective, the proof is analogous. 
 
 For the last part, note that if 
$\overline{\alpha}=(\alpha_1,\dotsc,\alpha_n),\overline{\beta}=(\beta_1,\dotsc,\beta_n)\in\Gamma^n$, then $A\in \M_{n\times n}(D)[\overline{\alpha}][\overline{\beta}]$ and $B\in \M_{n\times n}(D)[\overline{\beta}][\overline{\alpha}]$ can be regarded as gr-homomorphisms 
$$R(\beta_1)\oplus\dotsb\oplus R(\beta_n)\rightarrow R(\alpha_1)\oplus \dotsb \oplus R(\alpha_n),$$ $$R(\alpha_1)\oplus \dotsb \oplus R(\alpha_n)\rightarrow R(\beta_1)\oplus\dotsb\oplus R(\beta_n),$$ respectively.
 The fact that $AB=I_{r(\overline{\alpha})}$ implies that
 $A$ is surjective. Since both $R$-modules have the same pseudo-dimension $n$, the first part implies that $A$ is also injective. Therefore there exists $B' \in \M_{n\times n}(D)[\overline{\beta}][\overline{\alpha}]$ such that $B'A=I_{r(\overline{\beta})}$. Now $$B=I_{r(\overline{\beta})}B=B'(AB)=B'I_{r(\overline{\alpha})}=B',$$
 as desired.
\end{proof}

Let $R$ be a $\Gamma$-graded ring and $\overline{\alpha}=(\alpha_1,\dotsc,\alpha_m)\in\Gamma^m$, $\overline{\beta}=(\beta_1,\dotsc,\beta_n)\in\Gamma^n$.
Let $P_{r_{ij}(\overline{\alpha})}$ be the matrix obtained from $I_{r(\overline{\alpha})}$ by interchanging the rows $i$ and $j$. Note that if $A\in\M_{m\times n}(R)[\overline{\alpha}][\overline{\beta}]$, then
$P_{r_{ij}(\overline{\alpha})}A$ is the matrix obtained from $A$ by interchanging rows $i$ and $j$. Notice that $P_{r_{ij}(\overline{\alpha})}\in \M_{m\times m}(R)[\overline{\alpha'}][\overline{\alpha}],$ where $\overline{\alpha'}$ is obtained from $\overline{\alpha}$ interchanging $\alpha_i$ and $\alpha_j$, and $P_{r_{ij}(\overline{\alpha})}A\in \M_{m\times n}(R)[\overline{\alpha'}][\overline{\beta}]$.

Let $a\in R_{\gamma}$ with $d(\gamma)=r(\alpha_i)$. Let $D_{r_i(\overline{\alpha})}(a)$ be the matrix that agrees with $I_{r(\overline{\alpha})}$ except that it has an $a$ (instead of $1_{r(\alpha_i)}$) in the $(i,i)$-position. Notice that $D_{r_i(\overline{\alpha})}(a)A$ is the matrix obtained from $A$ by multiplying by $a$ the entries of the $i$-th row of $A$. Notice that $D_{r_i(\overline{\alpha})}(a)\in \M_{m\times m}(R)[\overline{\alpha'}][\overline{\alpha}]$ where $\overline{\alpha'}=(\alpha_1,\dotsc,\alpha_{i-1},\gamma\alpha_i,\alpha_{i+1},\dotsc, \alpha_m)$, and $D_{r_i(\overline{\alpha})}(a)A\in \M_{m\times n}(R)[\overline{\alpha'}][\overline{\beta}]$.

Let $a\in R_\gamma$ with $d(\gamma)=r(\alpha_i)$ and $r(\gamma)=r(\alpha_j)$ for  some $i\neq j$. Let $T_{r_{ij}(\overline{\alpha})}(a)$ be the matrix obtained from $I_{r(\overline{\alpha})}$ by replacing $row_j(I_{r(\overline{\alpha})})$ by
$a\cdot row_i(I_{r(\overline{\alpha})})+row_j(I_{r(\overline{\alpha})})$ and leaving the other rows intact. Note that $T_{r_{ij}(\overline{\alpha})}(a)A$ is the matrix obtained from $A$ by replacing  $row_j(A)$ by $a\cdot row_i(A)+row_j(A)$ and leaving the other rows intact. Notice that $T_{r_{ij}(\overline{\alpha})}(a)\in \M_{m\times m}(R)[\overline{\alpha'}][\overline{\alpha}]$ and
$T_{r_{ij}(\overline{\alpha})}(a)A\in \M_{m\times n}(R)[\overline{\alpha'}][\overline{\beta}]$, where $\overline{\alpha'}$ is obtained from $\overline{\alpha}$ replacing $\alpha_j$ by $\gamma\alpha_i$.

In the same way, one can define the $n\times n$ matrices obtained from $I_{r(\overline{\beta})}$ making elementary column operations and such that when they multiply $A$ on the right they perform that same operation on the columns of $A$.

\medskip

Let now $D$ be a $\Gamma$-graded division ring. Fix $\overline{\alpha}=(\alpha_1,\dotsc,\alpha_m)\in\Gamma^m$, $\overline{\beta}=(\beta_1,\dotsc,\beta_n)\in\Gamma^n$ and $A\in \M_{m\times n}(D)[\overline{\alpha}][\overline{\beta}]$.

For each $i=1,\dotsc,m$, one  can regard the  rows of $A$ as  homogeneous elements of the left $D$-module $(\beta_1^{-1})D\oplus \dotsb \oplus (\beta_n^{-1})D$. We define $\rho_r(A)$, the \emph{row rank} of $A$, as the pseudo-dimension of the graded left $D$-submodule generated by the rows of $A$. It can be computed multiplying $A$ on the left by adequate matrices $P_{r_{ij}(\overline{\alpha})}, D_{r_i(\overline{\alpha})}(a)$, $T_{r_{ij}(\overline{\alpha})}(a)$. 

The columns of $A$ are  homogeneous elements of the right $D$-module $D(\alpha_1)\oplus\dotsb\oplus D(\alpha_m)$. We define  $\rho_c(A)$, the \emph{column rank} of $A$, as the pseudo-dimension of the graded right $D$-submodule generated by the columns of $A$.

Consider now all possible $p\geq 0$, $\overline{\tau}\in \Gamma^p$ and matrices
$B\in \M_{m\times p}(D)[\overline{\alpha}][\overline\tau]$, $C\in \M_{p\times n}(D)[\overline{\tau}][\overline{\beta}]$ such that
\begin{equation}\label{eq:inner_rank}
    A=BC.
\end{equation}
We define $\rho(A)$, the \emph{inner rank} of $A$, as the least $p\geq 0$ for which there exist $\overline{\tau}$ and matrices $B,C$ as in \eqref{eq:inner_rank}. Note that we always have
$I_{r(\overline{\alpha})}A=A$ and $AI_{r(\overline \beta)}=A$. Thus, $\rho(A)\leq \min\{m,n\}$. 

Observe that \eqref{eq:inner_rank} means that the column of $A$ are a genuine linear combination of the columns of $B$, or that the rows of $A$ are a genuine linear combination of the rows of $C$.

We define $\rho_i(A)$ as the largest integer $s$ such that $A$ has an $s\times s$ invertible submatrix. 

The following result, shows that all for ranks just defined are equal over a graded division ring. We follow very close the proof in \cite[Exercises 13.13, 13.14]{Lamex} where the result is proved for (ungraded) division rings.

\begin{proposition}\label{prop:ranks}
    Let $D$ be a $\Gamma$-graded division ring. Let $m,n$ be positive integers, $\overline{\alpha}\in \Gamma^m$ and 
    $\overline{\beta}\in \Gamma^n$. For $A\in \M_{m\times n}(D)[\overline{\alpha}][\overline{\beta}]$, we have
    $$\rho_r(A)=\rho_c(A)=\rho(A)=\rho_i(A).$$
\end{proposition}

\begin{proof}
Suppose that $\rho_r(A)=r$ and that the rows
$i_1,\dotsc,i_r$ of $A$
 form a pseudo-basis of the graded left $D$-module  generated by the rows of $A$. Thus, all the rows of $A$ are a genuine left linear combination of these rows by Proposition~\ref{prop:pseudo_free}(3).  
 Let $\overline{\alpha'}=(\alpha_{i_1},\dotsc,\alpha_{i_r})\in \Gamma^r$ and  $C\in \M_{r\times n}(D)[\overline{\alpha'}][\overline{\beta}]$ be the submatrix of $A$ formed by the rows $i_1,\dotsc,i_r$.   
  Let
 $B\in \M_{m\times r}(D)[\overline{\alpha}][\overline{\alpha'}]$ such that 
 \begin{equation}\label{eq:ranks}
     A=BC.
 \end{equation}
 This equality shows that the columns of $A$ are a genuine right $D$-linear combination of the $r$ columns of $B$. Hence $\rho_c(A)\leq \rho_r(A)$. A similar argument shows that $\rho_r(A)\leq \rho_c(A)$. 

 By the definition of $\rho(A)$, equality \eqref{eq:ranks}, also shows that
 $\rho(A)\leq \rho_r(A)=\rho_c(A)$. But, on the other hand, observe that if $\rho(A)=s$, then the rows of $A$ are obtained as genuine left $D$-linear combinations of $s$ homogeneous elements. It implies that the pseudo-dimension of the graded $D$-module generated by the rows of $A$ has pseudo-dimension at most $s$. Thus $\rho(A)=\rho_r(A)$.

 Now we prove that $\rho_i(A)$ equals the other ranks. Suppose first that
 $A$ is of size $n\times n$. If
 $\rho_c(A)=n$, this means that the columns of $A$ form a pseudo-basis of $R(\alpha_1)\oplus \dotsb \oplus R(\alpha_n)$. Hence the homogeneous elements
$$\begin{pmatrix}
    1_{r(\alpha_1)} \\ 0 \\ \vdots \\ 0
\end{pmatrix},\ \begin{pmatrix}
   0  \\  1_{r(\alpha_2)}\\ \vdots \\ 0
\end{pmatrix}, \dotsc , 
\begin{pmatrix}
    0 \\ 0 \\ \vdots \\ 1_{r(\alpha_n)}
\end{pmatrix}$$
 which are of degrees $\beta_1^{-1},\beta_2^{-1},\dotsc,\beta_n^{-1}$, respectively, can be obtained as genuine linear combinations of the columns of $A$. That implies the existence of $B\in \M_{n\times n}(R)[\overline{\beta}][\overline{\alpha}]$ such that $AB=I_{r(\overline{\alpha})}$. By Corollary~\ref{coro:theo_of_grhomomorphism}, it implies that $A$ is invertible and therefore $\rho_i(A)=n$. Conversely, suppose that $\rho_i(A)=n$. Thus, $A$ is invertible. Hence, there exists 
 $B\in \M_{n\times n}(R)[\overline{\beta}][\overline{\alpha}]$ such that
 $BA=I_{r(\overline{\beta})}$. This implies that the columns of $A$ are pseudo-linearly independent. Therefore, $\rho_c(A)=n$. 

 Suppose now that $\rho_i(A)=s<n$. Let $M$ be an $s\times s$ submatrix of $A$ that is invertible. Suppose it is formed by the entries in the $i_1<\dotsb<i_s$ rows and
 $j_1<\dotsb<j_s$ columns.
  Let $C$ be the submatrix of $A$ formed by the $j_1,j_2,\dotsc,j_s$ columns of $A$. Then its columns are pseudo-right linearly independent. This implies $\rho_c(A)\geq s=\rho_i(A)$. If $\rho_c(A)>s$, then we can add one more column to $C$ to produce a new matrix $C'$ with $s+1$ pseudo-linearly independent columns of $A$. Since $s+1=\rho_c(C')=\rho_r(C')$, we can add one more row to the unique $s\times (s+1)$ submatrix of $C'$ containing the $s$ rows of $M$ to produce an $(s+1)\times (s+1)$ submatrix  $C''$ of $A$ such that
  $\rho_c(C'')=s+1$. But this is equivalent to $C''$ being invertible by what we have already proved. This  contradicts  the fact that $\rho_i(A)=s$. Therefore, $\rho_c(A)=s$, as desired.
\end{proof}

We end this section with the following observation. Let $A\in \M_{m\times n}(D)[\overline{\alpha}][\overline{\beta}]$ for some $\overline{\alpha}\in\Gamma^m$ and $\overline{\beta}\in\Gamma^n$. Notice that there could exist other $\overline{\alpha'}\in\Gamma^m$ and $\overline{\beta'}\in\Gamma^n$ such that $A\in \M_{m\times n}(D)[\overline{\alpha'}][\overline{\beta'}]$. Thus, it could seem that the inner rank of $A$, $\rho(A)$, depends on $\overline{\alpha}$ and $\overline{\beta}$. By Proposition~\ref{prop:ranks},  if $D$ is a gr-division ring, then it does not depend on $\overline{\alpha}$ and $\overline{\beta}$. Indeed, $\rho(A)=\rho_r(A)$ and $\rho_r(A)$ can be computed multiplying on the left by the matrices that define elementary row operations that can be performed in the same way  if $A$ is considered to belong to either $\M_{m\times n}(D)[\overline{\alpha}][\overline{\beta}]$  or $\M_{m\times n}(D)[\overline{\alpha'}][\overline{\beta'}]$.


\section{Structure of gr-semisimple rings}
\label{sec: art simp}

\emph{Throughout this section, let $\Gamma$ be a groupoid and $R=\bigoplus\limits_{\gamma\in\Gamma}R_\gamma$ be a $\Gamma$-graded ring.}

\subsection{Gr-simple modules and Schur's lemma}
Given  a $\Gamma$-graded $R$-module $S$, we say that $S$ is \emph{gr-simple}  if $S\neq0$ and its only graded submodules are $\{0\}$ and $S$.

An immediate consequence of the definition and Lemma \ref{lem: M(e)=M(gamma) como conj} is the following.

\begin{lemma}
\label{lem: shift de simples}
If $S$ is a gr-simple $R$-module, then there exists $e\in\Gamma_0$ such that $S=S(e)$. Furthermore, $S(\sigma)$ is gr-simple for each $\sigma\in e\Gamma$.\qed
\end{lemma}

This suggests that the definition of gr-simple module may be too restrictive. So we say that the $\Gamma$-graded $R$-module $S$ is \emph{$\Gamma_0$-simple} if $S(e)$ is a gr-simple $R$-module for each $e\in\Gamma'_0(S)$.


When considering rings and modules without a grading, it is useful to study simple modules up to isomorphism. In the graded context, one has to take into account the shifts of the simple modules. 
Inspired by the idea of \cite[p. 395]{artigoNastasescu}, we say that two $\Gamma$-graded $R$-modules $M$ and $N$ \emph{are in the same isoshift class} if there exists $\Sigma\subseteq\Gamma$ such that $\Sigma$ is fully 
$r$-unique for $N$, $\Sigma^{-1}$ is fully $r$-unique for $M$ and $M\cong_{gr}N(\Sigma)$. Note that this defines an equivalence relation for $\Gamma$-graded modules. The next proposition helps to understand how gr-simple isoshift classes are. Note that, by Lemma \ref{lem: shift de simples}, two gr-simple modules $M$ and $N$ are in the same isoshift class if and only if there exists $\sigma\in\Gamma$ such that $M\cong_{gr}N(\sigma)$.

\begin{lemma}
\label{lem: isoshift de simples}
Let $S,S'$ be two gr-simple $R$-modules. The following assertions are equivalent.
\begin{enumerate}[\rm (1)]
    \item $S$ and $S'$ are not in the same isoshift class.
    \item $\HOM(S,S')=0$.
    \item $\displaystyle \HOM_R\left(\bigoplus_{j\in J}S(\sigma_j),\bigoplus_{i\in I}S'(\sigma'_i)\right)=\{0\}$ for all $(\sigma_j)_{j\in J}\in \Gamma^J$ and $(\sigma'_i)_{i\in I}\in \Gamma^I$. 
\end{enumerate}
\end{lemma}

\begin{proof}
    $(1)\implies(2):$ Suppose there exists $\sigma\in\Gamma$ and nonzero $h\in \HOM_R(S,S')_\sigma$. So $0\neq h\in \Homgr(S,S'(\sigma))$ by Proposition \ref{prop: HOM(,) e Hom_gr}(1). In particular, $S'(\sigma)\neq0$. Therefore, it follows from Lemma \ref{lem: shift de simples} that $S'=S'(r(\sigma))$ and $S'(\sigma)$ is gr-simple. Since $\ker h$ is a proper graded submodule of $S$ and $\im h$ is a nonzero graded submodule of $S'(\sigma)$, it follows that $\ker h=0$ and $\im h=S'(\sigma)$. Hence, $h$ is a gr-isomorphism between $S$ and $S'(\sigma)$, contradicting (1).

    $(2)\implies(1):$ Suppose that $S$ and $S'$ are in the same isoshift class and let $\sigma\in\Gamma$ such that $S\cong_{gr}S'(\sigma)$. By Proposition \ref{prop: HOM(,) e Hom_gr}(1), we have $\HOM_R(S,S')_\sigma\cong\Homgr(S,S'(\sigma))\neq0$. 
        
    $(2)\implies(3):$ Using Proposition \ref{prop:endomorphism_rings} and Remark~\ref{rem: end rings for Gamma0 f.g.} (since gr-simple modules are generated by a single element) we have
    \[\HOM_R\left(\bigoplus_{j\in J}S(\sigma_j),\bigoplus_{i\in I}S'(\sigma'_i)\right)\cong_{gr}\Hgr_{I\times J}\left(\bigoplus_{j\in J}S(\sigma_j),\bigoplus_{i\in I}S'(\sigma'_i)\right)\]
    So it suffices to show that $\HOM_R(S(\sigma_j),S'(\sigma'_i))_\gamma=\{0\}$ for all $i\in I$, $j\in J$ and $\gamma\in\Gamma$. In fact, Proposition \ref{prop: HOM(,) e Hom_gr}(2) and (2) give us
    \[\HOM_R(S(\sigma_j),S'(\sigma'_i))_\gamma=\HOM_R(S,S')_{\sigma'_i\gamma\sigma_j^{-1}}=0.\]

    $(3)\implies(2):$ It is clear.
\end{proof}

Notice that it follows from Lemma~\ref{lem: isoshift de simples} that two gr-simple $R$-modules $S$ and $S'$ are in the same isoshift class if and only if $\HOM_R(S,S')\neq 0$.

Next we present a  result analoguous to Schur's Lemma in the groupoid graded context.

\begin{theorem}
\label{teo: schur}
Let $S$ be a $\Gamma_0$-simple $R$-module, $D:=\END_R(S)$ and $\Gamma'_0:=\Gamma'_0(D)=\Gamma'_0(S)$. The following assertions hold:
\begin{enumerate}[\rm (1)]
    \item $D$ is a gr-division ring.
    \item Consider the gr-primality relation $\sim$ defined on $\Gamma'_0$. Then $e\sim f$ in $\Gamma'_0$ if and only if $S(e)$ and $S(f)$ are in the same isoshift class.
    \item There exists a bijection between $\Gamma'_0/\sim$ and the isoshift classes of $\{S(e):e\in \Gamma'_0\}$ that sends each $[e]\in \Gamma'_0/\sim$ to the isoshift class of $S(e)$.
    \item $D$ is a gr-prime (resp. gr-simple) ring if and only if all $S(e)$  are in the same isoshift class for each $e\in\Gamma'_0$.
\end{enumerate}
\end{theorem}

\begin{proof}
(1) Let $\sigma\in\supp(D)$ (in particular, $r(\sigma),d(\sigma)\in\Gamma'_0$) and $0\neq g\in D_\sigma$.
Then $g':=g|_{S(\sigma^{-1})}\in \Homgr(S(\sigma^{-1}),S(r(\sigma)))$ because
\[g\left(S(\sigma^{-1})_\gamma\right)=g\left(S_{\sigma^{-1}\gamma}\right)\subseteq S_{\sigma\sigma^{-1}\gamma}=S(r(\sigma))_\gamma,\]
for all $\gamma\in\Gamma$ such that $r(\gamma)=r(\sigma)$.
We have that $\ker g'$ is a proper graded submodule of $S(\sigma^{-1})$ and $\im g'$ is a nonzero graded submodule of $S(r(\sigma))$. Since $S$ is $\Gamma_0$-simple, it follows that $S(\sigma^{-1})$ and $S(r(\sigma))$ are gr-simple $R$-modules. Therefore, $\ker g'=0$ and $\im g'=S(r(\sigma))$, i.e., $g'$ is a gr-isomorphism. Let $h'\in \Homgr(S(r(\sigma)),S(\sigma^{-1}))$ be the inverse of $g'$. Extend $h'$ to a $h\in\Homgr(S,S(\sigma^{-1}))$ defining $h(x)=0$ for all $x\in S(e)$ when $e\in\Gamma_0\setminus \{r(\sigma)\}$. Thus, $g\circ h: S\to S$ is the zero function on $S(e)$, for all $e\in\Gamma_0\setminus \{r(\sigma)\}$ and it is the identity on $S(r(\sigma))$, that is, $g\circ h=\mathds{1}_{r(\sigma)}$. On the other hand, $g$ (and therefore $h\circ g$) is the zero function on $S(e)$, for all $e\in\Gamma_0\setminus \{d(\sigma)\}$ and $h\circ g$ is the identity on $S(\sigma^{-1})$ ($=S(d(\sigma))$, as $R$-modules, by Lemma \ref{lem: M(e)=M(gamma) como conj}). Hence, $h\circ g=\mathds{1}_{d(\sigma)}$.

(2) Let $e,f\in\Gamma'_0$. Assume $e\sim f$. So there exists $\gamma\in e\Gamma f\cap \supp D$. By Proposition \ref{prop: HOM(,) e Hom_gr}(2), we get
\[0\neq D_\gamma=\HOM_R(S,S)_{e\gamma f}\cong\HOM_R(S(f),S(e))_\gamma.\]
This implies that $S(e)$ and $S(f)$ are in the same isoshift class, by Lemma \ref{lem: isoshift de simples}. 
Conversely, suppose that $S(e)$ and $S(f)$ are in the same isoshift class. Take $\sigma\in\Gamma$ such that $S(f)\cong_{gr}S(e)(\sigma)$. Again, by Proposition \ref{prop: HOM(,) e Hom_gr}, we obtain
\[D_{e\sigma f}=\HOM_R(S,S)_{e\sigma f}\cong\HOM_R(S(f),S(e))_\sigma=\Homgr(S(f),S(e)(\sigma))\neq0.\]
Therefore, $\sigma\in e\Gamma f\cap \supp D$ and it follows that $e\sim f$.

(3) It follows from (2).

(4) It is immediate from (3) and Proposition \ref{prop: quando D e'simples}(4).
\end{proof}

\begin{corollary}
\label{coro: schur para simples}
    If $S$ is a gr-simple $R$-module, then $D:=\END_R(S)$ is a gr-division ring with $\supp(D)\subseteq e\Gamma e$, where $e\in\Gamma_0$ is such that $S=S(e)$.\qed
\end{corollary}

Gr-simple modules also have the following interesting property that will be used later.

\begin{proposition}
\label{prop: Krull-Schmidt}
    Let $M$ be a $\Gamma$-graded $R$-module. Suppose $M=S_1\oplus\cdots\oplus S_n=T_1\oplus\cdots\oplus T_m$ where $n,m\in\mathbb{Z}_{>0}$ and $S_1,...,S_n,T_1,...,T_m$ are gr-simple graded submodules of $M$. Then $n=m$ and there exists a permutation $\pi$ of $\{1,...,n\}$ such that $S_i\cong_{gr}T_{\pi(i)}$ for each $i=1,...,n$.
\end{proposition}

\begin{proof}
    We proceed by induction on $n$.

    If $M=S_1=T_1\oplus\cdots\oplus T_m$ as in the statement, it is clear that $m=1$ and $T_1=S_1$.

    Let $n>1$ and assume that the result is valid for $n-1$. Suppose $M=S_1\oplus\cdots\oplus S_n=T_1\oplus\cdots\oplus T_m$ as in the statement. Note that $m>1$. Given $1\leq i \leq n$ and $1\leq j\leq m$, let $p_i:M\to S_i$, $p'_j:M\to T_j$ be the canonical projections and $\iota_i:S_i\to M$, $\iota'_j:T_j\to M$ be the canonical inclusions. Then
    \[id_{T_1}=p'_1\iota'_1=p'_1id_M\iota'_1=p'_1\left(\sum_{i=1}^n\iota_ip_i\right)\iota'_1=\sum_{i=1}^np'_1\iota_ip_i\iota'_1.\]
    Hence, there exists $1\leq k\leq n$ such that $p_k\iota'_1\neq0$. Therefore, $p_k\iota'_1:T_1\to S_k$ is a gr-isomorphism since $T_1$ and $S_k$ are gr-simple. On the other hand, note that $p_k\iota'_1=p_k|_{T_1}$. The injectivity of this function is equivalent to $T_1\cap\ker p_k=0$, that is,
    \begin{equation}
    \label{eq: KSA1}
        T_1\cap(S_1\oplus\cdots\oplus\widehat{S_k}\oplus\cdots\oplus S_n)=0.
    \end{equation}
    The surjectivity of $p_k|_{T_1}$ is equivalent to $p_k(T_1)=S_k$. Thus, given $t_1\in T_1$, we have $p_k(t_1)=t_1-\sum\limits_{\substack{i=1\\ i\neq k}}^np_i(t_1)$ and it follows that
    \begin{equation}
    \label{eq: KSA2}
        S_k\subseteq T_1+(S_1\oplus\cdots\oplus\widehat{S_k}\oplus\cdots\oplus S_n).
    \end{equation}
    By (\ref{eq: KSA2}), $M=S_1\oplus\cdots\oplus S_n\subseteq T_1+(S_1\oplus\cdots\oplus\widehat{S_k}\oplus\cdots\oplus S_n)$. So, from (\ref{eq: KSA1}), we get
    \[M=T_1\oplus(S_1\oplus\cdots\oplus\widehat{S_k}\oplus\cdots\oplus S_n).\]
    Hence, 
    \[S_1\oplus\cdots\oplus\widehat{S_k}\oplus\cdots\oplus S_n\cong_{gr}T_2\oplus\cdots\oplus T_m.\]
    By the induction hypothesis, $n-1=m-1$ and there exists a bijection \linebreak $\pi:\{1,...,k-1,k+1,...,n\}\to\{2,...,m\}$ such that $S_i\cong_{gr}T_{\pi(i)}$ for each $i=1,...,k-1,k+1,...,n$. Now just define $\pi(k)=1$. 
\end{proof}

By Corollary~\ref{coro: generation lemma para gr-ciclicos} for the case $|I|$ or $|J|$ are infinite and Proposition~\ref{prop: Krull-Schmidt} for the case $|I|$ and $|J|$ are finite, we obtain

\begin{proposition}
\label{prop: gr-ss de dim infinita}
    Let $M$ be a $\Gamma$-graded $R$-module. Suppose $M=\bigoplus\limits_{i\in I}S_i=\bigoplus\limits_{j\in J}T_j$ where $S_i,T_j$  are gr-simple graded submodules of $M$ for all $i\in I,j\in J$. Then $|I|=|J|$.\qed
\end{proposition}


\subsection{General results about gr-semisimple rings and modules}\label{sec:general results about gr-semisimple}


A $\Gamma$-graded $R$-module $M$ is said to be \emph{gr-semisimple} if $M$ is a sum of gr-simple graded submodules. That is, there exists a family of gr-simple submodules $\{S_i\}_{i\in I}$ such that $M=\sum_{i\in I} S_i$. In the same way, one can define left gr-semisimple modules. We say that the $\Gamma$-graded ring $R$ is a \emph{right gr-semisimple} ring if $R_R$ is a gr-semisimple module. That $R$ is a \emph{left gr-semisimple} ring if $_RR$ is a gr-semisimple module.

Let $M$ be a $\Gamma$-graded $R$-module. We say that $M$ is a \emph{gr-artinian} $R$-module if $M$ satisfies the descending chain condition on graded submodules. Note that if $M$ is gr-artinian, then $M(e)\neq0$ only for a finite number of $e\in\Gamma_0$. Thus, as in the case of gr-simplicity of modules, this motivates us to define that $M$ is a \emph{$\Gamma_0$-artinian} $R$-module if $M(e)$ is a gr-artinian $R$-module for all $e\in\Gamma_0$. We say that $R$ is a \emph{right $\Gamma_0$-artinian} ring if $R_R$ is a $\Gamma_0$-artinian $R$-module. Analogously, we define when $R$ is a \emph{left $\Gamma_0$-artinian} ring. 
 We observe that if $R$ is a right $\Gamma_0$-artinian ring, $\Delta_0\subseteq\Gamma_0$ and $\Delta=\{\gamma\in\Gamma \colon d(\gamma),r(\gamma)\in \Delta_0\}$, then $R_\Delta :=\bigoplus_{\gamma\in\Delta}R_\gamma$ is a right $\Delta_0$-artinian ring. Indeed an infinite strict descending chain of graded right ideals of $R_\Delta$ contained in $R_\Delta(e)$ for some $e\in\Delta_0$
$$I_1\supset I_2\supset \dotsb \supset I_n\supset \dotsb$$
implies the existence of the infinite strict descending chain of graded right ideals of $R$ contained in $R(e)$
$$I_1R\supset I_2R\supset \dotsb \supset I_nR\supset \dotsb.$$ 
We also note that if $R$ is right $\Gamma_0$-artinian, then $R$ is right gr-artinian if and only if $\Gamma'_0(R)$ is finite. Thus, being $\Gamma_0$-artinian implies being gr-locally artinian in the following sense: if $R=\bigoplus_{\gamma\in\Gamma}R_\gamma$ is a right $\Gamma_0$-artinian ring, then, for any finite subset $\Delta_0\subseteq \Gamma'_0(R)$, the $\Delta$-graded ring $R_\Delta=\bigoplus_{\gamma\in\Delta}R_\gamma$ is right gr-artinian. 
On the other hand, if $R$ is locally artinian in the foregoing sense, it does not imply that $R$ is $\Gamma_0$-artinian. For example, consider the ring $R:=\UT_{\mathbb{N}}(D)$ of countably infinite upper triangular matrices over a division ring $D$ with only a finite number of nonzero entries endowed with its natural $\mathbb{N}\times\mathbb{N}$ grading. Then $R$ is locally artinian, because for any finite subset $\Delta_0\subseteq \mathbb{N}$ we have $R_\Delta\cong_{gr}\UT_n(D)$, where $n=|\Delta_0|$. However, we have the strict descending chain of graded right ideals 
$$E_{11}R\supset E_{12}R\supset \dotsb \supset E_{1n}R\supset \dotsb.$$

The first results of this section are basic facts on gr-semisimple modules. 
The following result will be used throughout the paper and its proof can be found in \cite[Lemma~51 and Propositions~52-53]{CLP}.

\begin{proposition}\label{prop:basics_gr-semisimple_modules}
Let $M$ be a gr-semisimple module and $N$ be a graded submodule of $M$. Suppose that $M=\sum_{i\in I} S_i$ where $S_i$ is a gr-simple submodule of $M$ for all $i\in I$.    The following statements hold true.
\begin{enumerate}[\rm(1)]
    \item There exists  $I'\subseteq I$ such that 
    $M=N\oplus \left(\bigoplus\limits_{i\in I'} S_i\right)$. Hence $N$ is a graded direct summand of $M$.
    \item There exists $I_0\subseteq I$ such that $M=\bigoplus\limits_{i\in I_0} S_i$.
    \item There exists $I''\subseteq I$ such that $N\cong_{gr} \bigoplus\limits_{i\in I''}S_i$. \qed
\end{enumerate}
\end{proposition}





\begin{proposition}{\cite[Proposition 57]{CLP}}
    Let $M$ be a $\Gamma$-graded $R$-module. If $M$ is a semisimple $R$-module, then $M$ is a gr-semisimple $R$-module.\qed
\end{proposition}

In order to state the next result, we need a definition. We say that a $\Gamma$-graded $R$-module $Q$ is gr-injective if, for all gr-homomorphisms $j:M\to N$ and $g:M\to Q$ between $\Gamma$-graded $R$-modules with $j$ injective, there exists a gr-homomorphism $h:N\to Q$ such that $g=hj$ \cite[Propositions 42 and 44]{CLP}.

\begin{proposition}{\cite[Proposition 59]{CLP}}
\label{prop: CLP, Prop 59}
    The following assertions are equivalent:
    \begin{enumerate}[\rm (1)]
        \item $R$ is a right gr-semisimple ring.
        \item Every graded right ideal of $R$ is a graded direct summand of $R_R$.
        \item Every $\Gamma$-graded $R$-module is gr-injective.
        \item Every $\Gamma$-graded $R$-module is gr-projective.
        \item Every $\Gamma$-graded $R$-module is gr-semisimple.\qed
    \end{enumerate}
\end{proposition}

Now we proceed in a similar way to \cite[Section~4.6]{Bourbaki}, \cite[p. 35--36]{Lam1} and \cite[Section~2]{artigoNastasescu}.
Let $M$ be a $\Gamma$-graded $R$-module and $S$ be a gr-simple $R$-module. The \emph{isoshiftical component} of type $S$ of $M$, denoted by $M_S$, is the sum of the graded submodules $T$ of $M$ that are in the same isoshift class of $S$. If $S'$ is another gr-simple submodule of $M$ in the same isoshift class of $S$, then $M_S=M_{S'}$. Thus $M_S$ depends only on the isoshift class of $S$. 
We denote by $\mathcal{S}(R)$ the set of isoshift classes of gr-simple $R$-modules. Thus, if $j\in \mathcal{S}(R)$ is the isoshift class of $S$, we can write $M_j$ instead of $M_S$.

\begin{lemma}\label{lem:direct_sums_ordered_by_isofhift}
Let $M$ be a gr-semisimple $R$-module. The following statements hold true.
\begin{enumerate}[\rm(1)]
    \item $M=\bigoplus_{j\in\mathcal{S}(R)} M_j$.
    \item Let $\{S_i\}_{i\in I}$ be a family of gr-simple submodules of $M$ such that $M=\bigoplus_{i\in I}S_i$. For each $j\in\mathcal{S}(R)$, let $I(j)=\{i\in I\colon S_i\in j\}$. Then $M_j=\bigoplus_{i\in I(j)}S_i$.
\end{enumerate}  
\end{lemma}

\begin{proof}
(1) Since $M$ is gr-semisimple, $M=\sum_{j\in\mathcal{S}(R)}M_j$. We must show that this sum is direct. Let $j\in\mathcal{S}(R)$ and $S\in j$. 

Set $M_{j}':=\sum_{k\in\mathcal{S}(R)\setminus\{j\}} M_k$ and consider $M_j\cap M_{j}'$. On the one hand $M_j\cap M_{j}'$ is a graded submodule of the gr-semisimple module $M_j$. By Proposition~\ref{prop:basics_gr-semisimple_modules}(3), $M_j\cap M_{j}'=\bigoplus_{l\in L}S_l$ where $S_l\in j$. On the other hand, $M_j\cap M_{j}'$ is a graded submodule of the gr-semisimple module $M_j'$. By Proposition~\ref{prop:basics_gr-semisimple_modules}(3), $M_j\cap M_{j}'=\bigoplus_{i\in I''}S_i$ is a direct sum of gr-simple submodules $S_i$ not isomorphic to a shift of $S$ for each $i\in I''$. 
If $M_j\cap M_{j}'\neq 0$, pick $l_0\in L$. By Proposition~\ref{prop:basics_gr-semisimple_modules}(3), $S_{l_0}\cong_{gr}\bigoplus_{i\in I'''}S_i$ for some $I'''\subseteq I''$. Since, $S_{l_0}$ is gr-simple, $S_{l_0}\cong_{gr}S_{i_0}$ for some $i_0\in I'''$, a contradiction.

(2) By construction, $M_j \supseteq \bigoplus_{i\in I(j)}S_i$. Let $S$ be a gr-simple submodule of $M$ that belongs to $j$. Since $S$ is a cyclic $R$-module, $S\subseteq \bigoplus_{t=1}^n S_{i_t}$ for some $i_1,\dotsc,i_n\in I$. Note that if the composition $S\stackrel{\iota}{\hookrightarrow} \bigoplus_{t=1}^n S_{i_t}\stackrel{p_{i_t}}{\rightarrow}S_{i_t}$ is not zero, where $\iota$ and $p_{i_t}$ are the natural inclusion and projection, respectively, then $S\cong_{gr}S_{i_t}$. Thus, $i_1,\dotsc,i_n$ can be chosen to belong to $I(j)$ . Hence $S\subseteq \sum_{i\in I(j)}S_i$. Therefore, $M_j \subseteq \bigoplus_{i\in I(j)}S_i$.    
\end{proof}

Although the next result applies for general graded rings,  we turn our attention to gr-semisimple rings.

\begin{lemma}\label{lem:isoshiftical_component_is_a_graded_ideal}
Let $S$ be a minimal graded ideal of the $\Gamma$-graded ring $R$. Suppose that $S$ belongs to $j\in\mathcal{S}(R)$. 
\begin{enumerate}[\rm(1)]
    \item $R_j$ is a graded ideal of $R$.
    \item Let $S'$ be a minimal graded right ideal of $R$ that belongs to $j'\in\mathcal{S}(R)$. If $j\neq j'$, then $R_j\cdot R_{j'}=0$.
\end{enumerate}
\end{lemma}
\begin{proof}
    Notice that $S=S(e)$ for some $e\in\Gamma_0$ by Lemma~\ref{lem: shift de simples}. Consider the family
    \[\mathcal{F}:=\{T \textrm{ graded right ideal of $R$}:\textrm{there exists $\sigma\in e\Gamma$ such that $T_R\cong_{gr}S(\sigma)$}\}.\]
    
    Note that $\mathcal{F}\neq\emptyset$ because $S\in\mathcal{F}$, and, by definition, 
    \[R_j=\sum_{T\in\mathcal{F}}T,\]
    which is a graded right ideal of $R$. Let us see that $R_j$ is a graded ideal of $R$. Let $\delta\in\Gamma, a\in R_\delta$ and take $T\in\mathcal{F}$. If $aT=0$, we already have $aT\subseteq R_j$. So assume $aT\neq 0$ and consider the following gr-homomorphism of $\Gamma$-graded right $R$-modules
    \begin{align*}
        \varphi: T(\delta^{-1})&\longrightarrow R\\
        x&\longmapsto ax
    \end{align*}
    Since $aT\neq 0$, it follows that $T(d(\delta))=T$ is gr-simple. Therefore, by Lemma \ref{lem: shift de simples}, we have that $T(\delta^{-1})$ is gr-simple. Thus, $\varphi$ is injective, as it is nonzero. Therefore $aT\cong_{gr}T(\delta^{-1})\cong_{gr}S(\sigma)(\delta^{-1})=S(\sigma\delta^{-1})$ for some $\sigma\in e\Gamma$. It follows that $aT\in\mathcal{F}$ and $aT\subseteq R_j$. Hence
    \[aR_j=\sum_{T\in\mathcal{F}}aT\subseteq R_j.\]
    Since $a$ was an arbitrary homogeneous element of $R$, it follows that $R_j$ is also a left ideal and therefore it is a graded ideal of $R$. 

   (2) It is enough to show that $S\cdot S'=0$. Assume, on the contrary, that $s'S\neq 0$ for some homogeneous element $s'\in S'$. Since $S'$ is a minimal graded right ideal, the nonzero graded right ideal $s'S$ equals $S'.$ Moreover, if $\delta=\deg(s')$, then $\varphi\colon S\rightarrow S'$, $s\mapsto s's$, is such that $0\neq \varphi\in\HOM(S,S')_\delta,$ a contradiction by Lemma~\ref{lem: isoshift de simples}.
\end{proof}

\begin{lemma}
\label{lem: R ss -> R(e) é soma finita}
Let $R$ be a right gr-semisimple ring and suppose that $\{S_i:i\in I\}$ is a family of gr-simple graded right  $R$-submodules of $R$ such that $R=\bigoplus\limits_{i\in I}S_i$.
The following statements hold true.
\begin{enumerate}[\rm (1)]
    \item Each gr-simple $R$-module is gr-isomorphic to a shift of some $S_i$.
    \item For each $e\in\Gamma_0$, there exists a finite subset $I_e\subseteq I$ such that $R(e)=\bigoplus\limits_{i\in I_e}S_i$.
    \item $R$ is a right $\Gamma_0$-artinian ring.
\end{enumerate}
\end{lemma}

\begin{proof}
(1) Let $S'$ be a gr-simple $R$-module. Taking $\sigma\in\supp(S')$ and $0\neq x\in S'_\sigma$, we have the surjective gr-homomorphism
\begin{align*}
    \varphi: R&\longrightarrow S'(\sigma)\\
    r&\longmapsto xr
\end{align*}
 and thus $S'(\sigma)\cong_{gr}\frac{R}{\ker \varphi}$. On the other hand, since $R=\bigoplus\limits_{i\in I}S_i$, there exists $i\in I$ such that the projection $\pi: S_i\longrightarrow \frac{R}{\ker \varphi}$ is nonzero. Therefore, $\pi$ is a gr-isomorphism because it is a gr-homomorphism between gr-simple modules. Since $S'$ is gr-simple, we have $S'=S'(r(\sigma))=S'(\sigma)(\sigma^{-1})$ and therefore
\[S'\cong_{gr}\frac{R}{\ker \varphi}(\sigma^{-1})\cong_{gr}S_i(\sigma^{-1}).\]

(2) Set $e\in\Gamma_0$. We have $R(e)=\bigoplus\limits_{i\in I}S_i(e)$.
For each $i\in I$, since $S_i$ is gr-simple and $S_i(e)$ is a graded submodule of $S_i$, we have $S_i(e)=0$ or $S_i(e)=S_i$. Therefore, there exists $I'\subseteq I$ such that
$R(e)=\bigoplus\limits_{i\in I'}S_i$.
So there exists a finite subset $I_e\subseteq I'$ such that $1_e=\sum\limits_{i\in I_e}s_i$, for certain $s_i\in (S_i)_e\setminus\{0\}$, $i\in I_e$. Hence 
\[R(e)=1_eR\subseteq\bigoplus_{i\in I_e}s_iR=\bigoplus_{i\in I_e}S_i\subseteq R(e).\]

(3) It follows from (2).
\end{proof}

Our aim is now to express a gr-semisimple ring as a graded direct product of gr-simple rings. 
In the non-graded case, a finite direct product of semisimple rings  is also semisimple. Also, a direct product of a family of rings equals  their direct sum if and only if the family is finite. In the groupoid-graded context this equivalence is not true as the following example shows.

\begin{example}\label{ex:produto_infinito}
Let $\{R_i:i\in I\}$ be an infinite family of rings with unity. Each $R_i$ is a $I\times I$-graded ring via $R_i=(R_i)_{(i,i)}$ and we have
\[\sideset{}{^{gr}}\prod\limits_{k\in I}R_k=\bigoplus_{(i,l)\in I\times I}\left(\prod_{k\in I}(R_k)_{(i,l)}\right)=\bigoplus_{i\in I}\left(\prod_{k\in I}(R_k)_{(i,i)}\right)=\bigoplus_{k\in I}R_k.\] \qed
\end{example}

Motivated by this, we will say that a family $\{R_j:j\in J\}$ of $\Gamma$-graded rings is \emph{summable} if $\sideset{}{^{gr}}\prod\limits_{j\in J}R_j=\bigoplus\limits_{j\in J}R_j$. 
The following characterization of a summable families will be useful.
\begin{proposition}
\label{prop: quando familia e somavel}
Let $\{R_j:j\in J\}$ be a family of $\Gamma$-graded rings. The following statements are equivalent.
\begin{enumerate}[\rm (1)]
    \item The family $\{R_j:j\in J\}$ is summable.
    \item The set $\{j\in J: (R_j)_\sigma\neq0\}$ is finite for all $\sigma\in\Gamma$.
    \item The set $\{j\in J: (R_j)_e\neq0\}$ is finite for all $e\in\Gamma_0$.
    \item The set $\{j\in J: R_j(e)\neq0\}$ is finite for all $e\in\Gamma_0$.
\end{enumerate}
\end{proposition}

\begin{proof}
For each $\sigma\in\Gamma$, we have
\[\left(\sideset{}{^{gr}}\prod_{j\in J}R_j\right)_\sigma=\prod_{j\in J}\left(R_j\right)_\sigma\quad\quad\textrm{and}\quad\quad\left(\bigoplus_{j\in J}R_j\right)_\sigma=\bigoplus_{j\in J}\left(R_j\right)_\sigma.\]
Thus, $\{R_j:j\in J\}$ is summable if and only if $\{j\in J: (R_j)_\sigma\neq0\}$ is finite for all $\sigma\in\Gamma$. Therefore, we obtain $(1)\iff(2)$. 

$(2)\implies(3)$ is clear.

$(3)\iff(4)$ follows from Lemma \ref{lem: obj unit ---> unit} which gives us $(R_j)_e\neq0\Longleftrightarrow R_j(e)\neq0$, for all $e\in\Gamma_0$. 

Finally, if $J_e:=\{j\in J: R_j(e)\neq0\}$ is finite for all $e\in\Gamma_0$, then given $\sigma\in\Gamma$ we have that $(R_j)_\sigma\neq0$ implies $R_j(r(\sigma))\neq0$, i.e., $j\in J_{r(\sigma)}$. Thus, $(4)\implies(2)$.
\end{proof}

Now we are ready to show the main result of this subsection. We follow very close the proof of the ungraded case given in \cite[p.~35-36]{Lam1}.

\begin{theorem}\label{theo:gr-semisimple_as_products_of_gr-simple}
    Suposse that $R$ is a right gr-semisimple ring. The following statements hold true.
    \begin{enumerate}[\rm(1)]
        \item $R_j$ is a nonzero graded ideal of $R$ for each $j\in\mathcal{S}(R)$.
        \item $\{R_j:j\in\mathcal{S}(R)\}$ is a summable family of $\Gamma$-graded rings and $R=\prod_{j\in\mathcal{S}(R)}^{gr}R_j$.
        \item $R_j$ is a gr-simple right $\Gamma_0$-artinian ring for each $j\in\mathcal{S}(R)$.        
    \end{enumerate}
    \end{theorem}

\begin{proof}
(1) By Lemma~\ref{lem:isoshiftical_component_is_a_graded_ideal}(1), $R_j$ is a graded ideal  of $R$ for each $j\in\mathcal{S}(R)$. 
By Lemma~\ref{lem: R ss -> R(e) é soma finita}(1), $R_j\neq 0$ for each $j\in\mathcal{S}(R)$.

(2)  By Lemma~\ref{lem:direct_sums_ordered_by_isofhift}(1),
\begin{equation}\label{eq:direct_sum_R_lambda}
    R=\bigoplus_{j\in\mathcal{S}(R)}R_j.
\end{equation}  For each $e\in\Gamma_0$, $1_e=\sum_{j\in\mathcal{S}(R)} 1_{je}$ where $1_{je}\in (R_j)_e$ and $1_{je}\neq 0$ for only a finite number of $j\in\mathcal{S}(R)$. 
By Lemma~\ref{lem:isoshiftical_component_is_a_graded_ideal}(2), $1_{je}$  is an idempotent such that $1_{je}x=x$ and $y1_{je}=y$ for each $x\in (R_j)_{\gamma}$, $y\in (R_j)_\delta$ such that $r(\gamma)=e$ and $d(\delta)=e$. Hence $R_j$ is object unital for each $j\in\mathcal{S}(R)$. 
Since $1_{je}\neq 0$ for only a finite number of $j\in\mathcal{S}(R),$ $R_j(e)\neq 0$  for only a finite number of $j\in\mathcal{S}(R)$. Thus, $\{R_j:j\in\mathcal{S}(R)\}$ is a summable family of $\Gamma$-graded rings by Proposition~\ref{prop: quando familia e somavel}.

(3) Let $j\in\mathcal{S}(R)$. By Lemma~\ref{lem: R ss -> R(e) é soma finita}(3), $R$ is a $\Gamma_0$-artinian right $R$-module, thus $R_j$ is too because $R_j(e)\subseteq R(e)$ for each $e\in\Gamma_0$. By Lemma~\ref{lem:isoshiftical_component_is_a_graded_ideal}(2), this implies that $R_j$ is a $\Gamma_0$-artinian ring.

Let now $I\neq 0$ be a graded ideal of $R_j$. Observe that $I$ is also a graded ideal of $R$ by Lemma~\ref{lem:isoshiftical_component_is_a_graded_ideal}(2). Since $R$ is a right $\Gamma_0$-artinian ring,  $I$ contains a minimal graded right ideal $T$ of $R$. By Lemma~\ref{lem:direct_sums_ordered_by_isofhift}(2), $T\in j$. It is enough to show that $I\supseteq R_j$. Let $S=S(e)$ for some $e\in\Gamma_0'(R)$ be a gr-simple submodule of $R_R$ such that $S\in j$. There exists $\sigma\in\Gamma$ with $r(\sigma)=e$ and a gr-isomorphism $\varphi\colon T\rightarrow S(\sigma)$. Since $T$ is a graded direct summand of $R_R$ by Proposition~\ref{prop:basics_gr-semisimple_modules}(1), $T=aR$ for some homogeneous idempotent $a\in R$. Then $aT=a(aR)=a^2R=T$. Now $S(\sigma)=\varphi(T)=\varphi(aT)=\varphi(a)T$. Moreover, $S(\sigma)=S(r(\sigma))=S(e)=S$ as sets. Since $I$ is an ideal of $R$, $S=\varphi(a)T \subseteq\varphi(a)I \subseteq I$, as desired.
\end{proof}

\subsection{Structure of gr-simple $\Gamma_0$-artinian rings}
As we have just shown in Theorem~\ref{theo:gr-semisimple_as_products_of_gr-simple}, a  right gr-semisimple ring is the product of gr-simple  right $\Gamma_0$-artinian rings. Our objective is to characterize this latter class of graded rings as certain matrix rings over gr-division rings. 
Proposition~\ref{prop: M I(D)(E) gr-simples artiniano} describes such matrix rings over gr-division rings and characterizes when they are gr-simple right $\Gamma_0$-artinian rings. For that, we will need  the following result first.

\begin{lemma}
\label{lem:HOM(E_jR,E_iR)}
    Let $D$ be a $\Gamma$-graded ring and let $\overline{\Sigma}=(\Sigma_i)_{i\in I}\in\mathcal{P}(\Gamma)^I$ be a matricial sequence for $D$. Consider the $\Gamma$-graded ring $R:=\M_I(D)(\overline{\Sigma})$. Given $i,j\in I$, $\sigma_i\in\Sigma_i$ and $\tau_j\in\Sigma_j$ we have that    
    \[\HOM_R(E_{jj}^{r(\tau_j)}R,E_{ii}^{r(\sigma_i)}R)\cong_{gr}\M_{1\times 1}(D)(\sigma_i)(\tau_j).\]
    Moreover,
    \[\END_R(E_{ii}^{r(\sigma_i)}R)\cong_{gr} \M_1(D)(\sigma_i)\]
    as graded rings.
\end{lemma}

\begin{proof}
    It is well-known that if $e$ is an idempotent of a ring $X$ and $M$ is a right $X$-module, then $\Hom_R(eR,M)\rightarrow Me$, $f\mapsto f(e)$, is an isomorphism of additive groups and $\Hom_R(eR,eR)\rightarrow eRe$, $f\mapsto f(e)$, is an isomorphism of rings. We apply this to the idempotent $E^{r(\tau_j)}_{jj}$ of the ring $R$. 
    
    For each $\gamma\in\Gamma$, a homomorphism $f\in \HOM_R(E_{jj}^{r(\tau_j)}R,E_{ii}^{r(\sigma_i)}R)_\gamma$ is uniquely determined by the {element} $f(E_{jj}^{r(\tau_j)})\in (E_{ii}^{r(\sigma_i)}RE_{jj}^{r(\tau_j)})_{\gamma d(\tau_j)}$ because $E_{jj}^{r(\tau_j)}\in R_{d(\tau_j)}$. This means that
    \[\HOM_R(E_{jj}^{r(\tau_j)}R,E_{ii}^{r(\sigma_i)}R)\cong_{gr} E_{ii}^{r(\sigma_i)}RE_{jj}^{r(\tau_j)}\cong_{gr}\M_{1\times 1}(D)(\sigma_i)(\tau_j).\qedhere\]
\end{proof}

\begin{proposition}
\label{prop: M I(D)(E) gr-simples artiniano}
    Let $D$ be a gr-division ring and $\overline{\Sigma}=(\Sigma_i)_{i\in I}\in\mathcal{P}(\Gamma)^I$ a matricial sequence for $D$. Consider the $\Gamma$-graded ring $R:=\M_I(D)(\overline{\Sigma})$. The following assertions hold:
    \begin{enumerate}[\rm (1)]
        \item $\mathcal{F}:=\{E_{ii}^{r(\sigma_i)}R:i\in I, \sigma_i\in\Sigma_i\}$ is a family of gr-simple $\Gamma$-graded $R$-modules. Analogously, $\mathcal{F}':=\{RE_{ii}^{r(\sigma_i)}:i\in I, \sigma_i\in\Sigma_i\}$ is a family of gr-simple left $R$-modules.
        \item Given $i,j\in I$, $\sigma_i\in\Sigma_i$ and $\tau_j\in\Sigma_j$ we have that $E_{ii}^{r(\sigma_i)}R$ and $E_{jj}^{r(\tau_j)}R$ are in the same isoshift class if and only if $1_{r(\sigma_i)}D1_{r(\tau_j)}\neq0$.
        \item For each $i\in I$ and $\sigma\in\Sigma_i$ we have 
        \[\END_R((E_{ii}^{r(\sigma)}R)(\sigma^{-1}))\cong_{gr}1_{r(\sigma)}D1_{r(\sigma)}.\]
        \item $R=\bigoplus\limits_{S\in\mathcal{F}}S=\bigoplus\limits_{T\in\mathcal{F}'}T$.
        \item $R$ is a right and left gr-semisimple ring.
        \item If $D$ is a gr-prime ring, then $R$ is a gr-simple ring.
        \item $R$ is a gr-simple right (left) $\Gamma_0$-artinian ring if and only if $D$ is a gr-prime ring.
    \end{enumerate}
\end{proposition}

\begin{proof}
(1) Let $i\in I$ and $\sigma_i\in\Sigma_i$. Consider the nonzero $\Gamma$-graded right $R$-module
\[S:=\left(E_{ii}^{r(\sigma_i)}\right)R\]
generated by the homogeneous element $E_{ii}^{r(\sigma_i)}\in R_{d(\sigma_i)}$. Let $\gamma\in\supp(S)$ and $0\neq s\in S_\gamma$ (that is, $s=(s_{kl})_{kl}$ where $s_{il}\in D_{\sigma_i\gamma\Sigma_l^{-1}}$ and $s_{kl}=0$ if $k\neq i$). Then there exists $j\in I$ such that $s_{ij}\neq0$. So, for each $\alpha\in\Gamma$ and $x=(x_{kl})_{kl}\in S_\alpha$, since $x_{kl}=s_{kl}=0$ for all $k\neq i$, we have
\[x=\sum_{t\in I}s\cdot\left(s_{ij}^{-1}x_{it}E_{jt}\right)\in sR.\]
Thus, $S$ is gr-simple. 
For $\mathcal{F}'$ we have a similar proof.

(2) By (1) and Lemma~\ref{lem: isoshift de simples}, we have that $E_{jj}^{r(\tau_j)}R$ and $E_{ii}^{r(\sigma_i)}R$ are in the same isoshift class if and only if $\HOM_R(E_{jj}^{r(\tau_j)}R, E_{ii}^{r(\sigma_i)}R)\neq 0$. But this is equivalent to $\M_{1\times 1}(D)(\sigma_i)(\tau_j)\neq 0$ by Lemma~\ref{lem:HOM(E_jR,E_iR)}. Now, note that $1_{r(\sigma_i)}D1_{r(\tau_j)}$ and $\M_{1\times 1}(D)(\sigma_i)(\tau_j)$ are naturally identified.


(3) Using Proposition~\ref{prop: HOM(,) e Hom_gr}(2) and Lemma~\ref{lem:HOM(E_jR,E_iR)}, we obtain that
\begin{align*}
    \HOM_R(E_{ii}^{r(\sigma)}R(\sigma^{-1}),E_{ii}^{r(\sigma)}R(\sigma^{-1}))_\gamma & \cong \HOM(E_{ii}^{r(\sigma)}R,E_{ii}^{r(\sigma)}R)_{\sigma^{-1}\gamma\sigma} \\ 
    & \cong \M_1(D)(\sigma)_{\sigma^{-1}\gamma\sigma}\\
    & \cong D_{\sigma(\sigma^{-1}\gamma\sigma)\sigma^{-1}}=D_\gamma
\end{align*}
for all $\gamma \in r(\sigma)\Gamma r(\sigma)$. If $\gamma\notin r(\sigma)\Gamma r(\sigma)$, $\HOM_R(E_{ii}^{r(\sigma)}R(\sigma^{-1}),E_{ii}^{r(\sigma)}R(\sigma^{-1}))_\gamma=0$. Thus, it is  induced a gr-isomorphism of graded rings as in the statement.

(4) It is clear that 
\[\bigoplus_{i\in I}\left(\bigoplus_{\sigma_i\in\Sigma_i}E_{ii}^{r(\sigma_i)}R\right)=R=\bigoplus_{i\in I}\left(\bigoplus_{\sigma_i\in\Sigma_i}RE_{ii}^{r(\sigma_i)}\right).\]

(5) It follows from (1) and (4).

(6) Suppose that $D$ is a gr-prime ring and let $U$ be a nonzero graded ideal of $R$. Then there is $0\neq x\in U_\gamma$ for some $\gamma\in\Gamma$. Thus, there exists $i,j\in I$ such that the $(i,j)$-entry of $x$ is $x_{ij}\in D_{\sigma_i\gamma\tau_j^{-1}}\setminus\{0\}$, where $\sigma_i\in\Sigma_i$ and $\tau_j\in\Sigma_j$ are (the unique) such that $d(\sigma_i)=r(\gamma)$ and $d(\tau_j)=d(\gamma)$. 
So, 
\[x_{ij}E_{ij}=\left(E_{ii}^{r(\sigma_i)}\right)x\left(E_{jj}^{r(\tau_j)}\right)\in U \Longrightarrow E_{ii}^{r(\sigma_i)}=\left(x_{ij}E_{ij}\right)\left(x_{ij}^{-1}E_{ji}\right)\in U.\]
Take any $k\in I$ and $\sigma\in\Sigma_k$. Since $D$ is a gr-prime gr-division ring and $r(\sigma_i),r(\sigma)\in\Gamma'_0(D)$, it follows from Proposition \ref{prop: quando D e'simples} that there exists $0\neq y_{ik}\in\h(1_{r(\sigma_i)}D1_{r(\sigma)})$. So
\[y_{ik}E_{ik}=\left(E_{ii}^{r(\sigma_i)}\right)\left(y_{ik}E_{ik}\right)\in U \Longrightarrow E_{kk}^{r(\sigma)}=\left(y_{ik}^{-1}E_{ki}\right)\left(y_{ik}E_{ik}\right)\in U.\]
Therefore, for each $a=(a_{kl})_{kl}\in \h(R)$, we have
\[a=\sum_{k,l\in I}a_{kl}E_{kl}=\sum_{k,l\in I}\left(E_{kk}^{r(\deg a_{kl})}\right)\left(a_{kl}E_{kl}\right)\in U\]
and it follows that $U=R$.

(7) By Proposition~\ref{prop: quando D e'simples},  $D=\bigoplus\limits_{[e]\in\Gamma_0'/\sim}D_{[e]}$ where each $D_{[e]}$ is a graded ideal of $D$ which is a gr-simple gr-division ring and  $D$ is not gr-prime if and only if $\Gamma_0'/\sim$ posseses more than one class. This implies that $R$ is not gr-simple if $D$ is not gr-prime because $\M_I(D_{[e]})(\overline{\Sigma})$ is an ideal of $R$ for each $e\in\Gamma_0/\sim$.

Suppose $D$ is gr-prime. By (6), $R$ is gr-simple. By (5), $R$ is left and right gr-semisimple. Now,  $R$ is a right and left $\Gamma_0$-artinian ring by Lemma~\ref{lem: R ss -> R(e) é soma finita}(3).
\end{proof}

Now we proceed to give some results with different  characterizations of gr-simple right (left) $\Gamma_0$-artinian rings in view.

\begin{lemma}
\label{lem: artin -> semisimp}
The following assertions hold:
\begin{enumerate}[\rm (1)]
    \item If $R$ is a right $\Gamma_0$-artinian ring, then, for each $e\in\Gamma_0$, either   $R(e)=0$ or $R(e)$ contains a minimal  graded right  ideal of $R$. 
    \item If $R$ is a gr-simple ring and it has a minimal graded right ideal $S=S(e)$ for some $e\in\Gamma_0$, then there exists a $d$-finite sequence $(\sigma_i)_{i\in I}\in (e\Gamma)^I$ such that $R_R\cong_{gr}\bigoplus\limits_{i\in I}S(\sigma_i)$. In particular, $R$ is a right gr-semisimple ring.
\end{enumerate}
\end{lemma}

\begin{proof}
    (1) Assume $R$ is a right $\Gamma_0$-artinian ring and take $e\in\Gamma'_0(R)$. Since $R(e)$ is a gr-artinian right $R$-module, $R(e)$ contains a minimal graded right  $R$-submodule $S$, which is a minimal graded right ideal of $R$.
        
    (2) Suppose that $R$ is a gr-simple ring and it has a minimal graded right ideal $S$. Then $S=S(e)$ for some $e\in\Gamma_0$ by Lemma \ref{lem: shift de simples}. Consider the family
    \[\mathcal{F}:=\{T \textrm{ graded right ideal of $R$}:\textrm{there exists $\sigma\in e\Gamma$ such that $T_R\cong_{gr}S(\sigma)$}\}.\]
By Lemma~\ref{lem:isoshiftical_component_is_a_graded_ideal}(1), \[R_S=\sum_{T\in\mathcal{F}}T\]
   is a graded ideal of $R$. 
    Since $R$ is gr-simple and $R_S\neq0$, we have 
    \[R=\sum_{T\in\mathcal{F}}T\cong_{gr}\sum_{T\in\mathcal{F}}S(\sigma_T),\]
    where $(\sigma_T)_{T\in \mathcal{F}}\in (e\Gamma)^\mathcal{F}$ is such that $T_R\cong_{gr}S(\sigma_T)$ for each $T\in\mathcal{F}$. By Proposition~\ref{prop:basics_gr-semisimple_modules}(2), there exists $I\subseteq\mathcal{F}$ such that
    \begin{equation}
    \label{eq: R = soma direta dos Si}
        R_R\cong_{gr}\bigoplus_{i\in I}S(\sigma_i).
    \end{equation}
    Now note that (\ref{eq: R = soma direta dos Si}) gives us 
    \[R(f)\cong_{gr}\bigoplus_{i\in I}S(\sigma_i)(f)=\bigoplus_{\substack{i\in I \\ d(\sigma_i)=f}}S(\sigma_i)\] 
    for each $f\in\Gamma_0$.  On the other hand,  Lemma \ref{lem: R ss -> R(e) é soma finita}(2) tells us that $R(f)$ is a direct sum of a finite number of $S(\sigma_i)$.
    Therefore $\{i\in I:d(\sigma_i)=f\}$ is a finite set for all $f\in \Gamma_0$, that is, $(\sigma_i)_{i\in I}\in (e\Gamma)^I$ is $d$-finite. 
\end{proof}

\begin{theorem}
\label{teo: simp + art = semisimp}
Suppose that $R$ is a gr-simple ring. The following assertions are equivalent.
\begin{enumerate}[\rm (1)]
    \item $R$ is a right  gr-semisimple ring.
    \item $R$ is a right  $\Gamma_0$-artinian ring.
    \item $R$ has a minimal $\Gamma$-graded right ideal.
\end{enumerate}
The equivalence of the left version of the foregoing statements holds true.
\end{theorem}

\begin{proof}
$(1)\implies(2)$ it follows from \ref{lem: R ss -> R(e) é soma finita}(3).

$(2)\implies(3)\implies(1)$ is Lemma \ref{lem: artin -> semisimp}.

In order to get the left versions of the statements, it suffices to note that $R$ is a gr-simple ring if and only if $R^{op}$ is a gr-simple ring.
\end{proof}

We give now the main step towards obtaining the structure of gr-simple right (left) $\Gamma_0$-artinian rings.

\begin{proposition}
\label{prop: WA para simples}
If $R$ is a gr-simple right (resp. left) $\Gamma_0$-artinian ring, then there  exist $e\in\Gamma_0$, a $\Gamma$-graded division ring $D$ with $\supp(D)\subseteq e\Gamma e$ and a $d$-finite sequence $\overline{\sigma}=(\sigma_i)_{i\in I}\in (e\Gamma)^I$ such that
\[R\cong_{gr}\M_I(D)(\overline{\sigma}).\]
Furthermore, if $R$ has unity, then $I$ is finite.
\end{proposition}

\begin{proof}
Suppose that $R$ is a gr-simple right $\Gamma_0$-artinian ring. By Lemma \ref{lem: artin -> semisimp}, there exist $e\in\Gamma_0$,  a minimal graded right  ideal $S=S(e)$ of $R$ and a $d$-finite sequence $\overline{\sigma}=(\sigma_i)_{i\in I}\in (e\Gamma)^I$ such that
\begin{equation*}
\label{eq: WA-s}
    R_R\cong_{gr}\bigoplus_{i\in I}S(\sigma_i)=S(\overline{\sigma}).
\end{equation*}

By Corollary \ref{coro: schur para simples}, $D:=\END_R(S)$ is a $\Gamma$-graded division ring and $\supp(D)\subseteq e\Gamma e$.
By Lemma \ref{lem: R=END(R)} and Corollary \ref{coro: END M(sigma)}(1), we have gr-isomorphisms of $\Gamma$-graded rings
\[R\cong_{gr}\END(R_R)\cong_{gr}\END_R(S(\overline{\sigma}))\cong_{gr}\M_I(D)(\overline{\sigma}).\]

If $R$ is a gr-simple left $\Gamma_0$-artinian ring, then $R^{op}$ is a gr-simple right $\Gamma_0$-artinian ring. As we have just proved, there exist $e\in\Gamma_0$, a $\Gamma$-graded division ring $D$ with $\supp(D)\subseteq e\Gamma e$ and $d$-finite sequence $\overline{\sigma}=(\sigma_i)_{i\in I}\in (e\Gamma)^I$ such that 
\(R^{op}\cong_{gr}\M_I(D)(\overline{\sigma})\). Then we get form Proposition \ref{prop: oposto}(1) that 
\[R\cong_{gr}\M_I(D)(\overline{\sigma})^{op}\cong_{gr}\M_I(D^{op})(\overline{\sigma}).\]

Finally, note that if $R$ has unity, then $R(f)\neq0$ only for a finite number of $f\in\Gamma_0$. Then Lemma \ref{lem: R ss -> R(e) é soma finita}(2) guarantees  the existence of finite subsets $I_f\subseteq I$, $f\in\Gamma'_0(R)$, such that $I=\bigcup\limits_{f\in\Gamma'_0(R)}I_f$ is finite.
\end{proof}

Note that, in Proposition~\ref{prop: WA para simples}, $D$ is graded by the isotropy group $e\Gamma e$ and we have $r(\sigma_i)=e$ for all $i\in I$.  But the $\Gamma$-grading on $\M_I(D)(\overline{\sigma})$ may not be realized as a group grading (in the sense of \cite[Definition 1.8]{Elduque}). 
In fact, if there exist $i\neq j$ such that $d(\sigma_i)\neq d(\sigma_j)$, then $E^e_{ii}\in \M_I(D)(\overline{\sigma})_{d(\sigma_i)}$ and $E^e_{jj}\in \M_I(D)(\overline{\sigma})_{d(\sigma_j)}$. Therefore, $E^e_{ii}$ and $E^e_{jj}$ are in distinct homogeneous components. This cannot occur in a group grading because $E^e_{ii}$ and $E^e_{jj}$ are nonzero homogeneous idempotents and thus both should be in the component corresponding to the unity of the group.

On the other hand, the $\Gamma$-grading on $\M_I(D)(\overline{\sigma})$ is a refinement of a group grading, see \cite[Corollary~2 of Section~4]{CMGN}.   Indeed, for each $f\in\Gamma'_0:=\Gamma'_0(\M_I(D)(\overline{\sigma}))$, take $i_f\in I$ such that $d(\sigma_{i_f})=f$. Then $\M_I(D)(\overline{\sigma})$ is a $e\Gamma e$-graded ring if we define the component of degree $g\in e\Gamma e$ as $\bigoplus\limits_{f,f'\in\Gamma'_0}\M_I(D)(\overline{\sigma})_{\sigma_{i_f}^{-1}g\sigma_{i_{f'}}}$.
More generally, suppose that $R=\bigoplus\limits_{\gamma\in\Gamma}R_\gamma$ is a  $\Gamma$-graded ring such that there is
$e\in\Gamma_0'(R)$ with the property that for all $f\in\Gamma'_0(R)$ there exists $\sigma_f\in e\Gamma f$.  In other words, the groupoid $\Gamma':=\{\gamma\in\Gamma:1_{r(\gamma)},1_{d(\gamma)}\neq0\}$ is connected. Then $R$ is a
$e\Gamma e$-graded ring if we define the component of degree $g\in e\Gamma e$ as $\bigoplus\limits_{f,f'\in\Gamma'_0(R)}R_{\sigma_f^{-1}g\sigma_{f'}}$. This is a coarsening of the $\Gamma$-grading of $R$.



\begin{corollary}
\label{coro: art simp nao tem lado}
    Let $R$ be a gr-simple ring. Then $R$ is a right $\Gamma_0$-artinian ring if and only if $R$ is a left $\Gamma_0$-artinian ring.
\end{corollary}

\begin{proof}
     By Proposition \ref{prop: WA para simples}, if $R$ is either a right or a left $\Gamma_0$-artinian ring, then there exist $e\in\Gamma_0$, a  $\Gamma$-graded division ring $D$ with $\supp(D)\subseteq e\Gamma e$ and a $d$-finite sequence $\overline{\sigma}=(\sigma_i)_{i\in I}\in (e\Gamma)^I$ such that $R\cong_{gr}\M_I(D)(\overline{\sigma})$. But $\M_I(D)(\overline{\sigma})$ is a right and left gr-semisimple ring by Proposition \ref{prop: M I(D)(E) gr-simples artiniano}(5). It follows from Theorem \ref{teo: simp + art = semisimp} that $\M_I(D)(\overline{\sigma})$ is a right and left $\Gamma_0$-artinian ring.
\end{proof}

\begin{remark}
    From now on, in view of Corollary~\ref{coro: art simp nao tem lado}, we shall be at liberty to drop the adjectives ``left'' and ``right'' and just talk about \emph{gr-simple $\Gamma_0$-artinian rings}.\qed
\end{remark}

For  a $\Gamma$-graded division ring $D$ and   
$\Gamma$-graded right $D$-module $V$, we say that $V$ is \emph{$\Gamma_0$-finite dimensional over $D$} if $V(e)$ has finite pseudo-dimension for each $e\in\Gamma_0$.

Combining the previous results, we obtain the following characterization of gr-simple $\Gamma_0$-artinian rings which is the main result of this section.

\begin{theorem}
\label{teo: WA para simp -resumo}
The following statements are equivalent.
\begin{enumerate}[\rm (1)]
    \item $R$ is a gr-simple $\Gamma_0$-artinian ring.
    \item There exist $e\in\Gamma_0$,  a $\Gamma$-graded division ring $D$ with $\supp(D)\subseteq e\Gamma e$ and a $d$-finite sequence $\overline{\sigma}:=(\sigma_i)_{i\in I}\in (e\Gamma)^I$  such that
    \[R\cong_{gr}\M_I(D)(\overline{\sigma}).\]
    \item There exist  a gr-prime $\Gamma$-graded division ring $D$ and a matricial sequence $\overline{\Sigma}:=(\Sigma_i)_{i\in I}\in \mathcal{P}(\Gamma)^I$ for $D$   such that
    \[R\cong_{gr}\M_I(D)(\overline{\Sigma}).\]
    \item There exist  a $\Gamma$-graded division ring $D$ with $\supp(D)\subseteq e\Gamma e$ for some $e\in \Gamma_0$ and a $\Gamma$-graded right $D$-module $V$ which is $\Gamma_0$-finite dimensional over $D$ such that  
    \[R\cong_{gr}\END_D(V).\] 
    \item There exist  a gr-prime $\Gamma$-graded division ring $D$ and a $\Gamma$-graded right $D$-module $V$ which is $\Gamma_0$-finite dimensional over $D$ such that
    \[R\cong_{gr}\END_D(V).\] 
\end{enumerate}
\end{theorem}

\begin{proof}
    Proposition~\ref{prop: WA para simples} gives $(1)\implies(2)$ and it is clear that $(2)\implies(3)$.
    
    Proposition~\ref{prop: M I(D)(E) gr-simples artiniano}(7) gives $(3)\implies(1)$.

    Let us see that $(3)\Longleftrightarrow(5)$. If (3) holds, then just take $V:=D(\overline{\Sigma})$. Note that $V$ has $\Gamma_0$-finite dimension because $\overline{\Sigma}$ is $d$-finite. So, applying Corollary \ref{coro:graded_endomorphism_ring of modules}(4), we get $\END_D(V)=\END(D(\overline{\Sigma}))\cong_{gr}\M_I(D)(\overline{\Sigma})\cong_{gr}R$. Conversely, if (5) holds, then Theorem~\ref{theo:modules_over_division_rings}(1) and Proposition~\ref{prop:pseudo_free}(7) imply that $V\cong_{gr}\bigoplus\limits_{i\in I}D(\sigma_i)$ for a certain $\overline{\sigma}:=(\sigma_i)_{i\in I}\in\Gamma^I$ with $r(\sigma_i)\in\Gamma'_0(D)$ for each $i\in I$. Since $V$ is $\Gamma_0$-finite dimensional, it follows that $\overline{\sigma}$ is $d$-finite and therefore matricial for $D$. Then Corollary~\ref{coro:graded_endomorphism_ring of modules}(4) gives $\M_I(D)(\overline{\sigma})\cong_{gr}\END(D(\overline{\sigma}))\cong_{gr}\END_D(V)\cong_{gr}R$.

    The equivalence $(2)\Longleftrightarrow(4)$ is proved analogously to $(3)\Longleftrightarrow(5)$ using Corollary~\ref{coro: END M(sigma)}(2).
\end{proof}

In Section~\ref{sec: teo da den}, we present another proof of Theorem~\ref{teo: WA para simp -resumo} using a groupoid graded version of the Chevalley-Jacobson Density Theorem. In fact, we will obtain a slightly stronger result, for graded rings that are not necessarily object unital.

\medskip

We end this subsection with some consequences of
 Theorem~\ref{teo: WA para simp -resumo}  about some important graded subrings.

\begin{proposition}
\label{prop: R gr-art gr-simp --> 1eR1e tbm}
    Let $R$ be a gr-simple $\Gamma_0$-artinian ring. For each $\Delta_0\subseteq\Gamma_0$, the ring $1_{\Delta_0}R1_{\Delta_0}:=\bigoplus_{e,f\in\Delta_0}1_eR1_f$  is also gr-simple $\Gamma_0$-artinian. In particular, if $e\in\Gamma_0$, then $1_eR1_e$ is a gr-simple gr-artinian ring (as a group graded ring).
\end{proposition}

\begin{proof}
    By Theorem~\ref{teo: WA para simp -resumo}, there exist $e_0\in\Gamma_0$, a  $\Gamma$-graded division ring $D$ with $\supp(D)\subseteq e_0\Gamma e_0$ and a $d$-finite sequence $\overline{\sigma}=(\sigma_i)_{i\in I}\in (e_0\Gamma)^I$ such that $R\cong_{gr}\M_I(D)(\overline{\sigma})$.

    If $\Delta_0\subseteq\Gamma_0$, then 
    \[I_{\Delta_0}:=\{i\in I:d(\sigma_i)\in\Delta_0\}\]
    is such that $$1_{\Delta_0}R1_{\Delta_0}\cong_{gr}\M_{I_{\Delta_0}}(D)(\overline{\sigma}_{I_{\Delta_0}})$$ is a gr-simple $\Gamma_0$-artinian ring, by Theorem~\ref{teo: WA para simp -resumo}.
\end{proof}

Inspired by \cite[Subsection 1.4.1]{Hazrat}, the next example shows that there does not exist a full version of Proposition~\ref{prop: R gr-art gr-simp --> 1eR1e tbm} for the subrings $R_e$, $e\in\Gamma_0$. In Proposition~\ref{prop: R gr-art gr-simp --> Re tbm}, we give a characterization of when the subrings $R_e$ are simple artinian.

\begin{example}
    Let $K$ be a division ring and $G$ be a nontrivial group. Consider $K$ with the trivial $G$-grading and take $\sigma,\tau\in G$ such that $\sigma\neq\tau$. Then $\M_2(K)(\sigma,\tau)$ is a gr-simple gr-artinian ring, but 
    \[\M_2(K)(\sigma,\tau)_{e_G}=
    \begin{pmatrix}
    K & 0 \\
    0 & K \\
    \end{pmatrix}\]
    is not a simple ring.\qed
\end{example}

\begin{proposition}
\label{prop: R gr-art gr-simp --> Re tbm}
Let $R$ be a $\Gamma$-graded ring.
If $R$ is a gr-simple $\Gamma_0$-artinian ring, then $R_e$ is a semisimple ring for all $e\in\Gamma_0$. More precisely, suppose that  $R\cong_{gr}\M_I(D)(\overline{\sigma})$ where 
$e_0\in\Gamma_0$, $D$ is a $\Gamma$-graded division ring with $\supp(D)\subseteq e_0\Gamma e_0$ and $\overline{\sigma}:=(\sigma_i)_{i\in I}\in (e_0\Gamma)^I$ is a $d$-finite sequence.  Fix $e\in\Gamma'_0(R)$ and consider the finite set $I_e:=\{i\in I:d(\sigma_i)=e\}$. The following statements hold true:
    \begin{enumerate}[\rm (1)]
        \item There exist positive integers $n_1,...,n_k$ such that $n_1+\cdots+n_k=|I_e|$ and $R_e\cong\prod_{k=1}^{n}\M_{n_k}(D_{e_0})$. 
        \item $R_e$ is a simple artinian ring if and only if  $D_{\sigma_i\sigma_j^{-1}}\neq0$  for all $i,j\in I$ with $d(\sigma_i)=d(\sigma_j)=e$. In this case, $R_e\cong\M_{|I_e|}(D_{e_0})$.
    \end{enumerate}
\end{proposition}

\begin{proof}
Suppose that $R$ is a gr-simple $\Gamma_0$-artinian ring. The fact that $R_e$ is a semisimple ring for all $e\in\Gamma_0$ will follow from (1) and Proposition~\ref{prop: WA para simples}.

    (1) Consider the following equivalence relation in $I_e$:
    \[i\sim j\iff \sigma_i\sigma_j^{-1}\in\supp D.\]
    Let $I_1,...,I_n$ be the equivalence classes of this relation. 
    For each $k=1,...,n$ and $i,j\in I_k$, we have, by Proposition \ref{prop: HOM(,) e Hom_gr}(2), 
    \begin{align*}
        \Hom_{\textrm{gr-$D$}} (D(\sigma_j),D(\sigma_i))&\cong \HOM_D(D,D(\sigma_i))_{\sigma_j^{-1}}\\
        &=\Hom_{\textrm{gr-$D$}}(D,D(\sigma_i)(\sigma_j^{-1}))\\
        &=\HOM_D(D,D)_{\sigma_i\sigma_j^{-1}}\\
        &\cong D_{\sigma_i\sigma_j^{-1}}\neq0
    \end{align*}
    and it follows that $D(\sigma_j)\cong_{gr}D(\sigma_i)$.
    For each $k=1,...,n$, fix $i_k\in I_k$ and let $n_k:=|I_k|$. Set $\overline{\sigma'}:=(\sigma'_i)_{i\in I}\in (e_0\Gamma)^I$ where $\sigma'_i=\sigma_{i_k}$ if $i\in I_k$ and $\sigma_i'=\sigma_i$ if $i\notin I_e$. Then
    \begin{eqnarray*}
        D(\overline\sigma)& = & \bigoplus_{i\notin I_e}D(\sigma_i)\oplus \bigoplus_{i\in I_e}D(\sigma_i) \\
        &=& \bigoplus_{i\notin I_e}D(\sigma_i)\oplus \bigoplus_{k=1}^{n}\bigoplus_{i\in I_k}D(\sigma_i) \\
        &\cong_{gr}&  \bigoplus_{i\notin I_e}D(\sigma_i)\oplus\bigoplus_{k=1}^{n}D(\sigma_{i_k})^{(n_k)}\\
        &=&D(\overline{\sigma'}).
    \end{eqnarray*}
    Thus, by Corollary \ref{coro: END M(sigma)}(2), we obtain
    \[R\cong_{gr}\END_D(D(\overline{\sigma}))\cong_{gr}\END_D(D(\overline{\sigma'}))\cong_{gr}\M_I(D)(\overline{\sigma'}).\]
    Now, note that if $(a_{ij})_{ij}\in\M_I(D)(\overline{\sigma'})_e$, then 
    \[a_{ij}\neq0\implies D_{\sigma'_i{\sigma'}_j^{-1}}\neq0\implies i\sim j.\]
    Therefore
    \begin{equation}
    \label{eq:R_e simple art}
    R_e\cong\M_I(D)(\overline{\sigma'})_e\cong\prod_{k=1}^{n}\M_{n_k}(D_{e_0}).
    \end{equation}

    (2) follows from \eqref{eq:R_e simple art}.
\end{proof}

\subsection{Structure of gr-semisimple rings}


In this subsection, we obtain a version of Wedderburn-Artin Theorem for gr-semisimple rings. For that, we begin with some technical results.

\begin{proposition}
\label{prop: soma de ss e' ss}
Let $\{R_j:j\in J\}$ be a summable family of $\Gamma$-graded right (left) gr-semisimple rings. Then $R:=\sideset{}{^{gr}}\prod\limits_{j\in J}R_j$ is a right (resp. left) gr-semisimple ring.
\end{proposition}

\begin{proof}
We begin with the right case.
For each $j\in J$, we can write $R_j=\bigoplus\limits_{k\in K_j}S_{jk}$ for some family $\{S_{jk}:k\in K_j\}$ of gr-simple $\Gamma$-graded right $R_j$-modules.

Fix $j'\in J$ and $k\in K_{j'}$. We can make $S_{j'k}$ a $\Gamma$-graded right $R$-module via $s\cdot(r_j)_{j\in J}:=sr_{j'}$ for all $s\in S_{j'k}$ and $(r_j)_{j\in J}\in R$. Let us see that $S_{j'k}$ is a gr-simple $R$-module. Let $0\neq s \in \h(S_{j'k})$ and take any $x\in S_{j'k}$. Since $sR_{j'}=S_{j'k}$, there exists $r'\in R_{j'}$ such that $sr'=x$. Therefore, $(r_j)_{j\in J}\in R$, where $r_{j'}=r'$ and $r_j=0$ for all $j\neq j'$, is such that $s\cdot(r_j)_{j\in J}=x$. Thus, $sR=S_{j'k}$. From this we conclude that
\[R=\sideset{}{^{gr}}\prod_{j\in J}R_j=\bigoplus_{j\in J}R_j=\bigoplus_{j\in J}\bigoplus_{k\in K_j}S_{jk}\]
is direct sum of gr-simple right $R$-modules. Hence, $R$ is a right gr-semisimple ring.

Suppose now that $\{R_j:j\in J\}$ is a summable family of left gr-semisimple rings. So $\{(R_j)^{op}:j\in J\}$ is a summable family of right gr-semisimple rings. From what we have just proved, $\sideset{}{^{gr}}\prod\limits_{j\in J}(R_j)^{op}$ is a right gr-semisimple ring. By Proposition \ref{prop: oposto}(5), we have that $R^{op}$ is right gr-semisimple. It follows that $R$ is left gr-semisimple.
\end{proof}


\begin{corollary}
\label{coro: ss nao tem lado}
    The $\Gamma$-graded ring $R$ is right gr-semisimple ring if and only if it is left gr-semisimple ring.
\end{corollary}

\begin{proof}
    If $R$ is a right gr-semisimple ring,  it follows from Theorem~\ref{theo:gr-semisimple_as_products_of_gr-simple}
    that there exists a summable family $\{R_j:j\in J\}$ of gr-simple $\Gamma_0$-artinian rings such that $R\cong_{gr}\sideset{}{^{gr}}\prod\limits_{j\in J}R_j$. By Corollary \ref{coro: art simp nao tem lado} and Theorem \ref{teo: simp + art = semisimp}, each $R_j$ is a left gr-semisimple ring. It follows from Proposition \ref{prop: soma de ss e' ss} that $R$ is a left gr-semisimple ring.

    Conversely, if $R$ is a left gr-semisimple ring, then $R^{op}$ is a right gr-semisimple ring. So, as we have just proved, $R^{op}$ is a left gr-semisimple ring, that is, $R$ is a right gr-semisimple ring.
\end{proof}

\begin{remark} 
From now on, in view of Corollary~\ref{coro: ss nao tem lado}, we shall be at liberty to drop the adjectives ``left'' and ``right'' and just talk about \emph{gr-semisimple  ring}. \qed
\end{remark}

Combining previous results, we obtain the following characterization of gr-semisimple rings that can be seen as a version of the Wedderburn-Artin Theorem.

\begin{theorem}
\label{teo: wa para ss -resumo}
 The following statements are equivalent for the $\Gamma$-graded ring $R$.
\begin{enumerate}[\rm (1)]
    \item $R$ is a gr-semisimple ring.
    \item There exist $(e_j)_{j\in J}\in(\Gamma_0)^J$ and, for each $j\in J$, a $\Gamma$-graded division ring $D_j$ with $\supp(D_j)\subseteq e_j\Gamma e_j$ and a $d$-finite sequence $\overline{\sigma}_j:=(\sigma_{jk})_{k\in K_j}\in (e_j\Gamma)^{K_j}$ such that the family $\{\M_{K_j}(D_j)(\overline{\sigma}_j):j\in J\}$ is summable and
    \[R\cong_{gr}\sideset{}{^{gr}}\prod_{j\in J}\M_{K_j}(D_j)(\overline{\sigma}_j).\]
    \item There exists a set $J$ and, for each $j\in J$, a gr-prime $\Gamma$-graded division ring $D_j$ and a matricial sequence $\overline{\Sigma}_j:=(\Sigma_{jk})_{k\in K_j}\in \mathcal{P}(\Gamma)^{K_j}$ for $D_j$ such that the family $\{\M_{K_j}(D_j)(\overline{\Sigma}_j):j\in J\}$ is summable and
    \[R\cong_{gr}\sideset{}{^{gr}}\prod_{j\in J}\M_{K_j}(D_j)(\overline{\Sigma}_j).\]
    \item There exists a summable family $\{R_j:j\in J\}$ of gr-simple $\Gamma_0$-artinian rings such that
    \[R\cong_{gr}\sideset{}{^{gr}}\prod_{j\in J}R_j.\]
\end{enumerate}
\end{theorem}

\begin{proof}
    $(1)\implies(2)$: By Theorem \ref{theo:gr-semisimple_as_products_of_gr-simple}, the summable family $\{R_j\}_{j\in\mathcal{S}(R)}$ of gr-simple $\Gamma_0$-artinian rings is such that $R=\prod_{j\in \mathcal{S}(R)}^{gr}R_j$. By Theorem~\ref{teo: WA para simp -resumo}, for each $j\in\mathcal{S}(R)$, $R_j\cong_{gr} \M_{K_j}(D_j)(\overline{\sigma}_j)$ where $D_j$ is a $\Gamma$-graded division ring with $\supp(D_j)\subseteq e_j \Gamma e_j$ for some $e_j\in\Gamma_0$ and  $\overline{\sigma}_j=(\sigma_{jk})_{k\in K_j}\in (e_j \Gamma)^{K_j}$  is a $d$-finite sequence.  

    $(2)\implies(3)$: It is clear.
    
    $(3)\implies(4)$: By Proposition \ref{prop: M I(D)(E) gr-simples artiniano}(7), $R_j:=\M_{K_j}(D_j)(\overline{\Sigma}_j)$ is a gr-simple $\Gamma_0$-artinian ring for all $j\in J$.
    
    $(4)\implies(1)$: Theorem~\ref{teo: simp + art = semisimp} tell us that each $R_j$ is gr-semisimple. Therefore, $R$ is gr-semisimple by Proposition~\ref{prop: soma de ss e' ss}.
\end{proof}

Note that if $R$ is a gr-semisimple ring with unity, then $\Gamma'_0(R)$ is finite and it follows that, in Theorem \ref{teo: wa para ss -resumo}(2), we have that
\[K_j=\bigcup_{e\in\Gamma'_0(R)}\{k\in K_j:d(\sigma_{jk})=e\}\] is finite for all $j\in J$ and
\[J=\bigcup_{e\in\Gamma'_0(R)}\{j\in J:\M_{K_j}(D_j)(\overline{\sigma}_j)_e\neq0\}\]
is finite.


\begin{remark}
    One can prove $(1)\implies(2)$ in Theorem~\ref{teo: wa para ss -resumo}, using a more traditional argument. Suppose that $\{T_i:i\in I\}$ is a family of gr-simple graded right $R$-submodules of $R$ such that $R=\bigoplus\limits_{i\in I}T_i$. 
    Let $\{S_j:j\in J\}$ be a subset of $\{T_i:i\in I\}$ formed by exactly one representative of each isoshift class. 
    For each $j\in J$, let $e_j\in\Gamma_0$ be such that $S_j=S_j(e_j)$.
    Grouping together the modules that are in the same class, we obtain  a sequence $\overline{\sigma}_j:=(\sigma_{jk})_{k\in K_j}\in (e_j\Gamma)^{K_j}$ for each  $j\in J$ such that
    \begin{equation}
    \label{eq: R soma infinita de simples}
        R_R\cong_{gr}\bigoplus_{j\in J}S_j(\overline{\sigma}_j).
    \end{equation}
    From Lemma~\ref{lem: isoshift de simples}, we have that $\HOM_R\left(S_j(\overline{\sigma}_j),S_{j'}(\overline{\sigma}_{j'})\right)=\{0\}$, for distinct $j,j'\in J$. So we can apply Corollary~\ref{coro: END(+,+) infinito ortogonal} which, together with Lemma \ref{lem: R=END(R)}, gives us the following gr-isomorphisms of $\Gamma$-graded rings:
    \begin{equation}
    \label{eq: R = prod gr}
        R\cong_{gr} \END(R_R)\cong_{gr}\sideset{}{^{gr}}\prod_{j\in J}\END_R(S_j(\overline{\sigma}_j)).
    \end{equation}
    By Lemma~\ref{lem: R ss -> R(e) é soma finita}(2), for each $e\in\Gamma_0$,
    we have that $R(e)$ is gr-isomorphic to a direct sum of a finite number of $S_j(\sigma_{jk})$. Thus, it follows from 
    \eqref{eq: R soma infinita de simples} that, for all $j\in J$, $K_{j,e}:=\{k\in K_j:d(\sigma_{jk})=e\}$  and $J_e:=\{j\in J:K_{j,e}\neq\emptyset\}$ are finite sets.
    The finiteness of the sets $K_{j,e}$ implies that each $\overline{\sigma}_j$ is $d$-finite and therefore we can use Corollary~\ref{coro: END M(sigma)}(1). Hence, \eqref{eq: R = prod gr} gives us $\displaystyle R\cong_{gr} \sideset{}{^{gr}}\prod_{j\in J}\M_{K_j}(D_j)(\overline{\sigma}_j)$, where, for each $j\in J$, $D_j:=\END_R(S_j)$.
    The fact that $\{\M_{K_j}(D_j)(\overline{\sigma}_j):j\in J\}$ is a summable family follows from Proposition~\ref{prop: quando familia e somavel}, the finiteness of $J_e$ ($e\in\Gamma_0$) and the equality
    \[J_e=\{j\in J: \M_{K_j}(D_j)(\overline{\sigma}_j)_e\neq0\},\]
    for all $e\in\Gamma_0$. If $\M_{K_j}(D_j)(\overline{\sigma}_j)_e\neq0$, $e\in\Gamma_0$, then there exist $k,l\in K_j$ such that $(D_j)_{\sigma_{jk}e\sigma_{jl}^{-1}}\neq0$. In particular, $\sigma_{jk}e\sigma_{jl}^{-1}$ is defined. Thus, we obtain $k\in K_{j,e}$ and therefore $j\in J_e$. Conversely, if $j\in J_e$, then taking $k\in K_{j,e}$, i.e., $d(\sigma_{jk})=e$, we have that $0\neq E_{kk}^{e_j}\in \M_{K_j}(D_j)(\overline{\sigma}_j)_e$.\qed
\end{remark}

The following results are about the relation between the gr-semisimplicity of a groupoid graded ring and certain important graded subrings.

\begin{proposition}
\label{prop: R gr-ss --> 1eR1e, Re tbm}
    Let $R$ be a $\Gamma$-graded gr-semisimple ring. The following statements hold true.
    \begin{enumerate}[\rm (1)]
        \item If $\Delta_0\subseteq\Gamma_0$, then $1_{\Delta_0}R1_{\Delta_0}:=\bigoplus_{e,f\in\Delta_0}1_eR1_f$ is a gr-semisimple ring.
        \item If $e\in\Gamma_0$, then $1_eR1_e$ is a gr-semisimple ring.
        \item If $e\in\Gamma_0$, then $R_e$ is a semisimple ring.
    \end{enumerate}
\end{proposition}

\begin{proof}
    By Theorem \ref{teo: wa para ss -resumo}, there exists a summable family $\{R_j:j\in J\}$ of gr-simple $\Gamma_0$-artinian rings such that
    \[R\cong_{gr}\sideset{}{^{gr}}\prod_{j\in J}R_j.\]

    (1) Let $\Delta_0\subseteq\Gamma_0$. By Proposition \ref{prop: R gr-art gr-simp --> 1eR1e tbm}, $1_{\Delta_0}R_j1_{\Delta_0}$ is a gr-simple $\Gamma_0$-artinian ring, in particular gr-semisimple, for all $j\in J$. It is clear from Proposition \ref{prop: quando familia e somavel}, that $\{1_{\Delta_0}R_j1_{\Delta_0}:j\in J\}$ is a summable family. Then
    \[1_{\Delta_0}R1_{\Delta_0}\cong_{gr}\sideset{}{^{gr}}\prod_{j\in J}1_{\Delta_0}R_j1_{\Delta_0}\]
    is a gr-semisimple ring, by Proposition \ref{prop: soma de ss e' ss}.

    (2) Set $\Delta_0=\{e\}$. Since $1_{\Delta_0}R1_{\Delta_0}=1_eR1_e$, the result follows from (1).

    (3) Let $e\in \Gamma_0$. By Proposition \ref{prop: R gr-art gr-simp --> Re tbm}, $(R_j)_e$ is a semisimple ring for all $j\in J$. By Proposition \ref{prop: quando familia e somavel}, the set $J_e:=\{j\in J:(R_j)_e\neq0\}$ is finite and it follows that $R_e\cong\prod_{j\in J_e}(R_j)_e$ is a semisimple ring. 
\end{proof}

In general, the converse of items (1)--(3) in Proposition \ref{prop: R gr-ss --> 1eR1e, Re tbm} are not true as the following example shows.

\begin{example}
    Let $K$ be a division ring and $\Gamma:=\{1,2\}\times\{1,2\}$. Consider the ring $R:=\begin{pmatrix}
K & K \\
0 & K \\
\end{pmatrix}$ $\Gamma$-graded via $R_{(i,j)}:=E_{ii}RE_{jj}$ for each $1\leq i\leq j\leq 2$. Then $R_{(1,1)}$ and $R_{(2,2)}$ are semisimple rings. However, $R$ is not a gr-semisimple ring because $E_{12}R$ is not a direct summand of $R_R$.\qed
\end{example}

Before obtaining some important cases where the converse of items (1)--(3) in Proposition~\ref{prop: R gr-ss --> 1eR1e, Re tbm} hold, we need some definitions. Let $e\in\Gamma_0$. For a $\Gamma$-graded ring $R$, we say that $R$ is \emph{right $e$-faithful} if, for each, $\gamma\in e\Gamma$ and $0\neq a\in R_\gamma$, there exists $r\in R_{\gamma^{-1}}$ such that $0\neq ar\in R_e$. We will say that $R(e)$ is \emph{$\{e\}$-faithful} if, for each $0\neq a\in\h(R(e))$, there exists $r\in\h(R)$ such that $0\neq ar \in 1_eR1_e$. Clearly, if $R$ is right $e$-faithful, then $R(e)$ is $\{e\}$-faithful.  
We also say that $R$ is \emph{strongly $\Gamma$-graded} if $R_\gamma R_\delta=R_{\gamma\delta}$ for all $\gamma,\delta\in\Gamma$. It is not difficult to see that if $R$ is strongly $\Gamma$-graded, then $R$ is right $e$-faithful for all $e\in\Gamma_0$.  The previous concepts generalize items (2) and (3) of \cite[Definition 7 (p. 536)]{Balaba}.

\begin{proposition}
\label{prop: fidelidade e semissimplicidade}
Let $R$ be a $\Gamma$-graded ring.
    \begin{enumerate}[\rm (1)]
        \item If, for all $e\in\Gamma_0$, $R_e$ is a semisimple ring and $R$ is right $e$-faithful, then $R$ is a gr-semisimple ring. 
        \item If, for all $e\in\Gamma_0$, $1_eR1_e$ is a gr-semisimple ring and $R(e)$ is $\{e\}$-faithful, then $R$ is a gr-semisimple ring. 
    \end{enumerate}
\end{proposition}

\begin{proof}
    (1) Suppose that $e\in\Gamma_0$, $R$ is right $e$-faithful and $1_e=s_1+\cdots+s_n$, where $s_1,...,s_n\in R_e$ and $s_iR_e$ is a simple $R_e$-module for each $i=1,...,n$. We will show that each $s_iR$ is a gr-simple $R$-module. Fix $i=1,...,n$ and $0\neq a\in s_iR$. By $e$-faithfulness, there exists $r\in\h(R)$ such that $0\neq ar\in R_e$. Since $0\neq ar \in (s_iR)_e=s_iR_e$, it follows from simplicity that $arR_e=s_iR_e$. In particular, $s_i\in arR_e\subseteq aR$ and thus $aR=s_iR$. Hence $R(e)=1_eR=\sum_{i=1}^{n}s_iR$ is a gr-semisimple $R$-module.
    
    (2) Suppose that $e\in\Gamma_0$, $R(e)$ is $\{e\}$-faithful and $1_e=s_1+\cdots+s_n$, where $s_1,...,s_n\in R_e$ and $s_i(1_eR1_e)$ is a gr-simple $1_eR1_e$-module for all $i=1,...,n$. It suffices to show that each $s_iR$ is a gr-simple $R$-module. Fix $i=1,...,n$ and $0\neq a\in s_iR$. Take $r\in\h(R)$ such that $0\neq ar\in 1_eR1_e$. Then $0\neq ar \in 1_e(s_iR)1_e=s_i(1_eR1_e)$ and it follows that $ar(1_eR1_e)=s_i(1_eR1_e)$. Thus, $s_i\in aR$ and therefore $aR=s_iR$. 
\end{proof}

\begin{corollary}
\label{coro: fort grad e gr-ss}
    For a strongly $\Gamma$-graded ring $R$, the following assertions are equivalent:
    \begin{enumerate}[\rm (1)]
        \item $R$ is a gr-semisimple ring.
        \item $1_eR1_e$ is a gr-semisimple ring for all $e\in\Gamma_0$.
        \item $R_e$ is a semisimple ring for all $e\in\Gamma_0$.\qed
    \end{enumerate}
\end{corollary}





\begin{corollary}
    For a gr-prime strongly $\Gamma$-graded ring $R$, the following assertions are equivalent:
    \begin{enumerate}[\rm (1)]
        \item $R$ is a gr-simple $\Gamma_0$-artinian ring. 
        \item $1_eR1_e$ is a gr-simple gr-artinian ring for all $e\in\Gamma_0$.
    \end{enumerate}
\end{corollary}

\begin{proof}
    The result follows from Corollary~\ref{coro: fort grad e gr-ss}, noting that each $1_eR1_e$, $e\in\Gamma_0$, is a gr-prime ring.
\end{proof}

Following \cite[Definition 12]{CLP2}, we will say that the $\Gamma$-graded ring $R$ is an \emph{object crossed product} if, for each $\gamma\in\Gamma$, there exists an invertible element in $R_\gamma$. Object crossed products are strongly graded. In fact, for each $\sigma,\tau\in\Gamma$ with $d(\sigma)=r(\tau)$, taking an invertible element $u\in R_{\tau}$, we have
\[R_{\sigma\tau}=R_{\sigma\tau}1_{d(\tau)}=R_{\sigma\tau}u^{-1}u\subseteq R_\sigma R_\tau.\]
Furthermore, \cite[Proposition 16]{CLP2}, says that the object crossed products are precisely the $\Gamma$-graded rings of the form $A\rtimes^{\alpha}_{\beta} \Gamma$, where $(A,\Gamma,\alpha,\beta)$ is an {object crossed system} as in Example~\ref{ex: aneis graduados}(3). In this context, we have:

\begin{proposition}
\label{prop: prod cruz ss}
    Let $(A,\Gamma,\alpha,\beta)$ be an object crossed system. Then $A\rtimes^{\alpha}_{\beta} \Gamma$ is a gr-semisimple ring if and only if $A_e$ is a semisimple ring for all $e\in\Gamma_0$.
\end{proposition}

\begin{proof}
    The result follows from Corollary \ref{coro: fort grad e gr-ss}, because $(A\rtimes^{\alpha}_{\beta} \Gamma)_e\cong A_e$ for all $e\in\Gamma_0$.
\end{proof}

\begin{corollary}
\label{coro: prod cruz primo}
    Let $(A,\Gamma,\alpha,\beta)$ be an object crossed system. The following assertions hold:
    \begin{enumerate}[\rm (1)]
        \item If $\Gamma$ is connected and $A_e$ is a prime ring for all $e\in\Gamma_0$, then $A\rtimes^{\alpha}_{\beta} \Gamma$ is a gr-prime ring. The converse holds if, for each $e\in\Gamma_0$ and $\sigma\in e\Gamma e$, we have $\alpha_\sigma=id_{A_e}$.
        \item If $\Gamma$ is connected and $A_e$ is a simple artinian ring for all $e\in\Gamma_0$, then $A\rtimes^{\alpha}_{\beta} \Gamma$ is a gr-simple $\Gamma_0$-artinian ring. The converse holds if, for each $e\in\Gamma_0$ and $\sigma\in e\Gamma e$, we have $\alpha_\sigma=id_{A_e}$.
    \end{enumerate}
\end{corollary}

\begin{proof}
    (1) Suppose that $\Gamma$ is connected and $A_e$ is a prime ring for all $e\in\Gamma_0$. Let $\sigma,\tau\in\Gamma$, $0\neq a \in A_{r(\sigma)}$ and $0\neq b \in A_{r(\tau)}$. Take $\gamma\in d(\sigma)\Gamma r(\tau)$. Then $0\neq \alpha_\sigma(\alpha_\gamma(b)) \in A_{r(\sigma)}$. Since $A_{r(\sigma)}$ is prime, there exists $x\in A_{r(\sigma)}$ such that $ax\alpha_\sigma(\alpha_\gamma(b))\neq0$. Then $\alpha_\sigma^{-1}(x)\in A_{d(\sigma)}=A_{r(\gamma)}$ and
    \begin{align*}
        (au_\sigma)(\alpha_\sigma^{-1}(x)u_\gamma)(bu_\tau)=(ax\beta_{\sigma,\gamma}u_{\sigma\gamma})(bu_\tau)&=ax\beta_{\sigma,\gamma}\alpha_{\sigma\gamma}(b)\beta_{\sigma\gamma,\tau}u_{\sigma\gamma\tau}\\
        &=ax\alpha_\sigma(\alpha_\gamma(b))\beta_{\sigma,\gamma}\beta_{\sigma\gamma,\tau}u_{\sigma\gamma\tau}
    \end{align*}
    is nonzero because $ax\alpha_\sigma(\alpha_\gamma(b))\neq0$ and $\beta_{\sigma,\gamma},\beta_{\sigma\gamma,\tau}$ are invertible in $A_{r(\sigma)}$. Thus $A\rtimes^{\alpha}_{\beta} \Gamma$ is a gr-prime ring.

    Now suppose that $A\rtimes^{\alpha}_{\beta} \Gamma$ is a gr-prime ring and, for each $e\in\Gamma_0$ and $\sigma\in e\Gamma e$, we have $\alpha_\sigma=id_{A_e}$. $\Gamma$ is connected because
    \[(1_{A_e}u_e)A\rtimes^{\alpha}_{\beta} \Gamma(1_{A_f}u_f)\neq0\implies e\Gamma f\neq\emptyset\]
    for all $e,f\in\Gamma_0$. Fix $e\in\Gamma_0$ and let $a,b\in A_e\setminus\{0\}$. Since $A\rtimes^{\alpha}_{\beta} \Gamma$ is gr-prime, there exist $\sigma\in e \Gamma e$ and $x\in A_e$ such that $(au_e)(xu_\sigma)(bu_e)\neq0$. Then
    \[0\neq (au_e)(xu_\sigma)(bu_e)=(axu_\sigma)(bu_e)=axbu_\sigma\]
    because $\alpha_\sigma=id_{A_e}$, and it follows that $axb\neq0$. Hence, $A_e$ is a prime ring.

    (2) If $\Gamma$ is connected and $A_e$ is a simple artinian ring for all $e\in\Gamma_0$, it follows from (1) and Proposition \ref{prop: prod cruz ss} that $A\rtimes^{\alpha}_{\beta} \Gamma$ is a gr-prime gr-semisimple ring. Conversely, if $A\rtimes^{\alpha}_{\beta} \Gamma$ is a gr-simple $\Gamma_0$-artinian ring and, for each $e\in\Gamma_0$ and $\sigma\in e\Gamma e$, we have $\alpha_\sigma=id_{A_e}$, it follows from (1) and Proposition \ref{prop: prod cruz ss} that $A_e$ is a prime semisimple ring for all $e\in\Gamma_0$.
\end{proof}

We observe that Corollary \ref{coro: prod cruz primo} applies for \emph{object twisted groupoid rings} \cite[Definition 22]{CLP2}.

\medskip

If $D$ is a $\Gamma$-graded  division ring  and $\overline{\Sigma}=(\Sigma_i)_{i\in I}\in\mathcal{P}(\Gamma)^I$ is a matricial sequence for $D$, then $R:=\M_I(D)(\overline{\Sigma})$ is a gr-semisimple  ring by Proposition~\ref{prop: M I(D)(E) gr-simples artiniano}(5). On the other hand, $R$ can also be described as in Theorem~\ref{teo: wa para ss -resumo}(2). The next result gives an explicit way of passing from the former to the latter description of $R$.

\begin{proposition}\label{prop:matrix_rings_as_product_of_gr-simple}
    Let $D$ be a $\Gamma$-graded division ring, $\overline{\Sigma}:=(\Sigma_i)_{i\in I}\in\mathcal{P}(\Gamma)^I$ be a matricial sequence for $D$ and let $R:=\M_I(D)(\overline{\Sigma})$. Consider 
    \[\Xi:=\left\{[r(\sigma)]\in\Gamma'_0(D)/\sim\colon\sigma\in\bigcup_{i\in I}\Sigma_i\right\}\] 
    where $\sim$ is the gr-primality relation on $\Gamma_0'(D)$. Let $\Sigma:=\bigcup\limits_{i\in I}\{i\}\times\Sigma_i$. For each $\xi\in\Xi$, let $\Sigma_\xi:=\{(i,\sigma)\in\Sigma:r(\sigma)\in\xi\}$ and fix $(i_\xi,\sigma_\xi)\in\Sigma_\xi$.
    Then, for each $\xi\in\Xi$, there exists 
    $(\gamma_{i,\sigma_i})_{(i,\sigma_i)\in \Sigma_\xi}\in \prod\limits_{(i,\sigma_i)\in \Sigma_\xi} \supp(1_{r(\sigma_\xi)}D1_{r(\sigma_i)})$ 
    such that 
    \[R\cong_{gr}\sideset{}{^{gr}}\prod_{\xi\in \Xi}\M_{\Sigma_\xi}(1_{r(\sigma_\xi)}D1_{r(\sigma_\xi)})(\overline{\gamma}_\xi),\]
    where 
    $\overline{\gamma}_\xi:=(\gamma_{i,\sigma_i}\sigma_i)_{(i,\sigma_i)\in\Sigma_\xi}\in (r(\sigma_\xi)\Gamma)^{\Sigma_\xi}$.
\end{proposition}

\begin{proof}
    First notice that $\Sigma=\bigcup\limits_{\xi\in\Xi}\Sigma_\xi$ and this is a union of disjoint sets. Let $(i,\sigma_i)\in\Sigma$ and take (the unique) $\xi\in\Xi$ such that $r(\sigma_i)\in\xi$. Thus, $[r(\sigma_\xi)]=[r(\sigma_i)]$, i.e., there exists $\gamma_{i,\sigma_i}\in(\supp D)\cap r(\sigma_\xi)\Gamma r(\sigma_i)$. Then, taking $0\neq u_{i,\sigma_i}\in D_{\gamma_{i,\sigma_i}}$ we have that
    \begin{align*}
        E_{ii}^{r(\sigma_i)}R&\longrightarrow E_{i_\xi i_\xi}^{r(\sigma_\xi)}R\\
        x&\longmapsto (u_{i,\sigma_i}E_{i_\xi i})x
    \end{align*}
    is an isomorphism of degree $\sigma_\xi^{-1}\gamma_{i,\sigma_i}\sigma_i=\deg(u_{i,\sigma_i}E_{i_\xi i})$. Therefore,
    
    \[R_R=\bigoplus_{\substack{i\in I\\ \sigma_i\in\Sigma_i}}E_{ii}^{r(\sigma_i)}R =\bigoplus_{\xi\in\Xi}\left(\bigoplus_{(i,\sigma_i)\in\Sigma_\xi}E_{ii}^{r(\sigma_i)}R\right)\cong_{gr}\bigoplus_{\xi\in\Xi}S_\xi(\overline{\gamma}_\xi),\]
    where $S_\xi:=\left(E_{i_\xi i_\xi}^{r(\sigma_\xi)}R\right)(\sigma_\xi^{-1})$ and $\overline{\gamma}_\xi:=(\gamma_{i,\sigma_i}\sigma_i)_{(i,\sigma_i)\in\Sigma_\xi}\in (r(\sigma_\xi)\Gamma)^{\Sigma_\xi}$.
    By Proposition \ref{prop: M I(D)(E) gr-simples artiniano}(2), the gr-simple $R$-module $\left(E_{i_\xi i_\xi}^{r(\sigma_\xi)}R\right)(\sigma_\xi^{-1})$ is in the same isoshift class of $\left(E_{i_{\xi '}i_{\xi'}}^{r(\sigma_{\xi'})}R\right)(\sigma_{\xi'}^{-1})$ if and only if $1_{r(\sigma_\xi)}D1_{r(\sigma_{\xi'})}\neq0$, and this is equivalent to \mbox{$\xi=[r(\sigma_\xi)]=[r(\sigma_{\xi'})]=\xi'$}. 
    From Lemma~\ref{lem: isoshift de simples}, we have that $\HOM_R(S_\xi(\overline{\gamma}_\xi),S_{\xi'}(\overline{\gamma}_{\xi'}))=\{0\}$
    for distinct $\xi,\xi'\in \Xi$. So we can apply  Corollary~\ref{coro: END(+,+) infinito ortogonal} which, together with Lemma \ref{lem: R=END(R)}, gives us the following gr-isomorphisms of $\Gamma$-graded rings
      
    \begin{equation}
    \label{eq: R = prod gr1}
    R\cong_{gr} \END(R_R)\cong_{gr} \sideset{}{^{gr}}\prod_{\xi\in\Xi}\END_R(S_\xi(\overline{\gamma}_\xi)).
    \end{equation}

   Now note that $\END_R(S_\xi)=\END_R((E_{i_\xi i_\xi}^{r(\sigma_\xi)}R)(\sigma_\xi^{-1}))\cong_{gr}1_{r(\sigma_\xi)}D1_{r(\sigma_\xi)}$ by Proposition \ref{prop: M I(D)(E) gr-simples artiniano}(3).
      Hence, \eqref{eq: R = prod gr1} and Corollary \ref{coro: END M(sigma)}(1) give
   \[R\cong_{gr}\sideset{}{^{gr}}\prod_{\xi\in \Xi}\M_{\Sigma_\xi}(1_{r(\sigma_\xi)}D1_{r(\sigma_\xi)})(\overline{\gamma}_\xi).\qedhere\]
  \end{proof}

As a consequence, we have the following result, which shows how to decompose graded division rings as a product of matrix rings.

\begin{corollary}
\label{coro: D como prod de aneis de matrizes}
Let $D$ be a $\Gamma$-graded division ring and consider the gr-primality relation $\sim$ defined on $\Gamma'_0:=\Gamma'_0(D)$. Then, for each $[e]\in\Gamma'_0/\!\sim$, there exists $\overline{\gamma}_{[e]}:=(\gamma_f)_{f\in [e]}\in\prod\limits_{f\in[e]}\supp (1_eD1_f)$ such that 
\[D\cong_{gr}\sideset{}{^{gr}}\prod_{[e]\in \Gamma'_0/\sim}\M_{[e]}(1_eD1_e)(\overline{\gamma}_{[e]}).\]
\end{corollary}

\begin{proof}
By Remark~\ref{rem: R=M1(R)}(1), $D\cong_{gr}\M_1(D)(\Gamma'_0)$. Now apply  Proposition~\ref{prop:matrix_rings_as_product_of_gr-simple} for $\Xi=\Gamma'_0/\!\sim$, $\Sigma=\{1\}\times\Gamma'_0\cong\Gamma'_0$, $\Sigma_{[e]}=\{1\}\times[e]\cong[e]$ and $\sigma_{[e]}=e$ for each $[e]\in\Gamma'_0/\!\sim$.
\end{proof}

\begin{corollary}
\label{coro: decompondo M_I(D)(Sigma) com D primo}
    Let $D$ be a gr-prime $\Gamma$-graded division ring,  $\overline{\Sigma}:=(\Sigma_i)_{i\in I}\in\mathcal{P}(\Gamma)^I$ be a matricial sequence for $D$ and consider $R:=\M_I(D)(\overline{\Sigma})$. Set $\Sigma:=\bigcup\limits_{i\in I}\{i\}\times\Sigma_i$ and fix $(i_0,\sigma)\in\Sigma$. Then there exists $(\gamma_{i,\sigma_i})_{(i,\sigma_i)\in \Sigma}\in \prod\limits_{(i,\sigma_i)\in \Sigma} \supp(1_{r(\sigma)}D1_{r(\sigma_i)})$  such that 
    \[R\cong_{gr}\M_{\Sigma}(1_{r(\sigma)}D1_{r(\sigma)})(\overline{\gamma}),\] where 
    $\overline{\gamma}:=(\gamma_{i,\sigma_i}\sigma_i)_{(i,\sigma_i)\in\Sigma}\in (r(\sigma)\Gamma)^\Sigma$.
\end{corollary}

\begin{proof}
    It follows from Proposition \ref{prop:matrix_rings_as_product_of_gr-simple}, noting that, by Proposition \ref{prop: quando D e'simples}(4), the gr-primality relation on $\Gamma'_0(D)$ has a unique equivalence class.
\end{proof}

Corollaries \ref{coro: D como prod de aneis de matrizes} and \ref{coro: decompondo M_I(D)(Sigma) com D primo} provide another proof of 
Theorem~\ref{theo: anel com div primo = anel de matr}. Indeed, let $D$ be a gr-prime $\Gamma$-graded division ring. By Remark~\ref{rem: R=M1(R)}(1), $D\cong_{gr}\M_1(D){(\Gamma'_0)}$. Fix $e\in\Gamma'_0:=\Gamma'_0(D)$. By Corollary \ref{coro: D como prod de aneis de matrizes} or Corollary~\ref{coro: decompondo M_I(D)(Sigma) com D primo}, there exists $(\gamma_f)_{f\in \Gamma'_0}\in \prod\limits_{f\in \Gamma'_0} \supp(1_eD1_f)$ such that \[D\cong_{gr}\M_1(D){(\Gamma'_0)}\cong_{gr}\M_{\Gamma'_0}(1_eD1_e)\left(\overline{\gamma}\right),\]
where $\overline{\gamma}:=(\gamma_f)_{f\in\Gamma'_0}$.

\medskip

We end this subsection with a consequence of the results in Section~\ref{subsec:modules_graded_groupoid}. We need some definitions first. Suppose that $\Gamma$ is a connected groupoid and let $R$ be a $\Gamma$-graded ring.
Let $e_0\in \Gamma_0$, $G:=e_0\Gamma e_0$ and $\{\sigma_e\}_{e\in\Gamma_0}$ as in \eqref{eq:iso_connected_groupoid}.  A $(G\times \Gamma_0)$-graded $R$-module $M$ is gr-simple if $M\neq\{0\}$ and the only graded submodules of $M$ are $\{0\}$ and $M$. The $(G\times \Gamma_0)$-graded module $M$ is gr-semisimple if it is a sum of $(G\times \Gamma_0)$-graded gr-simple submodules. A $G$-graded $R$-module $M$ is gr-simple if $M\neq\{0\}$ and the only graded submodules of $M$ are $\{0\}$ and $M$. The $G$-graded module $M$ is gr-semisimple if it is a sum of $G$-graded gr-simple submodules.

\begin{corollary}\label{coro: semisimplicty of R, eGamma, G and GxGamma)}
Suppose that $\Gamma$ is a connected groupoid and let $R$ be a $\Gamma$-graded ring. Let $e_0\in \Gamma_0$, $G:=e_0\Gamma e_0$ and $\{\sigma_e\}_{e\in\Gamma_0}$ as in \eqref{eq:iso_connected_groupoid}. The following statements are equivalent.
    \begin{enumerate}[\rm(1)]
    \item $R$ is a gr-semisimple $\Gamma$-graded ring.
    \item  There exists $e\in\Gamma_0$ such that any graded right $R$-module in the full subcategory $e\Gamma-\grR R$ of $\Gamma-\grR R$ is gr-semisimple.
    \item Any $G$-graded right  $R$-module is gr-semisimple (as a $G$-graded module).
    \item Any $(G\times \Gamma_0)$-graded right  $R$-module is gr-semisimple (as a $(G\times \Gamma_0)$-graded module).
  \end{enumerate}
\end{corollary}

\begin{proof}
The equivalence of (1) and (2) follows from Proposition~\ref{prop: CLP, Prop 59} and Lemma~\ref{lem:eG_fG_modules}, upon observing that a $\Gamma$-graded right $R$-module $M$ is gr-semisimple if and only if $M(e)$  is gr-semisimple for every $e\in\Gamma_0$.

    The equivalence of (2), (3) and (4) follows from Proposition~\ref{prop:G_graded_eGamma_graded}.
\end{proof}

  \subsection{On the uniqueness of the representation as matrix rings}
The main aim of this subsection is to prove a kind of uniqueness of the decomposition 
in Theorem~\ref{teo: wa para ss -resumo}(2) for a gr-semisimple ring.
The following general result, together with Theorem~\ref{teo: iso entre aneis de matrizes}, provides  such result.

\begin{theorem}
\label{teo: quando ss sao isomorfos}    
    Let $\{R_j:j\in J\}$ and $\{T_j:j\in J'\}$ be summable families of gr-simple rings. Then $\sideset{}{^{gr}}\prod\limits_{j\in J}R_j\cong_{gr}\sideset{}{^{gr}}\prod\limits_{j\in J'}T_j$ if and only if there exists a bijection $\pi:J\to J'$ such that $R_j\cong_{gr}T_{\pi(j)}$ for each $j\in J$.
\end{theorem}

\begin{proof}
    Assume that $\Phi:\sideset{}{^{gr}}\prod\limits_{j\in J}R_j\longrightarrow\sideset{}{^{gr}}\prod\limits_{j\in J'}T_j$ is a gr-isomorphism of rings. We denote $R:=\sideset{}{^{gr}}\prod\limits_{j\in J}R_j$ and $T:=\sideset{}{^{gr}}\prod\limits_{j\in J'}T_j$. Given $j\in J$ and $k\in J'$, let $p_j:R\to R_j$, $p'_k:T\to T_k$ be the canonical projections and $\iota_j:R_j\to R$, $\iota'_k:T_k\to T$ be the canonical inclusions. 
    Fix $j_0\in J$. We have that $\iota_{j_0}(R_{j_0})$ is a graded ideal of $R$ and therefore $\Phi(\iota_{j_0}(R_{j_0}))$ is a graded ideal of $T=\bigoplus\limits_{j\in J'}T_j$. It is easy to see that 
    \[\Phi(\iota_{j_0}(R_{j_0}))=\bigoplus_{j\in J'}p'_j(\Phi(\iota_{j_0}(R_{j_0}))).\]
    Since $\Phi\circ\iota_{j_0}$ is an injective gr-homomorphism and $R_{j_0}$ is gr-simple, it follows that $\Phi(\iota_{j_0}(R_{j_0}))$ is a gr-simple ring and a graded ideal of $T$. Thus, there exists a unique $\pi(j_0)\in J'$ such that $p'_{\pi(j_0)}(\Phi(\iota_{j_0}(R_{j_0})))\neq\{0\}$. Then $p'_{\pi(j_0)}(\Phi(\iota_{j_0}(R_{j_0})))$ is a nonzero graded ideal of the gr-simple ring $T_{\pi(j_0)}$ and it follows that $p'_{\pi(j_0)}(\Phi(\iota_{j_0}(R_{j_0})))=T_{\pi(j_0)}$. Hence
    \begin{equation}
    \label{eq: prod de aneis}
        \Phi(\iota_{j_0}(R_{j_0}))=\iota'_{\pi(j_0)}(T_{\pi(j_0)}).
    \end{equation}
    In particular, $R_{j_0}\cong_{gr}T_{\pi(j_0)}$. It also follows from (\ref{eq: prod de aneis}) that the function $\pi:J\to J'$ is injective. Finally, note that
    \[\sideset{}{^{gr}}\prod\limits_{j\in J'}T_j=\Phi\left(\sideset{}{^{gr}}\prod\limits_{j\in J}R_j\right)=\Phi\left(\bigoplus\limits_{j\in J}\iota_j(R_j)\right)=\bigoplus\limits_{j\in J}\Phi(\iota_j(R_j))=\bigoplus_{j\in J}\iota'_{\pi(j)}(T_{\pi(j)})\]
    and it follows that $\pi$ is surjective as well.
\end{proof}

\begin{proposition}
\label{prop: estrutura dos aneis de matrizes}
    Let $H$ be a $\Gamma$-graded ring and $\overline{\Sigma}:=(\Sigma_i)_{i\in I}\in \mathcal{P}(\Gamma)^I$ be a fully matricial sequence for $H$. Let $R:=\M_I(H)(\overline{\Sigma})$. Fix $i_0\in I$ and consider $M:=\bigoplus\limits_{\sigma\in\Sigma_{i_0}}(E_{i_0i_0}^{r(\sigma)}R)(\sigma^{-1})$. The following assertions hold:
    \begin{enumerate}[\rm (1)]
        \item $E_{ii}^{r(\sigma_i)}R\cong_{gr}M(\sigma_i)$ for each $i\in I$ and $\sigma_i\in\Sigma_i$.
        \item $H\cong_{gr}\END_R(M)$ as $\Gamma$-graded rings.
        \item For each $\delta,\delta'\in \Gamma$ with $r(\delta),r(\delta')\in\Gamma'_0(H)$, if $M(\delta)\cong_{gr}M(\delta')$, then $\delta'\delta^{-1}\in\supp(H)$. And the converse holds if $H$ is a $\Gamma$-graded division ring.
    \end{enumerate}
\end{proposition}

\begin{proof}
    (1) Let $i\in I$, $\sigma_i\in\Sigma_i$ and take $\sigma\in\Sigma_{i_0}$ such that $r(\sigma)=r(\sigma_i)$. So it is easy to see that
    \begin{align*}
        E_{ii}^{r(\sigma_i)}R&\longrightarrow\left(E_{i_0i_0}^{r(\sigma)}R\right)(\sigma^{-1}\sigma_i)=M(\sigma_i)\\
        x&\longmapsto E_{i_0i}^{r(\sigma_i)}x
    \end{align*}
    is an isomorphism of right $R$-modules and it is graded since $E_{i_0i}^{r(\sigma_i)}\in R_{\sigma^{-1}\sigma_i}$.

    (2) Consider
    \begin{align*}
        \Phi: H&\longrightarrow \END_R(M)\\
        x&\longmapsto \Phi(x):M\to M\\
        &\phantom{aaaaaaaaaj}m\mapsto xm.
    \end{align*}
    Let us see that $\Phi$ is well-defined and is graded. Let $\gamma\in\Gamma$ and $x\in H_\gamma$. Let $\alpha\in\Gamma$ and $0\neq m=(m_{kl})_{kl}\in M_\alpha=\left(E_{i_0i_0}^{r(\sigma)}R\right)(\sigma^{-1})_\alpha\subseteq R_{\sigma^{-1}\alpha}$ where $\sigma\in\Sigma_{i_0}$ is such that $r(\sigma)=r(\alpha)$. For each $j\in I$, we have $m_{i_0j}\in H_{\Sigma_{i_0}\sigma^{-1}\alpha\Sigma_j^{-1}}=H_{\alpha\Sigma_j^{-1}}$. Therefore, $xm_{i_0j}\in H_{\gamma\alpha\Sigma_j^{-1}}=H_{\Sigma_{i_0}\tau^{-1}\gamma\alpha\Sigma_j^{-1}}$ where $\tau\in\Sigma_{i_0}$ is such that $r(\tau)=r(\gamma)$. Since $m_{kl}=0$ for all $k\neq i_0$, it follows that $xm\in \left(E_{i_0i_0}^{r(\tau)}R\right)_{\tau^{-1}\gamma\alpha}=M_{\gamma\alpha}$. Thus, $\Phi(x)\in\END_R(M)_\gamma$. Clearly $\Phi$ respects sums, products and units. Hence, $\Phi$ is a gr-homomorphism of rings. It is injective because if $x\in(\ker\Phi)_\gamma$ with $r(\gamma),d(\gamma)\in\Gamma'_0(H)$, then $E_{i_0i_0}^{d(\gamma)}\in M$ and it follows that $xE_{i_0i_0}^{d(\gamma)}=\Phi(x)(E_{i_0i_0}^{d(\gamma)})=0$ from where $x=0$. Finally, let us see that $\Phi$ is surjective. Suppose that $\gamma\in\Gamma$ and $0\neq g\in \END_R(M)_\gamma$. Then
    \[r(\gamma),d(\gamma)\in\Gamma'_0(\END_RM)=\Gamma'_0(M)=\{r(\sigma):\sigma\in\Sigma_{i_0}\}=\Gamma'_0(H).\]
    Let $\sigma,\tau\in\Sigma_{i_0}$ such that $r(\sigma)=d(\gamma)$ and $r(\tau)=r(\gamma)$. Since $E_{i_0i_0}^{r(\sigma)}\in M_\sigma$ and $\im g\subseteq M(r(\gamma))=\left(E_{i_0i_0}^{r(\tau)}R\right)(\tau^{-1})$, we have $g(E_{i_0i_0}^{r(\sigma)})\in M_{\gamma\sigma}=\left(E_{i_0i_0}^{r(\tau)}R\right)_{\tau^{-1}\gamma\sigma}$ and
    \[g(E_{i_0i_0}^{r(\sigma)})=E_{i_0i_0}^{r(\tau)}g(E_{i_0i_0}^{r(\sigma)})E_{i_0i_0}^{r(\sigma)}=xE_{i_0i_0}\]
    for some $x\in H_{\Sigma_{i_0}\tau^{-1}\gamma\sigma\Sigma_{i_0}^{-1}}=H_\gamma$. 
     If $m\in M(d(\gamma))=\left(E_{i_0i_0}^{r(\sigma)}R\right)(\sigma^{-1})$, we have
    \[g(m)=g(E_{i_0i_0}^{r(\sigma)}m)=g(E_{i_0i_0}^{r(\sigma)})m=xm=\Phi(x)(m).\]
    Since $g$ and $\Phi(x)$ have degree $\gamma$, it follows that they vanish on $M(e)$ for all $e\in\Gamma_0\setminus\{d(\gamma)\}$. Thus $g=\Phi(x)$.
        
    (3) First notice that, using Proposition \ref{prop: HOM(,) e Hom_gr}(1) and the previous item, we get
    \begin{align*}
        \Homgr (M(\delta),M(\delta'))&\cong \HOM_R(M,M(\delta'))_{\delta^{-1}}\\
        &=\Homgr(M,M(\delta')(\delta^{-1}))\\
        &=\HOM_R(M,M)_{\delta'\delta^{-1}}\\
        &\cong H_{\delta'\delta^{-1}}
    \end{align*}
    for all $\delta,\delta'\in\Gamma$. Furthermore, if $r(\delta),r(\delta')\in\Gamma'_0(H)=\Gamma'_0(\END_R(M))=\Gamma'_0(M)$, then $M(\delta),M(\delta')\neq0$. Thus, if $M(\delta)\cong_{gr}M(\delta')$, then $0\neq \Homgr (M(\delta),M(\delta'))\cong H_{\delta'\delta^{-1}}$ and therefore $\delta'\delta^{-1}\in\supp (H)$. If $H$ is a $\Gamma$-graded division ring, then $M$ is $\Gamma_0$-simple by Proposition \ref{prop: M I(D)(E) gr-simples artiniano}(1). In this case, if $\delta'\delta^{-1}\in\supp (H)$, then $\Homgr (M(\delta),M(\delta'))\cong H_{\delta'\delta^{-1}}\neq0$. But since $M(\delta)$ and $M(\delta')$ are gr-simple, it follows that every nonzero element of $\Homgr (M(\delta),M(\delta'))$ is a gr-isomorphism.
\end{proof}

\begin{theorem}
\label{teo: iso entre aneis de matrizes}
    Let $e,e'\in\Gamma_0$, $D,D'$ be $\Gamma$-graded division rings with $\supp(D)\subseteq e\Gamma e$, $\supp(D')\subseteq e'\Gamma e'$ and $\overline{\sigma}:=(\sigma_i)_{i\in I}\in (e\Gamma)^I$, $\overline{\delta}:=(\delta_i)_{i\in I'}\in (e'\Gamma)^{I'}$ be $d$-finite sequences. Then $\M_I(D)(\overline{\sigma})\cong_{gr}\M_{I'}(D')(\overline{\delta})$ if and only if there exist a bijection $\pi: I\to I'$ and $\tau\in e'\Gamma e$ such that 
    $D'\cong_{gr}\M_1(D)(\tau^{-1})$ and $\delta_{\pi(i)}\in \tau(\supp D)\sigma_i$ for each $i \in I$.
\end{theorem}

\begin{proof}
    Let $R:=\M_I(D)(\overline{\sigma})$ and suppose that $R\cong_{gr}\M_{I'}(D')(\overline{\delta})$. Fix $i_0\in I$ and let $S:=\left(E_{i_0i_0}^{r(\sigma_{i_0})}R\right)(\sigma_{i_0}^{-1})$ which is a gr-simple $R$-module by Proposition \ref{prop: M I(D)(E) gr-simples artiniano}(1). By Proposition \ref{prop: estrutura dos aneis de matrizes}, we have
    \[R_R=\bigoplus_{i\in I}E_{ii}^{r(\sigma_i)}R\cong_{gr}\bigoplus_{i\in I}S(\sigma_i)\quad\textrm{and}\quad D\cong_{gr}\END_R(S).\] 
    Since $R\cong_{gr}\M_{I'}(D')(\overline{\delta})$, we similarly obtain a gr-simple $R$-module $T$ such that
    \[R_R\cong_{gr}\bigoplus_{i\in I'}T(\delta_i)\quad\textrm{and}\quad D'\cong_{gr}\END_R(T).\]
    Fix $f\in\Gamma_0$ and consider the finite sets
    \[I_f:=\{i\in I:d(\sigma_i)=f\}\quad\textrm{and}\quad I'_f:=\{i\in I':d(\delta_i)=f\}.\]
    We have then
    \[R(f)\cong_{gr}\bigoplus_{i\in I_f}S(\sigma_i)\cong_{gr}\bigoplus_{i\in I'_f}T(\delta_i).\]
    Since $S$ and $T$ are gr-simple, it follows from Proposition \ref{prop: Krull-Schmidt} that $|I_f|=|I'_f|$ and there exists a bijection $\pi_f:I_f\to I'_f$ such that $S(\sigma_i)\cong_{gr}T(\delta_{\pi_f(i)})$ for each $i\in I_f$. Since $I=\bigcup\limits_{f\in\Gamma_0}I_f$ and $I'=\bigcup\limits_{f\in\Gamma_0}I'_f$ are disjoint unions it follows that we have a bijection $\pi:I\to I'$ given by $\pi(i)=\pi_f(i)$ when $i\in I_f$. 
    Now notice that, for each $i\in I$, we have $T=T(\delta_{\pi(i)})(\delta_{\pi(i)}^{-1})\cong_{gr}S(\sigma_i\delta_{\pi(i)}^{-1})$ and it follows that
    \[S(\sigma_i\delta_{\pi(i)}^{-1})\cong_{gr}S(\sigma_{i_0}\delta_{\pi(i_0)}^{-1}).\]
    By Proposition \ref{prop: estrutura dos aneis de matrizes}(3), we have $\sigma_{i_0}\delta_{\pi(i_0)}^{-1}\delta_{\pi(i)}\sigma_i^{-1}\in\supp D$ for all $i\in I$. That is, taking 
    \[\tau:=\delta_{\pi(i_0)}\sigma_{i_0}^{-1}\in e'\Gamma e\]
    we have $\delta_{\pi(i)}\in \tau(\supp D)\sigma_i$ for all $i \in I$. Finally, we have
    \[D'\cong_{gr}\END_R(T)\cong_{gr}\END_R(S(\sigma_{i_0}\delta_{\pi(i_0)}^{-1}))=\END_R(S(\tau^{-1}))\cong_{gr}\M_1(D)(\tau^{-1}),\]
    where the last gr-isomorphism follows from Corollary \ref{coro: END M(sigma)}(1).

    Conversely, assume that there exist a bijection $\pi: I\to I'$, $\tau\in e'\Gamma e$, a gr-isomorphism of rings $\varphi:\M_1(D)(\tau^{-1})\to D'$ and we have $\delta_{\pi(i)}\in \tau(\supp D)\sigma_i$ for all $i \in I$. For each $i\in I$, let $\gamma_i\in\supp D$ such that $\delta_{\pi(i)}=\tau\gamma_i\sigma_i$ and fix $u_i\in D_{\gamma_i}\setminus\{0\}$.
    Define the following isomorphism of additive groups
    \begin{align*}
        \Phi:\M_I(D)(\overline{\sigma})&\longrightarrow\M_{I'}(D')(\overline{\delta})\\
        dE_{ij}&\longmapsto \varphi(u_idu_j^{-1})E_{\pi(i)\pi(j)}.
    \end{align*}
    $\Phi$ is graded because, for each $\gamma\in\Gamma$, we have
    \begin{align*}
        dE_{ij}\in \M_I(D)(\overline{\sigma})_\gamma&\implies d\in D_{\sigma_i\gamma\sigma_j^{-1}}\\
        &\implies u_idu_j^{-1}\in D_{\gamma_i\sigma_i\gamma\sigma_j^{-1}\gamma_j^{-1}}=D_{\tau^{-1}\delta_{\pi(i)}\gamma\delta_{\pi(j)}^{-1}\tau}\\
        &\implies (u_idu_j^{-1})\in \M_1(D)(\tau^{-1})_{\delta_{\pi(i)}\gamma\delta_{\pi(j)}^{-1}}\\
        &\implies \varphi(u_idu_j^{-1})\in D'_{\delta_{\pi(i)}\gamma\delta_{\pi(j)}^{-1}}\\
        &\implies \varphi(u_idu_j^{-1})E_{\pi(i)\pi(j)}\in \M_{I'}(D')(\overline{\delta})_\gamma.
    \end{align*}
    Let us see now that $\Phi$ respects products. For this, let $d,\Tilde{d}\in D$ and $i,j,k,l\in I$. If $j\neq k$, then $\pi(j)\neq\pi(k)$ and we have
    \[\Phi((dE_{ij})(\Tilde{d}E_{kl}))=\Phi(0)=0=\varphi(u_idu_j^{-1})E_{\pi(i)\pi(j)}\varphi(u_k\Tilde{d}u_l^{-1})E_{\pi(k)\pi(l)}=\Phi(dE_{ij})\Phi(\Tilde{d}E_{kl}).\]
    If $j=k$, then
    \begin{align*}
        \Phi((dE_{ij})(\Tilde{d}E_{kl}))&=\Phi(d\Tilde{d}E_{il})\\
        &=\varphi(u_id\Tilde{d}u_l^{-1})E_{\pi(i)\pi(l)}\\
        &=\varphi(u_idu_j^{-1}u_k\Tilde{d}u_l^{-1})E_{\pi(i)\pi(l)}\\
        &=\varphi(u_idu_j^{-1})E_{\pi(i)\pi(j)}\varphi(u_k\Tilde{d}u_l^{-1})E_{\pi(k)\pi(l)}\\
        &=\Phi(dE_{ij})\Phi(\Tilde{d}E_{kl}).
    \end{align*}
    Finally, note that if $f\in \Gamma_0$, then $\delta_{\pi(i)}=\tau\gamma_i\sigma_i$ for all $i\in I$ implies $\pi(I_f)=I'_f$. Thus,
    \[\Phi(\mathds{I}_f)=\Phi\left(\sum_{i\in I_f}E_{ii}^e\right)=\sum_{i\in I_f}\varphi(1_e)E_{\pi(i)\pi(i)}=\sum_{i\in I_f}1_{e'}E_{\pi(i)\pi(i)}=\sum_{i\in I'_f}E_{ii}^{e'}=\mathds{I}_f.\]
    Hence, $\Phi$ is a gr-isomorphism of rings.
\end{proof}

\begin{remark}
    When $\Gamma$ is a group, Theorem \ref{teo: iso entre aneis de matrizes} can be stated as follows: \emph{Let $G$ be a group, $D,D'$ be $G$-graded division rings and $\overline{\sigma}:=(\sigma_1,...,\sigma_n)\in G^n$, $\overline{\delta}:=(\delta_1,...,\delta_m)\in G^m$. Then $\M_n(D)(\overline{\sigma})\cong_{gr}\M_m(D')(\overline{\delta})$ if and only if $n=m$ and there exist a permutation $\pi$ of $\{1,...,n\}$ and $\tau\in G$ such that $D'\cong_{gr}\M_1(D)(\tau^{-1})$ and $\delta_{\pi(i)}\in \tau(\supp D)\sigma_i$ for each $i=1,...,n$.} This is essentially what was achieved in \cite[p. 32-33]{Elduque}. Therefore, Theorem \ref{teo: iso entre aneis de matrizes} generalizes the group graded case and we believe that our proof here, when adapted to $\Gamma$ being a group, provides a more elementary proof (or at least based on more elementary facts) of the result of \cite{Elduque}.  \qed
\end{remark}

Next we present a generalization of Theorem~\ref{teo: iso entre aneis de matrizes}. As its proof has a more complicated notation, it follows the same idea of Theorem~\ref{teo: iso entre aneis de matrizes} and  the main results of this  section have already been proved, we present only a sketch of the proof.

\begin{theorem}
\label{teo: iso entre aneis de matrizes mais geral}
    Let $D,D'$ be $\Gamma$-graded division rings  and $\overline{\Sigma}:=(\Sigma_i)_{i\in I}\in \mathcal{P}(\Gamma)^I$, $\overline{\Delta}:=(\Delta_i)_{i\in I'}\in \mathcal{P}(\Gamma)^{I'}$ be fully matricial sequences for $D$ e $D'$, respectively. Consider $\Sigma:=\bigcup\limits_{i\in I}\{i\}\times\Sigma_i$ and $\Delta:=\bigcup\limits_{i\in I'}\{i\}\times\Delta_i$. Then $\M_I(D)(\overline{\Sigma})\cong_{gr}\M_{I'}(D')(\overline{\Delta})$ if and only if there exist a bijection $\pi: \Sigma\to \Delta$ and $\overline{\tau}:=(\tau_{e'})_{e'\in\Gamma'_0(D')}\in\prod\limits_{e'\in\Gamma'_0(D')} e'\Gamma (\Gamma'_0(D))$ such that $D'\cong_{gr}\M_{\Gamma'_0(D')}(D)(\overline{\tau}^{-1})$ for $\overline{\tau}^{-1}:=(\tau_{e'}^{-1})_{e'\in\Gamma'_0(D')}$, and $\delta\in \tau_{r(\delta)}(\supp D)\sigma$ for all $(i,\sigma) \in \Sigma$ and $(i',\delta)=\pi(i,\sigma)$. 
\end{theorem}

\begin{proof}
    Set $R:=\M_I(D)(\overline{\Sigma})$ and suppose that $R\cong_{gr}\M_{I'}(D')(\overline{\Delta})$. Fix $i_0\in I$ and let $S:=\bigoplus\limits_{\sigma\in\Sigma_{i_0}}(E_{i_0i_0}^{r(\sigma)}R)(\sigma^{-1})$ which is a  $\Gamma_0$-simple $R$-module by Proposition~\ref{prop: M I(D)(E) gr-simples artiniano}(1). By Proposition~\ref{prop: estrutura dos aneis de matrizes},  $R_R\cong_{gr}\bigoplus\limits_{i\in I}S(\Sigma_i)$
    and $D\cong_{gr}\END_R(S)$. Analogously, since $R\cong_{gr}\M_{I'}(D')(\overline{\delta})$, we obtain and $\Gamma_0$-simple $R$-module $T$ such that $R_R\cong_{gr}\bigoplus\limits_{i\in I'}T(\Delta_i)$ and $D'\cong_{gr}\END_R(T)$.
    Set $f\in\Gamma_0$ and consider the finite sets
    \[\Sigma_f:=\{(i,\sigma_i)\in \Sigma:d(\sigma_i)=f\}\quad\textrm{e}\quad \Delta_f:=\{(i,\delta_i)\in \Delta:d(\delta_i)=f\}.\]
    Then
    \[R(f)\cong_{gr}\bigoplus_{(i,\sigma_i)\in \Sigma_f}S(\sigma_i)\cong_{gr}\bigoplus_{(i,\delta_i)\in \Delta_f}T(\delta_i).\]
    Since $S$ and $T$ are $\Gamma_0$-simple, it follows from Proposition~\ref{prop: Krull-Schmidt} that $|\Sigma_f|=|\Delta_f|$ and there exists a bijection $\pi_f:\Sigma_f\to \Delta_f$ such that $S(\sigma_i)\cong_{gr}T(\delta_{i'})$ for all $(i,\sigma_i)\in \Sigma_f$, where $(i',\delta_{i'})=\pi_f(i,\sigma_i)$. Since the unions $\Sigma=\bigcup\limits_{f\in\Gamma_0}\Sigma_f$ and $\Delta=\bigcup\limits_{f\in\Gamma_0}\Delta_f$ are disjoint, we obtain a  bijection $\pi:\Sigma\to \Delta$ given by $\pi(i,\sigma_i)=\pi_{d(\sigma_i)}(i,\sigma_i)$. 
    Fix $e'\in\Gamma'_0(D')$, $i'_0\in I'$ and $\delta_{i'_0}\in\Delta_{i'_0}$ such that $r(\delta_{i'_0})=e'$. Take $(i_0,\sigma_{i_0})\in\Sigma$ such that $(i'_0,\delta_{i'_0})=\pi(i_0,\sigma_{i_0})$.
    Now note that, for each $i\in I$, $\sigma_i\in\Sigma_i$ and $(i',\delta_{i'}):=\pi(i,\sigma_i)$.  
    \[T(r(\delta_{i'}))\cong_{gr}S(\sigma_i\delta_{i'}^{-1}).\]
   It follows that $S(\sigma_i\delta_{i'}^{-1})\cong_{gr}T(e')\cong_{gr}S(\sigma_{i_0}\delta_{i'_0}^{-1})$ whenever $r(\delta_{i'})=e'$.
    In this event, by Proposition~\ref{prop: estrutura dos aneis de matrizes}(3),  $\sigma_{i_0}\delta_{i'_0}^{-1}\delta_{i'}\sigma_i^{-1}\in\supp D$. That is, 
    \[\tau_{e'}:=\delta_{i'_0}\sigma_{i_0}^{-1}\in e'\Gamma (\Gamma'_0(D))\]
    is such that $\delta_{i'}\in \tau_{e'}(\supp D)\sigma_i$. Therefore, we have
    \begin{eqnarray*}
        D'\;\cong_{gr}\;\END_R{T}&=&\END_R\left(\bigoplus_{e'\in\Gamma'_0(D')}T(e')\right)\\
        &\cong_{gr}&\END_R\left(\bigoplus_{e'\in\Gamma'_0(D')}S(\tau_{e'}^{-1})\right)\\
        &=&\END_R(S(\overline{\tau}^{-1}))\\
        &\cong_{gr}&\M_{\Gamma'_0(D')}(D)(\overline{\tau}^{-1}),
    \end{eqnarray*}
    where the last gr-isomorphism follows from Corollary~\ref{coro:graded_endomorphism_ring of modules}(2).

    Conversely, suppose that there exist  $\overline{\tau}:=(\tau_{e'})_{e'\in\Gamma'_0(D')}\in\prod\limits_{e'\in\Gamma'_0(D')} e'\Gamma (\Gamma'_0(D))$ and a bijection $\pi: \Sigma\to \Delta$, together with a gr-isomorphism $\varphi:\M_{\Gamma'_0(D')}(D)(\overline{\tau}^{-1})\to D'$ and $\delta\in \tau_{r(\delta)}(\supp D)\sigma$ whenever $(i,\sigma) \in \Sigma$ and $(i',\delta)=\pi(i,\sigma)$. For each $(i,\sigma)\in\Sigma$, let $\gamma_{i,\sigma}\in\supp D$ be such that $\delta=\tau_{r(\delta)}\gamma_{i,\sigma}\sigma$ where $(i',\delta)=\pi(i,\sigma)$. Fix also $u_{i,\sigma}\in D_{\gamma_{i,\sigma}}\setminus\{0\}$.
    It can be shown in the same way as in the proof of Theorem~\ref{teo: iso entre aneis de matrizes} that the map $\Phi:\M_I(D)(\overline{\Sigma})\longrightarrow\M_{I'}(D')(\overline{\Delta})$ defined by  
    \[\Phi(dE_{ij})=\varphi(u_{i,\sigma_i}du_{j,\sigma_j}^{-1}E_{r(\delta_{i'})r(\delta_{j'})})E_{i'j'}\]
    for all  $i,j\in I$ and $d\in 1_{r(\sigma_i)}D1_{r(\sigma_j)}$ and where $(i',\delta_{i'})=\pi(i,\sigma_i)$ and $(j',\delta_{j'})=\pi(j,\sigma_j)$.
\end{proof}


\section{The graded Jacobson-Chevalley density theorem}
\label{sec: teo da den}


In this section, in contrast to the other sections of this paper, $\Gamma$-graded rings need not be object unital and $\Gamma$-graded modules need not be unital. 


We begin by extending the definitions of gr-simple rings and modules for not necessarily object unital rings. Let $R$ be a $\Gamma$-graded ring and $M$ be a $\Gamma$-graded $R$-module. We say that $M$ is \emph{gr-simple} if $MR\neq 0$ and the only graded submodules of $M$ are $\{0\}$ and $M$. And $M$ is said to be \emph{faithful} if its right annihilator is zero, that is,
$$\ann_r(M):=\{a\in R\colon Ma=0 \}=\{0\}.$$
We say that $R$ is a \emph{right gr-primitive ring} if there exists a $\Gamma$-graded (right) $R$-module which is gr-simple and faithful. The graded ring $R$ is a \emph{gr-simple ring} if $R^2\neq 0$ and the only graded ideals of $R$ are $\{0\}$ and $R$. The concepts of $\Gamma_0$-artinian graded rings and modules are defined in the same way as in the object unital context. We observe that Theorem~\ref{teo: schur} is still valid for gr-simple $R$-modules not necessarily unital over $\Gamma$-graded rings not necessarily object unital. 

\begin{example}
\label{exem: de gr primitivo}
    Let $D$ be a $\Gamma$-graded division ring such that $\supp(D)\subseteq e\Gamma e$ for some $e\in\Gamma_0$ and $V$ be a $\Gamma$-graded  unital left $D$-module. Set
    $R:=\END_D(V)$. Then $V$ is a $\Gamma$-graded right $R$-module via $x\cdot t:=(x)t$ for all $x\in V$ e $t\in R$.  Note that $V=V(e)$. Moreover, if $x,y\in\h(V)$ with $x\neq0$, then, extending $(x)$ to a  pseudo-basis of $V$, there exists $g\in R$ tal que $(x)g=y$. Since $xR\neq0$, it follows that $V_R$ is gr-simple and faithful. Hence $R$ is a right gr-primitive ring. If $_DV$ is  $\Gamma_0$-finite dimensional, then it follows from Theorem~\ref{teo: WA para simp -resumo} that $\END_{D^{op}}(V^{op})$ is a gr-simple $\Gamma_0$-artinian  ring. And from Proposition~\ref{prop: oposto}(2), we get that $R^{op}$ is a gr-simple $\Gamma_0$-artinian ring. Therefore $R$ is a gr-simple $\Gamma_0$-artinian ring. \qed
\end{example}

The aim of this section is to show, via a graded version of the Density Theorem, that all right gr-primitive rings which are right $\Gamma_0$-artinian  are described in Theorem~\ref{exem: de gr primitivo}.



Let $D$ be a $\Gamma$-graded division ring and $V$ be a $\Gamma$-graded unital left $D$-module. A graded subring $T$ of $\END_D(V)$ is said to be a \emph{gr-dense subring}
of $\END_D(V)$ if for all $n>0$, $\gamma_1,...,\gamma_n\in\Gamma$, pseudo-linearly independent sequence $(u_1,...,u_n)\in V_{\gamma_1}\times\cdots\times V_{\gamma_n}$  and $(v_1,...,v_n)\in V_{\delta_1}\times\cdots\times V_{\delta_n}$ with $r(\delta_i)=r(\gamma_i)$ for all $i=1,...,n$ there exists $t\in T$ such that $(u_i)t=v_i$ for all $i=1,...,n$.  By Corollary~\ref{coro: ext de homo de D mod}, $\END_D(V)$ is a gr-dense subring of $\END_D(V)$. In fact, if $(u_1,...,u_n)\in V_{\gamma_1}\times\cdots\times V_{\gamma_n}$ is a pseudo-linearly independent sequence and $(v_1,...,v_n)\in V_{\delta_1}\times\cdots\times V_{\delta_n}$ with $r(\delta_i)=r(\gamma_i)$ for all $i=1,...,n$, then, for each $i=1,...,n$, there exists $t_i\in\END_D(V)_{\gamma_i^{-1}\delta_i}$ such that $(u_i)t_i=v_i$ and $(u_j)t_i=0$ for every $j\neq i$. It suffices to take $t:=t_1+\cdots +t_n$.


\begin{lemma}
\label{lema: pre teo da den}
Let $R$ be a $\Gamma$-graded ring, $S$ be a gr-simple $R$-module and $D:=\END_R(S)$. The following statements hold true. 
\begin{enumerate}[\rm (1)]
 \item $S=xR$ for all $0\neq x\in \h(S)$.
 \item If $_D V$ is a graded $D$-submodule of $_D S$ with $\pdim_D(V)<\infty$ and $x\in \h(S)\setminus V$, then there exists $a\in \h(R)$ such that $xa\neq0$ and $Va=0$.
    \end{enumerate} 
\end{lemma}

\begin{proof}
    (1) Let $0\neq x\in \h(S)$. Consider the graded $R$-submodule of $S$ $$X:=\{s\in S:sR=0\}.$$   
    Since $SR\neq0$, we have that $X\neq S$ and, thus, $X=0$ because $S$ is a gr-simple module. In particular, $xR\neq0$ and, therefore,  $xR=S$ because $S$ is a gr-simple $R$-module.

    
    (2) We prove the statement by induction on $n=\pdim_D V$.
    If $n=0$, then $V=\{0\}$ and $0\neq x\in\h(S)$. By (1), $xR=S\neq\{0\}$ and, hence, there exists $a\in\h(R)$ such that $xa\neq0$.
    
    
    
    Now suppose that $n\geq 1$ and that the result holds for $n-1$. Let $_D V$ be a graded $D$-submodule of $_D S$ with $\pdim_D V=n$ and $x\in \h(S)\setminus V$.
    Let $(v_1,...,v_n)$ be a pseudo-basis of $_DV$. If $n>1$, define 
    \[W:=\bigoplus_{i=1}^{n-1}Dv_i\]
    and if $n=1$, define $W=0$. Consider  
    \[A=\ann_r(W):=\{a\in R:Wa=0\}.\]
    Note that $W$ is a graded left $D$-submodule of $V$ and $A$ is a graded right ideal of $R$. We claim that 
  \[W= \{s\in S:sA=0\}.\]
  By definition, $W\subseteq \{s\in S:sA=0\}$. Moreover, if $s\in \h(S)\setminus W$, then the induction hypothesis implies the existence of $a\in\h(R)$ such that $sa\neq0$ and $Wa=0$, that is, $a\in A$ and $sA\neq0$. Thus the claim is proved.

Since $v_n\notin W$, the claim implies that $v_nA\neq0$. Now $v_nA$ is a nonzero graded $R$-submodule of the gr-simple $R$-module $S$ and, hence, $v_nA=S$. Finally, we prove the last step of the induction by way of contradiction. Suppose that there does not exist $a\in \h(R)$  such that $xa\neq0$ and $Va=0$. In other words, 
$$v_na=0\Longrightarrow Va=0\Longrightarrow xa=0$$
for all $a\in A$. Then the following homomorphism of $R$-modules is well-defined
\begin{align*}
        g:S&\longrightarrow S\\
        v_na&\longmapsto xa  \quad (a\in A).
    \end{align*}
    Note that $g\in D_{\deg(x)\deg(v_n)^{-1}}$. We also have $(x-g(v_n))a=0$ for all $a\in A$. Hence, $x-g(v_n)\in W$. But this implies that
    \[x=(x-g(v_n))+g(v_n)\in W\oplus Dv_n=V,\]
    a contradiction.    
\end{proof}
Now we have the following groupoid graded version of the Jacobson-Chevalley Density Theorem.


\begin{theorem}
\label{teo: teo da den}
    Let $R$ be a $\Gamma$-graded right gr-primitive ring. Let $S$ be a faithful gr-simple $R$-module and consider $D:=\END(S_R)$. Then $R$ is gr-isomorphic to gr-dense subring of $\END(_D S)$.
\end{theorem}

\begin{proof}
    Define
    \begin{align*}
        \varphi:R&\longrightarrow \END(_DS)\\
        r&\longmapsto\varphi_r:S\to S\\
        &\phantom{\longmapsto\varphi_r: i}x\mapsto xr.
    \end{align*}
    Clearly, $\varphi$ is a gr-homomorphism of rings. It is also injective because $\varphi_r=0$ implies that $r\in\ann_r(S)=0$. Hence, it is enough to show that $\im \varphi$ is a gr-dense subring of $\END(_DS)$. Notice also that $S=S(e)$ for some $e\in\Gamma_0$. Let  $n>0$, 
 a pseudo-linearly independent sequence over $D$ $(u_1,...,u_n)\in \h(S)^n$ and $(v_1,...,v_n)\in \h(S)^n$. For  $i\in\{1,...,n\}$, let $V_i:=\bigoplus\limits_{\substack{j=1\\j\neq i}}^nDu_j$. Since $u_i\notin V_i$, it follows from Lemma~\ref{lema: pre teo da den}(2) that there exists $a_i\in \h(R)$ such that $u_ia_i\neq0$ and $V_ia_i=0$. By Lemma~\ref{lema: pre teo da den}(1), we get $u_ia_iR=S$. Choose $c_i\in \h(R)$ such that $u_ia_ic_i=v_i$. Set
    \[r:=a_1c_1+\cdots+a_nc_n.\]
    Note that $i\neq j$ implies $u_ia_jc_j\in V_ja_jc_j=0$. Hence, 
    \[(u_i)\varphi_r=u_ir=\sum_{j=1}^{n}u_ia_jc_j=u_ia_ic_i=v_i,\]
    for all $i=1,...,n$.
\end{proof}

In order to obtain a graded version of the Wedderburn-Artin Theorem, we need the next result.

\begin{proposition}
\label{prop: subanel gr-denso artin}
    Let $D$ be a $\Gamma$-graded division ring, $V$ be a $\Gamma$-graded unital left $D$-module and $T$ be a gr-dense subring of $\END_D(V)$. 
   
    \begin{enumerate}[\rm (1)]
        \item If $T$ is a right $\Gamma_0$-artinian ring, then $V$ has $\Gamma_0$-finite dimension.
        \item If $V$ has $\Gamma_0$-finite dimension, then $T=\END_D(V)$.
    \end{enumerate}
\end{proposition}

\begin{proof}
    (1) Suppose, by way of contradiction, that $T$ is a right $\Gamma_0$-artinian ring but there exists $e\in\Gamma_0$ such that $(e)V$ is of infinite pseudo-dimension over $D$. Then there exists a pseudo-linearly independent sequence $(u_n)_{n\in\mathbb{N}}$ of homogeneous elements in $(e)V$. For each $n\in\mathbb{N}$, consider the following graded right ideal of $T$: 
    \[A_n:=\{t\in T:(u_i)t=0 \textrm{~for all~} 1\leq i\leq n\}.\]
    
    Since $(u_1,...,u_{n+1})$ is pseudo-linearly independent and $T$ is a  gr-dense subring of $\END_D(V)$, there exists $t_n\in T$ such that $(u_i)t_n=0$ for all $1\leq i \leq n$ and $(u_{n+1})t_n=u_{n+1}$. Clearly, we can suppose that $t_n\in T_{d(\deg(u_{n+1}))}=T_e$. We then have, $t_n\in A_n(e)\setminus A_{n+1}(e)$. In this way, we obtain the following strictly decreasing sequence of graded $T$-submodules of $T(e)$:
    \[A_1(e)\supsetneq A_2(e)\supsetneq A_3(e)\supsetneq\cdots,\]
    contradicting that $T(e)$ is gr-artinian.
    

(2) Now suppose that $_DV$ is of finite $\Gamma_0$-dimension and let us prove that $T=\END_D(V)$. Let  $\gamma\in\Gamma$ and $g\in\END_D(V)_\gamma$. Then $g$ is totally determined by the image of the elements of a pseudo-basis  $(v_1,...,v_n)$ of $(r(\gamma))V$ because $(V_\alpha)g\subseteq V_{\alpha\gamma}=\{0\}$ if $d(\alpha)\neq r(\gamma)$. Since $T$ is gr-dense in $\END_D(V)$, there exists $t\in T$ such that $(v_i)t=(v_i)g$ for all $1\leq i \leq n$. Therefore, $g=t\in T$.
\end{proof}

\begin{theorem}
\label{teo: wa via teo da den}
    Let $R$ be a $\Gamma$-graded right $\Gamma_0$-artinian  ring. The following statements are equivalent.
 \begin{enumerate}[\rm (1)]
        \item $R$ is a gr-simple ring.
        \item $R$ is a right gr-primitive ring.
        \item There exist a $\Gamma$-graded division ring $D$, $e\in \Gamma_0$ and 
        a $\Gamma$-graded unital left $D$-module $V$  of finite $\Gamma_0$-dimension over $D$ such that $\supp(D)\subseteq e\Gamma e$ and $R\cong_{gr}\END_D(V)$.
        \item There exist $e\in\Gamma_0$, a $\Gamma$-graded division ring $D$ with 
        $\supp(D)\subseteq e\Gamma e$ and a $d$-finite sequence $\overline{\sigma}:=(\sigma_i)_{i\in I}\in (e\Gamma)^I$  such that $R\cong_{gr}\M_I(D)(\overline{\sigma})$.
    \end{enumerate} 
\end{theorem}

\begin{proof}
$(1)\implies(2):$ The facts that $R$ is a gr-simple ring, $\ann_l(R)$ is a graded ideal of $R$ and $R^2\neq0$, implies that $\ann_l(R)=0$. Since $R$ is a nonzero $\Gamma_0$-artinian ring, we can take $e\in\Gamma_0$ such that $R(e)\neq0$ and a nonzero minimal graded submodule $V$ of $R(e)$. Then $VR\neq0$ because $V\nsubseteq \ann_l(R)=0$. Thus $V$ is a gr-simple $R$-module. Since $\ann_r(V)$ is a proper graded ideal of $R$ because $VR\neq0$, it follows from the gr-simplicity of $R$ that $\ann_r(V)=0$ and, therefore, $V_R$ is faithful. Hence, $R$ is a right gr-primitive ring.

    
$(2)\implies(3):$ By Theorem~\ref{teo: teo da den}, $R$ is gr-isomorphic to a gr-dense subring of $\END(_DS)$ where $S$ is a right faithful  gr-simple $R$-module and $D:=\END(S_R)$. Since $R$ is a right $\Gamma_0$-artinian ring, Proposition~\ref{prop: subanel gr-denso artin} implies that $_DS$ has finite $\Gamma_0$-dimension over $D$ and $R\cong_{gr}\END(_DS)$.


   $(3)\implies(4):$ If (3) holds, it follows from Proposition \ref{prop: oposto}(2) that $R^{op}\cong_{gr}\END_{D^{op}}V^{op}$. Since $D^{op}$ is a $\Gamma$-graded division ring with $\supp(D^{op})\subseteq e\Gamma e$ and $V^{op}$ is a $\Gamma_0$-finite dimensional $\Gamma$-graded right $D^{op}$-module, there exists a $d$-finite sequence $\overline{\sigma}:=(\sigma_i)_{i\in I}\in(e\Gamma)^I$ such that $V^{op}\cong_{gr}\bigoplus\limits_{i\in I}D^{op}(\sigma_i)$ and it follows from Corollary \ref{coro: END M(sigma)}(2) that 
   \[R^{op}\cong_{gr}\END_{D^{op}}(D^{op}(\overline{\sigma}))\cong_{gr}\M_I(D^{op})(\overline{\sigma}).\]
  So (4)  follows from Proposition \ref{prop: oposto}(1).
    
    

    $(4)\implies(1):$  This implication follows from Theorem~\ref{teo: WA para simp -resumo}.
\end{proof}

One can define $\Gamma_0$-simple modules in the same way as in the unital context. As a consequence, \emph{right $\Gamma_0$-primitive ring} can be defined as a $\Gamma$-graded ring for which there exists a faithful graded $\Gamma_0$-simple $R$-module. Clearly, every right gr-primitive ring is a right $\Gamma_0$-primitive ring. Note also that if $R$ is a right $\Gamma_0$-primitive ring and $S$ is a faithful right $\Gamma_0$-simple module such that all  $S(e)$, $e\in\Gamma_0'(S)$, are in the same isoshift class, then all $S(e)$ are faithful and, thus, $R$ is a right gr-primitive ring.




Following the idea of Example~\ref{exem: de gr primitivo}, it can be shown that if $D$ is a $\Gamma$-graded division ring, $V$ is a $\Gamma$-graded unital left $D$-module and  $R:=\END_D(V)$, then $V$ is a faithful $\Gamma_0$-simple right $R$-module. Hence $R$ is a right $\Gamma_0$-primitive ring. And if $D_V$ is of finite $\Gamma_0$-dimension, then $R$ is a $\Gamma_0$-artinian gr-simple ring.

With similar proofs as the ones of the results in this section, one can prove the following results.

\begin{theorem}  Let $R$ be a $\Gamma$-graded ring.  
    Suppose that $R$ i a right $\Gamma_0$-primitive ring. Let $S$ be a faithful $\Gamma_0$-simple $R$-module and set $D:=\END(S_R)$. Then $R$ is gr-isomorphic to a gr-dense subring of $\END(_D S)$.\qed
\end{theorem}


\begin{theorem}
    Let $R$ be a $\Gamma$-graded right $\Gamma_0$-artinian ring. The following statements are equivalent.
    \begin{enumerate}[\rm (1)]
        \item $R$  is a right $\Gamma_0$-primitive ring.
        \item There exist a $\Gamma$-graded division ring $D$ and a $\Gamma$-graded unital left $D$-module $V$ of finite $\Gamma_0$-dimension over $D$ such that $R\cong_{gr}\END_D(V)$.
        \item There exist a $\Gamma$-graded division ring and a $d$-finite sequence
        $\overline{\sigma}:=(\sigma_i)_{i\in I}\in \Gamma^I$  such that  $R\cong_{gr}\M_I(D)(\overline{\sigma})$. \qed
    \end{enumerate} 
\end{theorem}

\medskip

Suppose that $R$ is a $\Gamma$-graded right gr-primitive ring but not right $\Gamma_0$-artinian ring. Let $S$ be a faithful gr-simple right $R$-module and $D:=\END(S_R)$. By Theorem~\ref{teo: teo da den} and Proposition~\ref{prop: subanel gr-denso artin}(2), $_DS$ is not $\Gamma_0$-finite dimensional, that is, there exist $e\in\Gamma_0$  and an infinite pseudo-linear $D$-independent sequence 
 $(v_i)_{i\in\mathbb{N}}\in\h((e)S)^\mathbb{N}$. For each  $n\geq1$, set $V_n=\bigoplus_{i=1}^n Dv_i$,
$R_n=\{r\in R\colon V_n\cdot r\subseteq V_n\}$ and $I_n=\{r\in R\colon V_n\cdot r=0\}$. 
Then $R_n$ is a $\Gamma$-graded ring and $I_n$ is a graded ideal of $R_n$.
Note that, for each $i\geq1$, we have $Dv_i\cong_{gr}(\sigma_i^{-1})D$ via $v_i\mapsto1_{r(\sigma_i)}$, where $\sigma_i:=\deg(v_i)$. Then, by Proposition \ref{prop: oposto}(4), we have
\begin{align*}
    \END_D(V_n)\cong_{gr}\END_D\left(\bigoplus_{i=1}^n(\sigma_i^{-1})D\right)&\cong_{gr}\M_n(\END (_DD))(\sigma_1,\dotsc,\sigma_n)\\
    &\cong_{gr}\M_n(D)(\sigma_1,\dotsc,\sigma_n).
\end{align*}
By Theorem~\ref{teo: teo da den},  there exists a gr-isomorphism
\[
R_n/I_n\longrightarrow \END_D(V_n)\cong \M_n(D)(\sigma_1,\dotsc,\sigma_n)
\]
for each $n\geq 1$.


\section{Pseudo free module rings}
\label{sec:Pseudo_divsision_rings}


We say that the $\Gamma$-graded ring $R$ is a \emph{pseudo-free module ring}, or a \emph{pfm ring} for short, if every $\Gamma$-graded right $R$-module is pseudo-free.  For example, by Theorem~\ref{theo:modules_over_division_rings}, $\Gamma$-graded division rings are pfm rings. The aim of this section is to study the class of  pfm rings.

We begin with the next result that will be important when characterizing  pfm rings.

\begin{lemma}
\label{lem: pfm --> ss}
Let $R=\bigoplus\limits_{\gamma\in\Gamma}R_\gamma$ be a $\Gamma$-graded ring. The following statements hold true.
\begin{enumerate}[\rm(1)]
    \item If $R$ is a pfm ring, then $R$ is a gr-semisimple ring.
    \item Conversely, suppose that $R$ is a gr-semisimple ring. Then $R$ is a pfm ring if and only if every gr-simple right $R$-module is pseudo-free. 
\end{enumerate}
\end{lemma}

\begin{proof}
   (1) If every $\Gamma$-graded $R$-module is pseudo-free, then every $\Gamma$-graded $R$-module is gr-projective by Corollary~\ref{coro:pseudo-free is gr-projective}. By Proposition~\ref{prop: CLP, Prop 59}, it follows that $R$ is a gr-semisimple ring. 
   
   (2) Suppose that every gr-simple right $R$-module is pseudo-free. Since $R$ is gr-semisimple, it follows from Proposition \ref{prop: CLP, Prop 59} that every $\Gamma$-graded right $R$-module is gr-semisimple. Thus, every $\Gamma$-graded right $R$-module is a direct sum of pseudo-free modules.
\end{proof}

The foregoing result implies that  the following inclusion relationships hold
$$\left\{\begin{array}{c}
  \textrm{graded division rings}
\end{array} \right\} \subset \left\{\begin{array}{c}
      \textrm{pfm rings}
\end{array} \right\}\subset \left\{\begin{array}{c}
      \textrm{gr-semisimple rings}
\end{array} \right\}. $$
The next example shows that these inclusions are strict.

\begin{example}
\label{ex: exem onde todo mod e livre}
Let $D$ be a division ring.
\begin{enumerate}[(1)]
    \item Let $\Gamma=\{e\}$ be the trivial group(oid). If $n\geq 2$, then $R=\M_n(D)$ is a gr-semisimple ring which is not a pfm ring because $E_{11}R$ is not a free $R$-module.

\item
    Let $\Gamma$ be the groupoid $\{1,2\}\times\{1,2\}$. Consider $D$ as a $\Gamma$-graded division ring with support concentrated in $(1,1)$. Set \[R=\M_3(D)((1,1),(1,1),(1,2)).\]
    Then 
 \[R_{(1,1)}=\begin{bmatrix}
    D & D & 0 \\
    D & D & 0 \\
    0 & 0 & 0 \\
    \end{bmatrix},\quad  R_{(1,2)}=\begin{bmatrix}
    0 & 0 & D \\
    0 & 0 & D \\
    0 & 0 & 0 \\
    \end{bmatrix}, \]
\[    R_{(2,1)}=\begin{bmatrix}
    0 & 0 & 0 \\
    0 & 0 & 0 \\
    D & D & 0 \\
    \end{bmatrix},\quad  R_{(2,2)}=\begin{bmatrix}
    0 & 0 & 0 \\
    0 & 0 & 0 \\
    0 & 0 & D \\
    \end{bmatrix}.\]    
The identity elements of $R$ are
\[ 
\mathbb{I}_{(1,1)}=\begin{bmatrix}
    1 & 0 & 0 \\
    0 & 1 & 0 \\
    0 & 0 & 0 \\
    \end{bmatrix}=E_{11}+E_{22},\quad \mathbb{I}_{(2,2)}=\begin{bmatrix}
    0 & 0 & 0 \\
    0 & 0 & 0 \\
    0 & 0 & 1 \\
    \end{bmatrix}=E_{33}.
\]
On the one hand, $R$ is not a $\Gamma$-graded division ring because the homogenous element $E_{11}\in R_{(1,1)}$ is not invertible. Indeed, there does not exist $A\in R$ such that $E_{11}A=\mathbb{I}_{(1,1)}$. 
On the other hand, $R$ is a gr-semisimple ring by Theorem~\ref{teo: WA para simp -resumo}(2). Now, the $\Gamma$-graded $R$-module $S:=E_{33}=R((2,2))$ is gr-simple, pseudo-free and it is such that
\begin{eqnarray}
    R& = & E_{11}R\oplus E_{22}R\oplus E_{33}R \nonumber \\
    & = & E_{13}R\oplus E_{23}R\oplus E_{33}R \nonumber\\ 
    & \cong_{gr} & S((2,1))\oplus S((2,1))\oplus S.\label{eq:exemplo_pfm}
\end{eqnarray}
By \eqref{eq:exemplo_pfm} and Lemma~\ref{lem: R ss -> R(e) é soma finita}(1), every gr-simple $R$-module is gr-isomorphic to a shift of $S$.  Therefore, every gr-simple graded $R$-module is pseudo-free. By Lemma \ref{lem: pfm --> ss}(2), $R$ is a pfm ring.
\qed
\end{enumerate}
\end{example}

We point out that, in Example~\ref{ex: exem onde todo mod e livre}(2), $\mathcal{B}_1=\{\mathbb{I}_{(1,1)}\}$ and $\mathcal{B}_2=\{E_{13},\,E_{23}\}$ are two pseudo-bases of the $\Gamma$-graded $R$-module $S=R((1,1))$. Thus, there is no uniqueness of the cardinality of pseudo-bases of $\Gamma$-graded $R$-modules over pfm rings, and as a consequence, over gr-semisimple rings. Such uniqueness will characterize gr-division rings, but we will also be able to define an invariant similar to dimension for pfm rings, see Section~\ref{subsec:IPBN}.


\subsection{Characterization of pfm rings}
We already know that pfm rings are gr-semisimple. Moreover, gr-semisimple rings are the product of  summable families of gr-simple $\Gamma_0$-artinian graded rings by Theorem~\ref{teo: wa para ss -resumo}. The next result characterizes gr-simple pfm rings. We then will use it to provide a characterization of pfm rings not necessarily gr-prime. 

\begin{theorem}
\label{teo: pseudo aneis com div primos}
Let $R=\bigoplus\limits_{\gamma\in\Gamma}R_\gamma$ be a $\Gamma$-graded ring. The following statements are equivalent.
\begin{enumerate}[\rm (1)]
    \item $R$ is a gr-simple pfm ring.
    \item $R$ is a gr-prime pfm ring.
    \item There exist $e_0,e\in\Gamma_0$, a $\Gamma$-graded division ring $D$ with $\supp(D)\subseteq e_0\Gamma e_0$, a non-empty set $I$ and a $d$-finite sequence $\overline{\sigma}:=(\sigma_i)_{i\in I}\in (e_0\Gamma)^I$ such that
    \[R\cong_{gr}\M_I(D)(\overline{\sigma})\]
    and $I_e:=\{i\in I: d(\sigma_i)=e\}$ has exactly one element.
    \item $R$ is a gr-simple ring and there exists $e\in\Gamma_0$ such that $R(e)$ is a gr-simple $R$-module.
    \item There exists $e\in\Gamma_0$ such that $R(e)$ is a gr-simple $R$-module and $R_R$ is gr-isomorphic to a direct sum of shifts of $R(e)$.
    \item $R$ is a gr-simple $\Gamma_0$-artinian ring and there exists $e\in\Gamma_0$ such that $R_e$ is a division ring.
    \item $R$ is a right gr-primitive right $\Gamma_0$-artinian ring and there exists $e\in\Gamma_0$ such that $R_e$ is a division ring.
\end{enumerate}
\end{theorem}

\begin{proof}
    $(1)\implies (2)$: It is clear.
    
    $(2)\implies (3)$: Suppose that $R$ is a gr-prime pfm ring. 
    By Lemma~\ref{lem: pfm --> ss}(1), $R$ is a gr-semisimple ring. The fact that $R$ is gr-prime and Theorem~\ref{teo: wa para ss -resumo}(2) imply that there exist
$e_0\in\Gamma_0$, a $\Gamma$-graded division ring  $D$ with $\supp(D)\subseteq e_0\Gamma e_0$ and a $d$-finite sequence $\overline{\sigma}:=(\sigma_i)_{i\in I}\in (e_0\Gamma)^I$ such that $R\cong_{gr}\M_I(D)(\overline{\sigma})$. To ease the notation and without loss of generality, we can suppose that  
        \[R=\M_I(D)(\overline{\sigma}).\]

    Fix $i_0\in I$. Since $E_{i_0i_0}^{e_0}R$ has a pseudo-basis, there exist $\gamma_0\in\Gamma$ and a matrix $(a_{ij})_{ij}\in (E_{i_0i_0}^{e_0}R)_{\gamma_0}$ such that $(a_{ij})_{ij}\cdot X\neq0$ for all $0\neq X\in \mathbb{I}_{d(\gamma_0)}R$. We will show that $I_{d(\gamma_0)}$ contains exactly one element. Suppose, on the contrary, that $I_{d(\gamma_0)}$ has at least two different elements, say $j_1,j_2$. Hence, $d(\sigma_{j_1})=d(\sigma_{j_2})=d(\gamma_{0})$. If $a_{i_0j_t}=0$, for some $t=1,2$, define 
    \[X=E_{j_tj_t}^{e_0}\in R_{d(\sigma_{j_t})}=R_{d(\gamma_0)}\]
    and if, otherwise, $a_{i_0j_1}\neq0$ and $a_{i_0j_2}\neq0$, define
    \[X=a_{i_0j_1}^{-1}E_{j_1i_0}^{e_0}-a_{i_0j_2}^{-1}E_{j_2i_0}^{e_0}\in R_{\gamma_0^{-1}}.\]
    Notice that $(a_{ij})_{ij}$ is a matrix whose all nonzero entries are in row $i_0$. Thus, in both cases, we have $(a_{ij})_{ij}\cdot X=0$ where $X\in \mathbb{I}_{d(\gamma_0)}R\setminus\{0\}$, a contradiction.

    $(3)\implies (4)$: By Theorem~\ref{teo: WA para simp -resumo}, $R$ is a gr-simple ring. Moreover,  the fact that $I_e$ contains exactly one element implies that $R(e)=\bigoplus\limits_{i\in I_e}E_{ii}^{e_0}R$ is a gr-simple $R$-module by Proposition~\ref{prop: M I(D)(E) gr-simples artiniano}(1).            

    $(4)\implies (5)$: It follows from Lemma \ref{lem: artin -> semisimp}(2).

    $(5)\implies (1)$: Suppose that
    \[R_R\cong_{gr}\bigoplus\limits_{i\in I}R(e)(\sigma_i)\]
    for some $(\sigma_i)_{i\in I}\in\Gamma^I$. Since $R(e)$ is a gr-simple $R$-module,  $R$ is a gr-semisimple ring. By Lemma~\ref{lem: R ss -> R(e) é soma finita}(1), every  $\Gamma$-graded gr-simple right $R$-module  is gr-isomorphic to a shift of $R(e)$, and therefore,  pseudo-free. Now Lemma~\ref{lem: pfm --> ss}(2) implies (1). 
        
    $(4)\implies (6)$: $R$ is a right $\Gamma_0$-artinian ring by Theorem~\ref{teo: simp + art = semisimp}. By  Lemma~\ref{lem: R=END(R)} and Corollary~\ref{coro: schur para simples}, the fact that $R(e)$ is a gr-simple module implies that $1_eR1_e\cong_{gr}\END_R(R(e))$ is a gr-division ring. Therefore, $R_e$ is a division ring.

    $(6)\implies (3)$: By Theorem~\ref{teo: WA para simp -resumo}, 
    there exist $e_0\in\Gamma_0$, a   $\Gamma$-graded division ring $D$ and a $d$-finite sequence $\overline{\sigma}:=(\sigma_i)_{i\in I}\in (e_0\Gamma)^I$
such that $\supp(D)\subseteq e_0\Gamma e_0$ and $R\cong_{gr}\M_I(D)(\overline{\sigma})$. Then  
    \[\M_I(D)(\overline{\sigma})_e\cong R_e\]
is a division ring.
    Let $I_e:=\{i\in I: d(\sigma_i)=e\}$. Then $\mathbb{I}_e=\sum\limits_{i\in I_e}E_{ii}^{e_0}$ and $I_e\neq\emptyset$. If $i,j\in I_e$ then, since $\M_I(D)(\overline{\sigma})_e$ is domain, we get that $E_{ii}^{e_0}E_{jj}^{e_0}\neq0$.  This only happens if $i=j$. Therefore, $I_e$ consists of exactly one element.
    
    $(6)\iff (7)$ Follows from Theorem~\ref{teo: wa via teo da den}.   
\end{proof}

Before providing the general characterization of pfm rings, we would like to point out that the product of pfm rings is not a pfm ring in general. Indeed, let $\Gamma=\{e\}$ and $D$ be a ($\Gamma$-graded) division ring. Clearly, $D$ is a pfm ring. Set $R:=D\times D$. The ($\Gamma$-graded) ring $R$ is not a pfm ring because $D\times\{0\}$ is not a free $R$-module.
However, we have the following result.
\begin{proposition}
\label{prop:product_pfm_rings}
Let $\{R_j\colon j\in J\}$ be a family of $\Gamma$-graded rings. Set 
\mbox{$R:=\sideset{}{^{gr}}\prod\limits_{j\in J}R_j$}. If $R$ is a pfm ring, then $R_j$ is a pfm ring for each $j\in J$. The converse holds if $\Gamma'_0(R_j)\cap\Gamma'_0(R_k)=\emptyset$ for different $j,k\in J$.
\end{proposition}

\begin{proof}
Suppose that $R$ is a pfm ring. Let $j_0\in J$. Each $\Gamma$-graded $R_{j_0}$-module $M_{R_{j_0}}$ can be regarded as a right $R$-module via the action
\begin{equation}
 \label{eq:action_of_product}
 x\cdot (r_j)_{j\in J}=xr_{j_0}\ \textrm{ for all } x\in M,\ (r_j)_{j\in J}\in R.
\end{equation}
Thus $M_R$ must have a pseudo-basis. Because of the action \eqref{eq:action_of_product}, such pseudo-basis must be a pseudo-basis of $M_{R_{j_0}}$. Therefore $R_{j_0}$ is a pfm ring.

Suppose now that $R_j$ is a pfm ring for each $j\in J$ with $\Gamma'_0(R_j)\cap\Gamma'_0(R_k)=\emptyset$ for different $j,k\in J$. 
Observe that the fact that $\Gamma'_0(R_j)\cap\Gamma'_0(R_k)=\emptyset$ for different $j,k\in J$ implies that each (unital) right $R$-module $M$ is of the form $M=\bigoplus\limits_{j\in J}M_j$ where, for each $j\in J$, $M_j:=\bigoplus\limits_{e\in\Gamma'_0(R_j)}M1_e$ is a right $R_j$-module. The action is then given by
$(m_j)_{j\in J}(a_j)_{j\in J}=(m_ja_j)_{j\in J}$ for all $(m_j)_{j\in J}\in M,\ (a_j)_{j\in J}\in R$. Thus, if each $R_j$ is a pfm ring, then every $\Gamma$-graded right $R$-module is pseudo-free.
\end{proof}

Now we are ready to give the characterization of pfm rings not necessarily gr-prime.
We point out that items (2) and  (6)  of Theorem~\ref{teo: pseudo aneis com div} imply that $R$ is a $\Gamma$-graded pfm ring if and only if every $\Gamma$-graded left $R$-module is pseudo-free. In other words, being ``right pfm'' is synonymous with ``left pfm'' for $\Gamma$-graded rings.

\begin{theorem}
\label{teo: pseudo aneis com div}
Let $R=\bigoplus\limits_{\gamma\in\Gamma}R_\gamma$ be a $\Gamma$-graded ring. The following statementes are equivalent.
\begin{enumerate}[\rm (1)]
    \item $R$ is a pfm ring.
    \item There exist a family $\{K_j:j\in J\}$ of non-empty sets, sequences $(e_j)_{j\in J},(f_j)_{j\in J}\in(\Gamma_0)^J$ and, for each $j\in J$, there exist a $\Gamma$-graded division ring $D_j$ with $\supp(D_j)\subseteq e_j\Gamma e_j$ and a $d$-finite sequence $\overline{\sigma}_j:=(\sigma_{jk})_{k\in K_j}\in (e_j\Gamma)^{K_j}$ such that the family $\{\M_{K_j}(D_j)(\overline{\sigma}_j):j\in J\}$ is summable,
    \[R\cong_{gr}\sideset{}{^{gr}}\prod_{j\in J}\M_{K_j}(D_j)(\overline{\sigma}_j)\]
    and, for each $j\in J$, the set $K_{j,f_j}:=\{k\in K_j:d(\sigma_{jk})=f_j\}$ has exactly one element and
    $\{j'\in J:K_{j',f_j}\neq\emptyset\}=\{j\}$.
    \item There exists a summable family $\{R_j:j\in J\}$ of gr-prime pfm rings and $(f_j)_{j\in J}\in(\Gamma_0)^J$ such that
    \[R\cong_{gr}\sideset{}{^{gr}}\prod_{j\in J}R_j\]
    and, for each $j\in J$, $R(f_j)$ is gr-simple and $R_j(f_j)\neq0$.
    \item There exists a summable family $\{R_j:j\in J\}$ of gr-simple rings and $(f_j)_{j\in J}\in(\Gamma_0)^J$ such that
    \[R\cong_{gr}\sideset{}{^{gr}}\prod_{j\in J}R_j\]
    and, for each $j\in J$, $R(f_j)$ is gr-simple and $R_j(f_j)\neq0$.
    \item There exists $\Delta_0\subseteq\Gamma'_0(R)$ such that $R(\Delta_0)$ is $\Gamma_0$-simple and $R_R$ is gr-isomorphic to a direct sum of shifts of elements from $\{R(e):e\in\Delta_0\}$.
    \item There exists a summable family $\{R_j:j\in J\}$ of gr-simple $\Gamma_0$-artinian rings and $(f_j)_{j\in J}\in(\Gamma_0)^J$ 
    such that 
    \[R\cong_{gr}\sideset{}{^{gr}}\prod\limits_{j\in J}R_j\]
    and, for each $j\in J$, $R_{f_j}$ is a division ring and $(R_{j'})_{f_j}=0$ whenever $j'\in J\setminus\{j\}$.
\end{enumerate}
\end{theorem}

\begin{proof}


    
   $(1)\implies (2)$: Suppose that $R$ is a pfm ring. By Lemma~\ref{lem: pfm --> ss}(1),  $R$ is a gr-semisimple ring. By Theorem~\ref{teo: wa para ss -resumo}, we can suppose that there exist $(e_j)_{j\in J}\in(\Gamma_0)^J$ and, for each $j\in J$, a $\Gamma$-graded division ring $D_j$ with $\supp(D_j)\subseteq e_j\Gamma e_j$ and a $d$-finite sequence $\overline{\sigma}_j:=(\sigma_{jk})_{k\in K_j}\in (e_j\Gamma)^{K_j}$  such that the family $\{\M_{K_j}(D_j)(\overline{\sigma}_j):j\in J\}$ is summable and
   \[R=\sideset{}{^{gr}}\prod_{j\in J}\M_{K_j}(D_j)(\overline{\sigma}_j).\]
   Fix $j\in J$ and $k\in K_j$. Set $R_j:=\M_{K_j}(D_j)(\overline{\sigma}_j)$ and let $\iota_j:R_j\to R$ be the canonical inclusion. Since the $\Gamma$-graded right  $R$-module $\iota_j(E_{kk}^{e_j}R_j)$ has a pseudo-basis, there exist $\gamma_j\in\Gamma$ and $a_j\in (E_{kk}^{e_j}R_j)_{\gamma_j}$ such that $\iota_j(a_j)\cdot r\neq0$ for all $r\in 1_{d(\gamma_j)}R\setminus\{0\}$.
In particular, $a_j\neq0$ and, thus, $K_{j,d(\gamma_j)}\neq\emptyset$.
Let $j'\in J$ be such that there exists $k'\in K_{j',d(\gamma_j)}$. Then $0\neq E_{k'k'}^{e_{j'}}\in(R_{j'})_{d(\gamma_j)}$ and it follows that $\iota_j(a_j)\iota_{j'}(E_{k'k'}^{e_{j'}})\neq0$. But this is only possible if $j=j'$. Hence, $\{j'\in J:K_{j',f_j}\neq\emptyset\}=\{j\}$. Now observe that $a_j\cdot r_j\neq0$ for all $r_j\in R_j(d(\gamma_j))\setminus\{0\}$. Proceeding as in the proof of $(2)\implies(3)$ of Theorem~\ref{teo: pseudo aneis com div primos}, we get that $K_{j,d(\gamma_j)}$ contains exactly one element.

   $(2)\implies (3)$: Fix $j\in J$ and set $R_j:=\M_{K_j}(D_j)(\overline{\sigma}_j)$. Since $|K_{j,f_j}|=1$, it follows from Theorem~\ref{teo: pseudo aneis com div primos} that $R_j$ is a gr-prime pfm ring. We have
   \[R(f_j)\cong_{gr}\sideset{}{^{gr}}\prod_{j'\in J}R_{j'}(f_j)=\sideset{}{^{gr}}\prod_{j'\in J}\bigoplus_{k'\in K_{j',f_j}}E_{k'k'}^{e_{j'}}R_{j'}.\]
   Since $K_{j',f_j}=\emptyset$ for all $j'\neq j$ e $|K_{j,f_j}|=1$, then $R(f_j)\cong_{gr}R_j(f_j)$ is gr-simple by Proposition~\ref{prop: M I(D)(E) gr-simples artiniano}(1).
        
   $(3)\implies (4)$: Follows from Theorem~\ref{teo: pseudo aneis com div primos}.

    $(4)\implies (5)$: Set $\Delta_0:=\{f_j:j\in J\}\subseteq\Gamma'_0(R)$. Then $R(\Delta_0)$ is $\Gamma_0$-simple. Fix $j\in J$. Since $R_j$ is a gr-simple ring with a minimal graded right ideal $R_j(f_j)$, it follows from Lemma~\ref{lem: artin -> semisimp}(2) that $(R_j)_{R_j}$ is gr-isomorphic to a direct sum of shifts of $R_j(f_j)$. Hence $\iota_j(R_j)$ is gr-isomorphic to a direct sum of shifts of $R(f_j)$, where $\iota_j:R_j\to R$ is the canonical inclusion. Now (5) follows from 
    \[R\cong_{gr}\sideset{}{^{gr}}\prod_{j\in J}R_j=\bigoplus_{j\in J}R_j=\bigoplus_{j\in J}\iota_j(R_j)\]
    
    $(5)\implies (1)$: Suppose that 
    \[R_R\cong_{gr}\bigoplus\limits_{i\in I}R(\Delta_0)(\sigma_i)\]
    for some $(\sigma_i)_{i\in I}\in\Gamma^I$. Since $R(e)$ is a gr-simple $R$-module for each  $e\in\Delta_0$,  $R$ is then a gr-semisimple ring. By Lemma~\ref{lem: R ss -> R(e) é soma finita}(1), every $\Gamma$-graded gr-simple right $R$-module is gr-isomorphic to  a shift of some element in the set $\{R(e):e\in\Delta_0\}$. Hence,  it is pseudo-free. Statement (1) is now a consequence of Lemma~\ref{lem: pfm --> ss}(2).

     $(4)\implies (6)$: Since $R(f_j)$ is gr-simple and $R_j(f_j)\neq 0$ for each $j\in J$, then $R_j(f_j)$ is gr-simple for all $j\in J$. Thus, $R_j$ is a right $\Gamma_0$-artinian ring by Theorem \ref{teo: simp + art = semisimp}. By Lemma~\ref{lem: R=END(R)} and Corollary~\ref{coro: schur para simples}, the fact that  $R(f_j)$  is gr-simple implies that $1_{f_j}R1_{f_j}\cong_{gr}\END_R(R(f_j))$ is a gr-division ring. Hence, $R_{f_j}$ is a division ring. Since $R_{f_j}\cong\prod_{j'\in J}(R_{j'})_{f_j}$ and $f_j\in \Gamma'_0(R_j)$, it follows that $(R_{j'})_{f_j}=0$ if $j'\neq j$.

    $(6)\implies (2)$: Let $j\in J$. By Theorem~\ref{teo: WA para simp -resumo}, 
    there exist a non-empty set $K_j$, $e_j\in \Gamma_0$, a $\Gamma$-graded division ring $D_j$ with $\supp(D_j)\subseteq e_j\Gamma e_j$ and a $d$-finite sequence $\overline{\sigma}_j:=(\sigma_{jk})_{k\in K_j}\in (e_j\Gamma)^{K_j}$ such that
 \[R_j\cong_{gr}\M_{K_j}(D_j)(\overline{\sigma}_j).\]
 If $j'\in J$ is such that $K_{j',f_j}\neq\emptyset$, then there exists $k\in K_{j'}$ such that $d(\sigma_{j'k})=f_j$ and, thus, $0\neq E_{kk}^{e_{j'}}\in (R_{j'})_{f_j}$. This implies $j'=j$ by hypothesis. Now observe that 
    \[\M_{K_j}(D_j)(\overline{\sigma}_j)_{f_j}\cong  (R_j)_{f_j}\cong R_{f_j}\]
   is a division ring. Therefore, if $k,k'\in K_{j,f_j}$, then $E_{kk}^{e_j},E_{k'k'}^{e_j}\in\M_{K_j}(D_j)(\overline{\sigma}_j)_{f_j}$ and it follows that $E_{kk}^{e_j}E_{k'k'}^{e_j}\neq0$.  This only happens if $k=k'$. Therefore, $K_{j,f_j}$ has exactly one element.
\end{proof}




\subsection{Invariance of the number of elements of pseudo-bases}\label{subsec:IPBN}

Let $R=\bigoplus\limits_{\gamma\in\Gamma}R_\gamma$ be a $\Gamma$-graded ring. We say that $R$ has \emph{invariant pseudo-basis number}, or \emph{IPBN} for short, if any two pseudo-bases of a finitely generated $\Gamma$-graded pseudo-free (right) $R$-module have the same number of elements. In other words, if we have
$$R(\gamma_1)\oplus \dotsb\oplus R(\gamma_m)\cong_{gr}R(\delta_1)\oplus \dotsb\oplus R(\delta_n)$$
for some $\overline{\gamma}=(\gamma_1,\dotsc,\gamma_m)\in\Gamma^m$ and $\overline{\delta}=(\delta_1,\dotsc,\delta_n)\in\Gamma^n$ with $1_{r(\gamma_i)},1_{r(\delta_j)}\neq 0$, then $m=n$. 

Observe that, by Lemma~\ref{lem: generation}, IPBN implies that any two pseudo-basis of a graded pseudo-free (right) $R$-module have the same number of elements. 

We also remark that, by Proposition~\ref{prop:matrices_grhomomorphisms}, any gr-homomorphism  
$$R(\gamma_1)\oplus \dotsb\oplus R(\gamma_m)\longrightarrow R(\delta_1)\oplus \dotsb\oplus R(\delta_n)$$
can be uniquely expressed by a matrix in $\M_{n\times m}(R)[\overline{\delta}][\overline{\gamma}]$. Thus, $R$ fails to have IPBN if and only if there exist natural numbers $m\neq n$ and matrices
$A\in \M_{n\times m}(R)[\overline{\delta}][\overline{\gamma}]$ and $B\in\M_{m\times n}(R)[\overline{\gamma}][\overline{\delta}]$ such that $AB=I_{r(\overline{\delta})}$, $BA=I_{r(\overline{\gamma})}$ for some
$\overline{\gamma}=(\gamma_1,\dotsc,\gamma_m)\in\Gamma^m$ and $\overline{\delta}=(\delta_1,\dotsc,\delta_n)\in\Gamma^n$ with $1_{r(\gamma_i)}, 1_{r(\delta_j)}\neq 0$ for all $i,j$. Note that this formulation of IPBN does not involve right or left $R$-modules. In particular, we see that
``right IPBN'' is synonymous with ``left IPBN''.

We now give a characterization of $\Gamma$-graded division rings among $\Gamma$-graded pfm rings. 

\begin{theorem}\label{theo:pfm_gr_division_ring}
    Let $R=\bigoplus\limits_{\gamma\in\Gamma}R_\gamma$ be a $\Gamma$-graded ring. The following statements are equivalent.
    \begin{enumerate}[\rm(1)]
        \item $R$ is a $\Gamma$-graded division ring.
        \item $R$ is pfm ring that has IPBN.
    \end{enumerate}
\end{theorem}

\begin{proof}
By Theorem~\ref{theo:modules_over_division_rings}, it is enough to show $(2)\implies (1)$. Suppose that $R$ is a pfm ring that satisfies IPBN. By Theorem~\ref{teo: pseudo aneis com div}, there exists $\Delta_0\subseteq\Gamma'_0(R)$ such that $R(\Delta_0)$ is $\Gamma_0$-simple and $R_R$ is gr-isomorphic to a direct sum of shifts of elements from $\{R(e):e\in\Delta_0\}$. Let $(e_i)_{i\in I}\in (\Delta_0)^I$ and $(\sigma_i)_{i\in I}\in \Gamma^I$ such that $R_R\cong_{gr} \bigoplus\limits_{i\in I}R(e_i)(\sigma_i)$. By Lemma~\ref{lem: R ss -> R(e) é soma finita}(2), for any $f\in\Gamma_0'(R)$,
there exists a finite subset $I_f\subseteq I$ such that $R(f)=\bigoplus_{i\in I_f}R(e_i)(\sigma_i)$. Since 
each $R(f)$ is pseudo-free, with a pseudo-basis consisting of only one element, the IPBN property implies $|I_f|=1$. Thus $R(f)$ is gr-simple for all $f\in \Gamma_0'(R)$. Therefore, $R_R$ is $\Gamma_0$-simple. Hence, $R\cong_{gr}\END(R_R)$ is a gr-division ring by Lemma \ref{lem: R=END(R)} and Theorem \ref{teo: schur}(1).
\end{proof}

As we are going to see next, it is still possible to define an invariant for graded modules over pfm rings similar to the pseudo-dimension and that coincides with the graded length of a finitely generated graded module over a pfm ring. To that end, we begin by pointing out some facts about gr-semisimple modules.

Let $R$ be a $\Gamma$-graded ring and $M$ be a gr-semisimple $R$-module. Thus, 
\begin{equation}\label{eq:semisimple_decomposition}
    M=\bigoplus_{i\in I} M_i
\end{equation}
where $M_i$ is a gr-simple submodule of $M$ for each $i\in I$. If we have another decomposition of $M=\bigoplus_{j\in J} M_j'$ where  $M_j'$ is a gr-simple submodule of $M$ for each $j\in J$, then
$|I|=|J|$ by Proposition~\ref{prop: gr-ss de dim infinita}. We will then refer to the cardinality of the set $I$ in \eqref{eq:semisimple_decomposition} by the \emph{gr-simple dimension} of $M$ and it will be denoted by $\sdim(M)$. Furthermore, by Proposition~\ref{prop:basics_gr-semisimple_modules}, if $N$ is any graded submodule of $M$, then $N$ is gr-semisimple and there exists a graded submodule $N'$ of $M$ such that
$M=N\oplus N'$. Therefore, we obtain that
\begin{equation}
\label{eq:simple_dimension}
    \sdim_R(M)=\sdim_R(N)+\sdim_R(N')=\sdim_R(N)+\sdim_R(M/N).
\end{equation}

Suppose now that $X$ is a $\Gamma$-graded module. We say that a pseudo-linearly independent sequence $(x_i)_{i\in I}$ of homogeneous elements of $X$ is a \emph{gr-simple sequence} if $x_iR$ is a gr-simple $R$-module for all $i\in I$. If, moreover, $(x_i)_{i\in I}$ is a pseudo-basis of $X$ we say that it is a \emph{gr-simple pseudo-basis} of $X$. In this event, we will write $\spdim_R(X)=|I|$. We proceed to show that $\spdim_R(X)$ is a well-behaved invariant for pfm rings.

\begin{proposition}
\label{prop: a dim grsimples}
    Let  $R=\bigoplus\limits_{\gamma\in \Gamma}R_\gamma$ be a pfm ring and $M=\bigoplus\limits_{\gamma\in \Gamma}M_\gamma$ be a $\Gamma$-graded $R$-module. The following assertions hold:
    \begin{enumerate}[\rm (1)]
        \item $M$ has a gr-simple pseudo-basis.
        \item Any two gr-simple pseudo-basis of $M$ have the same cardinality $\sdim_R(M)$.
         \item Every pseudo-linearly independent gr-simple sequence of $M$ extends to a gr-simple pseudo-basis of $M$.
        \item If $N$ is a graded submodule of $M$, then  $$\spdim_R(N)+\spdim_R(M/N)=\spdim_R(M).$$
    \end{enumerate}
\end{proposition}

\begin{proof}
  (1) By Lemma \ref{lem: pfm --> ss}(1), $R$ is a gr-semisimple ring. It follows from Proposition \ref{prop: CLP, Prop 59} that $M$ is a gr-semisimple $R$-module. Suppose that $M=\bigoplus\limits_{i\in I}M_i$ where each $M_i$ is a gr-simple submodule of $M$. Since $R$ is a pfm ring,
   each $M_i$, $i\in I$, has a pseudo-basis consisting of exactly one element, say $x_i\in M_i$. Thus $(x_i)_{i\in I}$ is a gr-simple pseudo-basis of $M$.

    (2) Let $B_1:=(x_i)_{i\in I}$ and $B_2:=(y_j)_{j\in J}$ two gr-simple pseudo-bases  of $M$. 
    Then
    \[M=\bigoplus_{i\in I}x_iR=\bigoplus_{j\in J}y_jR\]
    are two decompositions of $M$ as direct sum of gr-simple $R$-modules. 
     Then $|I|=|J|$ by Proposition~\ref{prop: gr-ss de dim infinita}.

    (3) Let $(x_i)_{i\in I}$ be a pseudo-linearly independent gr-simple sequence of homogeneous elements of $M$. Then $N:=\bigoplus\limits_{i\in I}x_iR$ is a graded submodule of the gr-semisimple module $M$. 
   Hence there exists a gr-submodule $N'$ of $M$ such that $M=N\oplus N'$. By (1), $N'$ has a gr-simple pseudo basis. Then the union of this pseudo-basis with $(x_i)_{i\in I}$ forms a gr-simple pseudo-basis of $M$.
    
    (4) follows from (2) and \eqref{eq:simple_dimension}.
\end{proof}





\subsection{More on gr-division rings}

Our aim now is to give more characterizations of graded division rings. This first result will follow from our version of the Wedderburn-Artin Theorem. Furthermore, the comparison of Theorem~\ref{teo: carac aneis com div via WA}(5) and Theorem~\ref{teo: pseudo aneis com div}(2) enlightens the difference between gr-division rings and pfm rings.


\begin{theorem}
\label{teo: carac aneis com div via WA}
Let $R=\bigoplus\limits_{\gamma\in\Gamma}R_\gamma$ be a $\Gamma$-graded ring. The following statements are equivalent.
\begin{enumerate}[\rm (1)]
    \item $R$ is a gr-division ring.
    \item $R_R$ is $\Gamma_0$-simple $R$-module.
    \item $R$ is a gr-semisimple ring and  $1_eR1_e$ is an $e\Gamma e$-graded division ring for all $e\in\Gamma'_0(R)$.
    \item $R$ is a gr-semisimple ring and $R_e$ is a division ring for all $e\in\Gamma'_0(R)$.
    \item There exist a set $J$, a family of non-empty subsets $\{K_j:j\in J\}$, a sequence of idempotents $(e_j)_{j\in J}\in(\Gamma_0)^J$
    such that
    \[R\cong_{gr}\sideset{}{^{gr}}\prod_{j\in J}\M_{K_j}(D_j)(\overline{\sigma}_j),\]
    where  $D_j$ is a $\Gamma$-graded division ring with $\supp(D_j)\subseteq e_j\Gamma e_j$ and $\overline{\sigma}_j:=(\sigma_{jk})_{k\in K_j}\in (e_j\Gamma)^{K_j}$  for each $j\in J$, and the sets $K_{j,e}:=\{k\in K_j:d(\sigma_{jk})=e\}$ and $J_e:=\{j\in J:K_{j,e}\neq\emptyset\}$ have at most one element for all $j\in J$ and $e\in\Gamma_0$.
    \item There exists a family $\{R_j:j\in J\}$ of gr-prime gr-division rings such that $\supp(R_j)\cap \supp(R_{j'})=\emptyset$ for all different $j,j'\in J$ and 
    \[R=\bigoplus_{j\in J} R_j.\]
   \end{enumerate}
\end{theorem}

\begin{proof}
    $(1)\implies (2)$: For each $e\in\Gamma'_0(R)$, every nonzero element in $R(e)$ has an inverse and,  therefore, $R(e)$ is gr-simple.

    $(2)\implies (3)$: For each $e\in \Gamma'_0(R)$, $R(e)$ is a gr-simple $R$-module. Thus, $R=\bigoplus\limits_{e\in\Gamma_0}R(e)$ is a gr-semisimple ring. Moreover, for each $e\in\Gamma'_0(R)$, 
    every nonzero homogeneous element of the ring $1_eR1_e\subseteq R(e)$ has a right inverse. Hence, $1_eR1_e$ is an $e\Gamma e$-graded division ring for each $e\in\Gamma'_0(R)$.

    $(3)\implies (4)$: Straightforward.
    
    $(4)\implies (5)$: By Theorem~\ref{teo: wa para ss -resumo}, there exist families $\{K_j:j\in J\}$ and $\{e_j:j\in J\}$ of non-empty sets and idempotents of $\Gamma$, respectively, and there exist, for each $j\in J$, a $\Gamma$-graded division ring $D_j$ with  $\supp(D_j)\subseteq e_j\Gamma e_j$ and a $d$-finite sequence $\overline{\sigma}_j:=(\sigma_{jk})_{k\in K_j}\in (e_j\Gamma)^{K_j}$, such that  
    the family $\{\M_{K_j}(D_j)(\overline{\sigma}_j):j\in J\}$ is summable and 
    \[R\cong_{gr}\sideset{}{^{gr}}\prod_{j\in J}\M_{K_j}(D_j)(\overline{\sigma}_j).\]
    Fix $e\in\Gamma'_0(R)$. Then
    \[R_e\cong\prod_{j\in J_e}\M_{K_j}(D_j)(\overline{\sigma}_j)_e\]
    is a division ring. Thus $|J_e|=1$, say $J_e=\{j_0\}$. Given $k,l\in K_{j_0,e}$, then $E_{kk}^{e_{j_0}}E_{ll}^{e_{j_0}}\neq0$ and, therefore, $k=l$. Hence, $|K_{j_0,e}|=1$.

    $(5)\implies (6)$: For each $j\in J$, $R_j:=\M_{K_j}(D_j)(\overline{\sigma}_j)$ is a gr-prime gr-division ring, by Theorem~\ref{theo: anel com div primo = anel de matr}. Suppose now that $j,j'\in J$ and $\gamma\in\supp(R_j)\cap \supp(R_{j'})$. Since $|K_{j,e}|=1$ for all $e\in\Gamma'_0(R)$, there exists a unique $p\in K_j$ such that $d(\sigma_{jp})=r(\gamma)$. Analogously, there exists a unique $p'\in K_{j'}$ such that $d(\sigma_{j'p'})=r(\gamma)$. Thus, we obtain non-empty $K_{j,r(\gamma)}$ and $K_{j',r(\gamma)}$, hence $j,j'\in J_{r(\gamma)}$ and, therefore, $j'=j$.
        
    $(6)\implies (1)$: Since $\supp(R_j)\cap \supp(R_{j'})=\emptyset$ for all different $j,j'\in J$, it follows that, for each $\gamma\in\Gamma$, there exists a unique $j\in J$ such that $R_\gamma=(R_j)_\gamma$. 
    Hence, $R$ is a $\Gamma$-graded division ring because all  $R_j$ are. 
    \end{proof}

The next result shows that  to be a pfm ring is equivalent to be a gr-division ring for object crossed products.

\begin{proposition}
   Let $(A,\Gamma,\alpha,\beta)$ be an object crossed system. The following assertions are equivalent.
   \begin{enumerate}[\rm (1)]
       \item $A\rtimes^{\alpha}_{\beta} \Gamma$ is a pfm ring.
       \item $A_e$ is a division ring for all $e\in \Gamma_0$.
       \item $A\rtimes^{\alpha}_{\beta} \Gamma$ is a gr-division ring.
   \end{enumerate}
\end{proposition}

\begin{proof}
    $(1)\implies(2)$: 
    Suppose that $A\rtimes^{\alpha}_{\beta} \Gamma$ is a pfm ring. By Theorem \ref{teo: pseudo aneis com div}, there exists a summable family $\{R_j:j\in J\}$ of gr-simple $\Gamma_0$-artinian rings and $(f_j)_{j\in J}\in(\Gamma_0)^J$ such that $A\rtimes^{\alpha}_{\beta} \Gamma\cong_{gr}\sideset{}{^{gr}}\prod\limits_{j\in J}R_j$
    and, for each $j\in J$, $A_{f_j}\cong(A\rtimes^{\alpha}_{\beta} \Gamma)_{f_j}$ is a division ring and $(R_{j'})_{f_j}=0$ whenever $j'\in J\setminus\{j\}$. Let $e\in\Gamma_0$. Since $(A\rtimes^{\alpha}_{\beta} \Gamma)_e\neq0$, there exists $j\in J$ such that $(R_j)_e\neq0$. The gr-primeness of $R_j$ implies that $1_eR_j1_{f_j}\neq0$ and there exists $\sigma\in e\Gamma f_j$. Then $A_e=\alpha_\sigma(A_{f_j})\cong A_{f_j}$ is a division ring.

    $(2)\implies(3)$: If (2) holds, then $A\rtimes^{\alpha}_{\beta} \Gamma$ is a gr-semisimple ring by Proposition~\ref{prop: prod cruz ss} and $(A\rtimes^{\alpha}_{\beta} \Gamma)_e\cong A_e$ is a division ring for all $e\in\Gamma_0$. By Theorem \ref{teo: carac aneis com div via WA}, $A\rtimes^{\alpha}_{\beta} \Gamma$ is a gr-division ring.

    $(3)\implies(1)$: It follows from Theorem \ref{theo:modules_over_division_rings}.
\end{proof}

Recall that a (group graded) ring is a (group graded) division ring if and only if all its (graded) right modules are (graded) free. The next result is a generalization of such fact for groupoid graded rings. Indeed, if $\Gamma$ is a group and $e$ is the identity of $\Gamma$, then $e\Gamma e=\Gamma$ and $1_eR1_e=R$. Thus, Proposition~\ref{prop: carac ane com div}(2) implies that $R$ is a $\Gamma$-graded division ring.

\begin{proposition}
\label{prop: carac ane com div}
Let $R=\bigoplus\limits_{\gamma\in\Gamma}R_\gamma$ be a $\Gamma$-graded ring. The following statements are equivalent.
\begin{enumerate}[\rm (1)]
    \item $R$ is a gr-division ring.
    \item $R$ is a pfm ring and every $e\Gamma e$-graded right $1_eR1_e$-module is gr-free (as a group graded module) for all $e\in\Gamma'_0(R)$.
    
    \item $R$ is a pfm ring and every right $R_e$-module is free for all $e\in\Gamma'_0(R)$.
\end{enumerate}
\end{proposition}

\begin{proof}
Implications $(1)\Rightarrow (2)$ and $(1)\Rightarrow (3)$ hold by Theorem~\ref{theo:modules_over_division_rings}(1) and because, by Theorem~\ref{teo: carac aneis com div via WA}, (1) implies that $1_eR1_e$ is an $e\Gamma e$-graded division ring  and $R_e$ is a division ring for all $e\in\Gamma'_0(R)$. 

  $(2)\Rightarrow (1)$ (resp. $(3)\Rightarrow (1)$) holds because  of Theorem~\ref{teo: carac aneis com div via WA}, since every  pfm ring is gr-semisimple by Lemma~\ref{lem: pfm --> ss}(1).
\end{proof}

As Example~\ref{ex: exem onde todo mod e livre}(2) shows, the hypothesis about $1_eR1_e$-modules or $R_e$-modules are necessary in Proposition~\ref{prop: carac ane com div}. Moreover,  as we are going to see next, the hypothesis of $R$ being a pfm ring in Proposition~\ref{prop: carac ane com div} cannot be dropped either. 
We proceed to present some cases where some of these conditions are not necessary.

\begin{proposition}
\label{prop: quando vale a carac de ane com div}
Consider the following statements.
\begin{enumerate}[\rm (I)]
    \item $R$ is a $\Gamma$-graded division ring if and only if every $\Gamma$-graded right  $R$-module is pseudo-free.
    \item $R$ is a $\Gamma$-graded division ring if and only if, for each $e\in\Gamma'_0(R)$, every  $e\Gamma e$-graded right $1_eR1_e$-module is gr-free.
    \item $R$ is a $\Gamma$-graded division ring if and only if, for each $e\in\Gamma'_0(R)$, every right $R_e$-module is free.
\end{enumerate}
The following assertions hold:
\begin{enumerate}[\rm (1)]
    \item Statement (I) holds true for every $\Gamma$-graded ring $R$ if and only if 
     $\Gamma=\bigcup\limits_{e\in\Gamma_0}e\Gamma e$, that is, $\Gamma$ is a disjoint union of groups.
    \item Statement (II) holds true for every $\Gamma$-graded ring $R$ if and only if 
    $\Gamma=\bigcup\limits_{e\in\Gamma_0}e\Gamma e$.
    \item Statement (III) holds true for every $\Gamma$-graded ring $R$ if and only if 
     $\Gamma=\Gamma_0$.
\end{enumerate}
\end{proposition}

\begin{proof}
    (1) Suppose that $\Gamma=\bigcup\limits_{e\in\Gamma_0}e\Gamma e$, that is, $e\Gamma f=\emptyset$ for all different $e,f\in\Gamma_0$. 
    Let $R$ be a pfm ring. Then, by Theorem~\ref{teo: pseudo aneis com div}, there exist a summable family of $\Gamma$-graded gr-simple rings $\{R_j:j\in J\}$ and $(f_j)_{j\in J}\in(\Gamma_0)^J$ such that 
    \[R\cong_{gr}\sideset{}{^{gr}}\prod_{j\in J}R_j\]
   and, for each $j\in J$,  $R(f_j)$ is gr-simple and $R_j(f_j)\neq0$.
   Since $R_j(f_j)\neq0$ and $1_{f_j}R_j1_{f_j}$ is a nonzero graded ideal of $R_j$, it follows that $\supp R_j\subseteq f_j\Gamma f_j$. Since $R_j(f_j)$ is gr-simple, Theorem~\ref{teo: carac aneis com div via WA} implies that  $R_j=1_{f_j}R_j1_{f_j}$ is a $\Gamma$-graded division ring. Furthermore, if
   $j\neq j'$, then $R_{j'}(f_j)=0$ and $f_j\neq f_{j'}$. Thus, Theorem~\ref{teo: carac aneis com div via WA}(6) is satisfied and it follows that $R$ is a $\Gamma$-graded division ring. Therefore, statement (I) holds for all $\Gamma$-graded rings $R$.

    Conversely, suppose that $\Gamma\neq\bigcup_{e\in\Gamma_0}e\Gamma e$. Inspired by Example~\ref{ex: exem onde todo mod e livre}(2), we are going to construct a $\Gamma$-graded ring for which statement (I) does not hold. Let $D$ be a division ring and $\sigma\in\Gamma$ be such that $r(\sigma)\neq d(\sigma)$. Consider $D$ as a $\Gamma$-graded division ring with support concentrated in $\{r(\sigma)\}$. Set 
    \[R=\M_3(D)(r(\sigma),r(\sigma),\sigma).\]
    Then Theorem~\ref{teo: pseudo aneis com div primos}(3) holds and it follows that every $\Gamma$-graded right $R$-module is pseudo-free. But $R$ is not a gr-division ring because $E^{r(\sigma)}_{11}\in R_{r(\sigma)}$  is not invertible.

    (2) If $\Gamma=\bigcup\limits_{e\in\Gamma_0}e\Gamma e$, then $R=\bigoplus\limits_{e\in\Gamma_0}1_eR1_e$ for any $\Gamma$-graded ring $R$. Hence, $R$ is a $\Gamma$-graded division ring if and only if $1_eR1_e$ is an $e\Gamma e$-graded division ring for each $e\in\Gamma'_0(R)$. But this last condition amounts to say that every $e\Gamma e$-graded right $1_eR1_e$-module is  gr-free for each $e\in\Gamma'_0(R)$. Therefore, statement (II) holds for all $\Gamma$-graded rings $R$.

    Conversely, suppose that $\Gamma\neq\bigcup\limits_{e\in\Gamma_0}e\Gamma e$ and let $\sigma\in\Gamma$ be such that $r(\sigma)\neq d(\sigma)$. Let $D$ be a division ring and
        \[R=\begin{bmatrix}
    D & D \\
    0 & D \\
    \end{bmatrix}.\]
    We make of $R$ an object unital $\Gamma$-graded ring  via 
        \[R_{r(\sigma)}=\begin{bmatrix}
    D & 0 \\
    0 & 0 \\
    \end{bmatrix}, \quad R_\sigma=\begin{bmatrix}
    0 & D \\
    0 & 0 \\
    \end{bmatrix}, \quad R_{d(\sigma)}=\begin{bmatrix}
    0 & 0 \\
    0 & D \\
    \end{bmatrix}.\]
    Then $R_{r(\sigma)}$ and $R_{d(\sigma)}$ are division rings with identity element  $1_{r(\sigma)}:=E_{11}$ and $1_{d(\sigma)}:=E_{22}$, respectively. Thus, $1_{r(\sigma)}R1_{r(\sigma)}=R_{r(\sigma)}$ is a  $r(\sigma)\Gamma r(\sigma)$-graded division ring and $1_{d(\sigma)}R1_{d(\sigma)}=R_{d(\sigma)}$ is a $d(\sigma)\Gamma d(\sigma)$-graded division ring. But $R$ is not a $\Gamma$-graded division ring because
    $E_{12}=\begin{bmatrix}
    0 & 1 \\
    0 & 0 \\
    \end{bmatrix}\in R_\sigma$ is not invertible. Therefore such $R$ contradicts statement (II).
    
    (3) If $\Gamma=\Gamma_0$, then $R=\bigoplus\limits_{e\in\Gamma_0}R_e$ for any $\Gamma$-graded ring. In this event, $R$ is a $\Gamma$-graded division ring if and only if $R_e$ is a division ring for each $e\in\Gamma'_0(R)$. This amounts to say that every right   $R_e$-module is free for each $e\in\Gamma'_0(R)$. Therefore,  (III) holds for all $\Gamma$-graded rings $R$. 
    
    Conversely, suppose that $\Gamma\neq\Gamma_0$. Thus, there exists $\sigma\in\Gamma$ such that either 
    $\sigma^2$ is not defined or  $\sigma^2$ is defined but $\sigma^2\neq\sigma$. 
    If $r(\sigma)\neq d(\sigma)$, then the ring $R$ used in (2) shows that statement (III) does not hold. Hence suppose that $r(\sigma)=d(\sigma)$ and $\sigma^2\neq\sigma$. Let $D$ be a division ring and consider the ring
    \[R:=\frac{D[x]}{\langle x^2\rangle},\]
    endowed with a $\Gamma$-grading via
    \[R_{r(\sigma)}:=D+\langle x^2\rangle,\quad\quad R_\sigma:= Dx+\langle x^2\rangle.\]
    Then $R_{r(\sigma)}$ is a division ring and, therefore, every right $R_{r(\sigma)}$-module is free.
    But $R$ is not a $\Gamma$-graded division ring because $x+\langle x^2\rangle\in R_\sigma$ is not invertible. Therefore, statement (III) does not hold for all $\Gamma$-graded rings.
     \end{proof}


\section{Semisimple categories}
\label{sec: semisimple categories}

\emph{
Throughout this section, we fix a small preadditive category $\mathcal{C}$. We will use the following notation: $\mathcal{C}_0$ is the set of objects of $\mathcal{C}$, $I_X$ is the identity morphism of $X\in\mathcal{C}_0$ and $\mathcal{C}(A,B):=\Hom_\mathcal{C}(A,B)$ for all $A,B\in\mathcal{C}_0$.}

\emph{We denote by $\mathcal{A}b$ the category of abelian groups (in which we will always adopt additive notation).} 


We alert the reader that, in this section, the letter $f$ will denote a morphism in a category and not an idempotent in a groupoid as in previous sections.

\medskip

There are some concepts  in Ring Theory that are also defined in Category Theory. 
For example,    definitions of semisimple category, artinian category, free functor and others can be found in \cite{Gab, Mit}. In this section,  we explore some of these notions, relate them to the concepts of foregoing sections and show how they are related via the ring of the category when considered as a groupoid graded ring as in Example~\ref{ex: aneis graduados2}(1).

Many of the results in this section can be proved in a more general context. The authors have been  working on a paper in which they introduce the concept of groupoid graded categories, which generalize group graded categories, and results are obtained for such more general categories.

\subsection{(Bi)functors are graded (bi)modules}

Let $\Fun(\mathcal{C},\mathcal{A}b)$ be the category of additive covariant functors \mbox{$\mathcal{C}\to\mathcal{A}b$}. By  $\Fun(\mathcal{C}^{op},\mathcal{A}b)$, we denote the category of additive contravariant functors  \mbox{$\mathcal{C}\to\mathcal{A}b$}.

For each $A\in\mathcal{C}_0$, we consider the following functors $\mathcal{C}\to\mathcal{A}b$
\[\mathcal{C}(A,-):=\hom_\mathcal{C}(A,-)\quad\textrm{~and~}\quad\mathcal{C}(-,A):=\hom_\mathcal{C}(-,A).\]
 A \emph{right sieve on $A$ in $\mathcal{C}$} is an additive subfunctor (a subobject in the category of additive functors) of $\mathcal{C}(-,A)$ and a \emph{left sieve on $A$ in $\mathcal{C}$} is an additive subfunctor of $\mathcal{C}(A,-)$. 

In \cite{Mit}, $\mathcal{C}(A,-)$ and $\mathcal{C}(-,A)$ are denoted  by $\mathcal{C}_A$ and $\mathcal{C}_A^{~*}$, respectively. In \cite[p. 18]{Mit}, a right (left) ideal of $\mathcal{C}$ is defined as a subfunctor of $\mathcal{C}_A$ (of $\mathcal{C}_A^{~*}$) for some  $A\in\mathcal{C}_0$. We point out that we are using the opposite notation because our composition of morphisms in the category $\mathcal{C}$ is the usual from right to left, while \cite[p. 7]{Mit} composes morphisms of $\mathcal{C}$ from left to right.


In what follows, inspired by \cite[Proposition 2 (p. 347)]{Gab}, we present a possible motivation for the previous definitions. 

Let $R[\mathcal{C}]$ be the ring of the category $\mathcal{C}$ as defined in Example~\ref{ex: aneis graduados2}(1). Recall that $R[\mathcal{C}]$ is graded by the groupoid $\mathcal{G}:=\mathcal{C}_0\times\mathcal{C}_0$, see Example~\ref{ex:groupoids}(2), whose idempotents  are of the form 
$$\mathcal{\varepsilon}_A:=(A,A), \quad A\in\mathcal{C}_0.$$

Let  $M=\bigoplus_{(A,B)\in\mathcal{G}}M_{(A,B)}$ be a $\mathcal{G}$-graded right $R[\mathcal{C}]$-module. We build an additive contravariant functor $M(-,\mathcal{C}_0):\mathcal{C}\to\mathcal{A}b$ as follows: for each $X\in\mathcal{C}_0$, we define 
\[M(X,\mathcal{C}_0):=MI_X=\bigoplus_{A\in\mathcal{C}_0}M_{(A,X)}\]
and for each morphism $f:X\to Y$ in $\mathcal{C}$ we define
\begin{align*}
    M(f,\mathcal{C}_0):M(Y,\mathcal{C}_0)&\to M(X,\mathcal{C}_0)\\
    m&\mapsto mf.
\end{align*}
If $M=M(\varepsilon_A)$ for some $A\in\mathcal{C}_0$, then we will write $M(-,A)$ instead of $M(-,\mathcal{C}_0)$.

Consider the $\mathcal{G}$-graded ring $\mathbb{Z}^{(\mathcal{G}_0)}$ and the $\mathcal{G}$-graded abelian groups as left $\mathbb{Z}^{(\mathcal{G}_0)}$-modules
as in Remark~\ref{rem:examples of gradation}(2).
Note that, for each $X,Y\in\mathcal{C}_0$ and $f\in\mathcal{C}(X,Y)$, $M(X,\mathcal{C}_0)$ is a $\mathcal{G}$-graded left $\mathbb{Z}^{(\mathcal{G}_0)}$-module with support contained in $\mathcal{G}\varepsilon_X$ and $M(f,\mathcal{C}_0)$ is a homomorphism of degree $(Y,X)\in\mathcal{G}$.

We denote by $\Fungr(\mathcal{C}^{op},\mathcal{G}\varepsilon-\mathbb{Z}^{(\mathcal{G}_0)}\Rmod)$ the category whose objects are the additive contravariant functors $F:\mathcal{C}\to\mathcal{G}\varepsilon-\mathbb{Z}^{(\mathcal{G}_0)}\Rmod$ such that, if $X,Y\in\mathcal{C}_0$ and $f\in\mathcal{C}(X,Y)$, then $F(f)\in\HOM(F(Y),F(X))_{(Y,X)}$. The morphisms between two functors $F,G\in \Fungr(\mathcal{C}^{op},\mathcal{G}\varepsilon-\mathbb{Z}^{(\mathcal{G}_0)}\Rmod)$ are the natural transformations $(\alpha_X:F(X)\to G(X))_{X\in\mathcal{C}_0}$ such that $\alpha_X\in\Hom_{gr}(F(X),G(X))$.
Note that if $F\in\Fungr(\mathcal{C}^{op},\mathcal{G}\varepsilon-\mathbb{Z}^{(\mathcal{G}_0)}\Rmod)$, then for each $X\in\mathcal{C}_0$ we have $\supp F(X)\subseteq \mathcal{G}\varepsilon_X$, because $F(I_X)$ is the unity of the ring $\END(F(X))_{(X,X)}$ and therefore
\[0\neq a\in F(X)_{(Y,Z)}\implies 0\neq a=(a)F(I_X)\in F(X)_{(Y,Z)(X,X)}\implies Z=X.\]

Now, let $F:\mathcal{C}\to\mathcal{A}b$ be an additive contravariant functor and consider the additive group
\[M[F]:=\bigoplus_{X\in\mathcal{C}_0}F(X).\]
Given $X,Y,Z\in\mathcal{C}_0$, $m\in F(Z)$ and $f\in R[\mathcal{C}]_{(Y,X)}=\mathcal{C}(X,Y)$, we define
\[mf:=\begin{cases}
    (F(f))(m)\in F(X), &\textrm{~if $Z=Y$}\\
    0, &\textrm{~if $Z\neq Y$}
\end{cases}\,.\]
It is easy to see that this makes $M[F]$ a unital right $R[\mathcal{C}]$-module.
Thus, if we fix $A\in\mathcal{C}_0$, then $M[F]$ is a $\mathcal{G}$-graded right $R[\mathcal{C}]$-module
via
\[M[F]_{(A,X)}:=F(X)\]
for each $X\in\mathcal{C}_0$. 
If $F\in\Fungr(\mathcal{C}^{op},\mathcal{G}\varepsilon-\mathbb{Z}^{(\mathcal{G}_0)}\Rmod)$, then we consider $M[F]$ as a $\mathcal{G}$-graded right $R[\mathcal{C}]$-module
via
\[M[F]_{(Y,X)}:=F(X)_{(Y,X)}\]
for each $(Y,X)\in\mathcal{G}$.

The next result links the constructions of the previous paragraphs. Recall that, by Remark \ref{rem:examples of gradation}(1), every unital right $R[\mathcal{C}]$-module $M$ can be regarded as a $\mathcal{G}$-graded $R[\mathcal{C}]$-module making $M=M(\varepsilon_A)$ for some $\varepsilon_A\in\mathcal{G}_0$. Therefore, Theorem \ref{teo: funtores e modulos}(1) is a rephrasing of \cite[Proposition 2 (p. 347)]{Gab}.


\begin{theorem}
\label{teo: funtores e modulos}
    Let $\mathcal{G}:=\mathcal{C}_0\times\mathcal{C}_0$. 
    The following functors define an equivalence of categories
        \begin{align*}
          \mathcal{G}-\grR R[\mathcal{C}]&\leftrightarrows\Fungr(\mathcal{C}^{op},\mathcal{G}\varepsilon-\mathbb{Z}^{(\mathcal{G}_0)}\Rmod)\\
          M&\mapsto M(-,\mathcal{C}_0)\\
          M[F]&\mapsfrom F .
                    \end{align*}
   Moreover, if $A\in\mathcal{C}_0$,
    then such equivalence induces
    \begin{enumerate}[\rm (1)]
        \item an equivalence of categories
        \begin{align*}
            \varepsilon_A\mathcal{G}-\grR R[\mathcal{C}]&\leftrightarrows\Fun(\mathcal{C}^{op},\mathcal{A}b)\\
            M&\mapsto M(-,A)\\
            M[F]&\mapsfrom F;
        \end{align*}
        
        \item a bijection between the sets
        \[\{\textrm{graded right ideals of $R[\mathcal{C}]$ contained in $R[\mathcal{C}](\varepsilon_A)$}\}\rightarrow\{\textrm{right sieves on $A$ in $\mathcal{C}$}\}\]
         sending $R[\mathcal{C}](\varepsilon_A)$ to $\mathcal{C}(-,A)$.
    \end{enumerate}
\end{theorem}

\begin{proof}
Let $M$ and $N$ be objects of $\mathcal{G}-\grR R[\mathcal{C}]$. If $\alpha\in\Hom_{gr}(M,N)$, then $\alpha(MI_X)\subseteq NI_X$ for all $X\in\mathcal{C}_0$. Thus, $\alpha$ induces a natural transformation $M(-,\mathcal{C}_0)\rightarrow N(-,\mathcal{C}_0)$ given by restriction of $\alpha$, that is, $\Big(M(X,\mathcal{C}_0)\xrightarrow[]{\alpha_X} N(X,\mathcal{C}_0)\Big)_{X\in\mathcal{C}_0}$ where $\alpha_X=\alpha|_{MI_X}$. Indeed, if $X,Y\in\mathcal{C}_0$ and $f\in \mathcal{C}(X,Y)$, then $\alpha_X(mf)=\alpha_Y(m)f$ for all $m\in MI_Y$, i.e., the following diagram commutes
\[\xymatrix{ 
    MI_X=M(X,\mathcal{C}_0)\ar[d]_{\alpha_X} & M(Y,\mathcal{C}_0)=MI_Y\ar[l]_{M(f,\mathcal{C}_0)}\ar[d]^{\alpha_Y} \\ 
    NI_X=N(X,\mathcal{C}_0) & N(Y,\mathcal{C}_0)=NI_Y\ar[l]_{N(f,\mathcal{C}_0)}
}\]

Conversely, let $F,G\in\Fungr(\mathcal{C}^{op},\mathcal{G}\varepsilon-\mathbb{Z}^{(\mathcal{G}_0)}\Rmod)$ and $\Big(\alpha_X\colon F(X)\rightarrow G(X)\Big)_{X\in\mathcal{C}_0}$, where each $\alpha_X\in\Hom_{gr}(F(X),G(X))$, be a natural transformation $F\rightarrow G$. 
By the universal property of direct sums, we can define a homomorphism of $\Gamma$-graded additive groups
$\alpha:=\sum_{X\in\mathcal{C}_0}\alpha_X:M[F]\rightarrow M[G]$. Moreover, if $A,B,X,Y\in\mathcal{C}_0$, $f\in\mathcal{C}(X,Y)$ and $m\in M[F]_{(A,B)}$, then 
\[\alpha(mf)
=\begin{cases}
    \alpha_X((m)F(f)),&\textrm{if $B=Y$}\\
    \alpha(0),&\textrm{if $B\neq Y$}
\end{cases}
=\begin{cases}
    (\alpha_Y(m))G(f),&\textrm{if $B=Y$}\\
    0,&\textrm{if $B\neq Y$}
\end{cases}
=\alpha(m)f \]
by the commutativity of the diagrams in a natural transformation.


 It is straightforward to check that $M[M(-,\mathcal{C}_0)]\cong_{gr}M$ for each object $M$ of $\mathcal{G}-\grR R[\mathcal{C}]$, and $M[F](-,\mathcal{C}_0)\cong F$ for each object $F$ of $\Fungr(\mathcal{C}^{op},\mathcal{G}\varepsilon-\mathbb{Z}^{(\mathcal{G}_0)}\Rmod)$. Furthermore, such isomorphisms are natural.

(1) Fix $A\in\mathcal{C}_0.$ For $F\in\Fun(\mathcal{C}^{op},\mathcal{A}b)$, consider $M[F]$ as an object of $\varepsilon_A\mathcal{G}-\grR R[\mathcal{C}]$ via $M[F]_{(A,X)}:=F(X)$ for each $X\in\mathcal{C}_0$. 
 Then the assignments $M\mapsto M(-,A)$ and $F\mapsto M[F]$ define an equivalence between the categories $\varepsilon_A\mathcal{G}-\grR R[\mathcal{C}]$ and $\Fun(\mathcal{C}^{op},\mathcal{A}b)$.

(2) 
Note that if $M$ is a graded right ideal of $R[\mathcal{C}]$ with $M=M(\varepsilon_A)$ for some $A\in\mathcal{C}_0$, then $M(-,A)$ is a subfunctor of $\mathcal{C}(-,A)$, i.e., a right sieve on $A$ in $\mathcal{C}$. And, if $F$ is a right sieve on $A$ in $\mathcal{C}$, then $M[F]$ is  naturally gr-isomorphic to a graded right ideal of $R[\mathcal{C}]$ contained in $R[\mathcal{C}](\varepsilon_A)$. Thus, (2) follows from (1).
\end{proof}

\begin{remark}
\label{rem: funtores e modulos a esq}
    Analogously, we define the category $\Fungr(\mathcal{C},\,\varepsilon\mathcal{G}-\modR \mathbb{Z}^{(\mathcal{G}_0)})$ whose objects are the additive covariant functors $F:\mathcal{C}\to\varepsilon\mathcal{G}-\modR \mathbb{Z}^{(\mathcal{G}_0)}$ such that, if $X,Y\in\mathcal{C}_0$ and $f\in\mathcal{C}(X,Y)$, then $F(f)\in\HOM(F(X),F(Y))_{(Y,X)}$.
    Then, for a $\mathcal{G}$-graded left $R[\mathcal{C}]$-module $M$ we can build functors $M(\mathcal{C}_0,-)$ and $M(A,-)$. In this way, we prove that there exists an 
    equivalence  of categories
        \begin{align*}
            \mathcal{G}- R[\mathcal{C}]\Sgr&\leftrightarrows\Fungr(\mathcal{C},\,\varepsilon\mathcal{G}-\modR\mathbb{Z}^{(\mathcal{G}_0)})\\
            M&\mapsto M(\mathcal{C}_0,-)\\
             M[F]&\mapsfrom F,
        \end{align*}
    where $M[F]_{(X,Y)}:=F(X)_{(X,Y)}$ for each $(X,Y)\in\mathcal{G}$. If $A\in\mathcal{C}_0$, then 
    we get
    \begin{enumerate}[\rm (1)]
        \item an 
        equivalence of categories 
        \begin{align*}
            \mathcal{G}\varepsilon_A- R[\mathcal{C}]\Sgr&\rightarrow\Fun(\mathcal{C},\mathcal{A}b)\\
            M&\mapsto M(A,-).
        \end{align*}
        \item a bijection
        \[\{\textrm{graded left ideals of $R[\mathcal{C}]$ contained in $(\varepsilon_A)R[\mathcal{C}]$}\}\rightarrow\{\textrm{left sieves on $A$ in $\mathcal{C}$}\}\]
        that sends $(\varepsilon_A)R[\mathcal{C}]$ to $\mathcal{C}(A,-)$.\qed
    \end{enumerate}
\end{remark}

\medskip

Now we turn our attention to bimodules and bifunctors.

Let $M$ be a $\mathcal{G}$-graded $(R[\mathcal{C}],R[\mathcal{C}])$-bimodule. We define a bifunctor $$M(-,-):\mathcal{C}^{op}\times\mathcal{C}\to\mathcal{A}b$$ as follows. For each $X,Y\in\mathcal{C}_0$, we put 
\[M(X,Y):=M_{(Y,X)}\]
and, for morphisms $f:Z\to X$ and $g:Y \to W$ in $\mathcal{C}$, we define
\begin{align*}
    M(f^{op},g):M(X,Y)&\to M(Z,W)\\
    m&\mapsto g m f.
\end{align*}

Conversely, let $F$ be an additive bifunctor $\mathcal{C}^{op}\times\mathcal{C}\to \mathcal{A}b$. Consider the additive group
\[M[F]:=\bigoplus_{X,Y\in\mathcal{C}_0}F(X,Y).\]
For $A,B,X,Y,Z,W\in\mathcal{C}_0$, $m\in F(A,B)$, $f\in R[\mathcal{C}]_{(X,Z)}=\mathcal{C}(Z,X)$ and $g\in R[\mathcal{C}]_{(W,Y)}=\mathcal{C}(Y,W)$, we define
\[gmf:=\begin{cases}
    F(f^{op},g)(m)\in F(Z,W), &\textrm{~if $(A,B)=(X,Y)$}\\
    0, &\textrm{~if $(A,B)\neq(X,Y)$}
\end{cases}\,.\]
It is not difficult to verify that $M[F]$ is a $\mathcal{G}$-graded $(R[\mathcal{C}],R[\mathcal{C}])$-bimodule via
\[M[F]_{(X,Y)}:=F(Y,X)\]
for each $(X,Y)\in\mathcal{G}$. 

The next theorem links the two previous constructions. Before stating the result, we need some definitions.
Following \cite[p. 18]{Mit}, we define an \emph{ideal} of $\mathcal{C}$ as an additive subfunctor of the bifunctor $\mathcal{C}(-,-):\mathcal{C}^{op}\times \mathcal{C}\to\mathcal{A}b$. We denote  the category of additive bifunctors $\mathcal{C}^{op}\times\mathcal{C}\to \mathcal{A}b$ by $\Bifun(\mathcal{C},\mathcal{A}b)$. We also define the category $R[\mathcal{C}]\RgrR R[\mathcal{C}]$ whose objects are the $\mathcal{G}$-graded $(R[\mathcal{C}],R[\mathcal{C}])$-bimodules and the morphisms are the homomorphisms of bimodules $g:M\to N$ such that $g(M_\sigma)\subseteq N_\sigma$ for all $\sigma\in\Gamma$.  Note that, by Remark \ref{rem:examples of gradation}(1), the second part of the following theorem contains a characterization of ideals of $R[\mathcal{C}]$.

\begin{theorem}
\label{teo: bifuntores e ideais bilaterais}
 Let $\mathcal{G}:=\mathcal{C}_0\times\mathcal{C}_0$. 
    The following functors define an equivalence of categories
        \begin{align*}
            R[\mathcal{C}]\RgrR R[\mathcal{C}]&\leftrightarrows\Bifun(\mathcal{C},\mathcal{A}b)\\
            M&\mapsto M(-,-)\\
             M[F]&\mapsfrom F;
        \end{align*}
    Moreover, such equivalence induces a bijection
    \[\{\textrm{graded ideals of $R[\mathcal{C}]$}\}\longrightarrow\{\textrm{ideals of $\mathcal{C}$}\}\]
    that sends $R[\mathcal{C}]$ to $\mathcal{C}(-,-)$.
\end{theorem}

\begin{proof}
    The verification that $M\mapsto M(-,-)$ and $F\mapsto M[F]$ define  an equivalence is similar to the proof of Theorem \ref{teo: funtores e modulos}.
    
    For the second part of the statement, it suffices to note that if $M$ is a graded ideal of $R[\mathcal{C}]$, then $M(-,-)$ is a subfunctor of $\mathcal{C}(-,-)$ and, conversely, if $F$ is an ideal of $\mathcal{C}$, then $M[F]$ is  naturally gr-isomorphic to a graded ideal of $R[\mathcal{C}]$. 
\end{proof}

\subsection{Semisimple categories}

Recall that a nonzero object $A$ in an abelian category $\mathcal{A}$ is \emph{simple} if it has no proper, nonzero subobjects, and $A$ is \emph{semisimple} if it is a coproduct of simple objects.

Following \cite[p. 18]{Mit},  we will say that $\mathcal{C}$ is a \emph{(right) semisimple} category if $\mathcal{C}(-,A)$ is a semisimple object in the category $\Fun(\mathcal{C}^{op},\mathcal{A}b)$ for all $A\in\mathcal{C}_0$.

The next proposition follows immediately from Theorem \ref{teo: funtores e modulos}.

\begin{proposition}
\label{prop: (gr-)semisimplicidade em categorias}
  Let $\mathcal{G}:=\mathcal{C}_0\times\mathcal{C}_0$ and $A\in\mathcal{C}_0$. Then $\mathcal{C}(-,A)$ is a semisimple object in the category $\Fun(\mathcal{C}^{op},\mathcal{A}b)$ if and only if $R[\mathcal{C}](\varepsilon_A)$ is a $\mathcal{G}$-graded gr-semisimple $R[\mathcal{C}]$-module.\qed
\end{proposition}

For the next result, we need some definitions. 

If $R$ is a ring, not necessarily unital, then $\fgmodR R$ will denote the category of finitely generated unital right $R$-modules. 

Let $\{\mathcal{C}_j:j\in J\}$ be a family of preadditive categories with at least a zero object. Following \cite[p. 133]{Facchini}, let 
$\prod_{j\in J}^f\mathcal{C}_j$ be the full subcategory of $\prod_{j\in J}\mathcal{C}_j$ whose objects are of the form $(A_j)_{j\in J}$, where $A_j$ is a zero object of $\mathcal{C}_j$ for every $j\in J$ except for finitely many indices $j$.


We have the following characterization of semisimple categories.  The equivalence of items (1) and (8) was first proved in \cite[pp. 19-20]{Mit}.

\begin{theorem}
\label{teo: WA para categorias}
    Let $\mathcal{G}:=\mathcal{C}_0\times\mathcal{C}_0$. The following assertions are equivalent:
    \begin{enumerate}[\rm (1)]
        \item $\mathcal{C}$ is a semisimple category.
        \item $R[\mathcal{C}]$ is a gr-semisimple ring.
        \item Every object of the category $\Fungr(\mathcal{C}^{op},\mathcal{G}\varepsilon-\mathbb{Z}^{(\mathcal{C}_0)}\Rmod)$ is semisimple.
        \item Every object of the category $\Fun(\mathcal{C}^{op},\mathcal{A}b)$ is semisimple.
        \item Every object of the category $\Fungr(\mathcal{C},\varepsilon\mathcal{G}-\modR\mathbb{Z}^{(\mathcal{C}_0)})$ is semisimple.
        \item Every object of the category $\Fun(\mathcal{C},\mathcal{A}b)$ is semisimple.
        \item For all $A\in\mathcal{C}_0$, $\mathcal{C}(A,-)$ is a semisimple object in $\Fun(\mathcal{C},\mathcal{A}b)$.
        \item There exist a family $\{D_j:j\in J\}$ of division rings and a function $n:J\times \mathcal{C}_0\longrightarrow \mathbb{Z}_{\geq0}$ such that, for each $A,B\in \mathcal{C}_0$, $\sum\limits_{j\in J}n(j,A)<\infty$ and there exists an isomorphism of additive groups  
        \[\varphi_{(A,B)}:\mathcal{C}(A,B)\longrightarrow\prod_{j\in J}\M_{n(j,B)\times n(j,A)}(D_j)\]
        so that all these isomorphisms are compatible with products, i.e., $\varphi_{(A,B)}(fg)=\varphi_{(C,B)}(f)\cdot\varphi_{(A,C)}(g)$, for each $f\in\mathcal{C}(C,B)$, $g\in\mathcal{C}(A,C)$.
        \item There exists a family $\{D_j:j\in J\}$ of division rings such that $\mathcal{C}$ is isomorphic to a small full subcategory of $\fgmodR \bigoplus\limits_{j\in J}D_j$.
        \item There exists a family $\{D_j:j\in J\}$ of division rings such that $\mathcal{C}$ is isomorphic to a small full subcategory of $\prod_{j\in J}^f\fgmodR D_j$.
    \end{enumerate}
\end{theorem}

\begin{proof}
    $(1)\iff(2):$ By Proposition \ref{prop: (gr-)semisimplicidade em categorias}, $\mathcal{C}$ is a semisimple category if and only if $R[\mathcal{C}](\varepsilon_A)$ is a gr-semisimple $\mathcal{G}$-graded $R[\mathcal{C}]$-module for all $A\in\mathcal{C}_0$. By Lemma \ref{lem: R ss -> R(e) é soma finita}(2), this last condition is equivalent to the gr-semisimplicity of $R[\mathcal{C}]$ as a $\mathcal{G}$-graded ring.

    $(2)\implies(3)$: If $R[\mathcal{C}]$ is a gr-semisimple $\mathcal{G}$-graded ring, then, by Proposition \ref{prop: CLP, Prop 59}, all objects of $\mathcal{G}-\grR R[\mathcal{C}]$ are gr-semisimple. By Theorem \ref{teo: funtores e modulos}, we obtain (3).

    $(3)\implies (4)$: It follows from Theorem \ref{teo: funtores e modulos}.

    $(4)\implies (1)$: It is clear.

    The equivalence between (7), (2), (5) and (6) follows by the same argument as above, using Remark \ref{rem: funtores e modulos a esq} and the fact that the concepts of right gr-semisimplicity and left gr-semisimplicity coincide for groupoid graded rings. 

    $(2)\implies(8):$ Suppose that $R[\mathcal{C}]$ is a gr-semisimple $\mathcal{G}$-graded ring. By Theorem~\ref{teo: wa para ss -resumo}, there exist a set $J$, a family $\{K_j\colon j\in J\}$ of non-empty sets, a sequence $(A_j)_{j\in J}\in \mathcal{C}_0^J$, a summable family of $\mathcal{G}$-graded rings $\{\M_{K_j}(D_j)(\overline{\varsigma}_j)\colon j\in J\}$ and an gr-isomorphism of $\mathcal{G}$-graded rings
    \[\varphi\colon R[\mathcal{C}]\longrightarrow\sideset{}{^{gr}}\prod_{j\in J}\M_{K_j}(D_j)(\overline{\varsigma}_j)\]
    where $D_j$ is a $\mathcal{G}$-graded division ring with $\supp(D_j)\subseteq \varepsilon_{A_j}\mathcal{G} \varepsilon_{A_j}$ and 
    $\overline{\varsigma}_j=(\varsigma_{jk})_{k\in K_j}\in (\varepsilon_{A_j}\mathcal{G})^{K_j}$ is a $d$-finite sequence for each $j\in J$.

    Fix $j\in J$. 
    For each $k\in K_j$, let $B_{jk}\in\mathcal{C}_0$ such that $\varsigma_{jk}=(A_j,B_{jk})$. For each $A\in\mathcal{C}_0$, consider the sets $K_{j,A}:=\{k\in K_j:B_{jk}=A\}$ and $J_A:=\{j\in J:K_{j,A}\neq\emptyset\}$. 
    Note that 
    $K_{j,A}=\{k\in K_j:d(\varsigma_{jk})=\varepsilon_A\}$. It follows that $K_{j,A}$ and $J_A$ are finite for all $j\in J$ and $A\in\mathcal{C}_0$.  
    
    For each $j\in J$, $\supp(D_j)\subseteq \varepsilon_{A_j}\mathcal{G}\varepsilon_{A_j}$ implies $\supp(D_j)=\{\varepsilon_{A_j}\}$ and therefore $D_j$ is a division ring. We also have that the function
    \begin{align*}
        n:J\times \mathcal{C}_0&\longrightarrow\mathbb{Z}_{\geq0}\\
        (j,A)&\longmapsto |K_{j,A}|
    \end{align*}
    is such that $\{j\in J:n(j,A)\neq0\}=J_A$ is finite for all $A\in \mathcal{C}_0$. 
    In order to obtain the isomorphism $\varphi_{(A,B)}$ it suffices to compose $\varphi|_{\mathcal{C}{(A,B)}}$ with the isomorphism given by Lemma \ref{lem: M_n(D)_gamma} noting that, for each $j\in J$ and $A,B\in\mathcal{C}_0$, we have
    \[\M_{|K_{j,B}|\times|K_{j,A}|}(D_j)(\overline{\varsigma}_{j,B})(\overline{\varsigma}_{j,A})_{(B,A)}=\M_{n(j,B)\times n(j,A)}(D_j).\]
    
    $(8)\implies(2):$ Suppose that (8) holds.    
    Fix $j\in J$. For each $A\in \mathcal{C}_0$, define the following finite subset of $J\times \mathcal{C}_0\times\mathbb{N}$:
    \[K_{j,A}:=\begin{cases}\{(j,A,p):1\leq p\leq n(j,A)\}, & \textrm{if $n(j,A)\neq0$}\\
    \emptyset, & \textrm{if $n(j,A)=0$}\end{cases}.\]
    Consider the disjoint union
    \[K_j:=\bigcup_{A\in \mathcal{C}_0}K_{j,A}\]
    and, for each $k\in K_j$, define $B_{jk}=A$ if $k\in K_{j,A}$. Take $A_j\in \mathcal{C}_0$ and consider $D_j$ as a $\mathcal{G}$-graded division ring via $(D_j)_{(A_j,A_j)}:=D_j$. For each $k\in K_j$ and $A\in\mathcal{C}_0$, let $\varsigma_{jk}:=(A_j,B_{jk})$,  $\overline{\varsigma}_j:=(\varsigma_{jk})_{k\in K_j}$ and  $\overline{\varsigma}_{j,A}:=(\varsigma_{jk})_{k\in K_{j,A}}$. Note that $\overline{\varsigma}_j$ is $d$-finite. Therefore we can consider the $\mathcal{G}$-graded ring $\M_{K_j}(D_j)(\overline{\varsigma}_j)$.

    Since $\sum_{j\in J}n(j,A)<\infty$ for all $A\in\mathcal{C}_0$, the family $\{\M_{K_j}(D_j)(\overline{\varsigma}_j)\colon j\in J\}$ is summable. Therefore the $\mathcal{G}$-graded ring
    \[\sideset{}{^{gr}}\prod_{j\in J}\M_{K_j}(D_j)(\overline{\varsigma}_j)\]
    can be considered and it is gr-semisimple by Theorem~\ref{teo: wa para ss -resumo}.
        

    For each $j\in J$ and $A,B\in\mathcal{C}_0$, Lemma \ref{lem: M_n(D)_gamma} gives us an isomorphism of additive groups
    \begin{eqnarray*}
        \M_{K_j}(D_j)(\overline{\varsigma}_j)_{(B,A)} & \longrightarrow & \M_{|K_{j,B}|\times|K_{j,A}|}(D_j)(\overline{\varsigma}_{j,B})(\overline{\varsigma}_{j,A})_{(B,A)}\\ & & = \M_{n(j,B)\times n(j,A)}(D_j).  
    \end{eqnarray*}

    Now the compatibility of the isomorphisms  $\varphi_{(A,B)}$ induces the gr-isomorphism $R[\mathcal{C}]\rightarrow \prod_{j\in J}^{gr}\M_{K_j}(D_j)(\overline{\varsigma}_j)$.    

    $(8)\implies (9)$: It suffices to consider the functor $F:\mathcal{C}\to\fgmodR \bigoplus\limits_{j\in J}D_j$ defined on objects by  $F(A)=\bigoplus_{j\in J}D_j^{n(j,A)}$ and defined on morphisms using the maps $\varphi_{(A,B)}$.

    $(9)\implies (8)$: Let $M$ and $N$ be objects of $\fgmodR \bigoplus\limits_{j\in J}D_j$. It is easy to see that there exist sequences
    $(n(j,M))_{j\in J}$ and $(n(j,N))_{j\in J}$  of non-negative integers with $n(j,M)=0$ and $n(j,N)=0$ for almost all $j$ such that $M\cong \bigoplus_{j\in J}D_j^{n(j,M)}$ and $N\cong \bigoplus_{j\in J}D_j^{n(j,N)}$. Then 
    \[\Hom(M,N)\cong\bigoplus_{j\in J}\Hom_{D_j}(D_j^{n(j,M)},D_j^{n(j,N)})\cong\bigoplus_{j\in J}\M_{n(j,N)\times n(j,M)}(D_j)\]
    and we obtain (8).

    $(8)\implies (10)$: It suffices to consider the functor $F:\mathcal{C}\to\prod_{j\in J}^f\fgmodR D_j$ defined on objects by $F(A)=(D_j^{n(j,A)})_{j\in J}$, and on morphisms using the maps $\varphi_{(A,B)}$.

    $(10)\implies (8)$: Let $(M_j)_{j\in J}$ and $(N_j)_{j\in J}$ be objects of $\prod_{j\in J}^f\fgmodR D_j$. For each $j\in J$, take $n(j,M),n(j,N)\in\mathbb{Z}_{\geq0}$ such that $M_j\cong D_j^{n(j,M)}$ and $N_j\cong D_j^{n(j,N)}$. Then, the morphisms from $(M_j)_{j\in J}$ to $(N_j)_{j\in J}$ in $\prod_{j\in J}^f\fgmodR D_j$ are
    \[\prod_{j\in J}\Hom_{D_j}(M_j,N_j)\cong\prod_{j\in J}\Hom_{D_j}(D_j^{n(j,M)},D_j^{n(j,N)})\cong\prod_{j\in J}\M_{n(j,N)\times n(j,M)}(D_j)\]
    and (8) follows.
\end{proof}


\begin{remark}
    Let $R$ be a unital ring and let $\{S_j : j \in J\}$ be a complete set of representatives of the isomorphism classes of simple right $R$-modules. Set $D_j := \Hom_R(S_j, S_j)$ for all $j\in J$. 
    
    Suppose that $\mathcal{C}$ is a small full subcategory of $\fgmodR R$ whose objects are semisimple modules.     
    Given objects $M$ and $N$  of $\mathcal{C}$, there exist  sequences $(n(j,M))_{j \in J}$, $(n(j,N))_{j \in J} \in \mathbb{N}^J$ such that 
    \[M \cong \bigoplus_{j \in J} S_j^{n(j,M)}, \quad  \quad N \cong \bigoplus_{j \in J} S_j^{n(j,N)},\]
     $\sum\limits_{j\in J}n(j,M)<\infty$ and $\sum\limits_{j\in J}n(j,N)<\infty$.
    Then $\Hom_R(M, N)$ is naturally identified with $$\prod_{j\in J}\M_{n(j,N)\times n(j,M)}(D_j)$$  as in Theorem~\ref{teo: WA para categorias}(8).
    Hence, $\mathcal{C}$ is a semisimple category. \qed
\end{remark}

\subsection{Simple artinian categories}
\label{subsec: cat art simp}

Following \cite[p. 18-19]{Mit}, we will say that $\mathcal{C}$ is a \emph{right artinian category} if  $\mathcal{C}(-,A)$ is an artinian object  of the category $\Fun(\mathcal{C}^{op},\mathcal{A}b)$ for each $A\in\mathcal{C}_0$. That is,
the subobjects of $\mathcal{C}(-,A)$ satisfy the descending chain condition.
We will say that $\mathcal{C}$ is a \emph{simple category} if the only nonzero ideal of $\mathcal{C}$ is $\mathcal{C}(-,-)$.

The following result follows from Theorems~\ref{teo: funtores e modulos}(2) and \ref{teo: bifuntores e ideais bilaterais}.

\begin{proposition}
\label{prop: art simp em categorias}
    Let $\mathcal{G}:=\mathcal{C}_0\times\mathcal{C}_0$. The following statements hold true.
    \begin{enumerate}[\rm (1)]
        \item For each $A\in\mathcal{C}_0$, $\mathcal{C}(-,A)$ is an artinian object of $\Fun(\mathcal{C}^{op},\mathcal{A}b)$ if and only if $R[\mathcal{C}](\varepsilon_A)$ is a gr-artinian $\mathcal{G}$-graded $R[\mathcal{C}]$-module.
        \item $\mathcal{C}$ is a right artinian category if and only if $R[\mathcal{C}]$ is a right $\mathcal{G}_0$-artinian ring.
        \item $\mathcal{C}$ is a simple category if and only if $R[\mathcal{C}]$ is a gr-simple $\mathcal{G}$-graded ring.\qed
    \end{enumerate}
\end{proposition}

Items (2) and (3) of the previous result lead us to define that $\mathcal{C}$ is a \emph{simple artinian category} if $\mathcal{C}$ is right artinian and simple.


It is easy to see that $R[\mathcal{C}]$ is a gr-simple ring if and only if $\mathcal{C}$ is a simple category in the sense of \cite[Section 7.2]{Facchini}. Therefore, it follows from Proposition \ref{prop: art simp em categorias}(3) that  our concept of simple category coincides with that of \cite{Facchini}. In \cite[Section 7.2]{Facchini}, it was proved that if $R$ is a simple (unital) ring, then the category $\projR R$ of finitely generated projective right $R$-modules is a simple category \cite[Proposition 7.6]{Facchini} and that a preadditive category $\mathcal{P}$ is a simple category if and only if there exists a simple ring $R$ such that $\mathcal{P}$ is equivalent to a  full subcategory of $\projR R$ \cite[Theorem 7.5]{Facchini}. 
The following result characterizes the simple artinian categories.


\begin{theorem}
\label{teo: carac categorias art simp}
    Let $\mathcal{G}:=\mathcal{C}_0\times\mathcal{C}_0$. The following assertions are equivalent.
    \begin{enumerate}[\rm (1)]
        \item $\mathcal{C}$ is a simple artinian category.
        \item $\mathcal{C}$ is a simple and semisimple category.
        \item $R[\mathcal{C}]$ is a gr-simple $\mathcal{G}_0$-artinian ring.
        \item There exist a division ring $D$ and a function $n:\mathcal{C}_0\longrightarrow \mathbb{Z}_{\geq0}$ such that, for each $A,B\in \mathcal{C}_0$, there exists an isomorphism of additive groups  
        \[\varphi_{(A,B)}:\mathcal{C}(A,B)\longrightarrow \M_{n(B)\times n(A)}(D)\]
        and all these gr-isomorphisms are compatible with products, i.e., $\varphi_{(A,B)}(fg)=\varphi_{(C,B)}(f)\cdot\varphi_{(A,C)}(g)$, for each $f\in\mathcal{C}(C,B)$, $g\in\mathcal{C}(A,C)$.
        \item There exists a division ring $D$ such that $\mathcal{C}$ is isomorphic to a small full subcategory of $\fgmodR D$.
        \item There exists a simple artinian ring $R$ such that $\mathcal{C}$ is isomorphic to a small full subcategory of $\fgmodR R$.
    \end{enumerate}
\end{theorem}

\begin{proof}
    $(1)\iff(3)$: It follows from Proposition \ref{prop: art simp em categorias}.

    $(2)\iff(3)$: It follows from Proposition \ref{prop: art simp em categorias} and Theorem \ref{teo: simp + art = semisimp}.

    $(3)\iff(4)$: The equivalence is obtained by repeating the proof of $(2)\iff (8)$ in Theorem \ref{teo: WA para categorias} for $|J|=1$.

    $(4)\implies (5)$:  It suffices to consider the functor $F:\mathcal{C}\to\fgmodR D$ defined on objects by $F(A)=D^{n(A)}$, and defined on morphisms using the maps $\varphi_{(A,B)}$.

    $(5)\implies (6)$: It is clear.

    $(6)\implies (4)$: Suppose $\mathcal{C}$ be a small full subcategory of $\fgmodR R$, where $R$ is a simple artinian ring. Fix  a simple $R$-module $S$ and for each $M\in\mathcal{C}_0$ an isomorphism $M\cong S^{(n_M)}$ for a unique $n_M\in \mathbb{Z}_{\geq0}$. For any $M,N\in\mathcal{C}_0$, 
    \[\Hom_R(M,N)\cong\Hom_R(S^{(n_M)},S^{(n_N)})\cong\M_{n_N\times n_M}(D),\]
    where $D:=\End(S_R)$.
\end{proof}

\subsection{Free functors}

In this subsection,  we explore when $R[\mathcal{C}]$, the ring of the small preadditive category $\mathcal{C}$, is a pfm ring and relate this fact with the concept of free functors defined in \cite[p. 17-18]{Mit}. We proceed to recall the definition of that notion. 
Let $F:\mathcal{C}\to\mathcal{A}b$ be an additive contravariant functor.
We say that $(x_i)_{i\in I}\in\prod_{i\in I}F(A_i)$, where $(A_i)_{i\in I}\in\mathcal{C}_0^I$, is a sequence of \emph{generators of $F$} if, for all $A\in\mathcal{C}_0$ and $x\in F(A)$, there exists $(\lambda_i)_{i\in I}\in \bigoplus_{i\in I}\mathcal{C}(A,A_i)$ such that 
\[x=\sum_{i\in I}(F(\lambda_i))(x_i).\]
If, for all $A\in\mathcal{C}_0$ and $x\in F(A)$, the sequence $(\lambda_i)_{i\in I}$ is unique, then we say that $(x_i)_{i\in I}$ is a \emph{basis} of $F$ and that $F$ is \emph{free}.

The next result associates free functors with pseudo-free $R[\mathcal{C}]$-modules.
The equivalence of (1) and (3) was given in \cite[p.~18]{Mit}.

\begin{proposition}
\label{prop: funtores liv e modulos pseudoliv}
    Let $F:\mathcal{C}\to\mathcal{A}b$ be an additive contravariant functor and $A\in\mathcal{C}_0$. Consider the $\mathcal{C}_0\times\mathcal{C}_0$-graded right $R[\mathcal{C}]$-module $M[F]$ with $M[F]_{(A,B)}=F(B)$ for each $B\in\mathcal{C}_0$. The following statements are equivalent:
    \begin{enumerate}[\rm (1)]
        \item $F$ is a free functor with basis $(x_i)_{i\in I}\in\prod\limits_{i\in I}F(A_i)$.
        \item $M[F]$ is a pseudo-free right $R[\mathcal{C}]$-module with pseudo-basis $(x_i)_{i\in I}\in\prod\limits_{i\in I}M[F]_{(A,A_i)}$.
        \item The natural transformation
        \begin{align*}
            \bigoplus_{i\in I}\mathcal{C}(-,A_i)&\longrightarrow F\\
            I_{A_i}&\longmapsto x_i
        \end{align*}
        is an isomorphism in $\Fun(\mathcal{C}^{op},\mathcal{A}b)$.
    \end{enumerate}
\end{proposition}

\begin{proof}
    $(1)\iff(2)$: It is suffices to note that $x=\sum\limits_{i\in I}F(\lambda_i)(x_i)$ for some $B\in\mathcal{C}_0$, $x\in F(B)$ and $(\lambda_i)_{i\in I}\in\bigoplus_{i\in I}\mathcal{C}(B,A_i)$ is equivalent, by Theorem \ref{teo: funtores e modulos}(1), to $x=\sum\limits_{i\in I}x_i\cdot\lambda_i$ for some $B\in\mathcal{C}_0$, $x\in M[F]_{(A,B)}$ and $(\lambda_i)_{i\in I}\in\bigoplus\limits_{i\in I}R[\mathcal{C}]_{(A_i,B)}$.

    $(2)\iff(3)$: It follows from Proposition \ref{prop:pseudo_free}(7) and Theorem \ref{teo: funtores e modulos}(1).
\end{proof}

We have the following characterizations of small preadditive categories where all functors are free.

\begin{theorem}
\label{teo: todo funtor eh livre}
Let $\mathcal{G}:=\mathcal{C}_0\times\mathcal{C}_0$. The following statements are equivalent:
\begin{enumerate}[\rm (1)]
    \item All additive contravariant functors $\mathcal{C}\to\mathcal{A}b$ are free.
    \item All additive covariant functors $\mathcal{C}\to\mathcal{A}b$ are free.
    \item $R[\mathcal{C}]$ is a pfm $\mathcal{G}$-graded ring.
    \item There exist a set $J$, a sequence $(A_j)_{j\in J}\in\mathcal{C}_0^J$, a family $\{D_j:j\in J\}$ of division rings, a map $n\colon J\times \mathcal{C}_0\rightarrow \mathbb{Z}_{\geq0}$ with $\sum_{j\in J}n(j,A)<\infty$ for all $A\in\mathcal{C}_0$, $n(j,A_j)=1$ and $\mathcal{C}(A_j,A_j)\cong D_j$ for all $j\in J$, and there exist isomorphisms of additive groups
    \[\varphi_{(A,B)}:\mathcal{C}(A,B)\longrightarrow\prod_{j\in J}\M_{n(j,B)\times n(j,A)}(D_j)\]
    for each $A,B\in\mathcal{C}_0$. Moreover, all such isomorphisms are compatible with products, that is, $\varphi_{(A,B)}(fg)=\varphi_{(C,B)}(f)\cdot\varphi_{(A,C)}(g)$, for each $f\in\mathcal{C}(C,B)$, $g\in\mathcal{C}(A,C)$. 
    \item There exists a family $\mathcal{F}:=\{D_j:j\in J\}$ of division rings such that $\mathcal{C}$ is isomorphic to a small full subcategory of $\fgmodR \bigoplus\limits_{j\in J}D_j$ whose set of objects contains $\mathcal{F}$.
    \item There exists a family $\mathcal{F}:=\{D_j:j\in J\}$ of division rings such that $\mathcal{C}$ is isomorphic to a small full subcategory of $\prod_{j\in J}^f\fgmodR D_j$ whose set of objects contains  $\mathcal{F}$.
\end{enumerate}
\end{theorem}

\begin{proof}
    $(1)\implies(3)$: By Proposition \ref{prop: funtores liv e modulos pseudoliv} and Theorem \ref{teo: funtores e modulos}(1), we have that (1) implies that all $\mathcal{G}$-graded right $R[\mathcal{C}]$-modules $M$ of the form $M(\varepsilon_A)$, $A\in\mathcal{C}_0$, are pseudo-free. Now, (3) follows from Corollary  \ref{coro: M livre sse todo M(e) livre}.

    $(3)\implies(1)$: It follows immediately from Proposition \ref{prop: funtores liv e modulos pseudoliv}.

    $(2)\iff (3)$: It follows in the same way as $(1)\iff(3)$, using Remark \ref{rem: funtores e modulos a esq}.

    $(3)\implies (4)$ Proceed as in the proof of $(2)\implies(8)$ in Theorem~\ref{teo: WA para categorias} and note that there exists a sequence $(A_j)_{j\in J}\in \mathcal{C}_0^J$ such that $n(j,A_j)=|K_{j,A_j}|=1$ and $n(j',A_j)=|K_{j',A_j}|=0$ for all distinct $j,j'\in J$ by Theorem~\ref{teo: pseudo aneis com div}(2).
	
    $(4)\implies (3)$ Proceed as in the proof of $(8)\implies(2)$ in Theorem~\ref{teo: WA para categorias} and note that there exists a sequence $(A_j)_{j\in J}\in \mathcal{C}_0^J$ such that $|K_{j,A_j}|=n(j,A_j)=1$ and $|K_{j',A_j}|=n(j',A_j)=0$ for all distinct $j,j'\in J$.  By Theorem~\ref{teo: pseudo aneis com div}(2), this implies that $R[\mathcal{C}]$ is a pfm $\mathcal{G}$-graded ring.

    $(4)\implies (5)$: It suffices to consider the functor $F:\mathcal{C}\to\fgmodR \bigoplus\limits_{j\in J}D_j$ defined by $F(A)=\bigoplus_{j\in J}D_j^{n(j,A)}$ and to note that $F(A_j)=D_j$ for each $j\in J$.

    $(5)\implies (4)$: It follows the same idea of the proof of $(9)\implies(8)$ in Theorem~\ref{teo: WA para categorias}, noting that if $\mathcal{C}$ is a full subcategory of $\fgmodR \bigoplus\limits_{j\in J}D_j$ and $D_j\in\mathcal{C}_0$, then $\mathcal{C}(D_j,D_j)\cong D_j$ is a division ring.

    $(4)\implies (6)$: It suffices to consider the functor $F:\mathcal{C}\to\prod_{j\in J}^f\fgmodR D_j$ defined by $F(A)=(D_j^{n(j,A)})_{j\in J}$ and to note that, for each $j\in J$, $F(A_j)$ is the copy of $D_j$ in $\prod_{j\in J}^f\fgmodR D_j$.

    $(6)\implies (4)$: It is similar to $(5)\implies (4)$.
\end{proof}

 Let $R$ be a unital ring and let $\mathcal{C}$ be a small full subcategory of $\fgmodR R$.
Suppose that there exists a subset $\mathcal{S}\subseteq\mathcal{C}_0$ of simple right $R$-modules such that all the objects of $\mathcal{C}$ are isomorphic to finite direct sums of elements of $\mathcal{S}$.
 Then $R[\mathcal{C}]$, the ring of the category $\mathcal{C}$, satisfies Theorem~\ref{teo: todo funtor eh livre}(4). Thus, $R[\mathcal{C}]$ is a pfm ring. Therefore, suppose that $\mathcal{B}$ is a small preadditive category  such that $R[\mathcal{B}]$ is not pfm. Then $\mathcal{B}$ cannot be equivalent to the category of finitely generated semisimple modules of a ring.

The category of all finitely generated semisimple modules of a ring is an example of an amenable semisimple category  \cite[p.~20]{Mit}.  A characterization of such categories is given in \cite[Theorem 4.55]{Facchini}. We observe that small amenable semisimple categories satisfy the conditions of Theorem~\ref{teo: todo funtor eh livre}.

For small preadditive simple categories, we have the following result, which  follows from Proposition \ref{prop: art simp em categorias}(3) and Theorems \ref{teo: todo funtor eh livre} and \ref{teo: pseudo aneis com div primos}.

\begin{theorem}
\label{teo: cat pfm simples}
Suppose that $\mathcal{C}$ is a simple category and $\mathcal{G}:=\mathcal{C}_0\times\mathcal{C}_0$. The following statements are equivalent:
\begin{enumerate}[\rm (1)]
    \item Every additive contravariant functor $\mathcal{C}\to\mathcal{A}b$ is free.
    \item There exists $A\in\mathcal{C}_0$ such that $\mathcal{C}(-,A)$ is a simple object of $\Fun(\mathcal{C}^{op},\mathcal{A}b)$.
    \item $\mathcal{C}$ is a right artinian category and there exists $A\in\mathcal{C}_0$ such that $\mathcal{C}(A,A)$ is a division ring.
    \item There exist $A\in\mathcal{C}_0$, a division ring $D$, a map $n\colon \mathcal{C}_0\rightarrow \mathbb{Z}_{\geq0}$ with $n(A)=1$ and there exist isomorphisms of additive groups
    \[\varphi_{(A,B)}:\mathcal{C}(A,B)\longrightarrow\M_{n(B)\times n(A)}(D)\]
    for each $A,B\in\mathcal{C}_0$. Moreover, all such isomorphisms are compatible with products, that is, $\varphi_{(A,B)}(fg)=\varphi_{(C,B)}(f)\cdot\varphi_{(A,C)}(g)$, for each $f\in\mathcal{C}(C,B)$, $g\in\mathcal{C}(A,C)$. 
    \item There exists a division ring $D$ such that $\mathcal{C}$ is isomorphic to a small full subcategory $\mathcal{D}$ of $\fgmodR D$ with $D_D\in\mathcal{D}_0$.\qed
\end{enumerate}
\end{theorem}

\subsection{Division categories}
Suppose that a preadditive cateogry $\mathcal{D}$ has a nonzero object. We say that $\mathcal{D}$ is a \emph{division category} if  each nonzero morphism in $\mathcal{D}$ is an isomorphism. 
By Schur's Lemma, an example of a division category is the category of simple modules over a ring.  We characterize small division categories in the next result.

\begin{proposition}
\label{prop: categorias com div}
Let $\mathcal{G}:=\mathcal{C}_0\times\mathcal{C}_0$. The following statements are equivalent:
\begin{enumerate}[\rm(1)]
    \item $\mathcal{C}$ is a division category.
    \item $R[\mathcal{C}]$ is a $\mathcal{G}$-graded division ring.
    \item Every functor $F\in\Fun(\mathcal{C}^{op},\mathcal{A}b)$ is free and all basis of $F$ have the same cardinality.
    \item $\mathcal{C}$ is a semisimple category and $\mathcal{C}(A,A)$ is a division ring for each nonzero object $A\in \mathcal{C}_0$.
    \item There exists a family $\{\mathcal{C}_j:j\in J\}$ of simple division categories such that $\mathcal{C}=\coprod_{j\in J}\mathcal{C}_j$.
\end{enumerate}
\end{proposition}

\begin{proof}
$(1) \iff (2)$: it is clear.

$(2) \iff (3)$: By Theorem~\ref{theo:pfm_gr_division_ring}, (2) holds if and only if $R[\mathcal{C}]$ is a pfm $\mathcal{G}$-graded ring and, for each $\mathcal{G}$-graded $R[\mathcal{C}]$-module $M$ all pseudo-basis of $M$ have the same cardinality. By Theorem~\ref{teo: todo funtor eh livre} and Proposition~\ref{prop: funtores liv e modulos pseudoliv}, this is equivalent to (3).

$(2) \iff (4)$: By Theorem~\ref{teo: carac aneis com div via WA}, (2) holds if and only if $R[\mathcal{C}]$ is a gr-semisimple $\mathcal{G}$-graded ring and $R[\mathcal{C}]_{(A,A)}$ is a division ring for all nonzero $A\in\mathcal{C}_0$. By Theorem~\ref{teo: WA para categorias}, this is equivalent to $\mathcal{C}$ being a semisimple category and $\mathcal{C}(A,A)$ being a division ring for all $A$ such that $\mathcal{C}(A,A)\neq 0$.

$(2)\implies (5)$: By Theorem \ref{teo: carac aneis com div via WA}, there exists a family $\{R_j:j\in J\}$ of gr-simple gr-division rings such that $R[\mathcal{C}]=\bigoplus_{j\in J} R_j$ and $\supp(R_j)\cap \supp(R_{j'})=\emptyset$ for all different $j,j'\in J$. It suffices take, for each $j\in J$, the full subcategory $\mathcal{C}_j$ of $\mathcal{C}$ whose set of objects is $\{A\in\mathcal{C}_0:\varepsilon_A\in\mathcal{G}'_0(R_j)\}$.

$(5)\implies (2)$: For each $j\in J$, consider $R_j:=R[\mathcal{C}_j]$ as a $\mathcal{G}$-graded ring. Then $R[\mathcal{C}]=\bigoplus_{j\in J}R_j$ is a decomposition as in item (6) of Theorem \ref{teo: carac aneis com div via WA}.
\end{proof}


\begin{corollary}
\label{coro: C grupoide --> todo funtor livre}
    If $\mathcal{C}$ is a division category, then every additive contravariant functor $\mathcal{C}\to\mathcal{A}b$ is free. \qed
\end{corollary}


We comment on the decomposition in item (5) of Proposition \ref{prop: categorias com div}. Let $\mathcal{C}$ be a division category and $\mathcal{C}'_0$ the set of nonzero objects of $\mathcal{C}$. Similarly to Proposition \ref{prop: quando D e'simples}, we can define in $\mathcal{C}'_0$ the equivalence relation $A\sim B\iff \mathcal{C}(A,B)\neq0$. Let $\{A_j:j\in J\}$ be a family of all representatives of this relation. For each $j\in J$, let $\mathcal{C}_j$ be the full subcategory of $\mathcal{C}$ whose set of objects is $\{A\in\mathcal{C}_0:\mathcal{C}(A_j,A)\neq0\}$. Each $\mathcal{C}_j$ is a simple division category and $\mathcal{C}=\coprod_{j\in J}\mathcal{C}_j$. In particular, $\mathcal{C}$ is simple if and only if $\mathcal{C}(A,B)\neq0$ for all $A,B\in\mathcal{C}'_0$.

\begin{proposition}
\label{prop: carac categorias com div simp}
    Let $\mathcal{G}:=\mathcal{C}_0\times\mathcal{C}_0$. The following assertions are equivalent:
    \begin{enumerate}[\rm (1)]
        \item $\mathcal{C}$ is a simple division category.
        \item $R[\mathcal{C}]$ is a gr-simple $\mathcal{G}$-graded division ring.
        \item There exists a division ring $D$ such that $\mathcal{C}$ is isomorphic to a small full subcategory of $\fgmodR D$ whose objects have dimension 0 or 1.
        \item There exists a ring with unity $R$ such that $\mathcal{C}$ is isomorphic to a small full subcategory of $\fgmodR R$ whose objects are null or simple modules.
    \end{enumerate}
\end{proposition}

\begin{proof}
    $(1)\iff (2)$: It follows from Proposition \ref{prop: categorias com div} and Proposition \ref{prop: art simp em categorias}(3).

    $(2)\implies(3)$: By Theorem \ref{teo: carac categorias art simp}, there exists a division ring $D$ such that $\mathcal{C}$ is isomorphic to a small full subcategory of $\fgmodR D$. If $M$ and $N$ are nonzero $D$-modules, then the invertibility of all nonzero elements of $\Hom_D(M,N)$ is equivalent to $M$ and $N$ have dimension 1. Thus, (3) follows.

    $(3)\implies(4)$: It is clear.

    $(4)\implies(1)$: It follows from Schur's Lemma.
\end{proof}

By Proposition~\ref{prop: categorias com div}, $R[\mathcal{C}]$ is a $\mathcal{G}$-graded division ring if and only if for each $A,B\in\mathcal{C}_0$, all nonzero morphism in $\mathcal{C}(A,B)$ is invertible. On the other hand, Theorem~\ref{teo: todo funtor eh livre} tells us that if $R[\mathcal{C}]$ is a pfm ring, then there exists a family $\{A_j:j\in J\}$ of objects in $\mathcal{C}$, a family $\{D_j:j\in J\}$ of division rings, a map $n\colon J\times \mathcal{C}_0\rightarrow \mathbb{Z}_{\geq0}$ with $\sum_{j\in J}n(j,A)<\infty$ for all $A\in\mathcal{C}_0$, $n(j,A_j)=1$ and $\mathcal{C}(A_j,A_j)\cong D_j$ for all $j\in J$, and
\[\mathcal{C}(A_j,B)\cong\prod_{j'\in J}\M_{n(j',B)\times n(j',A_j)}(D_{j'})=\M_{n(j,B)\times 1}(D_{j});\]
\[\mathcal{C}(B,A_j)\cong\prod_{j'\in J}\M_{n(j',A_j)\times n(j',B)}(D_{j'})=\M_{1\times n(j,B)}(D_{j}).\]
That is, in this case we just know that, for each $j\in J$, the nonzero morphisms with domain $A_j$ are left invertible and the nonzero morphisms with codomain $A_j$ are right invertible.

\section*{Acknowledgments}

We thank the referee for helpful suggestions and corrections that improved the exposition.


\end{document}